\pdfoutput=1
\documentclass[12pt, a4paper]{amsart}
\usepackage[left=3cm, right=3cm, top=2.5cm, bottom=2cm]{geometry}
\usepackage{amsmath,amssymb,amsthm,enumerate}
\usepackage{graphicx}
\usepackage{mathtools}
\usepackage{pstricks-add}

\usepackage{tikz-cd}
\usepackage{tikz}
\usetikzlibrary{arrows,shapes,positioning,calc,3d,decorations.pathreplacing,decorations.markings,cd,patterns,intersections}
\tikzstyle{circ}=[shape=circle,draw,color=black,inner sep=2pt]
\tikzstyle{dot}=[shape=circle,draw,color=black,fill=black,inner sep=1.5pt]
\tikzstyle{->-}=[postaction={decorate}, decoration={markings,mark=at position .6 with {\arrow[scale=3,color=black]{to}}}]
\tikzstyle{-<-}=[postaction={decorate},decoration={markings,mark=at position .4 with {\arrowreversed[scale=3,color=black]{to};}}]

\usepackage{mathrsfs} \usepackage{xspace}
\usepackage{pdflscape}

\usepackage[pagebackref=true]{hyperref}
\usepackage{longtable}

\title[Hyperbolic generalized triangle groups]{Hyperbolic generalized triangle groups,\\ 
property~(T) and finite simple quotients}

\author[P.-E. Caprace]{Pierre-Emmanuel Caprace}
\thanks{P.-E.C. is a F.R.S.-FNRS senior research associate}
\address{Universit\'e catholique de Louvain, IRMP, Chemin du Cyclotron 2, bte L7.01.02, 1348 Louvain-la-Neuve, Belgique}
\email{pe.caprace@uclouvain.be}

\author[M. Conder]{Marston Conder}
\thanks{}
\address{University of Auckland, New Zealand}
\email{m.conder@auckland.ac.nz}

\author[M. Kaluba]{Marek Kaluba}
\thanks{M.K. has been supported by the National Science Center, Poland grant 2017/26/D/ST1/00103.}
\address{Adam Mickiewicz University in Poznan, Poland}
\address{Technische Universität Berlin, Chair of Discrete Mathematics/Geometry}
\email{kalmar@amu.edu.pl}

\author[S. Witzel]{Stefan Witzel}
\thanks{S.W. is funded through the DFG Heisenberg project WI 4079-6.}
\address{Mathematisches Institut, Universität Gießen, Arndtstr. 2, 35392 Gießen, Germany}
\email{Stefan.Witzel@math.uni-giessen.de}

\date{December 22, 2020}

\newtheorem{thm}{Theorem}[section]

\newtheorem{prop}[thm]{Proposition}
\newtheorem{lem}[thm]{Lemma}
\newtheorem{cor}[thm]{Corollary}

\theoremstyle{definition}

\newtheorem{qu}[thm]{Question}
\newtheorem*{qu*}{Question}
\newtheorem{rmk}[thm]{Remark}
\newtheorem{example}[thm]{Example}

\newcommand{\Zbb}{\mathbf{Z}}
\newcommand{\Fbb}{\mathbf{F}}

\newcommand{\Nbb}{\mathbf{N}}

\newcommand{\Sym}{\mathrm{Sym}}
\newcommand{\PSL}{\mathrm{PSL}}
\newcommand{\Sp}{\mathrm{Sp}}
\newcommand{\Alt}{\mathrm{Alt}}
\newcommand{\Aut}{\mathrm{Aut}}
\newcommand{\SL}{\mathrm{SL}}
\newcommand{\GL}{\mathrm{GL}}

\DeclareMathOperator{\Ker}{Ker}

\newcommand{\grp}[1]{\langle #1 \rangle}

\newcommand{\abs}[1]{\lvert {#1} \rvert}
\newcommand \vprime {{\mkern.5mu\backprime\mkern-1.5mu\prime}}

\newcommand{\newcomment}[4]{\newcounter{#2counter}
\expandafter\newcommand\csname #1\endcsname[1]{\refstepcounter{#2counter}{\color{#4}(#3\arabic{#2counter})}\marginpar{\scriptsize\raggedright\textbf{\color{#4}(#2 \arabic{#2counter}):} ##1}}}

\definecolor{darkgreen}{rgb}{0,0.6,0}
\newcomment{sw}{Stefan}{S}{darkgreen}
\newcomment{pec}{Pierre-Emmanuel}{PE}{blue}
\newcomment{mk}{Marek}{M}{red}

\begin{document}

\begin{abstract}
We construct several series of explicit presentations of infinite hyperbolic groups enjoying Kazhdan's property (T). Some of them are significantly shorter than the previously known shortest examples. Moreover, we show that some of those  hyperbolic Kazhdan groups   possess finite simple quotient groups of arbitrarily large rank; they constitute the first known specimens combining those properties. 
All the hyperbolic groups we consider are non-positively curved $k$-fold generalized triangle groups, i.e. groups that possess a simplicial action on a CAT($0$) triangle complex, which is sharply transitive on the set of triangles, and such that edge-stabilizers are cyclic of order $k$. 
\end{abstract}

\maketitle

\setcounter{tocdepth}{1}
{\small 
\tableofcontents
}

\section{Introduction}

It is a long-standing open question, going back to a remark of Gromov in his seminal monograph \cite[Remark~5.3.B]{Gromov_hyp}, whether every hyperbolic group is residually finite.  This question is equivalent to determining whether every non-trivial hyperbolic group has a non-trivial finite quotient (see \cite[Theorem~1.2]{KaWi} or \cite[Theorem~2]{Olsh_2000}). Using
Olshanskii's Common Quotient Theorem, those two questions can further  be shown equivalent to the following, which illustrates that the problem of residual finiteness of hyperbolic groups is related to the asymptotic properties of the finite simple groups (see \cite[\S5.3]{Cap_StAnd} for an expository account):
\begin{qu}\label{qu:S_d}
Let $\mathscr{S}_d$ be the collection of those finite simple groups that contain an isomorphic copy of the alternating group $\Alt(d)$. Does every non-elementary hyperbolic group admit a quotient in $\mathscr{S}_d$ for all $d$?
\end{qu}

A group that admits a finite quotient belonging to $\mathscr{S}_d$ for all $d$ is said to have \textbf{finite simple quotients of arbitrarily large rank}. 

Kazhdan's property (T) is relevant when trying to answer these questions negatively, i.e.\ finding a hyperbolic group that is not residually finite or does not admit finite simple quotients of arbitrarily large rank. Indeed, the groundbreaking work of Agol, Haglund and Wise implies that all compactly cubulated hyperbolic groups are residually finite, see \cite{Agol} and references therein. On the other hand, property (T) is incompatible with cocompact cubulations for an infinite group\footnote{Another known obstruction to the existence of a cocompact cubulation is provided by the work of Delzant--Gromov \cite{DG} and Delzant--Py \cite{DP} on K\"ahler groups; in particular, cocompact lattices in $SU(n, 1)$, with $n \geq 2$, are hyperbolic groups that are neither Kazhdan, nor virtually special.}, see  \cite{NibloReeves}. Moreover, the finite-dimensional unitary representations of Kazhdan groups are subjected to various rigidity theorems, see \cite{Wang75} and \cite{Rapinchuk99}. It is thus tempting to believe that a hyperbolic group with Kazhdan's property (T) should have fewer finite simple quotients than other hyperbolic groups. 

This circle of ideas caused us to systematically investigate finite quotients and property (T) for certain small hyperbolic groups. The condition ``small'' here means ``having a short presentation'' and is imposed for practical reasons: Many of our investigations involved computer-aided experiments and calculations, and groups with short presentations are generally easier to work with. On a related note, using \textsc{Magma}, we can check the existence of finite simple quotients only up to a certain order ($5 \cdot 10^7$) and we expect that for a small group the quotients in this region may give a meaningful impression of general finite simple quotients while for a larger group they will likely be noise.

For theoretic considerations, the structure of a group presentation is of course more relevant than its length. For that reason, we have focused our study on the class of \textbf{$k$-fold generalized triangle groups}, all of which have a presentation whose structure is fairly transparent. Following Lubotzky--Manning--Wilton~\cite{LMW}, we define such a group as the fundamental group of a triangle of finite groups with trivial face group,  cyclic edge groups   of order~$k$, and finite vertex groups\footnote{They should not be confused with the generalized triangle groups occuring, for example, in \cite{HowieWilliams08}, where the terminology has a completely different meaning.}. 
A $2$-fold generalized triangle group is a Coxeter group, and no infinite Coxeter group has (T) (see \cite{BJS}). The aforementionned work~\cite{LMW} provides infinitely many examples of hyperbolic $k$-fold generalized triangle groups with (T), for $k \geq 18$. In this paper, we obtain examples with $k= 5$, as well as infinite families of examples with $k$ any prime~$\geq 7$. 

For $k=5$, our examples include the following. 

\begin{thm}\label{thm:5-fold_with_(T)}
Each of the groups 
\begin{align*}
\mathscr H_{31} = \langle a, b, c  \mid  & a^5, b^5, c^5, [a, c], [b, c, b], [b, c, c, b], [b, c, c, c], \\
&a b a^2 b a^2 b a b^{-1} a b^{-1},\\
& b^2 a b a^{-1}b a^{-1} bab^2a,\\
& (bab^{-1}aba^{-1})^2
\rangle,
\end{align*}
\begin{align*}
\mathscr H_{109} = \langle a, b, c  \mid  & a^5, b^5, c^5, [a, c], [b, c, b], [b, c, c], \\
&   a  b  a  b^{-1}  a^{-1}  b  a  b  a^{-1}  b^{-1}  a^{-1}  b  a 
b^{-1}  a^{-1}  b^{-1}, \\
& 
     b  a  b  a  b^2  a^{-1}  b  a^2  b^{-2}  a^{-1}  b  a^{-1} 
b^{-1}  a^2,\\
&     b  a^{-1}  b  a  b^{-1}  a  b^2  a^{-1}  b  a  b^{-1}  a  b  a^{-1}      b^{-1}  a^2, \\
&     b  a  b^{-1}  a  b  a^{-1}  b  a^{-2}  b^{-1}  a^{-1}  b  a^{-1}
b^{-1}  a     b^{-1}  a^2,\\
&     b  a^{-1}  b  a^{-1}  b^{-2}  a  b^{-1}  a^{-1}  b^{-1}  a^{-1}  b 
a^{-2}  b^{-2}     a^2 ,\\
&     a  b  a^{-2}  b^{-1}  a^{-1}  b^{-1}  a^{-1}  b^{-2}  a  b^{-1}  a^{-2} b^2  a      b^{-1}, \\
&       a^{-2}  b^{-1}  a^{-2}  b  a  b^{-1}  a  b^{-1}  a^2  b^{-1} a b  a^{-2}     b^2
\rangle,
\end{align*}
is an infinite hyperbolic $5$-fold generalized triangle group satisfying Kazhdan's property (T). 
\end{thm}

It is easy to see from the presentation that each group from Theorem~\ref{thm:5-fold_with_(T)} is indeed a $5$-fold generalized triangle group. 
For $\mathscr H_p$, the vertex groups $\langle a, b\rangle$ and $\langle c, a\rangle$ are respectively isomorphic to $\PSL_2(p)$ and $C_5 \times C_5$ where $p$ is the subscript, while the vertex group $\langle b, c\rangle$ is isomorphic to a $5$-Sylow subgroup of $\Sp_4(5)$ in the first case and to a $5$-Sylow subgroup of $\SL_3(5)$ in the second case. The girths of the associated links are $10$, $8$ and $4$ for the first group and are $14$, $6$ and $4$ for the second group, so that the natural metric space on which the groups act geometrically is a CAT($-1$) triangle complex whose simplices are hyperbolic triangles of type with angles $\pi/5, \pi/4, \pi/2$ respectively $\pi/7, \pi/3, \pi/2$. For these groups, the assistance of computer calculations was required to provide a certified estimate of the spectral gap of the link  associated to $\grp{a,b}$. The presentations of the groups $\mathscr H_{31}$ and $\mathscr H_{109}$ have $10$ and $13$ relations respectively. They can be simplified to presentations on the same numbers of generators and relators of total relator lengths $88$ and $160$, respectively.

Let now $p$ be an odd prime. We next consider $p$-fold generalized triangles groups,  each of whose vertex groups are  isomorphic to the $p$-Sylow subgroups in any of the finite groups $\mathrm{SL}_2(\mathbf F_p) \times \mathrm{SL}_2(\mathbf F_p)$, $\mathrm{SL}_3(\mathbf F_p)$ or $\mathrm{Sp}_4(\mathbf F_p)$. Adopting a terminology suggested by Ershov--Jaikin-Zapirain~\cite{EJ}, we call them \textbf{Kac--Moody--Steinberg groups}, or \textbf{KMS groups} for short. We consider~$10$ infinite families of such groups, each indexed by the prime $p$. Three of those  families appear in the following.

\begin{thm}\label{thm:KMS:intro}
For each odd prime $p$, the groups
\begin{align*}
\mathscr G_{HC_2^{(1)}}(p) = \langle a, b, c  \mid  & a^p, b^p, c^p, [a, b, a], [a, b, b], \\
& [b, c, b], [b, c, c],  [a, c, a], [a, c, c, a], [a, c, c, c]\rangle,
\end{align*}
\begin{align*}
\mathscr G_{HB_2^{(2)}}(p) = \langle a, b, c    \mid & a^p, b^p, c^p, [a, b, a], [a, b, b], \\
 & [c, b, c], [c, b, b, c], [c, b, b, b ], [c, a, c], [c, a, a, c], [c, a, a, a]\rangle,
\end{align*}
and
\begin{align*}
\widetilde{\mathscr G}_{HBC_2^{(3)}}(p) = \langle t, a, b    \mid & t^3, a^p, tat^{-1}b^{-1},  [a, b, a], [a, b, b, a], [a, b, b, b ]\rangle,
\end{align*}
are infinite hyperbolic. Moreover, for all primes $p \geq 7$ (resp. $p \geq 11$), the groups $\mathscr G_{HC_2^{(1)}}(p)$ and $\mathscr G_{HB_2^{(2)}}(p)$ (resp. $\widetilde{\mathscr G}_{HBC_2^{(3)}}(p)$)  have Kazhdan's property (T). 
\end{thm}

The choice of notation  will be justified in Section~\ref{sec:KMS} below. 
The group $\mathscr G_{HC_2^{(1)}}(7)$ has a presentation on $3$ generators in which, after simplifications, the total  length of the relators is $87$. The generator $b$ of the group $\widetilde{\mathscr G}_{HBC_2^{(3)}}(p)$ is redundant. In fact, the group $\widetilde{\mathscr G}_{HBC_2^{(3)}}(11)$ has a presentation on $2$ generators and $5$ relators which, after simplifications, has a total relator length of $72$. This yields a significant improvement on the main result of  \cite{HypT}, which provides an example of a presentation of an infinite hyperbolic Kazhdan group which, after simplifications, has $4$ generators and $16$ relators, for a total relator length of  $555$. 

Similarly as in the proof of Theorem~\ref{thm:5-fold_with_(T)}, Property (T) is established for the KMS groups using the aforementioned criterion due to  Ershov--Jaikin-Zapirain, see Theorem~\ref{thm:EJ} below. In the case of KMS groups, the exact value of the representation angles can be computed  by hand, so that the proof of Theorem~\ref{thm:KMS:intro} is silicon-free. We refer to  Section~\ref{sec:KMS} for the details,  including other series of infinite hyperbolic Kazhdan groups. 

The following result reveals another remarkable feature of the KMS groups. 

\begin{thm}\label{thm:KMS:intro:2}
For each odd prime $p$, the group $\mathscr G_{HB_2^{(2)}}(p)$ has finite simple quotients of arbitrarily large rank.
\end{thm}

The proof is inspired by the seminal work of M.~Kassabov \cite[\S4.1]{Kassabov2007}, and uses also important results of Shangzhi Li~\cite{Li} on the subgroup structure of the special linear groups. 
Combining Theorems~\ref{thm:KMS:intro} and~\ref{thm:KMS:intro:2}, we obtain the following direct consequence. 

\begin{cor}\label{cor:intro}
There exists a hyperbolic Kazhdan group which possesses finite simple quotients of arbitrarily large rank.
\end{cor}

Further progress towards Question~\ref{qu:S_d} could be accomplished through a finer analysis of the finite simple quotients of specific hyperbolic groups. Indeed, it is conceivable that a given hyperbolic  group possesses quotients in $\mathscr S_d$ for all $d$, but that for $d$ large enough, the only such quotients are of a restricted type. This leads us to the following:

\begin{qu}\label{qu:KMS}
Let $p$ be an odd prime and $G$ be a hyperbolic KMS group over $\mathbf F_p$. Does $G$  have  quotients isomorphic to $\Alt(d)$, for infinitely many $d$'s? Is there a $d$ such that the Lie-type quotients of $G$ belonging to $\mathscr S_d$ are subjected to type or characteristic restrictions?  
\end{qu}

A significant portion of the paper is devoted to a systematic experimental study of the smallest  non-positively curved $3$-fold generalized triangle groups. In particular, we provide evidence that such a $3$-fold generalized triangle group  cannot have property (T).  Our experiments and their outcome are described in Section~\ref{sec:TriTri} and in the Appendices. Those investigations involved an extensive use of the \textsc{Magma} algebra system \cite{Magma}. They led us to propose a reformulation of the question whether all hyperbolic groups are residually finite, see Question~\ref{qu:RepVariety} below.

\medskip

The article is organized as follows. In Section~\ref{sec:EJ} we recall a criterion by Ershov--Jaikin-Zapirain for a group to have property (T) in terms of the \emph{representation angle} of certain subgroups. We relate the representation angle to the spectral gap of the associated coset graph, following ideas of Dymara--Januszkiewicz and I.~Oppenheim. We also compute the representation angle of various finite groups that will later appear as vertex stabilizers within generalized triangle groups. Generalized triangle groups are introduced in Section~\ref{sec:TriangleGroups}. The following two sections are concerned with $3$-fold generalized triangle groups: in Section~\ref{sec:small_cubic_graphs} we collect information on small $3$-regular graphs with edge-transitive automorphism groups, which in the sequel will play the roles of links and stabilizers of vertices respectively. In Section~\ref{sec:TriTri} we perform a systematic study of all possible $3$-fold generalized triangle groups that can be built out of these graphs. The list of groups is presented in Appendix~\ref{sec:PresentationList} while many of their properties are listed in Appendix~\ref{sec:Tables}. In Section~\ref{sec:five-fold} we provide examples of $5$-fold generalized triangle groups with property (T), proving in particular Theorem~\ref{thm:5-fold_with_(T)}. Finally, Section~\ref{sec:KMS} is devoted to Kac--Moody--Steinberg groups and contains the proofs of Theorem~\ref{thm:KMS:intro} and~\ref{thm:KMS:intro:2}.

\subsection*{Acknowledgements}

This project has benefited from conversations with Martin Deraux, Jason Manning, Nicolas Monod, Colva Roney-Dougal, Mikael de la Salle and Henry Wilton. It is a pleasure to thank them for their comments and suggestions. We also thank Izhar Oppenheim his comments.

\section{The Ershov--Jaikin-Zapirain criterion for property (T)}\label{sec:EJ}

A criterion for property (T) that is especially suited for the groups we shall consider is the following one, due to Ershov--Jaikin-Zapirain \cite{EJ} (see also   \cite{Kassabov}). We first need some terminology. 

\subsection{The representation angle between two subgroups}\label{sec:UnitaryAngle}

Given a  group $X$ generated by two subgroups $A, B \leq X$, and a unitary representation $(V, \pi)$ of $X$ without non-zero $X$-invariant vectors, we define
$$\varepsilon_X(A, B; \pi) = \sup\big\{\frac{|\langle u, v \rangle |}{\| u \| \| v \|} \mid u \in V^A \setminus \{0\}, v \in V^B \setminus \{0\}\big\}.
$$
That quantity should be interpreted as the cosine of the angle between the fixed spaces $V^A$ and $V^B$. If $V^A = \{0\}$ or $V^B = \{0\}$, we put $\varepsilon_X(A, B; \pi) =0$. If $\pi$ is an arbitrary unitary representation, we define  $\varepsilon_X(A, B; \pi)$ as $\varepsilon_X(A, B;\pi_0)$, where $\pi_0 \leq \pi$ is the sub-representation of $\pi$ defined on the orthogonal complement of the subspace of $X$-invariant vectors. 
The supremum of $\varepsilon_X(A, B; \pi) $ taken over all unitary representations $(V, \pi)$ of $X$ with $V^X = \{0\}$, is denoted by 
$$\varepsilon_X(A, B).$$
In case the group $X$ is finite, the quantity $\varepsilon_X(A, B)$ coincides with the supremum of   $\varepsilon_X(A, B; \pi) $ taken over all \textit{irreducible} non-trivial unitary representations $(V, \pi)$ of $X$. A spectral interpresentation of the quantity $\varepsilon_X(A, B; \pi) $ will be presented in Section~\ref{sec:Spec_Coset_graph}.

Let 
 $\alpha \in [0, \pi/2]$ be defined by $\alpha = \arccos(\varepsilon_X(A, B))$. The number $\alpha$ is called the \textbf{representation angle} associated with the triple $(X; A, B)$.

\begin{example}\label{ex:0}
Let $X = \langle a, b | a^m, b^n, aba^{-1}b^{-1}  \rangle$ be the direct product $C_m \times C_n$.   Then $\varepsilon_X(\langle  a \rangle, \langle b \rangle) = 0$. 
\end{example}

\begin{example}\label{ex:1}
Let $X = \langle a, b | a^2, b^2, (ab)^r \rangle$ be the dihedral group of order $2r$. Then $\varepsilon_X(\langle  a \rangle, \langle b \rangle) = \cos(\pi/r)$. 
\end{example}

\begin{example}\label{ex:2}
Let $p$ be a prime and $X = \langle a, b | a^p, b^p, [a, b, a], [a, b, b] \rangle$ be the Heisenberg group over $\mathbf F_p$. As shown in \cite[\S4.1]{EJ}, we have $\varepsilon_X(\langle  a \rangle, \langle b \rangle) = 1/\sqrt p$. 
\end{example}

\begin{example}\label{ex:3}
Let $p>2$ be a prime and $r >1$ be an integer such that $r$ divides $p-1$. Let also $\omega \in \mathbf N$  such that $\omega^r = 1 \mod p $ and $\omega^j \neq 1 \mod p$ for $j=1, \dots, r-1$. Then 
$$X = \langle a, b | a^r, b^r, (ab)^p, (ab)^\omega  a^{-1} b^{-1}\rangle$$
is the Frobenius group $C_{p} \rtimes C_r$ of order $pr$. The quantity $\varepsilon_X(\langle  a \rangle, \langle b \rangle) $ can be computed as follows. 

\begin{lem}\label{lem:Frobenius}
Let $\zeta = e^{2\pi i/p}$ and $\alpha =  \sum_{j=0}^{r-1} \zeta^{\omega^j} $. Let also $C(\alpha)$ denote the set of all conjugates of $\alpha$ in the cyclotomic field $\mathbf Q(\zeta)$.  
For any two distinct cyclic subgroups $A, B \leq X$ of order~$r$, we have  
$$\varepsilon_X(A,  B) =  \frac 1 r \sup\big \{ |\beta| \mid \beta \in C(\alpha)\big\}.$$
If $r = \frac{p-1} 2$, then 
$$\varepsilon_X(A,  B) = 
\left\{
\begin{array}{ll}
\frac{\sqrt p +1} {p-1} & \text{if } p \equiv 1 \mod 4,\\
\frac{\sqrt{p+1}}{p-1} & \text{if } p \equiv 3 \mod 4.
\end{array}
\right.$$
\end{lem}
\begin{proof}
The group $X \cong C_p \rtimes C_r$ is a normal subgroup of $C_p \rtimes C_{p-1} \cong \mathbf F_p \rtimes \mathbf F_p^*$. The latter acts doubly transitively on its unique conjugacy class of subgroups of order~$r$. Thus $\Aut(X)$ is doubly transitive on the set of cyclic subgroups of $X$ of order~$r$. It follows that  $\varepsilon_X(A,  B)$ is independent of the choice of the two distinct subgroups $A, B$ of order~$r$. 

Now we set $A = \langle a \rangle$ and $B = ab A b^{-1} a^{-1}$. The group $X$ has $r$ inequivalent (irreducible) representations of degree~$1$, and $\frac{p-1} r$ of degree~$r$. Each of the latter is obtained by inducing a degree~$1$ representation of the cyclic subgroup $C_p = \langle ab \rangle$ to the whole group. Those representations can be described as follows. 

Let $\zeta= e^{2\pi i/p}$  and $e_0, \dots, e_{r-1}$ be an orthonormal basis of $V := \mathbf C^r$. For each $n = 0, \dots, p-1$, a representation $\rho_n \colon X \to \mathrm{GL}_r(\mathbf C)$ is determined by setting $\rho_n(a)e_j = e_{j+1}$ (with indices taken modulo $r$), and $\rho_n(ab)e_j = \zeta^{n\omega^j} e_j$. 
The fixed space $V^A$ is of dimension~$1$; it is spanned by $\sum_{j=0}^{r-1} e_j$. Since $B =  ab A b^{-1} a^{-1}$, we have $V^B = \rho_n(ab)V^A$, which is spanned by $\rho_n(ab)(\sum_{j=0}^{r-1} e_j) = \sum_{j=0}^{r-1} \zeta^{n\omega^j} e_j$. We obtain
$$\varepsilon_X(A, B; \rho_n) =  \frac{| \sum_{j=0}^{r-1} \zeta^{n\omega^j} |} r.$$
The desired conclusion follows, since $C(\alpha) = \{ \sum_{j=0}^{r-1} \zeta^{n\omega^j} \mid n=1, \dots, p-1\}$. 

In the special case where $r = \frac{p-1} 2$, we observe that $\{1, \omega, \omega^2, \dots, \omega^{r-1}\}$ is an index~$2$ subgroup of the multiplicative group $\mathbf F_p^*$, which consists of the squares. It follows that $C(\alpha)$ consists of exactly two elements, namely $\alpha$ and $\beta = \sum_{j=0}^{r-1} \zeta^{a\omega^j}$, where $a$ is represents a non-square in $\mathbf F_p$.  We also recall that the  \textbf{quadratic Gauss sum} $g(s; p)$ is defined as
$$g(s; p) = \sum_{n=0}^{p-1} \zeta^{sn^2}.$$
Observe that $1+2\alpha = g(1;p)$ and $1+2\beta = g(a;p)$. Moreover $1+\alpha + \beta =\sum_{n=0}^{p-1} \zeta^n =  0$. A theorem of Gauss (see \cite[Theorem~1.2.4]{GaussSums}) ensures that 
$$
g(1;p ) = 
\left\{
\begin{array}{ll}
\sqrt p  & \text{if } p \equiv 1 \mod 4,\\
i \sqrt{p} & \text{if } p \equiv 3 \mod 4.
\end{array}
\right.$$
Therefore, we obtain $|\alpha|= \frac{\sqrt p -1} 2$ and $|\beta| = \frac{\sqrt p +1} 2$ if $p \equiv 1 \mod 4$, and $|\alpha|= |\beta| = \frac{\sqrt{p+1}} 2$ if $p \equiv 3 \mod 4$. The required result now follows from the first part of the lemma.
\end{proof}
\end{example}

The algebraic integer $\alpha$ is called a \textbf{Gaussian period}. Estimates of the absolute norm of Gaussian periods have been established in the literature: see \cite{Myerson}, \cite{Duke} and \cite{Habegger}.

\subsection{A criterion  for property (T)}

The relevance of the notion of unitary angle between subgroups comes from the following striking result,  due to Ershov--Jaikin-Zapirain.

\begin{thm}[{\cite[Theorem~5.9]{EJ}}] \label{thm:EJ}
Let $G$ be a locally compact group generated by three open subgroups $A_0, A_1, A_2$. For $i \mod 3$, set $X_i = \langle A_{i-1} \cup A_{i+1}\rangle$ and $\varepsilon_i = \varepsilon_{X_i}(A_{i-1}, A_{i+1})$. Assume that $X_i$ has relative property (T) in $G$ (e.g. $X_i$ is compact, or finite) for all $i$. If
$$\varepsilon_0^2 + \varepsilon_1^2+ \varepsilon_2^2+ 2\varepsilon_0 \varepsilon_1\varepsilon_2 < 1,$$
then $G$ has Kazhdan's property (T).	
\end{thm}

\begin{rmk}
Let 
 $\alpha_i \in [0, \pi/2]$ be defined by $\alpha_i = \arccos(\varepsilon_i)$. As observed by M.~Kassabov \cite{Kassabov}, the condition that $\varepsilon_0^2 + \varepsilon_1^2+ \varepsilon_2^2+ 2\varepsilon_0 \varepsilon_1\varepsilon_2 < 1$ is equivalent to the requirement that 
$$\alpha_0+ \alpha_1 + \alpha_2> \pi.$$
\end{rmk}

\begin{rmk}
In the reference \cite{EJ}, the result cited above is stated for an abstract group $G$; the same proof provides the version stated above, where $G$ is possibly non-discrete and the subgroups $A_0, A_1, A_2$ are open. In the rest of this paper, we will apply Theorem~\ref{thm:EJ} to  a discrete group $G$, except in Corollary~\ref{cor:SimplicialCx:(T)} (a result which will not be used elsewhere in the paper). In particular, the proofs of the main theorems stated in the introduction only rely on the application of Theorem~\ref{thm:EJ} to   discrete groups. 
\end{rmk}

\subsection{Representation angle and the spectrum of coset graphs}\label{sec:Spec_Coset_graph}

Given a group $X$ and two subgroups $A, B \leq X$, we define the \textbf{coset graph} of $X$ with respect to $\{A, B\}$ as the bipartite graph 
$$\Gamma_X(A, B)$$ 
with vertex set $X/A \sqcup X/B$ and edge set $X/A \cap B$, where the incidence relation is the relation of inclusion. 

In this section, we relate the number $\varepsilon_X(A, B)$ introduced in Section~\ref{sec:UnitaryAngle} with the spectrum of the coset graph $\Gamma_X(A, B)$. 

In the special case where $A \cap B = \{e\}$, the coset graph $\Gamma_X(A, B)$ is tightly related to the Cayley graph of $X$ with respect to $A \cup B \setminus \{e\}$. In order to describe that relation, we recall that the \textbf{line graph} associated with a graph $\mathcal G = (V, E)$ is the graph $\mathscr L(\mathcal G)$ with vertex set $E$, and where two vertices are adjacent if they represent edges sharing a vertex. The following observation is straightforward from the definition. 

\begin{lem}\label{lem:LineGraph}
Let $X$ be a group and $A, B\leq X$ be subgroups such that $X = \langle A \cup B\rangle$ and $A \cap B = \{e\}$. Then the line graph of the coset graph $\Gamma_X(A, B)$ is isomorphic to the Cayley graph of $X$ with respect to the generating set $A \cup B \setminus \{e\}$.\qed
\end{lem}

The following result is inspired by the work of I.~Oppenheim \cite{Oppenheim} and from the first step of the proof of Lemma 4.6 in \cite{DJ02}. 

\begin{thm}\label{thm:Spectrum}
Let $X$ be a finite group and $A, B \leq X$ be proper subgroups such that $X = \langle A \cup B \rangle$. Let also $\Delta$ be the combinatorial Laplacian on the coset graph $\Gamma_X(A, B)$. 
\begin{enumerate}[(i)]
\item For every  unitary representation $\pi$ of $X$, the real number $1-\varepsilon_X(A, B;\pi)$ is an eigenvalue of $\Delta$.

\item $1- \varepsilon_X(A, B)$ is the smallest positive eigenvalue of $\Delta$. 
\end{enumerate}
\end{thm}
\begin{proof}
Let $V$ be the vector space of the representation $\pi$, and let $V^A$, $V^B$ and $V^X$ be the subspaces constisting of the $A$-, $B$- and $X$-invariant vectors respectively. Let also $p_A$, $p_B$ and $p_X$ denote the orthogonal projections on $V^A$, $V^B$ and $V^X$. One checks that $\varepsilon_X(A, B;\pi)$ coincides with the operator norm $\|p_A p_B-p_X\|$, see \cite[Remark~3.8]{Oppenheim}. Since $X$ is a finite group, the representation  $\pi$ is a direct sum of irreducible subrepresentations. There is thus no loss of generality in assuming that $\pi$ is irreducible and non-trivial; in particular, we assume henceforth that $V$ is finite-dimensional.  

Set $P = p_A p_B-p_X$. 
For any non-zero vector $x \in V$, we have 
    $$\|Px\|^2  =  \langle Px, Px \rangle = \langle x, P^*Px\rangle 
     \leq \|x\| \|P^*Px \| $$
 by the Cauchy-Schwarz inequality. In particular, we have $\|Px\|^2 \leq \mu \|x\|^2$, where $\mu$ is the largest eigenvalue of the positive operator $P^*P$; moreover the equality case is achieved if $x$ is a $\mu$-eigenvector of $P^*P$.  Since $\pi$ is a subrepresentation of the left-regular representation $\lambda_X$ of $X$, we deduce from \cite[Lemma~4.19]{Oppenheim} that every eigenvalue of $P^*P$ is of the form $(1-\eta)^2$, where $\eta$ is an eigenvalue of $\Delta$. We deduce that $\|P\| = \sqrt \mu = 1-\eta$ for some eigenvalue $\eta$ of $\Delta$. This proves (i). 

Applying the same reasoning to the regular representation $\lambda_X$ of $X$, we deduce that $\varepsilon_X(A, B) = \varepsilon_X(A, B;\lambda_X) \leq 1-\eta_2$, where $\eta_2$ is the smallest positive   eigenvalue of $\Delta$. On the other hand, we know from \cite[Lemma~4.19]{Oppenheim} that if $\pi =\lambda_X$ is the left-regular representation, then the set of eigenvalues of $P^*P$ coincides with  $\{(1-\eta)^2 \mid \eta \text{ is an eigenvalue of } \Delta\}$. Thus the largest eigenvalue of $P^*P$ is equal to $(1-\eta_2)^2$, so that $\varepsilon_X(A, B) = 1-\eta_2$ as required. 
\end{proof} 

The following alternative argument does not rely on \cite[Lemma~4.19]{Oppenheim}, but is inspired instead by the first step of the proof of Lemma 4.6 in \cite{DJ02}. 

\begin{proof}[Alternative proof of Theorem~\ref{thm:Spectrum}(ii)]
Set $d_A = [A:A\cap B]$ and $d_B = [B:A\cap B]$. Consider the Hermitian space $\ell^2(\Gamma)$ of complex-valued functions  defined on the vertex set of $\Gamma$. The  inner product is defined by $\langle \phi , \psi \rangle_\Gamma = \sum_{v \in X/A} d_A \phi(v) \overline{\psi(v)} + \sum_{v \in X/B} d_B  \phi(v) \overline{\psi(v)}$. 
Let $\Delta$ be the combinatorial Laplacian on $\Gamma = \Gamma_X(A, B)$. Thus, for a complex-valued function $f$ defined on the vertex set of $\Gamma$, we have $\Delta(f)(v) = f(v) - \frac 1 {d_A} \sum_{w \in N(v)} f(w)$ if $v \in X/A$, and $\Delta(f)(v) = f(v) - \frac 1 {d_B} \sum_{w \in N(v)} f(w)$ if $v \in X/A$, where $N(v)$ denotes the set of neighbours of $v$ in the graph $\Gamma$. With respect to the inner product on $\ell^2(\Gamma)$ defined above, the Laplacian $\Delta$ is a positive operator. Moreover, its spectrum is symmetric around $1$ (see Proposition~2.14 in \cite{Oppenheim}). In particular the spectrum of $\Delta$ is contained in the interval $[0, 2]$, and we have $\delta \in (0, 1]$. 

 Let   $\ell^2_0(\Gamma)$ be the subspace defined as the orthogonal complement of the space of constant functions on $\Gamma$. Since $\Gamma$ is connected, the latter is the eigenspace associated with the eigenvalue $0$ of $\Delta$. We also consider   the right regular representation of $X$ on $\ell^2(X)$, and its $X$-invariant subspace   $\ell_0^2(X)$ consisting of the functions with zero sum. We normalize the counting measure on $X$ so that the full measure of the subgroup $A \cap B$ is~$1$. 
 With this normalization, the inner product $\langle \cdot , \cdot \rangle_X$ on $\ell^2(X)$ is such that for all $\phi, \psi \in \ell^2(X)$ that are $A \cap B$-invariant, we have $\langle \phi , \psi \rangle_X= \sum_{y \in X/A\cap B} \phi(y)\overline{\psi(y)}$, where $\phi$ and $\psi$ are viewed as functions on the coset space $X/A\cap B$. 

Since every non-trivial irreducible  unitary representation of $X$ is contained in the left regular representation $\lambda$ of $X$,  we have $\varepsilon_X(A, B)  = \varepsilon_X(A, B;\lambda)$.
 Let now $f_A, f_B \in \ell^2(X)$ be functions that are $A$- and $B$-invariant respectively. Then $f_A, f_B$ may be viewed as function on $X/A$ and $X/B$ respectively. Denoting by $F = f_A \sqcup f_B$ the function on $X/A \sqcup X/B$ defined in the natural way, we see that $F$ is an element of $\ell^2(\Gamma)$. In view of the normalizations of the inner products on $\ell^2(X)$ and $\ell^2(\Gamma)$ chosen above, we have $\|F\|^2_\Gamma = \|f_A\|_X^2 + \|f_B\|_X^2$. Borrowing a computation from  the proof of Lemma 4.6 in \cite{DJ02}, we obtain
\begin{align*}
    \langle \Delta F, F \rangle_\Gamma 
     = & \sum_{v \in X/A} d_A \big( f_A(v) - \frac 1 {d_A} \sum_{w \in N(v)} f_B(w)\big)\overline{f_A(v)}\\
     & + 
    \sum_{v \in X/B} d_B \big( f_B(v) - \frac 1 {d_B} \sum_{w \in N(v)} f_A(w)\big)\overline{f_B(v)}\\
     = & \  \| F \|_\Gamma^2 - \sum_{v \in X/A}\sum_{w \in N(v)} f_B(w) \overline{f_A(v)} - \sum_{v \in X/B}\sum_{w \in N(v)} f_A(w) \overline{f_B(v)}\\
     = & \ \| F \|_\Gamma^2 - \sum_{(v, w) \in E(\Gamma)} f_B(w)\overline{f_A(v)} + f_A(v) \overline{f_B(w)}\\
     = & \ \| F \|_\Gamma^2 - 2 \mathrm{Re}\big(\sum_{(v, w) \in E(\Gamma)} f_A(v) \overline{f_B(w)}\big)\\
     = & \ \| f_A \|_X^2 +  \| f_B \|_X^2 - 2 \mathrm{Re}\big( \langle f_A, f_B \rangle_X \big).
\end{align*}
If $f_A, f_B$ have norm~$1$, or more generally if  $\|F\|_\Gamma^2 = 2$, then we obtain $\mathrm{Re}\big( \langle f_A, f_B \rangle_X \big) = 1 - \frac { \langle \Delta F, F \rangle_\Gamma}{ \| F\|_\Gamma^2}$. 

We now suppose in addition that $f_A$ and $f_B$ have zero sum, i.e. they belong to $\ell^2_0(X)$. Then $F$ belongs to $\ell^2_0(\Gamma)$. 
By definition, the smallest eigenvalue of the restriction of $\Delta$ to $\ell^2_0(\Gamma)$ is $\delta$. Therefore we have $\frac { \langle \Delta F, F \rangle_\Gamma}{ \| F\|_\Gamma^2} \geq \delta$ since $\Delta$ is positive. We infer that $\mathrm{Re}\big( \langle f_A, f_B \rangle_X \big) \leq 1- \delta$ for all $f_A, f_B$ of norm~$1$ belonging to $\ell^2_0(X)$. For any two such functions $f_A, f_B$, we write there exists a complex number $\theta$ of modulus~$1$ such that $ \langle f_A , f_B\rangle = |\langle  f_A, f_B\rangle|e^{i\theta} $. Therefore $\langle e^{-i\theta} f_A, f_B\rangle$ is a positive real number, and since $e^{-i\theta} f_A$ is an $A$-invariant vector of norm~$1$ in $\ell_0^2(X)$, we deduce from the above that 
$|\langle  f_A, f_B\rangle| = \langle e^{-i\theta} f_A, f_B\rangle \leq 1-\delta$. In particular we have $\varepsilon_X(A, B) \leq 1-\delta$. 

To prove the converse inequality, we choose a $\delta$-eigenvector $F \in \ell^2_0(\Gamma)$ of $\Delta$. Up to scaling, we may  assume that $\|F\|_\Gamma^2 = 2$. Let $f_A$ (resp. $f_B$) be the restriction of $F$ to $X/A$ (resp. $X/B$). We view $f_A, f_B$ as elements of $\ell^2(X)$. By \cite[Proposition~2.16]{Oppenheim}, we have $f_A, f_B \in \ell^2_0(X)$.  

From the computation above, we deduce that $|\langle f_A, f_B \rangle_X| \geq \mathrm{Re}\big( \langle f_A, f_B \rangle_X \big)  = 1 - \frac { \langle \Delta F, F \rangle_\Gamma}{ \| F\|_\Gamma^2} = 1 - \delta$. Therefore $\varepsilon_X(A, B) \geq 1-\delta$. 
\end{proof}

Theorem~\ref{thm:Spectrum} has several useful consequences. First observe that if $\Gamma$ is any finite bipartite graph and $X \leq \Aut(X)$ acts edge-transitively, by preserving the canonical bipartition of $\Gamma$, then $\Gamma$ can  be identified with the coset graph $\Gamma_X(A, B)$, where $A, B$ are the stabilizers in $X$ of two adjacent vertices in $\Gamma$. Thus Theorem~\ref{thm:Spectrum} provides a way to compute the spectral gap of certain edge-transitive bipartite graphs using representation theory. For Cayley graphs, this is rather standard, see for example \cite{LaffertyRockmore} or  \cite{DSV}.

Theorem~\ref{thm:Spectrum} can also be used in the other direction, to compute the representation angle of some triple $(X, A, B)$ using spectral graph theory. We record the following. 

\begin{cor}\label{cor:Spec_Cayley}
Let $X$ be a finite group and $A, B \leq X$ be  subgroups such that $X = \langle A \cup B \rangle$. Assume that $[A:A \cap B] = [B : A \cap B] = k$. Let $\eta_2$ (resp.  $\lambda_2$) be the second largest eigenvalue of the adjacency matrix of the coset graph $\Gamma_X(A, B)$ (resp. the Cayley graph of $X$ with respect to the generating set $A \cup B \setminus \{e\}$). 

Then $\varepsilon_X(A, B) = \frac{\eta_2} k$. 

If in addition $A \cap B = \{1\}$, then  $\varepsilon_X(A, B) =  \frac{\lambda_2 - k +2} k$.
\end{cor}
\begin{proof}
Since $\delta = 1 - \frac{\eta_2} k$, where $\delta$ denote the smallest positive eigenvalue of the Laplacian on $\Gamma_X(A, B)$, we deduce directly from 
 Theorem~\ref{thm:Spectrum} that  $\varepsilon_X(A, B) = \frac{\eta_2} k$. We now assume that $A \cap B = \{1\}$. It then follows  from Lemma~\ref{lem:LineGraph} that the Cayley graph of $X$ with respect to the generating set $A \cup B \setminus \{e\}$ is isomorphic to the line graph $\mathscr L(\Gamma)$ of $\Gamma = \Gamma_X(A, B)$. The relation between the spectrum of a graph and the spectrum of its line graph is well known, see Theorem~1 in \cite{Cvtekovic75}. In particular we have $\eta_2 = \lambda_2 - k +2$. The result follows. 
\end{proof}

Theorem~\ref{thm:Spectrum} can also be combined with important results on spectral graph theory to provide upper bounds on the representation angle.

\begin{cor}\label{cor:Bounds}
Let $X$ be a finite group and $A, B \leq X$ be  subgroups such that $X = \langle A \cup B \rangle$. Assume that $[A:A \cap B] = [B : A \cap B] = k$. Let $D$ be the diameter of $\Gamma_X(A, B)$. Then 
$$\varepsilon_X(A, B) \geq \frac{2\sqrt{k-1}} k (1 - \frac 1 {\lfloor D/2 \rfloor}) + \frac 1 {k\lfloor D/2 \rfloor}.$$
\end{cor}
\begin{proof}
Immediate from Corollary~\ref{cor:Spec_Cayley} and Theorem~1 in \cite{Nilli}. 
\end{proof}

We refer to \cite{Friedman_Duke1993} for a slightly stronger bound.
 The lower bound from Corollary~\ref{cor:Bounds} is especially useful when the order of group $X$ is large compared to the order of $A$ and $B$, since for $k$ fixed,  the diameter $D$ is bounded below by a logarithmic function of the order of $X$. 
 
 Recall that a $k$-regular graph is a \textbf{Ramanujan graph} if the second largest eigenvalue of its adjacency matrix is at most~$2\sqrt{k-1}$. 

 \begin{cor}\label{cor:Ramanujan}
 Let $X$ be a finite group and $A, B \leq X$ be  subgroups such that $X = \langle A \cup B \rangle$. Assume that $[A:A \cap B] = [B : A \cap B] = k$. Then the coset graph $\Gamma_X(A, B)$ is Ramanujan if and only if $\varepsilon_X(A, B) \leq \frac{2\sqrt{k-1}} k$.
 \end{cor}
\begin{proof}
Immediate from Corollary~\ref{cor:Spec_Cayley}.
\end{proof}

We end this section with two very specific computation of a representation angle, relying on Corollary~\ref{cor:Spec_Cayley}. Their proofs have been computer-aided. 

\begin{prop}\label{prop:SL2(p)}
Let $(p,a,b,g,\varphi)$ be one of the tuples
\begin{align*}
\left(5,
\begin{pmatrix}
4   & 2 \\
3    & 3
\end{pmatrix},
\begin{pmatrix}
1   & 2 \\
0   & 1
\end{pmatrix},
6,
2.2360679775\right),\\
\left(9,
\begin{pmatrix}
  \zeta^5   & \zeta \\
2     & \zeta^6
\end{pmatrix},
\begin{pmatrix}
\zeta^3   & \zeta^6\\
\zeta^5    & \zeta^3
\end{pmatrix},
8,
3.16227766017\right)\text{,}
\end{align*}
where $\zeta \in \Fbb_9$ is a root of the Conway polynomial $x^2+2x+2 \in \Fbb_3[x]$.
Let $X = \SL_2(p)$ and let $A$ and $B$ be the cyclic subgroups generated by the matrices $a$ and $b$. Then
\begin{enumerate}[(i)]
\item $A$ and $B$ are cyclic of order~$5$.
\item  The girth of $\Gamma_X(A, B)$ equals $g$. 
\item $\abs{5\varepsilon_X(A, B) - \varphi} < 10^{-10}$.
\end{enumerate}
\end{prop}

\begin{rmk}
The exact values of $5\varepsilon_X(A, B)$ are probably $\sqrt 5$ and $\sqrt{10}$ respectively, for the two triples $(X, A, B )$ appearing in Proposition~\ref{prop:SL2(p)}.
\end{rmk}

\begin{proof}
Assertions (i) and (ii) are straightforward to check with \textsc{Magma}. The verification of (iii) also required computer calculations, but is more involved. The coset graph of $X$ with respect to the subgroup $A$ and $B$ was computed with \textsc{Magma}. Its eigenvalues then were computed in \textsc{Julia} using \textsc{Arblib}  \cite{Johansson2017arb}.
\end{proof}

\begin{prop}\label{prop:PSL2(p)}
Let $(p,a,b,g,\varphi)$ be one of the tuples
\begin{align*}
\left(31,
\begin{pmatrix}
8   & 14 \\
4     & 11
\end{pmatrix},
\begin{pmatrix}
23   & 0 \\
14    & 27
\end{pmatrix},
10,
3.85410196624\right),\\
\left(41,
\begin{pmatrix}
  0   & 28 \\
19     & 35
\end{pmatrix},
\begin{pmatrix}
38   & 27\\
2    & 9
\end{pmatrix},
10,
3.82842712474\right),\\
\left(109,
\begin{pmatrix}
  0   & 1 \\
-1     & 11
\end{pmatrix},
\begin{pmatrix}
  57   & 2 \\
52     & 42
\end{pmatrix},
14,
4.02260136849\right),\\
\left(131,
\begin{pmatrix}
  -58   & -24\\
-58     & 46
\end{pmatrix},
\begin{pmatrix}
  0   & -3 \\
44     & -12
\end{pmatrix},
14,
3.98383854575\right)\text{.}
\end{align*}
Let $X = \PSL_2(p)$ and let $A$ and $B$ be the cyclic subgroups generated by the natural images of the matrices $a$ and $b$. Then
\begin{enumerate}[(i)]
\item $A$ and $B$ are cyclic of order~$5$.
\item  The girth of $\Gamma_X(A, B)$ equals $g$. 
\item $\abs{5\varepsilon_X(A, B) - \varphi} < 10^{-10}$.
\end{enumerate}
\end{prop}

\begin{proof}
Assertions (i) and (ii) are straightforward to check with \textsc{Magma}. The verification of (iii) also required computer calculations, but is much more involved. Let $\lambda_2$ be the second largest eigenvalue of the adjacency matrix of the Cayley graph of $X$ with respect to $A \cup B \setminus\{e\}$. By Corollary~\ref{cor:Spec_Cayley}, we have $5\varepsilon_X(A, B) = \lambda_2 - 3$. Thus we must prove that $\abs{\lambda_2 - 3 - \varphi} < 10^{-10}$. For that purpose, we followed the following computer-assisted computational approach. 

A numerical estimate $\lambda_2 \approx \varphi + 3$ of the eigenvalue was obtained using standard ARPACK eigenvalue routines on the full Cayley graph $\Gamma_X(A,B)$. However, these computations lack certificates of the accuracy of the approximation. In order to obtain the required certification, we have computed the largest eigenvalue of the hermitian operator $\sum_{i=1}^4 \rho(a)^i + \rho(b)^i$ for each non-trivial irreducible representation $\rho$ of $X$ individually. Explicit realizations of those representations are described in \cite{Piat-Shap} (see also \cite{LaffertyRockmore}). One implementation of those irreducible representations was realized with \textsc{Magma}. Another, independent implementation \cite{RamanujanGraphs.jl} was realized in \textsc{Julia} \cite{Julia-2017}.
The certification, including provably correct bounds, was obtained using the \textsc{Arblib} library \cite{Johansson2017arb} and certified eigenvalue computations therein.
The largest eigenvalue among non-trivial irreducible representations of $X$ satisfies (iii).

The obtained values agree numerically with those obtained in \textsc{Magma}.
\end{proof}

\begin{rmk}\label{rem:PSL2(109)}
The group $\mathrm{PSL}_2(109)$ is the smallest finite simple quotient of the free product $C_5 * C_5$ such that the associated coset graph has girth~$\geq 14$. In view of Corollary~\ref{cor:Ramanujan}, we see that the coset graph $\Gamma_X(A, B)$ from the case $p = 109$ in Proposition~\ref{prop:PSL2(p)} is not a Ramanujan graph. The largest eigenvalue among non-trivial irreducible representations in this case is afforded by the principal representation associated to character $\nu_5\colon \mathbf{F}_{109}^* \to \mathbf{C}$, defined by $\nu_5(\alpha) = \zeta^5_{54}$ (where the generator of $\mathbf{F}_{109}^*$ $\alpha = 6_{109}$ was chosen).
\end{rmk}

\begin{rmk}\label{rem:PSL2(131)}
The three coset graphs $\Gamma_X(A,B)$ for $p \ne 109$ in Proposition~\ref{prop:PSL2(p)} are Ramanujan by Corollary~\ref{cor:Ramanujan}.
\end{rmk}

\subsection{Non-discrete groups}

Theorem~\ref{thm:Spectrum} is formulated for a finite group $X$. It is useful to observe that the result extends easily to a possibly non-discrete compact groups (this is actually the set-up originally adopted  in \cite{DJ02}). 

\begin{prop}\label{prop:CompactGroup}
Let $X$ be a compact group and $A, B \leq X$ be proper open subgroups such that $X = \langle A \cup B\rangle$. Let $K = \bigcap_{g \in X} g(A\cap B)g^{-1}$ be the normal core of $A\cap B$, which is an open normal subgroup (hence of finite index) in $X$. Set $\bar X = X/K$, $\bar A = A/K$ and $\bar B = B/K$. Then $\Gamma_X(A, B) = \Gamma_{\bar X}(\bar A, \bar B)$ and $\varepsilon_X(A, B) = \varepsilon_{\bar X}(\bar A, \bar B)$. 

In particular, we have $\varepsilon_X(A, B) = 1- \delta$, where $\delta$ is the smallest positive eigenvalue of the combinatorial Laplacian on the coset graph $\Gamma_X(A, B)$.
\end{prop}
\begin{proof}
The equality $\Gamma_X(A, B) = \Gamma_{\bar X}(\bar A, \bar B)$ is straightforward from the definitions. 

Since any unitary representation of $\bar X$ is also a continuous representation of $X$, we deduce that $\varepsilon_X(A, B) \geq \varepsilon_{\bar X}(\bar A, \bar B)$. In particular, the required equality holds if $\varepsilon_X(A, B) = 0$. Assume now that $\varepsilon_X(A, B) > 0$ and let $\pi$ be an irreducible non-trivial unitary representation of $X$ such that  $\varepsilon_X(A, B; \pi) >0$. Denoting by $V$ the vector space on which $\pi$ is defined, we must have $V^A \neq \{0\}$, since otherwise we have $\varepsilon_X(A, B; \pi)=0$ by definition. Given a non-zero $A$-fixed vector $v \in V$, the $X$-orbit of $v$ spans $V$ (by the irreducibility of $\pi$), so that the normal core $\bigcap_{g \in X} g Ag^{-1}$, which contains $K$, acts trivially on $V$. This shows that $K \leq \Ker(\pi)$, so that $\pi$ descends to a representation of the finite group $\bar X$. Therefore, we have $\varepsilon_X(A, B; \pi) \leq \varepsilon_{\bar X}(\bar A, \bar B)$. The required equality follows. 

The last assertion of the Proposition is a direct consequence of Theorem~\ref{thm:Spectrum}.
\end{proof}

Combining Proposition~\ref{prop:CompactGroup} with Theorem~\ref{thm:EJ}, we obtain a noteworthy criterion ensuring that certain (possibly non-discrete) automorphism groups of $2$-dimensional simplicial complexes have property (T). 

\begin{cor}\label{cor:SimplicialCx:(T)}
Let $Y$ be a $2$-dimensional, connected, locally finite, simplicial complex and $G\leq \Aut(Y)$ be a closed subgroup of its automorphism group. Assume that $G$ acts simplicially, and transitively on the $2$-simplices. Let $v_0, v_1, v_2 \in Y^{(0)}$ be vertices spanning a $2$-simplex. For each $i \in \{0, 1, 2\}$, we assume that the link $\mathrm{Lk}_Y(v_i)$ is connected and we denote by $\delta_i$   the smallest positive eigenvalue of the combinatorial Laplacian on $\mathrm{Lk}_Y(v_i)$.  If
$$(1-\delta_0)^2+(1-\delta_1)^2+(1-\delta_2)^2 + 2(1-\delta_0)(1-\delta_1)(1-\delta_2)<1,$$
then $G$ has Kazhdan's property (T).
\end{cor}
\begin{proof}
For each $i \mod 3$, let $A_i \leq G$ denote the stabilizer of the edge $[v_{i-1}, v_{i+1}]$ and $X_i = G_{v_i}$ be the stabilizer of the vertex $v_i$. Since $G$ is closed in $\Aut(Y)$ and since $Y$ is locally finite, it follows that $X_i$ and $A_i$ are compact groups. Since the $G$-action on $Y$ is simplicial and transitive on the $2$-simplices, the connectedness of the link $\mathrm{Lk}_Y(v_i)$ implies that $X_i$ is generated by $A_{i-1} \cup A_{i+1}$. Since $Y$ is connected, we deduce that $G$ is generated by $A_0\cup A_1 \cup A_2$. We may now invoke Theorem~\ref{thm:EJ}, whose hypotheses are satisfied in view of Proposition~\ref{prop:CompactGroup}. 
\end{proof}

\section{Generalized triangle groups}\label{sec:TriangleGroups}

\subsection{Non-positively curved triangles of finite groups}

A \textbf{triangle of groups} is a simple complex of groups $G(\mathcal T)$ over the poset $\mathcal T$ of all subsets of $\{1, 2, 3\}$ (see \cite[Example II.12.17(1)]{BH}). In this note, we shall assume that the group attached to the $2$-face is trivial. Thus $G(\mathcal T)$ consists in a collection of $6$ groups $A_0, A_1, A_2, X_0, X_1, X_2$ and $6$ monomorphisms $\varphi_{i; i-1} \colon A_i \to X_{i-1}$ and  $\varphi_{i; i+1} \colon A_i \to X_{i+1}$ with $i=0, 1, 2 \mod 3$. The fundamental group of $G(\mathcal T)$, denoted by $\widehat{G(\mathcal T)}$, is the direct limit of this system of groups and monomorphisms.  Triangles of groups appear in the work of Neumann--Neumann \cite{NN53}; the non-positively curved case has been studied by Gersten--Stallings (see \cite{Stallings1991} and \cite{GS91}). 

The following result is well known.

\begin{thm}\label{thm:NPC}
Let $G(\mathcal T) = (X_i, A_j ; \varphi_{i; i\pm 1})$ be a triangle of groups with trivial face group and let $G = \widehat{G(\mathcal T)}$ be its fundamental group.  
Assume that $X_i$ is generated by the images of $A_{i-1}$ and  $A_{i+1}$. For $i=0, 1, 2$, let $\Gamma_i$ be the coset graph $\Gamma_{X_i}\big(\varphi_{i-1; i}(A_{i-1}), \varphi_{i+1; i}(A_{i+1})\big)$ and $r_i$ be half of its girth. If 
$$\frac 1 {r_0} + \frac 1 {r_1}  +  \frac 1 {r_2} \leq 1,$$
then the following assertions hold. 
\begin{enumerate}[(i)]
	\item $G$  acts isometrically, by simplicial automorphisms, on a $\mathrm{CAT}(0)$ simplicial complex $Y(\mathcal T)$ of dimension~$2$. The action has a $2$-simplex $(v_0, v_1, v_2)$ as strict fundamental domain, which is isometric to a euclidean or hyperbolic triangle with angles $(\pi/r_0, \pi/r_1, \pi/r_2)$. For $i=0, 1, 2$, the link of $Y(\mathcal T)$ at the vertex $v_i$ is isomorphic $\Gamma_i$, and the stabilizer $G_{v_i}$ is isomorphic to $X_i$.  
	
	\item Every finite subgroup of $G$ is conjugate to a subgroup of   $G_{v_i}$ for some $i =0, 1, 2$. 
	
	\item If none of the monomorphisms $\varphi_{i; i\pm 1}$ is surjective, then $G$ is infinite. 
	
\item If there is a homomorphism $\psi \colon G \to F$ to a finite group, whose restriction to $G_{v_i}$ is injective for each $i=0, 1, 2$, then $G$ is virtually torsion-free. 

\item If   
$$\frac 1 {r_0} + \frac 1 {r_1}  +  \frac 1 {r_2} <1,$$
then $G$ is non-elementary hyperbolic, and the simplicial complex $Y(\mathcal T)$ carries a $G$-invariant $\mathrm{CAT}(-1)$ metric. 
\end{enumerate}

\end{thm}
\begin{proof}
For (i), (ii) and (v), see \cite[Theorem~II.12.28]{BH}. 

For (iii), observe that this condition implies that the graphs $\Gamma_i$ have no vertices of degree~$1$. In view of (i), we deduce from \cite[Proposition~II.5.10]{BH} that  every geodesic segment in $Y(\mathcal T)$ is contained in a bi-infinite geodesic. In particular $Y(\mathcal T)$ is unbouded, hence $G$ is infinite by (i). 

(iv) Let $\psi \colon G \to F$  be a homomorphism to a finite group whose restriction to $X_i$ is injective for each $i=0, 1, 2$. If $g \in G$ is a torsion element, then $xgx^{-1} \in X_i$ for some $x \in G$ and $i\in \{0, 1, 2\}$ by (ii).  Since $\psi |_{X_i}$ is injective, we deduce that if $g$ is non-trivial, then $\psi(x) \psi(g) \psi(x)^{-1}$ is non-trivial, hence   $\psi(g)$ is equally non-trivial. Thus $\Ker(\psi)$ is a torsion-free subgroup of finite index in $G$. 
\end{proof}

Following  \cite{LMW}, in case the edge-groups $A_0, A_1, A_2$ are all cyclic, and none of the monomorphisms $\varphi_{i; i\pm 1}$ is surjective, we say that    $\widehat{G(\mathcal Y)}$ is   a \textbf{generalized triangle group}. When the $A_i$'s are all cyclic of order~$k$, we say that $G$ is a \textbf{$k$-fold triangle group}. The triple $(r_0, r_1, r_2)$ will be called the \textbf{half girth type} of the generalized triangle group $G$. The triple $(\Gamma_0, \Gamma_1, \Gamma_2)$ is called its \textbf{link type}. If the half girth type satisfies the inequality $1/r_0 + 1/r_1 + 1/r_2 \leq 1$, we say that the triangle group $G(\mathcal T)$ is \textbf{non-positively curved}. 

In the special case $k=3$, which is of core interest in this paper,  we say that  $\widehat{G(\mathcal Y)}$ is   a \textbf{trivalent triangle group}. 

If the half girth type of $G$ is one of $(3,3,3)$, $(2,4,4)$, or $(2,3,6)$, Theorem~\ref{thm:NPC} does not provide any conclusion regarding the hyperbolicity of $G$. As we shall see, a triangle group  can be hyperbolic or not. To that end, we will use the following critera in order to construct subgroups isomorphic to $\Zbb^2$.

\begin{lem}\label{lem:translation_length}
Consider the setup of Theorem~\ref{thm:NPC} and assume that $r_0=r_1=r_2=3$. Choose the metric on $Y(\mathcal T)$ so that edges have length $1$. Let $\{i,j,k\} = \{0,1,2\}$. Let $a,a',a'',a''' \in A_i \setminus \{1\}$, $b,b',b'',b''',b^{\prime\vprime},b^\vprime \in A_j \setminus \{1\}$ and $c,c'',c'' \in A_k \setminus \{1\}$.
\begin{enumerate}
\item The element $abcb'$ acts on $Y(\mathcal T)$ with translation length $\sqrt{3}$.
\item The element $abca'b'c'$ acts on $Y(\mathcal T)$ with translation length $3$.
\item The element $abcb'\; a'b''c'b'''$ acts on $Y(\mathcal T)$ with translation length $2\sqrt{3}$.\label{item:trans_2}
\item The element $abca'b'c'\; a''b''c''b'''$ acts on $Y(\mathcal T)$ with translation length $\sqrt{21}$.
\item The element $abcb' \;  a'b''c'b''' \; a''b^{\prime\vprime}c''b^{\vprime}$ acts on $Y(\mathcal T)$ with translation length $3\sqrt{3}$.\label{item:trans_3}
\end{enumerate}
In each case there is a non-empty open subset $U$ of the triangle labeled $1$ such that the element moves every point of $\overline{U}$ by the translation length and, in particular, is hyperbolic.
\end{lem}

\begin{figure}[h]
\centering
\begin{tikzpicture}[scale=1.25]
\draw (90:1) -- (210:1) -- (330:1) -- cycle;
\node at (0,0) {$1$};
\node[dot,fill=blue] at (210:1) {};
\begin{scope}[shift={($(30:1)$)}]
\draw (30:1) -- (150:1) -- (270:1) -- cycle;
\node at (0,0) {$a$};
\node[dot,fill=red] at (150:1) {};
\end{scope}
\begin{scope}[shift={($(30:1) + (330:1)$)}]
\draw (90:1) -- (210:1) -- (330:1) -- cycle;
\node at (0,0) {$ab$};
\node[dot,fill=yellow] at (210:1) {};
\end{scope}
\begin{scope}[shift={($(30:2) + (330:1)$)}]
\draw (30:1) -- (150:1) -- (270:1) -- cycle;
\node at (0,0) {$abc$};
\node[dot,fill=red] at (270:1) {};
\end{scope}
\begin{scope}[shift={($(30:2) + (330:1) + (90:1)$)}]
\draw (90:1) -- (210:1) -- (330:1) -- cycle;
\node at (0,0) {$abcb'$};
\node[dot,fill=blue] at (210:1) {};
\node[dot,fill=red] at (90:1) {};
\node[dot,fill=yellow] at (330:1) {};
\end{scope}
\draw[dashed,gray] (210:1) -- ($(210:1) + (30:3) + (330:1.5) + (90:1.5)$);
\draw[dashed,gray] (330:1) -- ($(330:1) + (30:2) + (330:1) + (90:1)$);
\end{tikzpicture}
\begin{tikzpicture}[scale=1.25]
\draw (90:1) -- (210:1) -- (330:1) -- cycle;
\node at (0,0) {$1$};
\node[dot,fill=blue] at (210:1) {};
\begin{scope}[shift={($(30:1)$)}]
\draw (30:1) -- (150:1) -- (270:1) -- cycle;
\node at (0,0) {$a$};
\node[dot,fill=red] at (150:1) {};
\end{scope}
\begin{scope}[shift={($(30:1) + (330:1)$)}]
\draw (90:1) -- (210:1) -- (330:1) -- cycle;
\node at (0,0) {$ab$};
\node[dot,fill=yellow] at (210:1) {};
\end{scope}
\begin{scope}[shift={($(30:2) + (330:1)$)}]
\draw (30:1) -- (150:1) -- (270:1) -- cycle;
\node at (0,0) {$abc$};
\node[dot,fill=blue] at (150:1) {};
\end{scope}
\begin{scope}[shift={($(30:2) + (330:2)$)}]
\draw (90:1) -- (210:1) -- (330:1) -- cycle;
\node at (0,0) {$abca'$};
\node[dot,fill=red] at (210:1) {};
\end{scope}
\begin{scope}[shift={($(30:3) + (330:2)$)}]
\draw (30:1) -- (150:1) -- (270:1) -- cycle;
\node at (0,.2) {$abca'b'$};
\node[dot,fill=yellow] at (150:1) {};
\end{scope}
\begin{scope}[shift={($(30:3) + (330:3)$)}]
\draw (90:1) -- (210:1) -- (330:1) -- cycle;
\node at (0,-.2) {$abca'b'c'$};
\node[dot,fill=blue] at (210:1) {};
\node[dot,fill=yellow] at (330:1) {};
\node[dot,fill=red] at (90:1) {};
\end{scope}
\draw[dashed,gray,thick] ($(210:1) + (30:1.5pt)$) -- ($(330:4) + (30:3)+(150:1.5pt)$);
\draw[dashed,gray,thick] ($(90:1) + (210:1.5pt)$) -- ($(90:1) + (330:3) + (30:3)+ (330:1.5pt)$);
\end{tikzpicture}

\bigskip

\begin{tikzpicture}[scale=1.25]
\draw (90:1) -- (210:1) -- (330:1) -- cycle;
\node at (0,0) {$1$};
\node[dot,fill=blue] at (210:1) {};
\begin{scope}[shift={($(30:1)$)}]
\draw (30:1) -- (150:1) -- (270:1) -- cycle;
\node at (0,0) {$a$};
\node[dot,fill=red] at (150:1) {};
\end{scope}
\begin{scope}[shift={($(30:1) + (330:1)$)}]
\draw (90:1) -- (210:1) -- (330:1) -- cycle;
\node at (0,0) {$ab$};
\node[dot,fill=yellow] at (210:1) {};
\end{scope}
\begin{scope}[shift={($(30:2) + (330:1)$)}]
\draw (30:1) -- (150:1) -- (270:1) -- cycle;
\node at (0,0) {$abc$};
\node[dot,fill=blue] at (150:1) {};
\end{scope}
\begin{scope}[shift={($(30:2) + (330:2)$)}]
\draw (90:1) -- (210:1) -- (330:1) -- cycle;
\node at (0,0) {$abca'$};
\node[dot,fill=red] at (210:1) {};
\end{scope}
\begin{scope}[shift={($(30:3) + (330:2)$)}]
\draw (30:1) -- (150:1) -- (270:1) -- cycle;
\node at (0,.2) {$abca'b'$};
\node[dot,fill=yellow] at (150:1) {};
\end{scope}
\begin{scope}[shift={($(30:3) + (330:3)$)}]
\draw (90:1) -- (210:1) -- (330:1) -- cycle;
\node at (0,.2) {$t=$};
\node at (0,-.2) {$abca'b'c'$};
\node[dot,fill=blue] at (210:1) {};
\end{scope}
\begin{scope}[shift={($(30:4) + (330:3)$)}]
\draw (30:1) -- (150:1) -- (270:1) -- cycle;
\node at (0,0) {$ta''$};
\node[dot,fill=red] at (150:1) {};
\end{scope}
\begin{scope}[shift={($(30:4) + (330:4)$)}]
\draw (90:1) -- (210:1) -- (330:1) -- cycle;
\node at (0,0) {$ta''b''$};
\node[dot,fill=yellow] at (210:1) {};
\end{scope}
\begin{scope}[shift={($(30:5) + (330:4)$)}]
\draw (30:1) -- (150:1) -- (270:1) -- cycle;
\node at (0,.2) {$ta''b''c''$};
\node[dot,fill=red] at (270:1) {};
\end{scope}
\begin{scope}[shift={($(30:5) + (330:4) + (90:1)$)}]
\draw (90:1) -- (210:1) -- (330:1) -- cycle;
\node at (0,-.25) {$ta''b''c''b'''$};
\node[dot,fill=red] at (90:1) {};
\node[dot,fill=yellow] at (330:1) {};
\node[dot,fill=blue] at (210:1) {};
\end{scope}
\begin{scope}[overlay]
\path[name path=edge] ($(30:5) + (330:4) + (90:2)$) -- ($(30:5) + (330:5) + (90:1)$);
\path[name path=axis] (210:1) -- ($(210:1) + (30:10) + (330:8) + (90:2)$);
\draw[dashed,gray, name intersections={of=axis and edge, by=E}] (210:1) -- (E);
\draw[dashed,gray] (330:1) -- ($(330:1) + (30:5) + (330:4) + (90:1)$);
\end{scope}
\end{tikzpicture}

\caption{The complex $Z$ in the various cases of \protect{Lemma~\ref{lem:translation_length}}. The complexes for cases \eqref{item:trans_2} and \eqref{item:trans_3} are not drawn; they consists of two respectively three of the complexes for the first case glued together.}
\label{fig:kms_complex}
\end{figure}
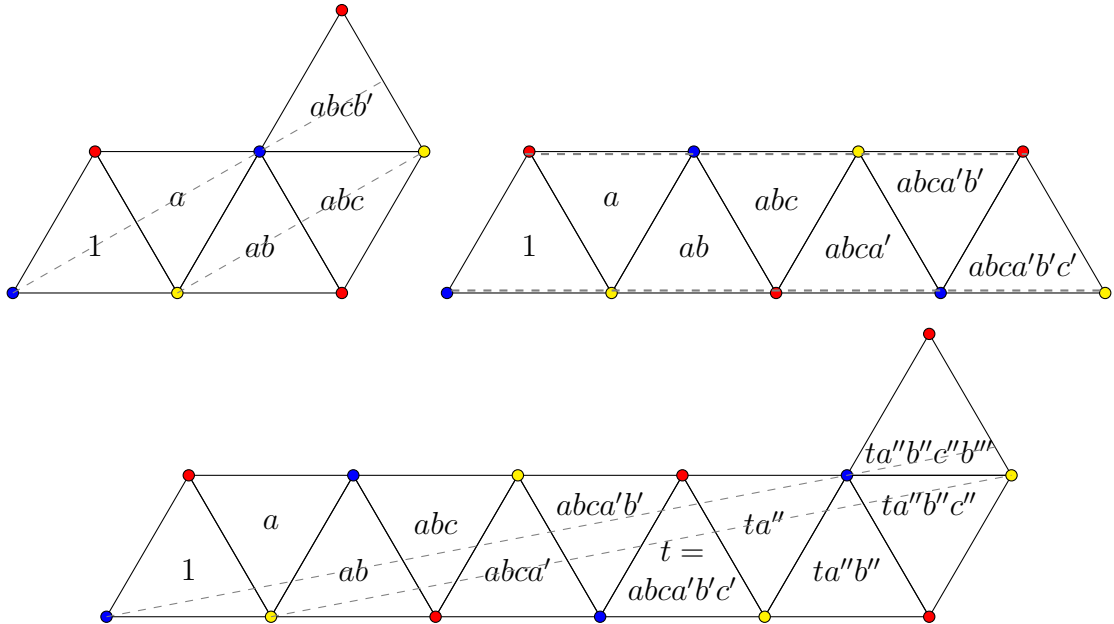

\begin{proof}
We take $\sigma$ to be the base triangle in $Y(\mathcal T)$ whose vertex stabilizers contain $\grp{a,b}$, $\grp{a,c}$ and $\grp{b,c}$ and label the triangle $x.\sigma$ by $x$. Let $g$ be one the words from the statement and let $Z$ be the complex consisting of triangles labeled by prefixes of $g$ with identifications as in $Y(\mathcal T)$, see Figure~\ref{fig:kms_complex}.

Then $Z$ is a subcomplex of $Y(\mathcal{T})$ and we equip it with the induced metric. Let $\pi \colon Z \to \mathbb{E}^2$ be the map to the Euclidean plane indicated in the figure. We claim that if $C \subseteq Z$ is such that $\pi(C)$ is convex, then $\pi|_C$ is an isometry. In order to prove this, note first that $\pi$ is locally $1$-Lipschitz since every vertex link of $Y(\mathcal{T})$ has girth $6$. Let $\hat{Z}$ be $Z$ equipped with the length metric (the metric on $Z$ and on $\pi(Z)$ induce the same length metric). Then the identity map $\iota \colon \hat{Z} \to Z$ is $1$-Lipschitz. Now if $\pi(C)$ is convex, then $(\pi \circ \iota)|_C$ is an isometry, hence $\iota|_C$ is a local isometry. The claim now follows by applying \cite[Proposition~II.4.14]{BH}.

Let $U$ be the open region obtained by intersecting $\sigma$ with the region between the two dashed lines. Taking $C$ to be convex hull of $\overline{U}$ and $g\overline{U}$, it follows from the above discussion that every point of $\overline{U}$ is moved by the claimed distance. Since the displacement function $d_g \colon Y(\mathcal T) \to [0,\infty)$ is convex and constant on $U$, it attains its minimum in $U$: If $x \in U$ and there were a point $y$ of smaller displacement, convexity would imply that $d_g(z) < d_g(x)$ for every point $z \in (x,y]$.
\end{proof}

\begin{lem}\label{lem:translation_length_c}
Consider the setup of Theorem~\ref{thm:NPC} and assume that $r_0=r_1=2, r_2=4$. Choose the metric on $Y(\mathcal T)$ so that the long edges have length $1$. Let $\{i,j\} = \{0,1\}$. Let $a,a',a'',a''' \in A_i \setminus \{1\}$, $b,b',b'',b''' \in A_j \setminus \{1\}$ and $c,c'',c'',c''' \in A_2 \setminus \{1\}$.
\begin{enumerate}
\item The element $acbc'$ acts on $Y(\mathcal T)$ with translation length $\sqrt{2}$.
\item The element $aca'bc'b'$ acts on $Y(\mathcal T)$ with translation length $2$.
\item The element $acbc' \; a'c''b'c'''$ acts on $Y(\mathcal T)$ with translation length $2\sqrt{2}$.\label{item:trans_2a_c}
\item The element $aca'bc'a''b'c''b''c'''$ acts on $Y(\mathcal T)$ with translation length $\sqrt{10}$.
\item The element $aca'bc'b' \; a''c''a'''b''c'''b'''$ acts on $Y(\mathcal T)$ with translation length $4$.\label{item:trans_2b_c}
\end{enumerate}
In each case there is a non-empty open subset $U$ of the triangle labeled $1$ such that the element moves every point of $\overline{U}$ by the translation length and, in particular, is hyperbolic.
\end{lem}

\begin{proof}
The proof is completely analogous to that of Lemma~\ref{lem:translation_length}, see Figure~\ref{fig:kms_complex_c}.
\end{proof}

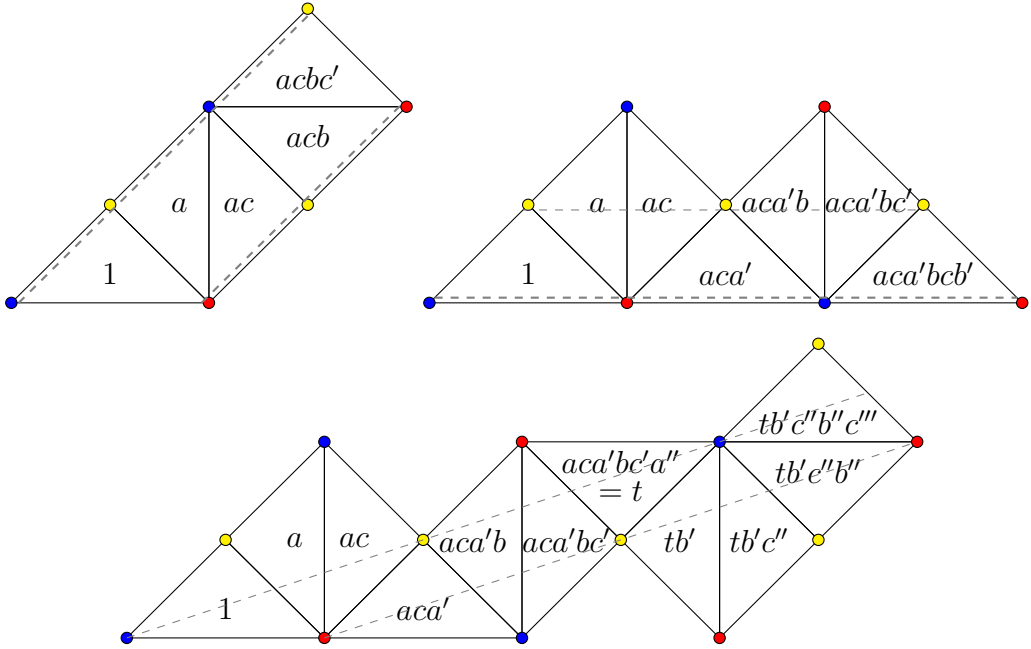
\begin{figure}
\centering
\begin{tikzpicture}[scale=1.3]
\draw (90:1) -- (0:1) -- (180:1) -- cycle;
\node at (90:.3) {$1$};
\node[dot,fill=blue] at (180:1) {};
\begin{scope}[shift={($(0:1)+(90:1)$)}]
\draw (180:1) -- (90:1) -- (270:1) -- cycle;
\node at (180:.3) {$a$};
\node[dot,fill=yellow] at (180:1) {};
\end{scope}
\begin{scope}[shift={($(0:1)+(90:1)$)}]
\draw (90:1) -- (0:1) -- (270:1) -- cycle;
\node at (0:.3) {$ac$};
\node[dot,fill=red] at (270:1) {};
\end{scope}
\begin{scope}[shift={($(0:2)+(90:2)$)}]
\draw (270:1) -- (0:1) -- (180:1) -- cycle;
\node at (270:.3) {$acb$};
\node[dot,fill=yellow] at (270:1) {};
\end{scope}
\begin{scope}[shift={($(0:2)+(90:2)$)}]
\draw (90:1) -- (0:1) -- (180:1) -- cycle;
\node at (90:.3) {$acbc'$};
\node[dot,fill=blue] at (180:1) {};
\node[dot,fill=red] at (0:1) {};
\node[dot,fill=yellow] at (90:1) {};
\end{scope}
\draw[dashed,gray,thick] ($(180:1) + (315:1.5pt) + (45:1.5pt)$) -- ($(0:2) + (90:3) +  (315:1.5pt)$);
\draw[dashed,gray,thick] ($(0:1) + (135:1.5pt) + (225:1.5pt)$) -- ($(0:3) + (90:2) +  (135:1.5pt)$);
\end{tikzpicture}
\begin{tikzpicture}[scale=1.3]
\draw (90:1) -- (0:1) -- (180:1) -- cycle;
\node at (90:.3) {$1$};
\node[dot,fill=blue] at (180:1) {};
\begin{scope}[shift={($(0:1)+(90:1)$)}]
\draw (180:1) -- (90:1) -- (270:1) -- cycle;
\node at (180:.3) {$a$};
\node[dot,fill=yellow] at (180:1) {};
\end{scope}
\begin{scope}[shift={($(0:1)+(90:1)$)}]
\draw (90:1) -- (0:1) -- (270:1) -- cycle;
\node at (0:.3) {$ac$};
\node[dot,fill=blue] at (90:1) {};
\end{scope}
\begin{scope}[shift={($(0:2)$)}]
\draw (90:1) -- (0:1) -- (180:1) -- cycle;
\node at (90:.3) {$aca'$};
\node[dot,fill=red] at (180:1) {};
\end{scope}
\begin{scope}[shift={($(0:3)+(90:1)$)}]
\draw (180:1) -- (90:1) -- (270:1) -- cycle;
\node at (180:.5) {$\smash{aca'b}\mathllap{\phantom{a}}$};
\node[dot,fill=yellow] at (180:1) {};
\end{scope}
\begin{scope}[shift={($(0:3)+(90:1)$)}]
\draw (90:1) -- (0:1) -- (270:1) -- cycle;
\node at (0:.45) {$\smash{aca'bc'}\mathllap{\phantom{a}}$};
\node[dot,fill=red] at (90:1) {};
\end{scope}
\begin{scope}[shift={($(0:4)$)}]
\draw (90:1) -- (0:1) -- (180:1) -- cycle;
\node at (90:.3) {$aca'bcb'$};
\node[dot,fill=yellow] at (90:1) {};
\node[dot,fill=blue] at (180:1) {};
\node[dot,fill=red] at (0:1) {};
\end{scope}
\draw[dashed,gray,thick] ($(180:1) + (90:1.5pt) + (0:1.5pt)$) -- ($(0:5) + (90:1.5pt) + (180:1.5pt)$);
\draw[dashed,gray] ($(90:1) + (270:1.5pt) + (180:1.5pt)$) -- ($(90:1) + (0:4) + (270:1.5pt) + (0:1.5pt)$);
\end{tikzpicture}
\bigskip

\begin{tikzpicture}[scale=1.3]
\draw (90:1) -- (0:1) -- (180:1) -- cycle;
\node at (90:.3) {$1$};
\node[dot,fill=blue] at (180:1) {};
\begin{scope}[shift={($(0:1)+(90:1)$)}]
\draw (180:1) -- (90:1) -- (270:1) -- cycle;
\node at (180:.3) {$a$};
\node[dot,fill=yellow] at (180:1) {};
\end{scope}
\begin{scope}[shift={($(0:1)+(90:1)$)}]
\draw (90:1) -- (0:1) -- (270:1) -- cycle;
\node at (0:.3) {$ac$};
\node[dot,fill=blue] at (90:1) {};
\end{scope}
\begin{scope}[shift={($(0:2)$)}]
\draw (90:1) -- (0:1) -- (180:1) -- cycle;
\node at (90:.3) {$aca'$};
\node[dot,fill=red] at (180:1) {};
\end{scope}
\begin{scope}[shift={($(0:3)+(90:1)$)}]
\draw (180:1) -- (90:1) -- (270:1) -- cycle;
\node at (180:.5) {$aca'b$};
\node[dot,fill=yellow] at (180:1) {};
\end{scope}
\begin{scope}[shift={($(0:3)+(90:1)$)}]
\draw (90:1) -- (0:1) -- (270:1) -- cycle;
\node at (0:.45) {$aca'bc'$};
\node[dot,fill=blue] at (270:1) {};
\end{scope}
\begin{scope}[shift={($(0:4)+(90:2)$)}]
\draw (270:1) -- (0:1) -- (180:1) -- cycle;
\node at (270:.2) {$aca'bc'a''$};
\node at (270:.5) {$=t$};
\node[dot,fill=red] at (180:1) {};
\end{scope}
\begin{scope}[shift={($(0:5)+(90:1)$)}]
\draw (180:1) -- (90:1) -- (270:1) -- cycle;
\node at (180:.4) {$tb'$};
\node[dot,fill=yellow] at (180:1) {};
\end{scope}
\begin{scope}[shift={($(0:5)+(90:1)$)}]
\draw (90:1) -- (0:1) -- (270:1) -- cycle;
\node at (0:.4) {$tb'c''$};
\node[dot,fill=red] at (270:1) {};
\end{scope}
\begin{scope}[shift={($(0:6)+(90:2)$)}]
\draw (270:1) -- (0:1) -- (180:1) -- cycle;
\node at (270:.3) {$tb'c''b''$};
\node[dot,fill=yellow] at (270:1) {};
\end{scope}
\begin{scope}[shift={($(0:6)+(90:2)$)}]
\draw (90:1) -- (0:1) -- (180:1) -- cycle;
\node at (90:.2) {$tb'c''b''c'''$};
\node[dot,fill=blue] at (180:1) {};
\node[dot,fill=red] at (0:1) {};
\node[dot,fill=yellow] at (90:1) {};
\end{scope}
\begin{scope}[overlay]
\path[name path=axis] ($(180:1)$) -- ($(0:11) + (90:4)$);
\path[name path=edge] ($(0:7) + (90:2)$) -- ($(0:6) + (90:3)$);
\draw[dashed,gray, name intersections={of=axis and edge, by=E}] (180:1) -- (E);
\draw[dashed,gray] ($(0:1)$) -- ($(0:7) + (90:2)$);
\end{scope}
\end{tikzpicture}

\caption{The relevant complexes for \protect{Lemma~\ref{lem:translation_length_c}}. The complexes for cases \eqref{item:trans_2a_c} and \eqref{item:trans_2b_c} are not drawn as they consists of several copies of the complexes for the other case glued together.}
\label{fig:kms_complex_c}
\end{figure}

\begin{cor}\label{cor:z_squared}
Let $x$ and $y$ be two elements as in Lemma~\ref{lem:translation_length} or~\ref{lem:translation_length_c} with translation lengths $\abs{x}$ and $\abs{y}$. If $x$ and $y$ commute then $\grp{x,y} \cong \Zbb \times \Zbb$ unless there are $k,\ell \in \Nbb$ relatively prime such that $\abs{x}/\abs{y} = k/\ell$ and $x^{\ell} \ne y^{\pm k}$.
\end{cor}

\begin{proof}
Assume that $x$ and $y$ commute and do not span $\Zbb \times \Zbb$. Then there are $k',\ell' \in \Nbb$ such that $x^{\ell'} = y^{k'}$ or $x^{\ell'} = y^{-k'}$. From now on assume without loss the former. Then $x$ and $y$ have a common axis and it follows that $\abs{x}/\abs{y} = k'/\ell'$. Now let $k = k'/\operatorname{gcd}(k',\ell')$ and $\ell = \ell'/\operatorname{gcd}(k',\ell')$.

There is a point $p$ in the interior of the triangle labeled $1$ that lies on an axis for $x$ as well as an axis for $y$. This follows from the facts that $x$ and $y$ have a common axis and that each has an axis that meets the interior of the triangle labeled $1$. It follows that $x^\ell.p = y^k.p$ and hence that the triangle labeled $1$ is taken to the same triangle by $x^\ell$ and $y^k$. Since the action of $G$ on triangles is free, it follows that $x^\ell = y^k$.
\end{proof}

\subsection{Acylindrical hyperbolicity}

Under the hypotheses of Theorem~\ref{thm:NPC}, the group  $G = \widehat{G(\mathcal T)}$ acts properly and cocompactly on the $2$-dimensional CAT($0$) complex $Y(\mathcal T)$. If $\frac 1 {r_0} + \frac 1 {r_1}  +  \frac 1 {r_2} <1,$ then the complex is piecewise hyperbolic and the length metric on $Y(\mathcal T)$ is globally CAT($-1$). As recorded in Theorem~\ref{thm:NPC}, the group $G$ is then non-elementary  hyperbolic. In particular, it is acylindrically hyperbolic (see \cite{Osin} for an extensive account of this notion). The latter property actually also holds in most cases if we have $\frac 1 {r_0} + \frac 1 {r_1}  +  \frac 1 {r_2} =1$: indeed, the only exception occurs when all  the vertex links are generalized polygons, in which case the complex $Y(\mathcal T)$ is a $2$-dimensional Euclidean building. The following result combines the Rank Rigidity theorem of Ballmann--Brin \cite{BB} with a result of A.~Sisto \cite{Sisto}.

\begin{thm}\label{thm:Acyl}
Retain the notation of Theorem~\ref{thm:NPC} and assume that $\frac 1 {r_0} + \frac 1 {r_1}  +  \frac 1 {r_2} \leq 1$. If the equality holds, assume in addition that the graphs $\Gamma_0$, $\Gamma_1$ and $\Gamma_2$ are not all generalized polygons. Then $G = \widehat{G(\mathcal T)}$ contains an element acting as a rank one isometry on the $\mathrm{CAT}(0)$ complex $Y(\mathcal T)$. In particular $G$ is acylindrically hyperbolic, hence SQ-universal. 
\end{thm}
\begin{proof}
The discussion preceding the statement covers the case 
$\frac 1 {r_0} + \frac 1 {r_1}  +  \frac 1 {r_2} <1$. We now assume that we have  equality: $\frac 1 {r_0} + \frac 1 {r_1}  +  \frac 1 {r_2} =1$. The complex $Y(\mathcal T)$ is then piecewise Euclidean, and the length metric on $Y(\mathcal T)$ is globally CAT($0$). We may then invoke the main result of \cite{BB}, which yields the following dichotomy: either the links $\Gamma_0$, $\Gamma_1$ and $\Gamma_2$ are all generalized polygons, and then $Y(\mathcal T)$ is a Euclidean building, or it is not the case, and then the group $G$ contains an element that acts as a rank one isometry on $Y(\mathcal T)$. The latter property implies in turn that $G$ is acylindrically hyperbolic (this follows from the main results of \cite{Sisto}, see also \cite[Cor.~3.4]{CapHum}). It follows from \cite[Th.~8.1]{Osin} that $G$ is SQ-universal. 
\end{proof}

\subsection{Distinguishing isomorphism classes}\label{sec:Isom-triangle-groups}

Generalized triangle groups are generated by torsion elements. It follows that 
isomorphisms between non-positively curved triangle groups can be analyzed with the help of the  Bruhat--Tits Fixed Point Theorem. 

\begin{prop}\label{prop:HomomorphismTriangleGroups}
Let $k \geq 2$ (resp. $k' \geq 2$) be an integer, let $G$ (resp. $G'$) be a non-positively curved $k$-fold (resp. $k'$-fold) triangle group of half girth type $(r_0, r_1, r_2)$ (resp. $(r'_0, r'_1, r'_2)$). Assume that none of the vertex groups of $G$ is cyclic. Let $Y$ (resp. $Y'$) the associated CAT($0$) simplicial complex, as in Theorem~\ref{thm:NPC}(i). Let also $(v_0, v_1, v_2)$ be a $2$-simplex in $Y$, and let $\psi \colon G \to G'$ be a homomorphism whose restriction to $G_{v_i}$ is injective for all $i \in \{0, 1, 2\}$. 

Assume that $r_i, r'_j \in \{2, 3, 4\}$ for all $i, j \in \{0, 1, 2\}$, with $r_0 \leq r_1 \leq r_2$ and $r'_0 \leq r'_1 \leq r'_2$. Then the following assertions hold.
\begin{enumerate}[(i)]
    \item $k$ divides $k'$. 
        \item For each vertex $y$ in $Y$, the group $\psi(G_y)$ fixes a unique vertex $y'$ in $Y'$.
            \item If $1/r'_0 + 1/r'_1 + 1/r'_2 <1$, then $(v'_0, v'_1, v'_2)$ is a $2$-simplex of $Y'$. 
    
    \item If $\psi$ is an isomorphism, then $(v'_0, v'_1, v'_2)$ is a $2$-simplex of $Y'$, and the assignments $y \mapsto y'$ extend to a $\psi$-equivariant isometry $Y \to Y'$. 
\end{enumerate}
\end{prop}
\begin{proof}
Since the group $\psi(G_{v_i})$ is finite, we deduce from Theorem~\ref{thm:NPC}(ii) that it fixes some vertex $v'_i$ of the complex $Y'$. The group $\psi(G_{v_i}) \cap \psi(G_{v_{i+1}}) = \psi(G_{v_i, v_{i+1}})$ is cyclic of order~$k$, and fixes pointwise the geodesic segment $[v'_i, v'_{i+1}]$. Since a point in the interior of a $2$-simplex of $Y'$ has a trivial stabilizer in $G'$, we infer that $[v'_i, v'_{i+1}]$ is entirely contained in the $1$-skeleton of $Y'$. Since $G_{v_i}$ is not cyclic whereas the stabilizer of every edge of $Y'$ in $G'$ is cyclic, we infer that $v'_i$ is the unique vertex of $Y'$ fixed by $\psi(G_{v_i})$. Since $(v_0, v_1, v_2)$ is a strict fundamental domain for the $G$-action on $Y$ by Theorem~\ref{thm:NPC}(i), we deduce that $\psi(G_y)$ fixes a unique vertex $y'$ in $Y'$ for every vertex $y$ in $Y$. This proves (i) and (ii). 

From the previous paragraph, we deduce that  for all $i$, the Alexandrov angle $\alpha_i := \angle_{v'_i}(v'_{i-1}, v'_{i+1})$ is an integer multiple of $\pi/r'_j$ for some $j \in \{0, 1, 2\}$. Since $Y'$ is CAT($0$), it follows from \cite[Proposition~II.1.7(4)]{BH} that $\alpha_0 + \alpha_1 + \alpha_2 \leq \pi$. Moreover, we must have $\alpha_i >0$ for all $i$ since otherwise  $\psi(G_{v_i})$, which is generated by $\psi(G_{v_i, v_{i-1}}) \cup \psi(G_{v_i, v_{i+1}})$, would fix an edge of $Y'$. This is is impossible since $G_{v_i}$ is not cyclic, whereas the stabilizer of every edge of $Y'$ is cyclic. Therefore, since $\alpha_i$ is an integer multiple of $\pi/r'_j$ and $r'_j \in \{2, 3, 4\}$, the are only two possibilities: either  $\alpha_i = \pi/r'_j$, or $(r'_0, r'_1, r'_2) = (4, 4, 4)$ and $\alpha_i = \pi/2$. The second possibility is however excluded by the Flat Triangle Lemma, see \cite[Proposition~II.2.9]{BH}, since in that case $Y'$ is CAT($-1$) and, therefore, it does not contain any flat triangle.

Recall from Theorem~\ref{thm:NPC}(i) that the link of $Y$ at $v_i$ is isomorphic to the coset graph of $G_{v_i}$ with respect to  the cyclic groups $G_{v_i, v_{i-1}}$ and $G_{v_i, v_{i+1}}$. From the previous paragraph, we deduce that $\psi$ induces a simplicial embedding of the link of $Y$ at $v_i$ into the link of $Y'$ at $v'_i$. 

Let us now assume in addition that $1/r'_0 + 1/r'_1 + 1/r'_2 <1$, and assume now for a contradiction that the geodesic triangle $\Delta = (v'_0, v'_1, v'_2)$ does not span  a $2$-simplex of $Y'$. 
Then, by the previous paragraph, the geodesic path $[v'_{i-1}, v'_{i+1}]$, which is contained in the $1$-skeleton of $Y'$, must contain at least one vertex different from $v'_{i-1}$ and $v'_{i+1}$. Therefore, any simply connected subcomplex of $Y'$ containing $v'_0$, $v'_1$ and $ v'_2$ has at least~$4$ simplices of dimension~$2$. Let us fix a smallest possible such  subcomplex, and view it  as a ruled surface bounded by $\Delta$ in the sense of \cite[III.H.2, p.~426]{BH}. We see that the area of that ruled surface is at least $4$ times the area of a hyperbolic triangle with angles $(\pi/r'_0, \pi/r'_1, \pi/r'_2)$. In particular it is at least $4(\pi -2\pi/3-\pi/4) = 4\pi/12 = \pi/3$. 

Let us now consider the area of  a comparison triangle for $\Delta$ in the hyperbolic plane. Since $\alpha_i \in \{\pi/4, \pi/3\}$, we deduce from \cite[Proposition~II.1.7(4)]{BH} that, in this  comparison triangle, every inner angle is at least $\pi/4$. Therefore the area of that comparison triangle is at most $\pi - 3/4\pi = \pi/4$. Thus we see that the area of the ruled surface bounded by $\Delta$ in $Y'$ is strictly greater than the area of a comparison triangle in the hyperbolic plane. This contradicts  \cite[Proposition~III.H.2.16]{BH}.

Let us finally  assume  that $\psi$ is an isomorphism (but relax the hypothesis that $1/r'_0 + 1/r'_1 + 1/r'_2 <1$). Applying the  discussion on the links above to the inverse map $\psi^{-1}$, we deduce that $\psi$ induces a simplicial isomorphism from the link of $Y$ at $v_i$ to the link of $Y'$ at $v'_i$. In particular, we have $r_i =r'_i$. By the definition of the complexes $Y$ and $Y'$, the metric is completely determined by the girth type. More precisely, the Alexandrov angle $\angle_{v_i}(v_{i-1}, v_{i+1})$ (resp. $\angle_{v'_i}(v'_{i-1}, v'_{i+1})$) is equal to $\pi/r_i$ (resp. $\pi/r'_i$). Therefore $\psi$ induces an isometry from the link of $Y$ at $v_i$ to the link of $Y'$ are $v'_i$. Assertions (iv) follows, in view of the fact that   local isometries extend to global isometries (see \cite[Proposition~II.4.14]{BH}). 
\end{proof}

\begin{cor}\label{cor:isomorphism}
Let $k \geq 2$ (resp. $k' \geq 2$) be an integer, let $G$ (resp. $G'$) be a non-positively curved $k$-fold (resp. $k'$-fold) triangle group of half girth type $(r_0, r_1, r_2)$ (resp. $(r'_0, r'_1, r'_2)$), none of whose vertex groups of $G$ is cyclic. We denote by $A_0, A_1, A_2$ (resp. $A'_0, A'_1, A'_2$) the natural images of the defining edge groups into $G$ (resp. $G'$). Assume that  $r_i$ and $ r'_j$ belong to $ \{2, 3, 4\}$ for all $i, j \in \{0, 1, 2\}$. 
Let  $\psi \colon G \to G'$ be an isomorphism. 

Then there exist an element  $g \in G'$ and a permutation $\sigma \in \Sym(\{0, 1, 2\})$ such that 
$$g \psi(A_i) g^{-1} = A'_{\sigma(i)}$$
for all $i \in \{0, 1, 2\}$.   
In particular $k = k'$. 
\end{cor}
\begin{proof}
Let $Y$ (resp. $Y'$) the associated CAT($0$) simplicial complex, as in Theorem~\ref{thm:NPC}(i). Let also $(v_0, v_1, v_2)$ be the $2$-simplex in $Y$ whose edges $[v_1, v_2]$, $[v_2, v_0]$ and $[v_0, v_1]$ are respectively fixed by $A_0, A_1$ and $A_2$. Similarly let $(w_0, w_1, w_2)$ be be the $2$-simplex in $Y'$ whose edges are respectively fixed by $A'_0, A'_1$ and $A'_2$. By Proposition~\ref{prop:HomomorphismTriangleGroups}(iv), there exists a $2$-simplex $(v'_0, v'_1, v'_2)$ in $Y'$ such that $\psi(G_{v_i}) \leq G'_{v'_i}$ for all $i \in \{0, 1, 2\}$. Since $(w_0, w_1, w_2)$ is a strict fundamental domain for the $G'$-action on $Y'$, there  exist an element  $g \in G'$ and a permutation $\sigma \in \Sym(\{0, 1, 2\})$ such that $gv'_i = w_{\sigma(i)}$ for all $i \in \{0, 1, 2\}$. The conclusion follows by Proposition~\ref{prop:HomomorphismTriangleGroups}.
\end{proof}

It is convenient to reformulate Corollary~\ref{cor:isomorphism} using the following terminology. Let $G(\mathcal T) = (X_i, A_j; \varphi_{i, i\pm 1})$ and $G(\mathcal T') = (X'_i, A'_j; \varphi'_{i, i\pm 1})$ be triangles of groups with trivial face groups. We say that $G(\mathcal T)$ and $G(\mathcal T')$ are \textbf{equivalent} there is a permutation $\sigma\in \Sym(\{0, 1, 2\})$ and, for each $i \mod 3$, an isomorphism $\alpha_i \colon A'_i \to A_{\sigma(i)}$ and an isomorphism $\beta_i \colon X_i \to X'_{\sigma^{-1}(i)}$  such that 
$$\varphi'_{i, i\pm 1} = \beta_{\sigma(i\pm 1)} \circ \varphi_{\sigma(i), \sigma(i\pm 1)} \circ \alpha_i.$$
Denoting by $\rho \in \Sym(\{0, 1, 2\})$ the $3$-cycle defined by $\rho(i)=i+1$, we have $\sigma(i \pm 1) = \sigma \rho^{\pm 1}(i)$. Thus the image of $\beta_{\sigma(i\pm 1)}$ is $X'_{\rho^{\pm 1}(i)} = X'_{i\pm 1}$ as required. 

It is clear that two equivalent triangles of groups have isomorphic fundamental groups. The conclusion of Corollary~\ref{cor:isomorphism} is that, under the corresponding hypotheses, the converse holds: two triangles of groups have isomorphic fundamental groups if and only if they are equivalent.

A well-known theorem of Z.~Sela ensures that one-ended torsion-free hyperbolic groups are co-Hopfian (see \cite[Theorem~4.4]{Sela97}). We recover a very special variation on that result. 

\begin{cor}\label{cor:co-Hopf}
Let $k \geq 2$  be an integer, let $G$ be a $k$-fold triangle group of type $(r_0, r_1, r_2)$,  none of whose  vertex groups is cyclic.  If $r_i \in \{3, 4\}$ for all $i \in \{0, 1, 2\}$ and $(r_0, r_1, r_2) \neq (3, 3, 3)$, then every injective homomorphism $G \to G$ is surjective. In other words $G$ is co-Hopfian. 
\end{cor}
\begin{proof}
Immediate from Proposition~\ref{prop:HomomorphismTriangleGroups}. 
\end{proof}

\section{Small edge-regular cubic graphs}\label{sec:small_cubic_graphs}

In view of the previous section, in order to build the smallest non-positively curved $3$-fold generalized  triangle  groups, we should determine which of the small edge-transitive  bipartite trivalent graphs (also called \textbf{cubic graphs}) admit a group $X$ acting regularly on the edge set, preserving the bipartition. Such graphs will a fortiori coincide with the coset graph of $X$ with respect to the stabilizers $A, B$ of two adjacent vertices, that must both be of order~$3$ by construction. We shall rely on classification results that describe all cubic graphs admitting an edge-transitive automorphism group, up to a certain size (see \cite{ConderMorton95}, \cite{CB02} and  \cite{CMM06}).

We consider graphs of order~$\leq 54$. Since the order of a bipartite cubic graph of girth $2r$ is at least $2^{r+1}-2$, we obtain $r \leq 4$. 
Therefore, it suffices to consider graphs of girth $4$, $6$ and $8$ in the context we have adopted. For graphs of girth~$\leq 6$, we shall moreover limit ourselves to graphs of order~$\leq 30$.

In the following, we provide the exact values of the cosine of the representation angle associated with the triple $(X, A, B)$ under consideration, computed formally by investigating systematically all irreducible representations of $X$. Corollary~\ref{cor:Spec_Cayley}, together with the \textsc{Magma} tools computing the spectrum of graphs, can be used to provide a computational confirmation. 

\subsection{Girth $4$}
Among the small cubic graphs with an edge-transitive automorphism group, only two have girth~$4$. 

\subsubsection{Order $6$: the complete bipartite graph}

Let $$X = \langle a, b | a^3, b^3, a b a^{-1} b^{-1}\rangle$$ and set $A = \langle a \rangle$ and $B = \langle b \rangle$. 

The group $X$ is the direct product $C_3 \times C_3$. The coset graph $\Gamma_X(A, B)$ is the \textbf{complete bipartie graph} $K_{3, 3}$. By Example~\ref{ex:0}, we have 
$$\varepsilon_X(A, B) = 0.$$
The corresponding angle is $90^\circ$. 

\subsubsection{Order $8$: the cube}\label{sec:cube}

Let $$X = \langle a, b | a^3, b^3, a b ab\rangle$$ and set $A = \langle a \rangle$ and $B = \langle b \rangle$. 

The group $X$ is the alternating group $\Alt(4)$. The coset graph $\Gamma_X(A, B)$ is the $1$-skeleton of the \textbf{cube}. Computations show that 
$$\varepsilon_X(A, B) = \frac 1 3.$$
The corresponding angle is $70.53^\circ$.

\subsection{Girth $6$}

In girth~$6$, we focus on graphs of order~$\leq 30$. We see from \cite{ConderMorton95} and \cite{CMM06} that there are exactly $6$ such graphs, respectively of order $14, 16, 18, 20, 24$ and $26$. The graph of order~$20$ is the \textbf{Desargues graph}. It can be viewed as the coset graph of $X = \Alt(5)$ with respect to $A = \langle (1, 2), (1, 2, 3) \rangle$ and $B = \langle (3,4), (3, 4, 5) \rangle$. We claim that it does not admit any automorphism group acting regularly on the edges preserving the bipartition, so it is excluded from our list.  Indeed, such a group would have order~$30$. In every group of order~$30$, any Sylow $3$-subgroup, which is cyclic of order~$3$, is normal. Therefore, no group of order~$30$ is generated by a pair of cyclic subgroups of order~$3$. Thus the Desargues graph cannot be the coset graph of a group of order~$30$, as claimed. 

Each of the $5$ remaining graphs is discussed below. 

\subsubsection{Order $14$: the Heawood graph}\label{sec:Heawood}

Let $$X = \langle a, b | a^3, b^3, a b a^{-1} b^{-1} ab\rangle$$ and set $A = \langle a \rangle$ and $B = \langle b \rangle$. 

The group $X$ is the Frobenius group of order~$21$. The coset graph $\Gamma_X(A, B)$ is the \textbf{Heawood graph}, which is the incidence graph of the projective plane of order~$2$. By Lemma~\ref{lem:Frobenius}, we have 
$$\varepsilon_X(A, B) = \sqrt 2/3.$$

The corresponding angle is $\approx 61.87^\circ$. 

\subsubsection{Order $16$: the M\"obius--Kantor graph}

Let $$X = \langle a, b | a^3, b^3, abab^{-1}a^{-1}b^{-1} \rangle$$ and set $A = \langle a \rangle$ and $B = \langle b \rangle$.

The group $X$ is isomorphic to $\mathrm{SL}_2(\mathbf F_3)$, its order is~$24$. The coset graph $\Gamma_X(A, B)$ is the \textbf{M\"obius--Kantor graph}. This group coincides with the complex reflection group with Coxeter diagram
\begin{center}
\begin{figure}[h]
\psset{xunit=1cm,yunit=1cm,algebraic=true,dimen=middle,dotstyle=o,dotsize=5pt 0,linewidth=1.6pt,arrowsize=3pt 2,arrowinset=0.25}
\begin{tikzpicture}[scale=3]
\node[circle,draw,thick] (l) at (0,0) {$3$};
\node[circle,draw,thick] (r) at (1,0) {$3$};
\draw (l) edge[thick] node[above] {$3$} (r);
\end{tikzpicture}
	\end{figure}
\end{center}
Computations show that 
$$\varepsilon_X(A, B) = \sqrt 3/3.$$
The infimum is achieved by a faithful representation as a complex reflection group in $\mathrm{SU}(2)$.

The corresponding angle is $\approx 54.74^\circ$. 

The group $X$ has one quotient that occurred before, namely $\Alt(4)$. A homomorphism to the presentation of  $\Alt(4)$ from Section~\ref{sec:cube} is obtained by mapping $a$ to $a$ and $b$ to $b^{-1}$. 

\subsubsection{Order $18$: the Pappus graph}\label{sec:Pappus}

Let $$X = \langle a, b | a^3, b^3, (ab)^3, (ab^{-1})^3 \rangle$$ and set $A = \langle a \rangle$ and $B = \langle b \rangle$. 

The group $X$ is isomorphic to the Heisenberg group over $\mathbf F_3$, its order is~$27$. The coset graph $\Gamma_X(A, B)$ is the \textbf{Pappus graph}. As mentionned in Example~\ref{ex:2} above, we have
$$\varepsilon_X(A, B) = \sqrt 3/3.$$

The corresponding angle is $\approx 54.74^\circ$. 

The group $X$ has one quotient that occurred before, namely the direct product $C_3 \times C_3$. 

\subsubsection{Order $24$: the Nauru graph}

Let $$X = \langle a, b | a^3, b^3, (ab)^3, ab a^{-1} b a^{-1}b^{-1}ab^{-1} \rangle$$ and set $A = \langle a \rangle$ and $B = \langle b \rangle$. 

The group $X$ is isomorphic to the direct product $\Alt(4) \times C_3$, its order is~$36$. The coset graph $\Gamma_X(A, B)$ is the \textbf{Nauru graph}. Computations show that 
$$\varepsilon_X(A, B) = 2/3.$$

The corresponding angle is $\approx 48.19^\circ$. 

The group $X$ has two quotients that occurred before, namely $C_3 \times C_3$ and $\Alt(4)$. 

\subsubsection{Order $26$}\label{sec:F26A}

Let $$X = \langle a, b | a^3, b^3, (ab)^3, a  b  a^{-1}  b  a^{-1}  b  a^{-1}  b^{-1} \rangle$$ and set $A = \langle a \rangle$ and $B = \langle b \rangle$. 

The group $X$ is isomorphic to the Frobenius group of order $39$. The coset graph $\Gamma_X(A, B)$ is the cubic graph denoted by \textbf{F26A} in the Foster census.  Using Lemma~\ref{lem:Frobenius}, we find that 
$$\varepsilon_X(A, B) = \frac{|\zeta^2 + \zeta^5+ \zeta^6| }{3},$$
where $\zeta = e^{2\pi i/13}$. 

The corresponding angle is $\approx 46.26^\circ$.

\subsection{Girth $8$}

In girth~$8$, we focus on graphs of order~$\leq 54$. We see from from \cite{CB02} and \cite{CMM06}  that there are exactly $4$ such graphs, respectively of order $30, 40, 48$ and $54$. The graph of order~$30$ is the Tutte $8$-cage $\mathcal T$, which is also the incidence graph of the generalized quadrangle of order~$2$.   It does not admit any automorphism group acting regularly on the edges preserving the bipartition, so it is excluded from our list. Indeed, since $\mathcal T$ is of order~$30$, an edge-regular automorphism group must have order~$45$; on the other hand, the index~$2$ subgroup of $\Aut(\mathcal T)$ preserving the bipartition is isomorphic to $\Sym(6)$, which does not have any subgroup of order~$45$. Each of the three remaining ones is discussed below.

\subsubsection{Order $40$: the double cover of the dodecahedron}

Let $$X = \langle a, b | a^3, b^3, (ab^{-1} a b)^2, (a^{-1} b^{-1} a b^{-1})^2 \rangle$$ and set $A = \langle a \rangle$ and $B = \langle b \rangle$. 

The group $X$ is isomorphic to $\Alt(5)$, which is of order $60$: indeed, there is an isomorphism mapping $a$ to $(1, 2, 3)$ and $b$ to $(3, 4, 5)$. The coset graph $\Gamma_X(A, B)$ is a double cover of the $1$-skeleton of the dodecahedron.  Computations show that 
$$\varepsilon_X(A, B) = \sqrt 5/3.$$
The infimum is achieved by the irreducible representations in degree~$3$.

The corresponding angle is $\approx 41.81^\circ$.

\subsubsection{Order $48$}

Let $$X = \langle a, b | a^3, b^3, (ab)^2(a^{-1}b^{-1})^2\rangle$$ and set $A = \langle a \rangle$ and $B = \langle b \rangle$. 

The group $X$ is isomorphic to the direct product $\mathrm{SL}_2(\mathbf F_3) \times C_3$, its order is~$72$. The coset graph $\Gamma_X(A, B)$ is of girth~$8$. This group coincides with the complex reflection group with Coxeter diagram
	\begin{center}
	\begin{figure}[h]
\psset{xunit=1cm,yunit=1cm,algebraic=true,dimen=middle,dotstyle=o,dotsize=5pt 0,linewidth=1.6pt,arrowsize=3pt 2,arrowinset=0.25}
\begin{tikzpicture}[scale=3]
\node[circle,draw,thick] (l) at (0,0) {$3$};
\node[circle,draw,thick] (r) at (1,0) {$3$};
\draw (l) edge[thick] node[above] {$4$} (r);
\end{tikzpicture}
\end{figure}
\end{center}
Computations show that 
$$\varepsilon_X(A, B) = \sqrt{2/3}.$$
The infimum is achieved by a faithful representation as a complex reflection group in $\mathrm{SU}(2)$.

The corresponding angle is $\approx  35.26^\circ$. 

The group $X$ has $4$ quotients that occurred before: adding the relation $(ab)^4$ gives a presentation of $SL_2(\Fbb_3)$ (order $24$, acting on the M\"obius-Kantor graph of order $16$). A homomorphism to the presentation above is given by taking $a$ to $a$ and $b$ to $b^{-1}$. The second quotient is obtained by adding the relation $(ab^{-1})^3$. This is a presentation of the group $\Alt(4) \times C_3$ (order $36$, acting on the Nauru graph of order $24$). A homomorphism to the previous presentation is given by taking $a$ to $a$ and $b$ to $b^{-1}$. One can then postcompose with quotient maps mentioned above, and obtain quotients isomorphic to $C_3 \times C_3$ and $\Alt(4)$.

\subsubsection{Order $54$: the Gray graph}

Let $$X = \langle a, b |  a^3, b^3,
a  b  a^{-1}  b^{-1}  a^{-1}  b  a  b^{-1}, 
(a  b  a^{-1}  b)^3\rangle$$
and set $A = \langle a \rangle$ and $B = \langle b \rangle$. 

The group $X$ is isomorphic to the wreath product $C_3 \wr C_3$. It is also isomorphic to the $3$-Sylow subgroup in $\mathrm{Sp}_4(\mathbf F_3)$. Alternatively, it can be viewed as a non-trivial split extension of the Heisenberg group over $\mathbf F_3$ by $C_3$. Its order is~$81$. The coset graph $\Gamma_X(A, B)$ is  the \textbf{Gray graph}. An alternative presentation will be provided in Proposition~\ref{prop:PresentationsUnipotent}(ii) below. Moreover, 
by Proposition~\ref{prop:AngleUnipotent}(ii), we have 
$$\varepsilon_X(A, B) = \sqrt{2/3}.$$

The corresponding angle is $\approx 35.26^\circ$. 

The group $X$ has two quotients that occurred before. Adding the relation $(ab)^3$ gives a presentation of the Heisenberg group over $\Fbb_3$ (order $27$, acting on the Pappus graph of order $18$). The homomorphism to the   presentation from Section~\ref{sec:Pappus} is given by taking $a$ to $a$ and $b$ to $b$. Post-composing with a quotient map mentioned above, we also obtain $C_3 \times C_3$ as a quotient. 

\begin{rmk}\label{rem:F54A}
We emphasize that the Gray graph is \textit{not} vertex-transitive (it is semi-symmetric, but not symmetric), while each of  the other $9$ graphs above  admits a vertex-transitive automorphism group. Moreover, the Gray graph should not be confused with the graph denoted by \textbf{F54A} in the Foster census. The latter graph is bipartite, of order~$54$, symmetric, and of girth~$6$ (and therefore excluded from our list, since fro graphs of girth~$\leq 6$, we only considered graphs of order~$\leq 30$).
\end{rmk}

\begin{rmk}\label{rem:39=exception}
For $9$ of the $10$ triples $(X, A, B)$ catalogued above, we have $$\varepsilon_X(A, B) \in \{\frac{\sqrt n} 3 \mid n=0, 1, 2, 3, 4, 5, 6\}.$$ 
The only exception is the Frobenius group of order~$39$, whose associated graph is $F26A$, see Section~\ref{sec:F26A}. This is also the only group involving the prime $13$ (the other groups involve only $2, 3, 5, 7$).
\end{rmk}

\begin{rmk}\label{rem:Raman}
It follows from Corollary~\ref{cor:Ramanujan} that the $10$ graphs above are all Ramanujan graphs. 
\end{rmk}

\section{Non-positively curved trivalent triangle groups: an experimental case study}\label{sec:TriTri}
\subsection{Overview}
A $3$-fold generalized triangle group is called a \textbf{trivalent triangle group}. 
We have undertaken a systematic enumeration of non-positively curved trivalent triangle groups with the smallest possible vertex links. In this introduction, we present an overview of our experimental set-up and of the outcome of our computations, and refer to the next subsections of the text for details.

In a trivalent triangle group, 
every vertex link is an edge-regular trivalent graph (also called cubic graph).  In order to satisfy the non-positive curvature condition, the inequality $\frac{2\pi}{g_0} + \frac{2\pi}{g_1} + \frac{2\pi}{g_2} \leq \pi$ must be satisfied, where $g_i$ denotes the girth of the $i$-th vertex link in the triangle of groups under consideration (see Theorem~\ref{thm:NPC}). The bound on the order of the graphs we have imposed implies that $g_i \leq 8$ for all $i$. The non-positive curvature condition therefore implies that the only girths to be considered are $4$,  $6$ and $8$. The previous section describes  the $2$ smallest edge-regular cubic graphs of girth~$4$, the $5$ smallest edge-regular cubic graphs of girth~$6$, and the $3$ smallest edge-regular cubic graphs of girth~$8$. (It is important to emphasize that not every edge-transitive cubic graph is edge-regular.) This leads us to a set $\mathcal X$ of $10$ graphs, and a corresponding set of $10$ finite groups acting regularly on the edges of those graphs.

In the following, we describe an enumeration of all the possible trivalent triangle groups, all of whose  vertex groups belong to $\mathcal X$. There are exactly $252$ inequivalent triangles of groups with vertex groups in $\mathcal X$. Corollary~\ref{cor:isomorphism} ensures that two inequivalent triangles of groups from that list yield non-isomorphic fundamental groups. We have thus obtained a list of $252$ non-isomorphic trivalent triangle groups, that are all infinite. A list of their presentations is included in Section~\ref{sec:PresentationList}. 

When the girths of the vertex links satisfy the strict inequality $\frac{2\pi}{g_0} + \frac{2\pi}{g_1} + \frac{2\pi}{g_2} < \pi$, the corresponding trivalent triangle group is hyperbolic (see Theorem~\ref{thm:NPC}). Among the $252$ groups of our list, $149$ satisfy that condition. Among the remaining $103$ groups, $38$ are hyperbolic while $65$ contain a subgroup isomorphic to $\mathbf Z \times \mathbf Z$ and are thus not hyperbolic. It is noteworthy that there exist pairs $(H_1, H_2)$ of trivalent triangle groups sharing the same triple of vertex groups, such that $H_1$ is hyperbolic but $H_2$ is not (see for example the pairs $(G^{14, 14, 24}_0, G^{14, 14, 24}_1)$ or $(G^{24, 24, 24}_0, G^{24, 24, 24}_1)$ in Appendix~\ref{sec:Tables}).

It is important to underline that four of the groups from our list had been introduced by M.~Ronan~\cite{Ronan} and studied by various authors. These Ronan groups are those obtained by imposing that   the three vertex links are all isomorphic to the Heawood graph, which is the smallest graph among the $8$ graphs we have selected. The Ronan groups act properly, chamber-transitively on $\widetilde A_2$-buildings. Therefore, they  have Kazhdan's property (T), and they are \textit{not} hyperbolic. Moreover, it is conjectured that there exists some $d$ such that none of them has a quotient in $\mathscr S_d$ (two of the Ronan groups are arithmetic, see \cite{KMW84}, hence  for them, this conjectural assertion follows from Serre's conjecture on the congruence subgroup problem \cite[Conjecture~1]{Rapinchuk92}; for the other two Ronan groups, see \cite[Conjecture~1.5]{BCL}).

Let us now describe our findings regarding property (T). Theorem~\ref{thm:EJ}  confirms that the four Ronan groups have property (T), but  happens to be inconclusive for the remaining $248$ groups on our list. On the other hand, by enumerating subgroups of low index, we found that the majority of those $248$ have finite index subgroups with infinite abelianization, and thus fail to have property (T). 
On the remaining groups, 
we have also run\footnote{The code used in the project is available under \url{https://git.wmi.amu.edu.pl/kalmar/SmallHyperbolic}} the algorithmic tools developed in \cite{Kaluba2019}
to check property (T), without reaching any conclusion. Moreover, further results confirming the absence of property (T) for a large subset of those groups have more recently been obtained by C.~Ashcroft~\cite[Corollary~C]{Ash}. 

In addition, we could also prove that $4$ of our groups have an unbounded isometric action on the real hyperbolic $3$-space, while $9$ of them have an unbounded isometric action on the complex hyperbolic plane, and $5$ have an unbounded isometric action on the complex hyperbolic $3$-space (see Section~\ref{sec:HyperbolicActions}). The existence of such actions is also an obstruction to property (T). 
Among those groups, one has all its vertex groups isomorphic to $\mathrm{SL}_2(\mathbf F_4) \cong \Alt(5)$, and has a Kleinian quotient group. Two others have all their vertex groups   isomorphic to $\mathrm{SL}_2(\mathbf F_3)$, and admit quotient groups that are arithmetic and non-arithmetic lattices in $\mathrm{SU}(2, 1)$. These lattices were first introduced and studied by G.~Mostow~\cite{Mostow}. We view those specific triangle groups as relatives of the $k$-fold generalized triangle groups considered in \cite{LMW}: indeed, they share the property of having all their vertex groups isomorphic to $\mathrm{(P)SL}_2(\mathbf F_q)$ or  $\mathrm{(P)GL}_2(\mathbf F_q)$ for some prime power $q$.

According to an unpublished result of Y.~Shalom and T.~Steger, the four Ronan groups are \textbf{hereditarily just-infinite} (this means that each of their proper quotients is finite, and that this property is inherited by their finite index subgroups). None of the other $248$ groups on our list has this property: indeed, all of them are acylindrically hyperbolic, hence SQ-universal (see Theorem~\ref{thm:Acyl}). Thus, as far as \textit{infinite} quotients are concerned, the four Ronan groups constitute an exception in our sample.  

Let us now describe our findings regarding finite simple quotients. All the data we collected is displayed in Section~\ref{sec:Tables}. We performed a systematic search of finite simple quotients of order~$\leq 5\cdot 10^7$,
and a systematic search of alternating quotients of degree~$\leq 30$. For various subclasses, this upper bound could be extended up to degree~$\leq 40$. The outputs of those computations show that when the girth triple $(g_0, g_1, g_2)$ is $(8, 8, 8)$, the corresponding triangle group has a tremendous amount of finite simple quotients, including most alternating groups of degree between $20$ and $40$. Moreover, all of them are virtually torsion-free. When the girth triple is not $(8, 8, 8)$ the situation is less clear. For each group with girth triple $(6, 8, 8)$ on our list, we could find some (and usually many) non-abelian finite simple quotients.  On the other hand, for some of the hyperbolic groups with girth triple $(4, 8, 8)$, $(6, 6, 6)$ and $(6, 6, 8)$ on our list, we could not find any (non-abelian) finite simple quotient at all. We underline that our investigations of the finite simple quotients of the groups from our sample was primarily based on a systematic enumeration of all simple quotients of finite order up to a certain upper bound (namely $5 \cdot 10^7)$, and a systematic enumeration of all subgroups of finite index up to a certain upper bound (typically $30$ or $40$). Occasionally, we have also found individual finite simple quotients, usually of much larger order, by exhibiting explicit linear representations in degree $d \leq  6$ in characteristic~$0$, and then by extracting congruence quotients. The latter construction was achieved by investigating certain representation varieties of the groups under consideration, using the \textsc{Magma} tools in algebraic geometry (see Section~\ref{sec:LinearReps} and Remark~\ref{rem:deg6} below). 

Our list of trivalent triangle groups interpolates between two extremes. At one end of the spectrum, we have the four Ronan groups,  that are very rigid non-hyperbolic groups with conjecturally no finite simple quotient of arbitrarily large rank. At the opposite end, we have the $17$ trivalent groups from our list all of whose vertex links have girth $8$. Each of the latter has a tremendous amount of finite simple quotients compared to the other groups from the list. It is tempting to believe that those $17$ groups admit $\Alt(d)$-quotients for  all but finitely many $d$, and that they are all virtually special (the latter conjecture has recently been partly verified by C.~Ashcroft, see \cite[Corollary~C]{Ash}). The other groups on our list have properties that appear to be  intermediate compared to those extremes. This led  some of us to speculate, at an intermediate stage of this project,  that some of them, with girth triple $(6, 6, 6)$ or $(6, 6, 8)$, could fail to admit quotients in $\mathscr S_d$ for all $d$. We emphasize that the two hyperbolic groups that are `as close as possible' to the four Ronan groups are those denoted by $G_0^{14, 14, 24}$ and $G_0^{14, 14, 26}$ on our list.

\subsection{Enumerating small trivalent triangle  groups}
\label{sec:TriTriSetup}

Let $\mathcal L$ be the set consisting of the~$10$ graphs listed in Section~\ref{sec:small_cubic_graphs}. We have performed a systematic enumeration of all the non-positively curved triangles $G(\mathcal Y)$ of finite groups with trivial face group and cyclic edge groups of order~$3$, such that the link in the local development at every vertex belongs to $\mathcal L$. Their half girth type is thus an element of the set 
$$\{(2, 4, 4), (3, 3, 3), (3, 3, 4), (3, 4, 4), (4, 4, 4)\}.$$  
Let $\mathbf Y$ be the collection of those $G(\mathcal Y)$. It follows from Corollary~\ref{cor:isomorphism} that if two elements $G(\mathcal Y_1), G(\mathcal Y_2) \in \mathbf Y$ are inequivalent, their fundamental groups $\widehat{G(\mathcal Y_1)}$ and $\widehat{G(\mathcal Y_2)}$ are not isomorphic (see Section~\ref{sec:Isom-triangle-groups} for the definition of the notion of \textit{equivalence} of triangles of groups).

For each graph $L \in \mathcal L$, let $|L|$ be its order (i.e. the cardinality of its vertex set) and  $X^{|L|} = \langle a, b\rangle$ be the group acting regularly on the edges of $L$, generated by a pair  of elements of order~$3$, as it appears in the list from Section~\ref{sec:small_cubic_graphs}. We need to describe, for each such group $X^{|L|}$, the group $\Aut(X^{|L|})_{\{A, B\}}$ of those automorphisms of $X^{|L|}$ that stabilizer the pair $\{A, B\}$, where $A = \langle a\rangle$  and $B = \langle b \rangle$. By definition, every element $\alpha \in \Aut(X^{|L|})_{\{A, B\}}$ permutes the set $\{a, a^{-1}, b, b^{-1}\}$; moreover $\alpha$ is uniquely determined by its action on that set. Thus $\Aut(X^{|L|})_{\{A, B\}}$ is isomorphic to a subgroup of the dihedral group of order~$8$. In Table~\ref{table:Aut(X)}, we provide a generating set for $\Aut(X^{|L|})_{\{A, B\}}$ as a collection of permutations of the set $\{a, a^{-1}, b, b^{-1}\}$. The case-by-case verification is straightforward.

\begin{table}[h]
    \centering
    \begin{tabular}{|c|c|c|}
    \hline
       $X = \langle A, B\rangle$  & Order of $\Aut(X)_{\{A, B\}}$  & Generators of $\Aut(X)_{\{A, B\}}$\\
       \hline
       $X^6, X^{18}, X^{40}$ & 8 & $(a \ b)(a^{-1}\ b^{-1}), (a\ a^{-1})$\\
$X^8, X^{16}, X^{24}, X^{48}$ & 4 & $(a \ b)(a^{-1}\ b^{-1}), (a\ a^{-1})(b\ b^{-1})$\\
$X^{14}, X^{26}$ & 2 & $(a \ b^{-1})(a^{-1}\ b)$\\
$X^{54}$ & 4 & $(a \ a^{-1}), (b\ b^{-1})$\\
\hline
    \end{tabular}
    
    \vspace{.3cm}
    \caption{Descrition of $\Aut(X^{|L|})_{\{A, B\}}$ for $L \in \mathcal L$}
    \label{table:Aut(X)}
\end{table}

We are now ready for the following.

\begin{prop}\label{prop:enumeration}
The set $\mathbf Y$ consists of $252$ inequivalent triangles of finite groups. Their fundamental groups are pairwise non-isomorphic. Their presentations are those listed in Appendix~\ref{sec:PresentationList}.
\end{prop}
\begin{proof}[About the proof]
A triangle of groups $G(\mathcal{Y}) \in \mathbf{Y}$ is determined by the following data: the three edge groups $A_i = \grp{c_i}$ which are cyclic of order $3$; the vertex groups $X_i = \grp{a_i,b_i}$, each of which is one of the $10$~groups whose presentation is given in Section~\ref{sec:small_cubic_graphs}, and the homomorphisms $\varphi_{i-1,i}, \varphi_{i+1,i}$ which amount to identifying the pair $(a_i,b_i)$ with one of the eight pairs $(c_{i-1}^{\pm 1}, c_{i+1}^{\pm1}), (c_{i+1}^{\pm 1}, c_{i-1}^{\pm 1})$. Once the $(A_i)_{i=0, 1, 2}$ and $(X_i)_{i=0, 1, 2}$ are fixed this leads to $8^3$ possible triangles of groups. However, replacing $c_i$ by $c_i^{-1}$ in a given triangle of groups and also inverting $c_i$ in the identifications coming from the $\varphi_{i,j}$ leads to a triangle of groups that is obviously equivalent to the original one. Thus at most $4^3 = 64$ of these triangles of groups are inequivalent. Therefore, once the triple $(X_0, X_1, X_2)$ of vertex groups is fixed, we may encode all the possible triangles of groups $G(\mathcal{Y})$ by an element of the $6$-dimensional vector space $\mathbf F_2^6$ as follows. To a vector $v= (v_0, v_1, \dots, v_5) \in \mathbf F_2^6$, we associate a unique triangle of groups $G(\mathcal{Y})$  defined as follows. For each $i \in \{ 0, 1, 2\}$, we rename the generating pair $(a_i,b_i)$ in $X_i$ according to the  rule:
$$(x_i, y_i) = \left\{
\begin{array}{cl}
(a_i, b_i) & \text{if } v_{2i} = 0\\
(b_i, a_i) & \text{if } v_{2i} = 1.
\end{array}\right.$$
The connecting homomorphisms $\varphi_{i-1,i}, \varphi_{i+1,i}$ are then uniquely determined by the following identifications:
$$
\begin{array}{cl}
y_i = x_{i+1} & \text{if } v_{2i+1} = 0\\
y_i = x_{i+1}^{-1} & \text{if } v_{2i+1} = 1,
\end{array}$$
where, as usual, the index $i$ is taken modulo~$3$. 

Using that parametrization, we can now determine the equivalence classes of triangles of groups, keeping the triple  $(X_0, X_1, X_2)$ fixed, as follows. The equivalence classes of triangles of groups coincide with the orbits of a finite group $\Delta$ determined by the groups $\Aut(X_i)_{\{A_i, B_i\}}$ described in Table~\ref{table:Aut(X)}, for $i=0, 1, 2$, and the permutations of $\{X_0, X_1, X_2\}$ permuting  identical vertex groups. Let us illustrate this by two examples. If the group $\Aut(X_0)_{\{A_0, B_0\}}$ contains the permutation $(a_0\  b_0^{-1})(a_0^{-1} b_0)$, it follows that for every vector $v= (v_0, v_1, \dots, v_5)$, the triangle of group determined by $v$ is equivalent to the triangle of groups determined by $(v_0+1, v_1+1, v_2, v_3, v_4, v_5+1)$. This means that $\Delta$ contains the translation $v \mapsto v + (1, 1, 0, 0, 0, 1)$. Similarly, if the groups $X_0$ and $X_1$ are identical, so that the assignments $(a_0, b_0) \mapsto (a_1, b_1)$ define an isomorphism, then $\Delta$ contains the linear transformation $(v_0, v_1, v_2, v_3, v_4, v_5) \mapsto (v_2, v_3, v_0, v_1, v_4, v_5)$. Therefore, using the coding we have introduced above, we see that the group $\Delta$ acts on $\mathbf F_2^6$ by affine transformations. The equivalence classes of triangle presentations with vertex groups $(X_0, X_1, X_2)$ are nothing but the $\Delta$-orbits on the space $\mathbf F_2^6$. This computation is now easily implemented in \textsc{Magma}.

It is straightforward to obtain a presentation for the fundamental group of a triangle of groups $G(\mathcal Y)$ given by these data: it is generated by the $a_i$ and presented by the relations of the $X_i$ with the appropriate identifications.

Section~\ref{sec:PresentationList} lists these fundamental groups in the form $G^{m_1,m_2,m_3}_\ell$ where the number $m_i$ is the order of the cubic graph on which the groups $X_i$ acts edge-transitively and the number $\ell \in \{0,\ldots,63\}$ corresponds to an element of the index set $\mathbf F_2^6$ as described above (explicitly $(v_0,\ldots,v_5)$ corresponds to $32(1-v_0) + 16v_1 + 8(1-v_2)+4v_3 + 2(1-v_4) + v_5$).

In view of Proposition~\ref{prop:HomomorphismTriangleGroups}, two inequivalent triangles of groups have non-isomorphic fundamental groups. 
This leads to a computation of isomorphism classes; for each class Section~\ref{sec:PresentationList} lists the representative with smallest index $\ell$.

An independent verification has also been realized using  the \textsc{Magma} call \sloppy  \texttt{SearchForIsomorphism($G^{m_1,m_2,m_3}_\ell$,$G^{m_1',m_2',m_3'}_{\ell'}$,$3$)}, for all $\ell \neq  \ell' \in \{0, 1, \dots 63\}$. This  searches for  an isomorphism taking the generators of $G^{m_1,m_2,m_3}_\ell$ to generators or inverses of generators of $G^{m_1',m_2',m_3'}_{\ell'}$. Although there is no guaranty that every such isomorphism will be found, it turns out that the outcome confirms the list displayed in Section~\ref{sec:PresentationList}.

We observe that $\Delta$ depends on $X_i$ only through the group $\Aut(X_i)_{\{A_i, B_i\}}$. For example, the number of equivalence classes for the triple $(X^{14}, X^{16}, X^{18})$ must be equal to the number of equivalence classes with the triple $(X^{26}, X^{40}, X^{48})$ since, in view of Table~\ref{table:Aut(X)}, the affine group $\Delta$ will be identical in those two cases. 
\end{proof}

Information about these groups is tabulated in Section~\ref{sec:Tables}. That information was mostly obtained by basic use of \textsc{Magma} and is described before the tables. There are a few exceptions. Information on Kazhdan's property (T) is incomplete. The four groups $G^{14,14,14}_\ell, \ell \in \{0,1,2,6\}$ are Ronan's groups \cite{Ronan} that are uniform lattices on $\tilde{A}_2$-buildings and therefore are well-known to have property (T). This is also recovered by Theorem~\ref{thm:EJ} using that $\varepsilon_{X_i}(A_{i-1},A_{i+1}) = \sqrt{2}/3$, see Section~\ref{sec:Heawood}. For all the other groups Theorem~\ref{thm:EJ} is inconclusive. It turns out that many of them have a finite index subgroup with infinite abelianization. Moreover, some of them admit unbounded isometric actions on real or comlex hyperbolic spaces, which is also an obstruction is property (T) (see Section~\ref{sec:HyperbolicActions} below).

The information on alternating quotients is obtained by running a systematic search of subgroups of small index, using the \textsc{Magma} call \texttt{LowIndexSubgroups}, then extracting the corresponding coset action and testing whether the corresponding quotient group is alternating. That procedure was run up to a certain upper bound on the index, which was fixed for each half girth type, and is clearly indicated in the tables.

The table for groups of half girth type $(2, 4, 4)$ and $(3,3,3)$ also contains information on hyperbolicity of the groups. It uses an automatic structure that is found (for all trivalent triangle groups) by Holt's \textsc{kbmag} using the \textsc{Magma} call \texttt{$\mathit{isaut}, \mathit{GA}$ := IsAutomaticGroup($G$)}. For those groups that are hyperbolic it can be confirmed using \texttt{IsHyperbolicGroup($\mathit{GA}$ :\ MaxTries:=$20$)}. For the groups that are not, the table provides two elements that generate $\Zbb \times \Zbb$ by Corollary~\ref{cor:z_squared}, where the hypotheses can be verified using the automatic group $\mathit{GA}$.

Although this is not relevant for verifying non-hyperbolicity, we briefly explain how we found these elements. A copy of $\Zbb \times \Zbb$ in a generalized triangle group of half-girth type  $(3,3,3)$ or $(2, 4, 4)$ will act on a flat plane of the CAT($0$) complex $Y$ associated with $G(\mathcal{Y})$, leaving the vertex coloring by conjugacy classes of stabilizers invariant. It is therefore canonically a finite-index subgroup of the isomorphism group of that colored tiled plane, which is the root lattice $\Lambda$ of type $\tilde{A}_2$ or $\tilde C_2$, respectively. If the plane contains the base simplex $\sigma$, the vectors of the five shortest lengths in $\Lambda$ are represented by words as in Lemma~\ref{lem:translation_length}, see Figure~\ref{fig:root_lattice}.
In searching for generators of $\Zbb \times \Zbb$, we enumerated pairs $(g_1,g_2)$ of such words with $\abs{g_1} \ge \abs{g_2}$ in lexicographic order of their translation lengths. Using the automatic structure, we could then have \textsc{Magma} check whether they commute and use Lemma~\ref{lem:translation_length} to decide whether the cyclic groups they generate are commensurate.

\begin{figure}
    \centering~\hfill
    \includegraphics{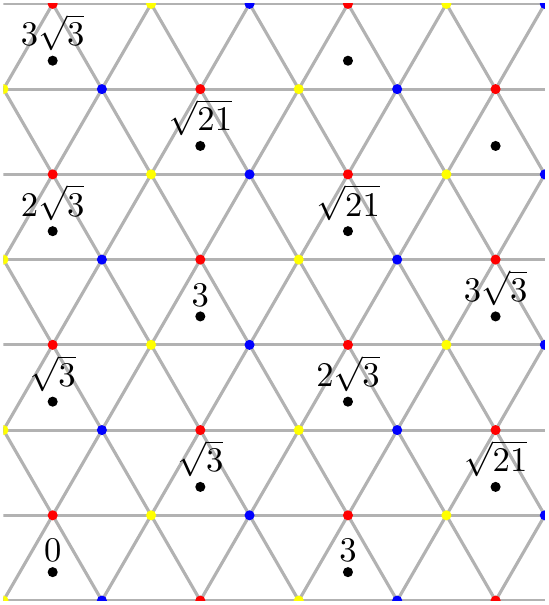}\hfill
    \includegraphics{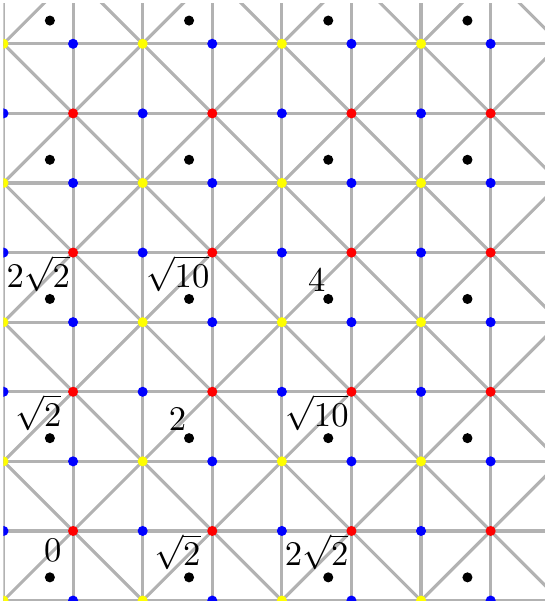}\hfill\phantom{a}
    \caption{Short vectors in the root lattices of type $\tilde{A}_2$ and $\tilde{C}_2$ (the (longer) edges have length $1$).}
    \label{fig:root_lattice}
\end{figure}

\begin{rmk}
As can be seen in the table in Section~\ref{sec:tab_333}, the elements of the first pair we found always satisfy $\abs{g_1} = \abs{g_2}$. It is not a priori clear to us why we would never find pairs with $\abs{g_1} = 2\sqrt{3}$ and $\abs{g_2} = \sqrt{3}$ or with $\abs{g_1} = 2\sqrt{3}$ and $\abs{g_2} = 3$ first.
\end{rmk}

\begin{rmk}
If we take edge lengths to be $1$, the covolume of the root lattice $\Lambda$ of type $\tilde{A}_2$ is $\frac{3\sqrt{3}}{2}$. As a consequence of the previous remark, the covolumes of the copies of $\Zbb \times \Zbb$ that we first find in trivalent triangle groups of half-girth type $(3,3,3)$ are $\frac{3\sqrt{3}}{2} \cdot \{1, 3, 4, 7, 9\}$. These are generally not the copies of $\Zbb \times \Zbb$ with smallest covolume: Ronan's group $G_1^{14,14,14}$ contains a copy of $\Zbb \times \Zbb$ of covolume $\frac{3\sqrt{3}}{2} \cdot 3$ (with $(\abs{g_1},\abs{g_2}) = (2\sqrt{3},\sqrt{3})$) while our search first finds one of covolume $\frac{3\sqrt{3}}{2} \cdot 4$ (with $(\abs{g_1},\abs{g_2}) = (3,3)$).

It would be interesting to know which covolumes (asymptotically) appear with which multiplicity in a given generalized triangle group. Such information is not known even in the case of Ronan's groups.
\end{rmk}

\subsection{Linear representations with infinite image in characteristic~$0$}\label{sec:LinearReps}

In the following sections,  we construct, for certain trivalent triangle groups, explicit low degree representations in characteristic~$0$ with infinite image. We have used the following methods. First, the \textsc{Magma} call \texttt{L2Quotients} computes the finite quotients of  a finitely presented group $G = \langle S | R\rangle$ of the form $\mathrm{PSL}_2(q)$ or $\mathrm{PGL}_2(q)$. When this algorithm ensures the existence of infinitely many such quotients in infinitely many different characteristics, one expects the group $G$ to have a representation in $\mathrm{PGL}_2(\mathbf C)$ with infinite image. This is confirmed in our sample (see Section~\ref{sec:Tables} and Proposition~\ref{prop:ActionOnH^3} below). If the group $G$ is $2$-generated, there is a similar function \texttt{L3Quotients}. This does not apply to any trivalent triangle group from our sample, since they fail to be $2$-generated, but it does apply to certain extensions described in Section~\ref{sec:CyclicExt_trivalent}. Fifty-four  groups from our sample have a subgroup of index~$3$ which is $2$-generated, but the presentation of that subgroup happens to be too complicated for the \texttt{L3Quotients}-algorithm to reach a conclusion in a reasonable computing time. 
An alternative approach attempts to build explicit points of the representation variety of $G$ in $\mathrm{GL}_d(\mathbf C)$, exploiting  the fact that triangle groups are generated by torsion subgroups. More precisely, given a triangle group $G = \langle a, b, c \mid R\rangle$, we can start with an explicit representation $\rho \colon X \to \mathrm{GL}_d(F)$ of the finite group $X = \langle a, b \rangle$, where $F$ is a number field. The possibility to define $\rho(c)$ in such a way that $\rho$ extends to a representation of $G$ can be explored as follows. First define $C$  as a matrix whose entries are $d^2$ indeterminates. Each of the relators from the presentation of $G$ yields an identity, each of which corresponds to $d^2$ polynomial equations in the indeterminates above. The collection of those matrices $C$ satisfying those identities is thus a complex algebraic variety, which can be studied with the \textsc{Magma} tools in algebraic geometry. In particular, the dimension of that variety can be computed. When the dimension is~$-1$, the representation $\rho$ cannot be extended to $G$, whereas if the dimension is~$0$,  \textsc{Magma} finds the rational points $C$ over the ground field $F$. We can then set $\rho(c) = C$ and obtain a representations of $G$. Of course, there is no guarantee that the representation $\rho$ constructed in that way has infinite image. This method has been used to identify some of the representations described below. 

\subsection{Actions on real and complex hyperbolic spaces}\label{sec:HyperbolicActions}

The goal of this section is to explain that some of the groups appearing in the enumeration above admit representations with unbounded image in $\mathrm{SO}(3, 1)$,  in $\mathrm{U}(2, 1)$ or in $\mathrm U(3, 1)$. We start be recalling that the existence of such a representation is an obstruction to Kazhdan's property (T). 

\begin{thm}[{\cite[Theorem~2.7.2]{BHV}}]\label{thm:SO(n, 1)-SU(n, 1)}
Let $n \geq 2$.  For any group $\Gamma$ with Kazhdan's property (T), the image of any homomorphism of $\Gamma$ to $O(n, 1)$ or $U(n, 1)$ has compact closure.  
\end{thm}

In other words, if a group $\Gamma$ is capable of acting by isometries on a real or complex hyperbolic space without a global fixed point, then $\Gamma$ does not have (T).

\begin{prop}\label{prop:ActionOnH^3}
For $\Gamma \in \{G_0^{8, 40, 40}, G_0^{16, 40, 40}, G_0^{24, 40, 40}, G_0^{40, 40, 48}\}$, there is a representation $\Gamma \to \mathrm{SO}(3, 1)$ whose image is a non-discrete, Zariski dense subgroup. 

The group $G_0^{40, 40, 40}$ has a representation $\Gamma \to \mathrm{SO}(3, 1)$ whose image is a cocompact lattice (namely, an index $2$ subgroup of the compact hyperbolic Coxeter group of type $(3, 5, 3)$). 

In particular, none of those five groups has Kazhdan's property (T). 
\end{prop}
\begin{proof}
Let $G = \mathrm{SO}(3, 1)$. 
In the hyperbolic $3$-space $X$, consider a geodesic triangle $\mathcal T$ with vertices $v_0, v_1, v_2$ and, for $i \mod 3$, let  $\ell_i$ be the geodesic line through $v_{i-1}, v_{i+1}$. Let $\rho_i \in G$ be a rotation of an angle $2\pi/3$ around $\ell_i$. Notice that the stabilizer $G_{v_i}$ is isomorphic to $O(3)$. Its action on the unit tangent sphere at $v_i$ is transitive on the set of ordered pairs at any given distance. Therefore, the isomorphism type of the subgroup $\langle \rho_{i-1}, \rho_{i+1} \rangle \leq G_{v_i}$ depends only on the angle formed by $\ell_{i-1}$ and $\ell_{i+1}$. 

If we choose $v_0, v_1, v_2$ so that the cosines of the inner angles of $\mathcal T$ are respectively $1/3$, $\sqrt 5/3$ and $\sqrt 5/3$ (which is possible since the sum of those three angles is~$< \pi$), it follows that the subgroup $\langle \rho_1, \rho_2\rangle \cong \Alt(4)$ and $\langle \rho_0, \rho_1\rangle \cong \langle \rho_0, \rho_2\rangle \cong \Alt(5)$. Therefore $\Lambda = \langle \rho_0, \rho_1, \rho_2\rangle$ is a quotient of the trivalent triangle group $G_0^{8, 40, 40}$,  which is the unique trivalent triangle group whose vertex groups are respectively isomorphic to $ \Alt(4)$, $\Alt(5)$, $\Alt(5)$. In view of the epimorphisms recorded in Section~\ref{sec:Epimorphisms}, this implies that $\Lambda$ is a common quotient of $G_0^{8, 40, 40}$, $G_0^{16, 40, 40}$, $G_0^{24, 40, 40}$ and $G_0^{40, 40, 48}$. 

By construction $\Lambda$ does not fix any point in $X$ or in the ideal boundary of $X$. Moreover $\Lambda$ does not preserve any non-empty  closed convex subset strictly contained in $X$. It then follows from the Karpelevich--Mostow Theorem that $\Lambda$ is Zariski-dense (see \cite[Proposition~2.8]{CaMo_discrete}). The fact that $\Lambda$ is non-discrete follows from the classification in \cite{GM}. 

If we choose $v_0, v_1, v_2$ so that the cosines of the inner angles of $\mathcal T$ are all equal to $\sqrt 5/3$, then the group $\langle \rho_i, \rho_{i+1}\rangle \cong \Alt(5)$ for all $i$ and $\Lambda' = \langle \rho_0, \rho_1, \rho_2\rangle$ is a cocompact lattice in $G$ contained, as an index~$2$ subgroup in the compact hyperbolic Coxeter group of type $(3, 5, 3)$ (this follows from the classification in \cite{GM}). Since $G_0^{40, 40, 40}$ is the unique trivalent triangle group  whose vertex groups are all isomorphic to $ \Alt(5)$, we deduce that $\Lambda'$ is a quotient of $G_0^{40, 40, 40}$.
\end{proof}

We also note that the representations afforded by Proposition~\ref{prop:ActionOnH^3} also provide a theoretical confirmation of the occurrence of  infinitely many $L_2(q)$-quotients found by the computer calculations (see Section~\ref{sec:Tables}). 

\begin{cor}\label{cor:StrongApprox}
Each of the groups $G_0^{8, 40, 40}$, $G_0^{16, 40, 40}$, $G_0^{24, 40, 40}$,  $G_0^{40, 40, 40}$ and $G_0^{40, 40, 48}$ admits finite simple quotients of the form $\mathrm{PSL}_2(q)$ for infinitely many values of $q$. 
\end{cor}

\begin{proof}
By Proposition~\ref{prop:ActionOnH^3}, each of those groups has a Zariski dense representation in $ \mathrm{SO}(3, 1)$ (in the case of  $G_0^{40, 40, 40}$, this follows from the Borel density theorem), hence in $ \mathrm{SL}_2(\mathbf C)$. The conclusion follows from the Strong Approximation of Weisfeiler and Nori (see \cite{Weisfeiler} and \cite{Nori}). 
\end{proof}

\begin{prop}\label{prop:ComplexHyperbolic}
For $\Gamma \in \{G_0^{16, 16, 16}, G_1^{16, 16, 48}, G_0^{16, 48, 48}, G_1^{48, 48, 48}\}$, there is a representation $\Gamma \to \mathrm{SU}(2, 1)$ whose image is a lattice. In particular $\Gamma$ does not have (T), and $\Gamma$ admits finite simple quotients of the form $A_2(q)$ or ${}^2 A_2(q^2)$ for infinitely many values of $q$. 
\end{prop}
\begin{proof}
The assignment $x \to a$ and $y \to b^{-1}$ extends to a surjective homomorphism $\langle x, y \mid x^3, y^3, xyxyx^{-1}y^{-1}x^{-1}y^{-1}\rangle \to \langle a, b \mid a^3, b^3, abab^{-1}a^{-1} b^{-1}  \rangle$ whose kernel is the normal closure of $(xy)^4$. 
Therefore, we have surjective homomorphisms
$$G_1^{48, 48, 48} \to G_0^{16, 48, 48} \to G_1^{16, 16, 48} \to G_0^{16, 16, 16}.$$
In particular it suffices to prove the required assertion for $\Gamma = G_0^{16, 16, 16}$. 
As observed above, the three vertex groups of $\Gamma$ are each isomorphic to a complex reflection group, and the existence of a quotient of $\Gamma$ embedding as a lattice in $U(2, 1)$  follows from \cite[Theorem~A]{Mostow} (it is in fact easy to arrange that this lattice be contained in $\mathrm{SU}(2, 1)$, see \cite[Theorem~4.7]{Parker}). Explicit representations are provided in \cite[\S 9.1]{Mostow} or in \cite[Theorem~4.7]{Parker}. The  assertion on the failure of property (T) then follows from Theorem~\ref{thm:SO(n, 1)-SU(n, 1)}. By Borel density, the image of $\Gamma$ is Zariski dense in $SU(2, 1)$. Strong approximation yields  finite simple quotients of the form $A_2(q)$ or ${}^2 A_2(q^2)$ for infinitely many values of $q$. 
\end{proof}

\begin{rmk}\label{rem:G_16_16_16_0}
The group $G_0^{16, 16, 16}$ has an autormorphism of order~$3$ that cyclically permutes the generators. One checks that the corresponding semi-direct product  $G_0^{16, 16, 16} \rtimes C_3$ admits the following presentation:
$$\langle a, b \mid a^3, b^3, a b^{-1} a  b  a  b^{-1}  a^{-1}  b  a^{-1}  b^{-1}  a^{-1} b \rangle.$$
Since that group is $2$-generated, we may invoke the \texttt{L3Quotients} algorithm in \textsc{Magma}, which confirms the occurrence of infinitely many $A_2(q)$-quotients. Any of those quotients descends to a quotient of $G_0^{16, 16, 16}$, since a non-abelian simple group does not have proper subgroups of index~$3$. 
\end{rmk}

\begin{rmk}
By Theorem~\ref{thm:NPC}, the groups $G_0^{16, 16, 16}$ and $G_1^{16, 16,16}$ act cocompactly on $2$-dimensional simplicial complexes, all of whose vertex links are isomorphic to the M\"obius--Kantor graph. A  systematic study of such complexes has been conducted by Sylvain Barr\'e and Mika\"el Pichot, see \cite{BarrePichot} and references therein. 
\end{rmk}

\begin{prop}\label{prop:ComplexHyperbolic:2}
For $\Gamma \in \{G_0^{6, 48, 48}, G_0^{16, 16, 48},  G_1^{16, 48, 48}, G_1^{24, 48, 48}, G_0^{48, 48, 48}\}$, there is a representation $\Gamma \to \mathrm{U}(2, 1)$ whose image does not have compact closure. In particular $\Gamma$ does not have (T).
\end{prop}
\begin{proof}
As in the proof of Proposition~\ref{prop:ComplexHyperbolic} (see also Section~\ref{sec:Epimorphisms}), we have surjective homomorphisms
$$G_0^{48, 48, 48} \to G_1^{24, 48, 48} \to G_0^{6, 48, 48}.$$
The three vertex groups of the trivalent triangle group  
\[
G_0^{6, 48, 48} =   \langle a, b, c\mid{} a^{3}, b^{3}, c^{3}, bab^{-1}a^{-1},\ (cb)^{2}(c^{-1}b^{-1})^{2}, (ac)^{2}(a^{-1}c^{-1})^{2}\rangle
\]
are complex reflection groups, so that the group $G_0^{6, 48, 48}$ is a complex hyperbolic triangle group with Coxeter--Mostow diagram 
\begin{center}
\begin{figure}[h]
\psset{xunit=1cm,yunit=1cm,algebraic=true,dimen=middle,dotstyle=o,dotsize=5pt 0,linewidth=1.6pt,arrowsize=3pt 2,arrowinset=0.25}
\begin{tikzpicture}[scale=3]
\node[circle,draw,thick] (l) at (0,0) {$3$};
\node[circle,draw,thick] (c) at (1,0) {$3$};
\node[circle,draw,thick] (r) at (2,0) {$3$};
\draw (l) edge[thick] node[above] {$4$} (c);
\draw (c) edge[thick] node[above] {$4$} (r);
\end{tikzpicture}
	\end{figure}
\end{center}
Following Mostow \cite{Mostow}, we obtain  a representation in $\mathrm{GL}_3(\mathbf C)$ mapping $(a, b, c)$ to $(A, B, C)$, where 
$A = \left(
\begin{array}{ccc}
\omega & (\omega -1)\frac{\sqrt 6} 3 & 0\\
0 & 1 & 0\\
0 & 0 & 1
\end{array}\right)$, 
$B = \left(
\begin{array}{ccc}
1 & 0 & 0\\
0 & 1 & 0\\
0 & (\omega -1)\frac{\sqrt 6} 3 & \omega 
\end{array}\right)$, 
and 
$C = \left(
\begin{array}{ccc}
1 & 0 & 0\\
(\omega -1)\frac{\sqrt 6} 3 & \omega  & (\omega -1)\frac{\sqrt 6} 3\\
0 & 0 & 1
\end{array}\right)$, 
and $\omega = \omega = e^{2\pi i/3}$. 
The matrices $A, B, C$ preserve the Hermitian form with Gram matrix
$\left(
\begin{array}{ccc}
1 & \frac{\sqrt 6} 3 & 0\\
\frac{\sqrt 6} 3 & 1 & \frac{\sqrt 6} 3\\
0 & \frac{\sqrt 6} 3 & 1
\end{array}\right)$, which is non-degenerate with signature $(2, 1)$. 
Moreover, by (2.3.3) in \cite{Mostow}, the image of the representation  acts irreducibly on $\mathbf C^3$, and therefore its closure is not compact. The failure of property (T) follows from Theorem~\ref{thm:SO(n, 1)-SU(n, 1)}.

Similarly, we have surjective homomorphisms
$$G_0^{48, 48, 48} \to G_1^{16, 48, 48} \to G_0^{16, 16, 48}.$$ 
The group $G^{16,16,48}_{0} = \langle a, b, c\mid{} a^{3}, b^{3}, c^{3}, baba^{-1}b^{-1}a^{-1}, cbcb^{-1}c^{-1}b^{-1}, (ac)^{2}(a^{-1}c^{-1})^{2}\rangle$ can be represented as a complex hyperbolic triangle group in $\mathrm{GL}_3(\mathbf C)$ by mapping $(a, b, c)$ to the triple $(A, B, C)$, where 
$A = \left(
\begin{array}{ccc}
\omega & (\omega -1)\frac{\sqrt 6} 3 & (\omega -1)\frac{\sqrt 3} 3 \\
0 & 1 & 0\\
0 & 0 & 1
\end{array}\right)$, 
$B = \left(
\begin{array}{ccc}
1 & 0 & 0\\
0 & 1 & 0\\
 (\omega -1)\frac{\sqrt 3} 3 & (\omega -1)\frac{\sqrt 3} 3 & \omega 
\end{array}\right)$, 
and 
$C = \left(
\begin{array}{ccc}
1 & 0 & 0\\
(\omega -1)\frac{\sqrt 6} 3 & \omega  & (\omega -1)\frac{\sqrt 3} 3\\
0 & 0 & 1
\end{array}\right)$. The matrices $A, B, C$ preserve the Hermitian form with Gram matrix
$\left(
\begin{array}{ccc}
1 & \frac{\sqrt 6} 3 & \frac{\sqrt 3} 3\\
\frac{\sqrt 6} 3 & 1 & \frac{\sqrt 3} 3\\
\frac{\sqrt 3} 3 & \frac{\sqrt 3} 3 & 1
\end{array}\right)$, which is non-degenerate with signature $(2, 1)$. The other conclusions follow as before. 
\end{proof}

\begin{prop}\label{prop:ComplexHyperbolic:3}
For $\Gamma \in \{G_1^{24, 24, 24}, G_0^{24, 24, 48}, G_1^{24, 48, 48}, G_0^{48, 48, 48}\}$, there is a representation  $\Gamma \to  \mathrm{U}(3, 1)$ whose image does not have compact closure. In particular $\Gamma$ does not have (T).
\end{prop}
\begin{proof}   
As before, we first notice the existence of surjective homomorphisms $G_0^{48, 48, 48} \to  G_1^{24, 48, 48} \to G_0^{24, 24, 48} \to G_1^{24, 24, 24}$, so that it suffices to consider the case where $\Gamma = G_1^{24, 24, 24}$. 
Set $\omega = e^{2\pi i/3}$. The representation sends $(a, b, c)$ to $(A, B, C)$, where 
$A = \left(\begin{array}{cccc}
      0   &   0 & 1     &  0\\
      1  & 0  &    0    &  0\\
0  &    1      & 0   & 0\\
      0 & 0 & 0 & 1
\end{array}\right)$, 
$B = \left(\begin{array}{cccc}
      0   &   0 & \omega     &  0\\
      -\omega  & 0  &    0    &  0\\
0  &    -\omega      & 0 & 0\\
      0 & 0 & 0 & 1
\end{array}\right)$ 
and 
$$C = 
\left(\begin{array}{cccc}
      \omega   &   -1  &-\omega - 2     &  4\\
      \omega  &-\omega - 2  &    -1     &  4\\
2\omega + 1   &    \omega      & \omega   & -4\omega\\
      \omega   &   -1   &   -1  &-\omega + 3
\end{array}\right).$$ 
The Hermitian form whose Gram matrix is the diagonal matrix with coefficients $(1, 1, 1, -4)$ is preserved by  $Q = \langle A, B, C\rangle$, so that $Q$ is contained in $U(3, 1)$. 
The finite group $\langle A, B\rangle$ acts irreducibly on the $3$-dimensional subspace spanned by the first three vectors of the canonical basis. Since that subspace is not invariant under $C$, it follows that $Q$ acts irreducibly on $\mathbf C^4$. This implies that $Q$ is not conjugate to a compact subgroup of $U(3, 1)$. 
\end{proof}

\begin{rmk}\label{rem:CongruenceImage}
Observe that $Q$ is contained in $\mathrm{GL}_4(\mathbf Z[\omega])$. It is thus a discrete subgroup. Using the \textsc{Magma} call \texttt{CongruenceImage}, followed by \texttt{LMGChiefFactors}, one deduces from Proposition~\ref{prop:ComplexHyperbolic:3} that $\Gamma$ has finite simple quotients isomorphic to $\mathrm{PSU}_4(5)$, $\mathrm{PSL}_4(7)$, $\mathrm{PSU}_4(11)$, $\mathrm{PSL}_4(13)$ and $\mathrm{PSU}_4(17)$.
\end{rmk}

\begin{rmk}
Using the representation variety approach described  in Section~\ref{sec:LinearReps}, we found that the groups  $G_1^{26, 26, 26}$ and $G_{21}^{26, 26, 26}$ both have also  a representation to $ \mathrm{U}(3, 1)$ with unbounded image. The coefficients are however too long to be included here. As above, we can compute the first few congruence quotients of those linear images, and deduce that $G_1^{26, 26, 26}$ and $G_{21}^{26, 26, 26}$ both have finite simple quotients isomorphic to $\mathrm{PSU}_4(25)$, $\mathrm{PSU}_4(49)$, $\mathrm{PSU}_4(121)$, $\mathrm{PSU}_4(17)$ and $\mathrm{PSU}_4(361)$. 
\end{rmk}

\subsection{Cyclic extensions of triangle groups}\label{sec:CyclicExt_trivalent}

The existence of an automorphism of order~$3$ that cyclically permutes the generators, observed in Remark~\ref{rem:G_16_16_16_0} for the group $G_0^{16, 16, 16}$, actually applies to each group of the form $G_0^{k, k, k}$ from our sample, namely with $k \in \{14, 16, 18, 24, 26, 40, 48, 54\}$. The corresponding semi-direct product $G_0^{k, k, k} \rtimes C_3$ is denoted by $\widetilde G_0^{k, k, k}$. A presentation of $\widetilde G_0^{k, k, k}$ can be obtained as follows. Denoting by $R^k$ the relators involving only $a$ and $b$ in the presentation of $G_0^{k, k, k}$, we have
$$\widetilde G_0^{k, k, k}\cong \langle t, a, b | R^k, t^3, tat^{-1}b^{-1}\rangle.$$
Clearly, the generator $b$ is redundant, and $\widetilde G_0^{k, k, k}$ is a $2$-generator group. In fact, after simplifications, the presentation of $\widetilde G_0^{k, k, k}$ is usually much shorter than the presentation of $G_0^{k, k, k}$ (see Section~\ref{sec:Presentation_extended}). 

Let $X$ be the subgroup of $\widetilde G_0^{k, k, k}$ generated by $a$ and $b$. Thus $X$ is a finite group with  presentation $X\cong \langle a, b| R^k\rangle$.
The group $\widetilde G_0^{k, k, k}$ is a quotient of the HNN-extension $\langle t, a, b | R^k,  tat^{-1}b^{-1}\rangle$ of the finite group $X$ (which is a virtually free group) by the single extra relation $t^3 = 1$. 
That very specific structure of  $\widetilde G_0^{k, k, k}$ can be used to construct homomorphisms  $\rho \colon \widetilde G_0^{k, k, k} \to H$ to a given target $H$, as follows.    Assume given a homomorphism $\rho\colon X \to H$ to a group $H$. Assume moreover that there is an element $\tau \in H$ that conjugates $\rho(a)$ to $\rho(b)$. Then any element 
$T$ belonging to the coset $\tau C_H(\rho(a))$ of the centralizer $C_H(\rho(a))$
conjugates $\rho(a)$ to $\rho(b)$. Therefore, the assignment $\rho(t) = T$ defines a homomorphism $\rho \colon \widetilde G_0^{k, k, k} \to H$ if and only if $T^3=1$. This method has been implemented to construct the representation described in the following.

\begin{prop}\label{prop:Rep_E^18}
Let $\omega = e^{2\pi i/3}$ and $\zeta = e^{2\pi i/9}  \in \mathbf C$. The assignments $(a, b, c) \mapsto (A, B, C)$, where 
$$
A=
\left(\begin{array}{ccc}
0 & 0 & 1\\
\omega & 0 & 0\\
0 & 1 & 0
\end{array}
\right), 
\hspace{.5cm}
B=
\left(\begin{array}{ccc}
0 & 1 & 0\\
0 & 0 & 1\\
\omega & 0 & 0
\end{array}
\right), 
\hspace{.5cm}
C=
\left(\begin{array}{ccc}
\omega+1 & 0 & -\omega\\
0 & 0 & \omega\\
\omega+1 & \omega & \omega^2
\end{array}
\right),
$$ 
define a homomorphism $\rho\colon G_0^{18, 18, 18} \to \mathrm{PGL}_3(\mathbf Z[\omega])$ with Zariski dense image. 

Moreover, the extra assignment $t \mapsto  T$, where 
$$
T = \left(\begin{array}{ccc}
 2\zeta^4 + \zeta^3 + \zeta - 1 &    2\zeta^4 + \zeta^3 + \zeta - 1   &   -\zeta^4 + \zeta^3 + \zeta + 2\\
-\zeta^4 - 2\zeta^3 - 2\zeta - 1   &  -\zeta^4 + \zeta^3 + \zeta + 2   & 2\zeta^4 + \zeta^3 + \zeta - 1\\
 2\zeta^4 + \zeta^3 + \zeta - 1  &  2\zeta^4 + \zeta^3 + \zeta - 1  & -\zeta^4 - 2\zeta^3 - 2\zeta - 1
\end{array}
\right),
$$
defines an extension of $\rho$ to a homomorphism $\rho\colon \widetilde G_0^{18, 18, 18} \to \mathrm{PGL}_3(\mathbf Z[\zeta])$. 
\end{prop}
\begin{proof}
We retain the notation from the discussion made before the proposition. 
We have $X = \langle a, b | a^3, b^3, (ba)^3, (ba^{-1})^3\rangle$. Computations show that $A^3$ and $B^3$ are both equal to the scalar matrix $\omega \mathrm{Id}$. Moreover we have $(BA)^3 = \mathrm{Id} = (BA^{-1})^3$. Thus $\rho$  defines a homomorphism $X \to \mathrm{PGL}_3(\mathbf Z[\omega])$. 

One further computes that $T^3$ is a scalar matrix, and that $T$  conjugates $A$ to $B$, and $B$ to $C$. Therefore $\rho$ indeed defines a homomorphism $\widetilde G_0^{18, 18, 18} \to \mathrm{PGL}_3(\mathbf Z[\zeta])$ whose restriction to $G_0^{18, 18, 18}$ takes values in the group $\mathrm{PGL}_3(\mathbf Z[\omega])$. 

To verify that the image of $\rho$ is Zariski dense, one computes that $(ABAC)^3 = \left(\begin{array}{ccc}
  1 & 0 & 3\omega^2\\
 0 & 1 & -3\omega^2\\
 0 & 0 & 1
\end{array}
\right)$. In particular the Zariski closure of the cyclic group $\langle (ABAC)^3 \rangle$ is a one-dimensional unipotent subgroup  of $\mathrm{SL}_3(\mathbf C)$. It is then straightforward to check that this subgroup, together with its conjugates under $\langle A, B, C \rangle$, generates the entire group $\mathrm{SL}_3(\mathbf C)$.  The required assertion follows.
\end{proof}

As in Corollary~\ref{cor:StrongApprox}, combining  Proposition~\ref{prop:Rep_E^18} with Strong Approximation, we deduce that the groups $G_0^{18, 18, 18}$ and $\widetilde G_0^{18, 18, 18}$ have quotients of the form $\mathrm{PGL}_3(\mathbf F_q)$ for infinitely many finite fields $\mathbf F_q$, each of which is of  degree~$\leq 6$ over its prime field.  This is confirmed by calling \texttt{L3Quotients} for the group $\widetilde G_0^{18, 18, 18}$ in \textsc{Magma}. 

We also remark that, in contrast with the representations studied in Section~\ref{sec:HyperbolicActions}, Proposition~\ref{prop:Rep_E^18} does not provide an unbounded action of the group $G_0^{18, 18, 18}$ on real or complex hyperbolic spaces, but it rather provides an action on a symmetric space of higher rank. In particular, it does not yield any conclusion on the possible failure of property (T) for that group. 

We conclude this section by underlining another feature of the cyclically extended groups. Except for the case $k=54$, the associated simplicial complex associated to the group  $G_0^{k, k, k}$ via Theorem~\ref{thm:NPC} satisfies the hypotheses of J.~Swiatkowski's main theorem in \cite{Swiat}. In view of the regularity properties of the vertex links that can be consulted in \cite{ConderMorton95}, we deduce that for $k =16, 18, 24, 26, 40, 48$, the full automorphism group of the simplicial complex associated with the group $G_0^{k, k, k}$ is discrete. On the other hand, the complex associated with $G_0^{14, 14, 14}$ is a $2$-adic Bruhat--Tits building, whose automorphism group is  non-discrete. 

\subsection{A representation in degree $6$}\label{sec:Deg6rep}

The following result was obtained using the representation variety approach described  in Section~\ref{sec:LinearReps}.

\begin{prop}\label{prop:deg6}
For $\Gamma \in \{G_4^{14, 14, 18}, G_4^{14, 14, 54}\}$, there is an irreducible representation $\Gamma \to \mathrm{U}(6)$ whose image is  infinite. 
\end{prop}
\begin{proof}
There is a surjective homomorphism $G_4^{14, 14, 54} \to G_4^{14, 14, 18}$, so it suffices to prove the statement for $\Gamma = G_4^{14, 14, 18}$. We next observe that $\Gamma$ has an automorphism of order~$2$ fixing $b$ and swapping $a$ and $c^{-1}$. The corresponding extension, which is an overgroup of index~$2$ of $\Gamma$, has the following presentation:
$$\widetilde \Gamma = \langle t, a, b| t^2, a^3, b^3, t btb^{-1},
b  a  b^{-1}  a^{-1}  b a, 
(a  t  a^{-1}  t)^3,
(a  t)^6\rangle.$$
Set $\omega = e^{2\pi i/3}$ and $\zeta = e^{2\pi i/12}$. Let $A_1 =\left(\begin{array}{ccc}
0 & 0 & 1\\
1 & 0 & 0\\
0 & 1 & 0\end{array}\right)$, 
$A_2 = \left(\begin{array}{ccc}
1 & 0 & 0\\
0 & \omega^2 & 0\\
0 & 0 & \omega\end{array}\right)$ and
$I = \left(\begin{array}{ccc}
1 & 0 & 0\\
0 & 1 & 0\\
0 & 0 & 1
\end{array}\right)$. Define the $6\times 6$-matrices 
$A = \left(\begin{array}{cc}
A_1 & 0\\
0 & A_2
\end{array}\right)$
and 
$T = \left(\begin{array}{cc}
0 & I\\
I & 0
\end{array}\right)$, where each entry represents a $3\times 3$-block. Finally, define
$$B = \left(\begin{array}{cccccc}
\frac{\zeta^2 - 1} 3  & \frac{-4 \zeta^2 + 1} 6 & -\frac 1 6 &  -\frac  \zeta 3 & -\frac{\zeta^3} 6 & \frac{-3 \zeta^3 + 4 \zeta}{6} \\
\frac{-2 \zeta^2 + 3}{6} & \frac{-\zeta^2}{3} & \frac{-2 \zeta^2 + 2}{3} & \frac{-\zeta^3}{6} & \frac{\zeta^3 + 2 \zeta}{6} & \frac{-\zeta^3}{6} \\
\frac{2 \zeta^2 - 5}{6} &  \frac{\zeta^2 - 1}{3} & \frac{1}{3} & \frac{-\zeta^3}{6} & \frac{\zeta^3 + 2 \zeta}{6} & \frac{-\zeta^3}{6} \\
\frac{-\zeta}{3} & \frac{-\zeta^3}{6} & \frac{-3 \zeta^3 + 4 \zeta}{6} & \frac{\zeta^2 - 1}{3} & \frac{-4 \zeta^2 + 1}{6} & \frac{-1}{6} \\
\frac{-\zeta^3}{6} & \frac{\zeta^3 + 2 \zeta}{6} & \frac{-\zeta^3}{6} & \frac{-2 \zeta^2 + 3}{6} & \frac{-\zeta^2}{3} & \frac{-2 \zeta^2 + 2}{3} \\
\frac{-\zeta^3}{6} & \frac{\zeta^3 + 2 \zeta}{6} & \frac{-\zeta^3}{6} & \frac{2 \zeta^2 - 5}{6} & \frac{\zeta^2-1}{3} & \frac{1}{3} 
\end{array}\right).$$
One  verifies that the assignments $(t, a, b) \mapsto (T, A, B)$ defines a representation $\rho \colon \widetilde \Gamma \to \mathrm U(6)$. Setting $C = T A^{-1} T$, one also verifies that the matrix $A^{-1}BCB$ has eigenvalues that are not roots of unity, so that the group $\langle T, A, B\rangle$ is infinite. The only two non-trivial invariant subspaces of the subgroup $\langle A, C\rangle$ are the $3$-dimensional subspaces respectively spanned by the first and the last $3$ vectors of the canonical basis. Since none of them is $B$-invariant, it follows that $\langle A, B, C\rangle$ acts irreducibly. Therefore, the restriction of $\rho$ to $\Gamma$ defines an irreducible unitary representation, as required. 
\end{proof}

\begin{rmk}\label{rem:deg6}
As in Remark~\ref{rem:CongruenceImage}, we can compute the first few congruence quotients of the linear group $\langle T, A, B\rangle$,  and deduce that the groups $G_4^{14, 14, 18}$ and $ G_4^{14, 14, 54}$ both have  finite simple quotients isomorphic to $\mathrm{PSp}_6(5)$, $\mathrm{PSp}_6(7)$, $\mathrm{PSp}_6(11)$,  $\mathrm{PSp}_6(13)$, 
$\mathrm{PSp}_6(17)$, $\mathrm{PSp}_6(19)$. Using Theorem~\ref{thm:NPC}(iv), one can derive that $G_4^{14, 14, 18}$ is virtually torsion-free. Notice that the  systematic searches for small finite simple quotients, and alternating quotients of small degree, of those groups, reported on in Section~\ref{sec:Tables}, did not identify any non-abelian finite simple quotient for $G_4^{14, 14, 18}$ and $G_4^{14, 14, 54}$. 
\end{rmk}

\subsection{On representations of hyperbolic quotients of $\PSL_2(\mathbf Z)$}
 
 The relative success of the representation variety approach we followed in the previous sections suggests to consider the following.

\begin{qu}\label{qu:RepVariety}
Let $d \geq 1$ be an integer, let $\omega = e^{2\pi i/3}$ and  $\mathcal R_d = \mathbf C[X_1, \dots, X_{9d^2}]$ be the polynomial ring in $9d^2$ indeterminates over $\mathbf C$. Let also $\widetilde \Gamma = \langle a, x \mid a^3\rangle \cong C_3 *\mathbf Z$ and   $\rho \colon \widetilde \Gamma\to \mathrm{GL}_{3d}\big(\mathcal{R}_d\big)$ be the representation of $\widetilde \Gamma$ defined by $\rho(a) = \left(\begin{array}{ccc} 
I_d & 0 & 0\\
0 & \omega I_d & 0\\
0 & 0 & \omega^2 I_d
\end{array}\right)$ and $\rho(x) = \left(\begin{array}{ccc} 
X_1 & \cdots & X_{3d}\\
\vdots & \ddots & \vdots\\
X_{9d^2-3d+1} & \cdots & X_{9d^2}
\end{array}\right)$, where $I_d$ denotes the $d\times d$-identity matrix. 

Let now $r_1, \dots, r_m \in \widetilde \Gamma$ be such that the quotient group $\Gamma = \widetilde \Gamma/\langle\!\langle x^2, r_1, \dots, r_m\rangle\!\rangle \cong \langle a, x \mid a^3,  x^2, r_1, \dots, r_m\rangle$ is  non-trivial and hyperbolic (in particular, $\Gamma$ is a hyperbolic quotient of  $C_3 * C_2 \cong \mathrm{PSL}_2(\mathbf Z)$). 
Let $\mathcal I_d$ be the ideal in $\mathcal R_d$ generated by the $(m+1)9d^2$ polynomials corresponding to the entries of the $m+1$ matrices in the set $\{\rho(x^2)-I_{3d}\} \cup \big\{\rho(r_j)-I_{3d} \mid j=1, \dots, m\big\}$.

Does there exist $d \geq 1$ such that the quotient ring $\mathcal R_d/\mathcal I_d$ is non-zero? 
\end{qu}

\begin{rmk}
Question~\ref{qu:RepVariety} is formally equivalent to the question whether every hyperbolic group is residually finite. Indeed, if the answer to Question~\ref{qu:RepVariety} is positive, then every non-trivial hyperbolic quotient of $\mathrm{PSL}_2(\mathbf Z)$ has a non-trivial finite-dimensional linear representation over $\mathbf C$, and hence a non-trivial finite quotient. Since every non-elementary hyperbolic group $G$  has a non-elementary hyperbolic quotient in common with $\mathrm{PSL}_2(\mathbf Z)$ by Olshanskii's Common Quotient Theorem, it follows that $G$ has a non-trivial finite quotient. It then follows that all hyperbolic groups are residually finite, see \cite[Theorem~1.2]{KaWi} or \cite[Theorem~2]{Olsh_2000}.  Conversely, if every hyperbolic group is residually finite, then the group $\Gamma$ from Question~\ref{qu:RepVariety} has a non-trivial finite quotient $Q$ in which the cyclic group $\langle a \rangle$ injects. In particular the order of $Q$ is $3d$ for some integer $d\geq 1$, and the image of $a$ in the regular representation of $Q$  is conjugate to the matrix $\left(\begin{array}{ccc} 
I_d & 0 & 0\\
0 & \omega I_d & 0\\
0 & 0 & \omega^2 I_d
\end{array}\right)$. Therefore the representation variety of $\Gamma$ whose coordinate ring is $\mathcal R_d/\mathcal I_d$, is non-empty, and hence the ring $\mathcal R_d/\mathcal I_d$ is non-zero.
\end{rmk}

The potential asset of the reformulation provided by Question~\ref{qu:RepVariety} stems from the possibility to approach the problem by Gr\"obner bases computations. Investigating Question~\ref{qu:RepVariety} for random quotients of $\mathrm{PSL}_2(\mathbf Z)$, such as those considered in \cite{Oll_GAFA}, would be highly interesting.

\section{Five-fold hyperbolic triangle groups with property (T)}\label{sec:five-fold}

As mentioned in Section~\ref{sec:TriangleGroups}, the only trivalent triangle groups from our sample for which Theorem~\ref{thm:EJ} applies and guarantees property (T) are the four Ronan groups. For the majority of the other groups, we could find a finite index subgroup with infinite abelianization and/or an isometric action on a real or complex hyperbolic space, which show that property (T) fails.

Let us also remark that we cannot expect  Theorem~\ref{thm:EJ} to  apply and guarantee property (T) for non-positively curved trivalent triangle group with a vertex group of very large order, since by Corollary~\ref{cor:Bounds}, one of the representation angles is bounded above by $\arccos(\frac{\sqrt 8} 3-\epsilon) \approx 19.47^\circ$. We consider this as evidence that there in order for an infinite hyperbolic $k$-fold generalized triangle group to have property (T) it is necessary that $k \ge 4$. 

\begin{rmk}
A possible approach to confirm that a hyperbolic $k$-fold generalized triangle groups with $k$ small cannot have property (T)  would be to show that the conformal dimension of the boundary of those hyperbolic groups is at most~$2$. This is known to be an obstruction to property (T), see \cite{Bourdon2016} and \cite[Theorem~1.3(3)]{BFGM}.
\end{rmk}

In this section we will see that hyperbolic $k$-fold generalized triangle groups with property (T) do exist for $k = 5$. The following remark clarifies how we obtained experimental evidence that such examples could be constructed using Theorem~\ref{thm:EJ}. 

\begin{rmk}
To a finite group $X$ and two subgroups $A$ and $B$ we can associate two kinds of angles: one is $2 \pi/ g_X(A,B)$ where $g_X(A,B)$ is the girth of the bipartite coset graph $\Gamma_X(A, B)$. The other is $\arccos \varepsilon_X(A,B)$. In order for a generalized triangle group associated to $A_0,A_1,A_2,X_0,X_1,X_2$ to be non-positively curved according to Theorem~\ref{thm:NPC} the sum over the three angles of the first kind needs to be $\le \pi$. In order for Theorem~\ref{thm:EJ} to guarantee property~(T), the sum over the three angles of the second kind needs to be $>\pi$. So in order for both properties to be satisfied, for at least one triple $(X,A,B) = (X_i,A_{i-1},A_{i+1})$ one needs
\begin{equation}\label{eq:girth_repangle}
\arccos \varepsilon_{X}(A,B) > \frac{2\pi}{g_{X}(A,B)}\text{.}
\end{equation}
Among groups of small order this is rarely satisfied. Indeed, numerical evidence suggests that the only groups $X$ of order $\abs{X} \le 2000$ that admit subgroups $A$, $B$ of order $5$ satisfying \eqref{eq:girth_repangle} are $\mathscr U_3(5)$, $\SL_2(\Fbb_5)$, (both with girth $6$ and representation angle $> 60^\circ$), $\mathscr U_4(5)$, $\SL_2(\Fbb_9)$, and a polycyclic group of order $800$ (all three with girth $8$ and representation angle $> 45^\circ$). In fact, Proposition~\ref{prop:SL2(p)}  provides certified estimates for $\SL_2(\Fbb _5)$ and $\SL_2(\Fbb_9)$. The groups $\mathscr U_3(5)$ and $\mathscr U_4(5)$ will be introduced and studied in Section~\ref{sec:KMS} below where their exact representation angle is determined.

From those experiments, it follows that the only candidates for being $5$-fold generalized triangle groups with all vertex groups of order~$\leq 2000$ that would both be hyperbolic and have property (T) by an application of Theorem~\ref{thm:EJ} would have to be of half girth type $(3, 3, 3)$. Moreover, all of their vertex groups would be isomorphic to $\mathscr U_3(5)$ or $\SL_2(\Fbb_5)$. We shall see in Section~\ref{sec:KMS} below that, up to isomorphism, there is only one non-positively curved $5$-fold triangle group with all vertex groups isomorphism to $\mathscr U_3(5)$. That group has (T), but it  is \textit{not} hyperbolic (see Proposition~\ref{prop:ZxZ}). 
\end{rmk}

Nonetheless, it turns out that if we allow (much) larger vertex groups, then we can indeed construct hyperbolic $5$-fold generalized triangle groups with property (T). 

\begin{prop}\label{prop:5-fold}
Let $X$ be a finite group generated by two elements $a, b$ of order~$5$. Let $A = \langle a \rangle$ and $B = \langle b \rangle$. Assume that:
\begin{enumerate}[(i)]
    \item The girth of the coset graph $\Gamma = \Gamma_X(A, B)$ is at least~$14$. 
    \item $\varepsilon_X(A, B)< \frac{2\sqrt 5} 5$. 
\end{enumerate}
Then every generalized $5$-fold triangle group with vertex groups respectively isomorphic to $X$, $C_5 \times C_5$ and the Heisenberg group over $\mathbf F_5$, is infinite hyperbolic with Kazhdan's property (T). 
\end{prop}
\begin{proof}
Let $G(\mathcal T)$ be $5$-fold triangle group as in the statement. The coset graph of $C_5 \times C_5$ with respect to a generating pair of cyclic subgroups of order~$5$ is the complete bipartite graph $K_{5, 5}$. Its girth is $4$. Moreover, the corresponding representation angle is $\pi/2$, see Example~\ref{ex:0}. 

We shall see in Proposition~\ref{prop:Girth:Moufang} that the coset graph of the Heisenberg group over $\mathbf F_p$ with respect to a generating pair of cyclic subgroups of order~$p$ is of girth $6$. Moreover, the cosine of the corresponding representation angle is $\frac {\sqrt p} p$ by Example~\ref{ex:2}. 

Assume that (i) and (ii) hold. The half-girth type of $G(\mathcal T)$ is $(2, 3, r)$, where $r$ is half the girth of $\Gamma$. Therefore the fundamental group $\widehat{G(\mathcal T)}$ is infinite hyperbolic by Theorem~\ref{thm:NPC}, in view of (i).

Since $\varepsilon_X(A, B) <  \frac{2\sqrt 5} 5$ by hypothesis,  it follows from Theorem~\ref{thm:EJ} that $\widehat{G(\mathcal T)}$ has (T). 
\end{proof}

Analogously one can verify using Proposition~\ref{prop:AngleUnipotent} below:

\begin{prop}\label{prop:5-fold_bis}
Let $X$ be a finite group generated by two elements $a, b$ of order~$5$. Let $A = \langle a \rangle$ and $B = \langle b \rangle$. Assume that:
\begin{enumerate}[(i)]
    \item The girth of the coset graph $\Gamma = \Gamma_X(A, B)$ is at least~$10$.
    \item $\varepsilon_X(A, B)< \frac{\sqrt{15}}{5}$.
\end{enumerate}
Let $G(\mathcal T)$ be a generalized $5$-fold triangle group generated with vertex groups isomorphic to $X$, $C_5 \times C_5$ and the $5$-Sylow subgroup $\mathscr  U_4(5)$ of $\Sp_2(5)$, respectively. If the edge groups embed into $\mathscr  U_4(5)$ as groups that do not both commute with their commutator subgroup then $G$ is infinite hyperbolic with Kazhdan's property (T). 
\end{prop}
\begin{proof}
Again $C_5 \times C_5$ has a coset graph of girth $4$ and a representation angle of $\pi/2$. The group $\mathscr U_4(p)$ with respect to the described subgroups has as coset graph of girth $8$ by Proposition~\ref{prop:Girth:Moufang} and a representation angle of $\arccos(\sqrt{2/p})$ by Proposition~\ref{prop:PresentationsUnipotent}.

The half-girth type of $G(\mathcal T)$ is $(2, 4, r)$, where $r \ge 5$ by (i). Therefore the fundamental group $\widehat{G(\mathcal T)}$ is infinite hyperbolic by Theorem~\ref{thm:NPC}.

The cosines of the representation angles are $(0,\sqrt{2/5},\varepsilon)$ where $\varepsilon < \sqrt{3/5}$ by (ii). It follows from Theorem~\ref{thm:EJ} that $\widehat{G(\mathcal T)}$ has (T). 
\end{proof}

Two examples of triples $(X, A, B)$ satisfying the hypotheses of Proposition~\ref{prop:5-fold} are given by the cases $p = 109$ and $p = 131$ of Proposition~\ref{prop:PSL2(p)}. Two examples of triples $(X,A,B)$ satisfying the hypotheses of Proposition~\ref{prop:5-fold_bis} are given by the cases $p = 31$ and $p = 41$ the same proposition. This leads to several hyperbolic $5$-fold generalized triangle groups with property (T). Two of these feature in Theorem~\ref{thm:5-fold_with_(T)} from the introduction, whose proof can now be completed.

\begin{proof}[Proof of Theorem~\ref{thm:5-fold_with_(T)}]
The presentations of $\mathscr H_p$ make it  clear that the groups are $5$-fold generalized triangle group.

Let $L =\grp{a, b \mid a^5, b^5, R}$, where $R$ denotes the set consisting of the $7$ relators of $\mathscr H_{109}$ involving both $a$ and $b$. The following procedure allows one to verify with \textsc{Magma} that $L$  is isomorphic $X \cong \mathrm{PSL}_2(109)$. First, one computes that the assignments
\[
a \mapsto 
\left(\begin{array}{cc}
  0   & 1 \\
-1     & 11
\end{array}\right)
\quad \text{and}\quad
b \mapsto 
\left(\begin{array}{cc}
  57   & 2 \\
52     & 42
\end{array}\right)
\]
define a surjective homomorphism $L \to X \cong \mathrm{PSL}_2(109)$. On the other hand, the \textsc{Magma} command \texttt{\#L} confirms that $L$ is a finite group of order $647460 = \abs{\mathrm{PSL}_2(109)}$. The required assertion follows.

By Proposition~\ref{prop:PSL2(p)}, the girth of $\Gamma_X(A, B)$ is $14$ and $5\varepsilon_X(A, B) < 2\sqrt 5 \approx 4.47213595$. The conclusion for $\mathscr H_{109}$ follows from Proposition~\ref{prop:5-fold}.

Similarly if $a,b,X,A,B$ are as in the $p = 31$ case of Proposition~\ref{prop:PSL2(p)} one verifies that the group presented by the generators and relations of $\mathscr H_{31}$ involving only $a$ and $b$ is isomorphic to $X \cong \PSL_2(31)$ where the generators are represented by matrices with the same letter. We conclude that the coset graph has girth $14$ and that $5\varepsilon_X(A,B) < \sqrt{15} \approx 3.8729833462$. Let $M = \grp{b,c \mid S}$ where $S$ are the relations of $\mathscr H_{31}$ involving only $b$ and $c$. We will see in the next section that $M \cong \mathscr U_4(5)$ and that $[b,c]$ does not commute with $c$. Therefore we can conclude with Proposition~\ref{prop:5-fold_bis}.
\end{proof}

\begin{rmk}
Clearly, the method used above yields more examples of  hyperbolic $5$-fold generalized triangle groups with property (T) than those recorded in Theorem~\ref{thm:5-fold_with_(T)}. We may indeed use $\PSL_2(41)$ instead of $\PSL_2(31)$, and $\PSL_2(131)$ instead of $\PSL_2(109)$, as a consequence of Proposition~\ref{prop:PSL2(p)}. Moreover, for each triple of vertex groups, we can take advantage of the freedom we have in defining the homomorphisms identifying an edge group to a generator of the vertex group containing it. However, we have not been able to construct an \textit{infinite family} of hyperbolic $5$-fold generalized triangle groups with property (T).
\end{rmk}

\section{Kac--Moody--Steinberg groups of rank~$3$}\label{sec:KMS}

\subsection{Coset graphs from Moufang polygons}

The Pappus graph and the Gray graph are members of an infinite family of graphs, some of which were considered in \cite{LU93}. Those can be constructed as follows. 

Given two elements $x, y$ in a group $G$, we denote by $[x, y] = x^{-1}y^{-1}xy$ their commutator. Given subgroups $A, B\leq G$, we denote by $[A, B]$ the subgroup generated by all elements of the form $[a, b]$ with $a \in A$ and $b \in B$. 

Let $n \in \{3, 4, 6, 8\}$ and $\mathcal G$ be a Moufang $n$-gon. Let $(U, U_1, \dots, U_n)$ be the \textbf{root group sequence} associated with $\mathcal G$, as defined in \cite[(8.10)]{TW}. This means in particular that $U$ is a group, and that $U_i $ is a subgroup of $U$ for each $i$, that is called a \textbf{root group}. Moreover the product map $U_1 \times \dots \times U_n \to U$ is bijective, and for all $i<j$ we have $[U_i, U_j ] \leq U_{i+1} \dots U_{j-1}$. 

\begin{prop}\label{prop:Girth:Moufang}
The (possibly disconnected) bipartite coset graph $\Gamma_U(U_1, U_n)$ is the subgraph of the generalized $n$-gon $\mathcal G$ spanned by all the edges opposite the unique edge fixed by $U$. In particular the girth of $\Gamma_U(U_1, U_n)$ is~$2n$. 
\end{prop}

\begin{proof}
The proof, formulated in a special case in  \cite[Theorem~3.1 and Proposition~3.2]{LU93}, applies in full generality. 
\end{proof}

We now focus on two specific examples. 

Let $q$ be a power of a prime $p$. We denote by $\mathscr U_3(q)$ a copy of the $p$-Sylow subgroup in $\mathrm{SL}_3(q)$ and by $\mathscr  U_4(q)$ a copy of the $p$-Sylow subgroup in $\mathrm{Sp}_4(q)$. In those groups, each of the root groups $U_1, \dots, U_n$ is isomorphic to the additive group of a field $k$ of order~$q$. Denoting by $x_i \colon k \to U_i$ an isomorphism, then the non-triv
ial commutation relations between the root subgroups of $\mathscr U_n(q)$ are as follows, for all $a, b \in k$ (see \cite[(16.2)]{TW}):
$$[x_1(a), x_3(b)] = x_2(ab)$$ 
if $n=3$, and 
\begin{align*}
[x_2(a), x_4(b)^{-1}]  &= x_3(2ab)\\
[x_1(a), x_4(b)^{-1}]  &= x_2(ab)x_3(ab^2) 
\end{align*}
if $n=4$ (the  root subgroups  not involved by those relations commute). It is easy to see from those commutation relations that the group $\mathscr U_3(q)$ is generated by $U_1$ and $U_3$, and similarly that if $q>2$, then $\mathscr U_4(q)$ is generated by $U_1$ and $U_4$. Hence the graphs $\Gamma_{\mathscr U_3(q)}(U_1, U_3)$ (resp.  $\Gamma_{\mathscr U_4(q)}(U_1, U_4)$
 with $q>2$) are connected. 
 The group $\mathscr U_3(q)$ is nothing but a Heisenberg group over $k$. 
 
 In the special case where $q$ is a prime, the groups $\mathscr U_3(q)$ and $\mathscr  U_4(q)$  admit the following presentations, where the symbol $[x_1, \dots, x_n]$ denotes the $n$-th commutator $[[[x_1, x_2], \dots], x_n]$. 

\begin{prop}\label{prop:PresentationsUnipotent}
Let $p$ be a prime. Then:
\begin{enumerate}[(i)]
\item $\mathscr U_3(p) \cong \langle a, b \mid a^p, b^p, [a, b, a], [a, b, b]\rangle.$ 

\item If $p>2$, then $\mathscr U_4(p) \cong \langle a, b \mid a^p, b^p, [a, b, a], [a, b, b, a], [a, b, b, b]\rangle.$
\end{enumerate}

\end{prop}
\begin{proof}
That $\langle a, b \mid a^p, b^p, [a, b, a], [a, b, b]\rangle$ is a presentation of the Heisenberg group over $\mathbf F_p$ is well known and easy to see. 

Let $U = \langle a, b \mid a^p, b^p, [a, b, a], [a, b, b, a], [a, b, b, b]\rangle.$ Observe that $z = [a, b, b]$ commutes with $a$ and $b$, and is thus a central element of $U$. The quotient $U /\langle z \rangle$ is isomorphic to $ \langle a, b \mid a^p, b^p, [a, b, a], [a, b, b]\rangle \cong \mathscr U_3(p)$. Moreover, since $z$ is central, it follows that $z^n = [a, b, b^n]$ for all $n$. In particular $z^p=1$. We infer that $|U| \leq p^4$. 

On the other hand, the assignment $a \mapsto x_1(1)$ and $b \mapsto x_4(-1)$ extends to a homomorphism $\varphi \colon U \to \mathscr U_4(p)$. Since $p>2$, it follows from the commutation relations described above that  $\mathscr U_4(p)$ is generated by $U_1$ and $U_4$, so that $\varphi$ is surjective. Since $|\mathscr U_4(p)| = p^4$, we deduce from the previous paragraph that $\varphi$ is an isomorphism. 
\end{proof}

We next compute the representation angles. 

\begin{prop}\label{prop:AngleUnipotent}
Let $p$ be a prime. Then:
\begin{enumerate}[(i)]
\item $\varepsilon_{\mathscr  U_3(p)}(U_1, U_3)=1/\sqrt p$. 

\item If $p>2$, then $\varepsilon_{\mathscr  U_4(p)}(U_1, U_4)=\sqrt{2/p}$. 
\end{enumerate}
In particular, the graphs $\Gamma_{\mathscr U_3(p)}(U_1, U_3)$ and 
$\Gamma_{\mathscr U_4(p)}(U_1, U_4)$ are Ramanujan graphs. 
\end{prop}
\begin{proof}
For the first assertion, see \cite[\S 4]{EJ}. We focus on $\mathscr  U_4(p)$.

Let $a =  x_1(1)$ and $b = x_4(-1)$. Let $W$ be the subgroup generated $U_1 \cup U_2 \cup U_3$. Thus $W$ is abelian of order~$p^3$ and $\mathscr  U_4(p)$ is a semi-direct product $W \rtimes U_4$. We view $W$ as a vector space over~$\mathbf F_p$. In view of Proposition~\ref{prop:PresentationsUnipotent}, we see that the action of $b$ on $W$  is represented by the matrix
$$B =  \left(
\begin{array}{ccc}
1 & 1 & 0\\
0 & 1 & 1\\
0 & 0 & 1
\end{array}
\right)$$
with respect to the basis $([a, b, b], [a, b], a)$. 

We  claim that every irreducible representation of $U = \mathscr  U_4(p)$ is either of degree $1$, or is a representation of degree $p$ induced by a	 degree~$1$ representation of $W$. Indeed, 

Let $\chi$ be a character of $W$ and $\pi = \mathrm{Ind}_W^{U}(\chi)$ be the representation of $U$  induced by $\chi$. The representation $\pi$ can be realized as follows. Let $(e_0, e_1, \dots, e_{p-1})$ denote a basis of $\mathbf C^p$. Then $\pi(b)e_n = e_{n+1}$ (where the indices are taken modulo $p$), and $\pi(a) e_n = \chi(b^{-n}ab^n)e_n$. 

According to \cite[Corollary to Proposition~23]{Serre}, the representation $\pi = \mathrm{Ind}_W^{U}(\chi)$ is irreducible if and only if the caracters $\chi^{b^n} \colon W \to \mathbf C^* : w \mapsto \chi(b^{-n} w b^n)$ are pairwise distinct for distinct values of $n  \mod p$. In particular, if $\chi([a, b]) \neq 1$ or $\chi([a, b, b]) \neq 1$, then $\pi$ is irreducible. The number  characters of $W$ satisfying that condition is $p^3 - p$. Those characters from $p^2-1$ orbit under the action of $\langle b \rangle$ by automorphisms. Thus we obtain $p^2-1$ pairwise non-isomorphic irreducible representations of $U$ in this way. Since the abelianization of  $U$ has order~$p^2$, it follows that $U$ has also $p^2$ representations of degree $1$. Since $p^4 = p^2+ (p^2-1)p^2$, it follows that every irreducible representation of $U$ is either of degree~$1$, or is of the form $\pi = \mathrm{Ind}_W^{U}(\chi)$ for some character $\chi$ of $W$ such that $\chi([a, b]) \neq 1$ or $\chi([a, b, b]) \neq 1$. Therefore, in order to finish the proof, we may fix such a representation $\pi$ and show that $\varepsilon_U(U_1, U_4; \pi) \leq \sqrt{2/p}$. 

The fixed-point space of $U_4 = \langle b \rangle $ under the representation $\pi$ is the one-dimensional subspace spanned by $\sum_{n=0}^{p_1} e_n$. Since the action of $U_1 = \langle a \rangle$ under $\pi$ is diagonal in the basis $(e_0, \dots, e_{p-1})$, the fixed-point space of $U_1$ is spanned by those  $e_n$ which are fixed by $a$. We have $\pi(a)e_n = \chi(b^{-n} a b^n) e_n$. Moreover $\chi(b^{-n} a b^n)= \chi(a [a, b]^n [a, b, b]^{n(n-1)/2}) = \chi(a) \chi([a, b])^n \chi([a, b, b])^{n(n-1)/2}$. We know that $\chi(a)$, $\chi([a, b])$ and $\chi([a, b, b])$ are three $p^\mathrm{th}$ roots of unity, and moreover $\chi([a, b])$ and $\chi([a, b, b])$ are not both equal to~$1$. The number of values of $n \mod p$ such that $\chi(b^{-n} a b^n)=1$ is thus the number of solutions of a quadratic equation in the prime field of order~$p$. Hence the fixed point space of $a$ has dimension~$\leq 2$. If its dimension is~$0$ we have $\varepsilon_U(U_1, U_4; \pi)  =0$.  If its dimension is $1$, it is spanned by some $e_n$ and we obtain
$\varepsilon_U(U_1, U_4; \pi)  = \frac{|\langle e_n, \sum_k e_k\rangle |}{\| e_n\| \|\sum_k e_k\|} = 1/\sqrt p$. If its dimension is $2$, it is spanned by a pair $e_m, e_n$ and we obtain
\begin{align*}
\varepsilon_U(U_1, U_4; \pi)   = &  \sup\{\frac{|\langle \lambda e_m +\mu e_n, \sum_k e_k\rangle |}{\|  \lambda e_m +\mu e_n\| \|\sum_k e_k\|}  \mid \lambda, \mu \in \mathbf C, \ (\lambda, \mu) \neq (0, 0)\}\\
= & \sup\{\frac{| \lambda +\mu  |}{ \sqrt{|\lambda|^2 + |\mu^2|}\sqrt p}  \mid \lambda, \mu \in \mathbf C , \ (\lambda, \mu) \neq (0, 0)\}\\
= & \sqrt{2/ p} 
\end{align*}
by the Cauchy--Schwarz inequality. 
Therefore $\varepsilon_U(U_1, U_4) = \sqrt{2/ p}$.

The last assertion now follows from Corollary~\ref{cor:Ramanujan}. 
\end{proof}

\begin{rmk}
For $p$ odd and $q$ any power of $p$, the estimate  
$$\varepsilon_{\mathscr  U_4(q)}(U_1, U_4) \leq \sqrt{\frac{ 1 +\sqrt p } p }$$ 
can be extracted from the work of Ershov--Rall \cite{ER18}. 
\end{rmk}

\subsection{Kac--Moody--Steinberg groups of rank~$3$}\label{sec:KMS-presentations}

We now show that all the trivalent triangle groups considered in Section~\ref{sec:TriangleGroups}, all of whose vertex groups are $3$-groups, belong to a broad infinite family that contains numerous examples of infinite hyperbolic Kazhdan groups.

Let  $G(\mathcal T)$ be a triangle of groups with all edge groups cyclic of order~$p$, and the vertex group $X_i$ isomorphic to   $\mathscr U_{r_i}(p)$  for $r_i \in \{2, 3,4\}$, where $i = 0, 1, 2$ and where $\mathscr U_{2}(p)$ is defined as the direct product $C_p \times C_p$. Let $G = \widehat{G(\mathcal T)}$ be its fundamental group. Thus $G$ is a generalized $p$-fold triangle group. In view of   Theorem~\ref{thm:NPC} and Proposition~\ref{prop:Girth:Moufang}, the half girth type of  $G(\mathcal T)$ is $(r_0, r_1, r_2)$. We denote by $\mathcal C$ the collection of all non-positively curved generalized triangle groups obtained in this way.

Corollary~\ref{cor:isomorphism} allows us to classify the groups in $\mathcal C$ up to isomorphism. In order to facilitate the statement, we introduce additional notation. Let us first list $10$ Dynkin diagrams of rank~$3$, see Figure~\ref{fig:Dynkin}. The notation for the diagrams of affine type, namely $ \widetilde{A_2}$, $\widetilde{B_2}$, $\widetilde{C_2}$ and $\widetilde{BC_2}$, is standard. The notation for the $6$ other diagrams is inspired by Kac--Moody theory;  the Weyl group corresponding to each of those $6$ diagrams is a hyperbolic triangle Fuchsian group. 

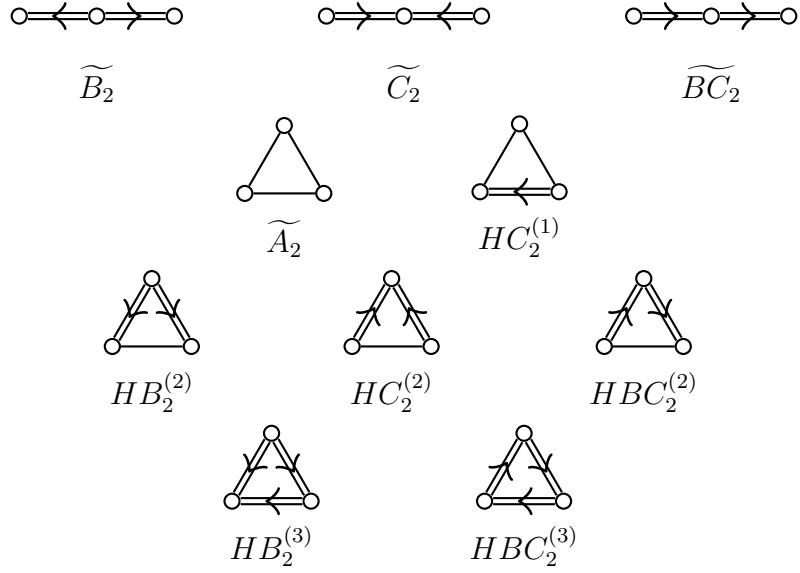
\begin{figure}[h]
\begin{tikzpicture}[scale=.6]
\node[circ,thick] (l) at (-1.7, 0) {};
\node[circ,thick] (m) at (0, 0) {};
\node[circ,thick] (r) at (1.7, 0) {};
\path (l) edge[double distance=1.5pt,thick] (m);
\path (l) edge[-<-,color=white] (m);
\path (m) edge[double distance=1.5pt,thick] (r);
\path (m) edge[->-,color=white] (r);
\node at (0,-1.5) {$\widetilde{B_2}$};
\end{tikzpicture}
\hspace{1.5cm}
\begin{tikzpicture}[scale=.6]
\node[circ,thick] (l) at (-1.7, 0) {};
\node[circ,thick] (m) at (0, 0) {};
\node[circ,thick] (r) at (1.7, 0) {};
\path (l) edge[double distance=1.5pt,thick] (m);
\path (l) edge[->-,color=white] (m);
\path (m) edge[double distance=1.5pt,thick] (r);
\path (m) edge[-<-,color=white] (r);
\node at (0,-1.5) {$\widetilde{C_2}$};
\end{tikzpicture}
\hspace{1.5cm}
\begin{tikzpicture}[scale=.6]
\node[circ,thick] (l) at (-1.7, 0) {};
\node[circ,thick] (m) at (0, 0) {};
\node[circ,thick] (r) at (1.7, 0) {};
\path (l) edge[double distance=1.5pt,thick] (m);
\path (l) edge[->-,color=white] (m);
\path (m) edge[double distance=1.5pt,thick] (r);
\path (m) edge[->-,color=white] (r);
\node at (0,-1.5) {$\widetilde{BC_2}$};
\end{tikzpicture}

\begin{tikzpicture}[scale=.6]
\node[circ,thick] (t) at (90:1) {};
\node[circ,thick] (l) at (210:1) {};
\node[circ,thick] (r) at (330:1) {};
\draw (t) edge[thick] (l);
\draw (t) edge[thick] (r);
\draw (l) edge[thick] (r);
\node at (0,-1.5) {$ \widetilde{A_2}$};
\end{tikzpicture}
\hspace{1.5cm}
\begin{tikzpicture}[scale=.6]
\node[circ,thick] (t) at (90:1) {};
\node[circ,thick] (l) at (210:1) {};
\node[circ,thick] (r) at (330:1) {};
\draw (t) edge[thick] (l);
\draw (t) edge[thick] (r);
\path (l) edge[double distance=1.5pt,thick] (r);
\path (l) edge[-<-,color=white] (r);
\node at (0,-1.5) {$HC_2^{(1)}$};
\end{tikzpicture}

\begin{tikzpicture}[scale=.6]
\node[circ,thick] (t) at (90:1) {};
\node[circ,thick] (l) at (210:1) {};
\node[circ,thick] (r) at (330:1) {};
\draw (t) edge[double distance=1.5pt,thick] (l);
\draw (t) edge[->-,color=white] (l);
\draw (t) edge[double distance=1.5pt,thick] (r);
\draw (t) edge[->-,color=white] (r);
\path (l) edge[thick] (r);
\node at (0,-1.5) {$HB_2^{(2)}$};
\end{tikzpicture}
\hspace{1.5cm}
\begin{tikzpicture}[scale=.6]
\node[circ,thick] (t) at (90:1) {};
\node[circ,thick] (l) at (210:1) {};
\node[circ,thick] (r) at (330:1) {};
\draw (t) edge[double distance=1.5pt,thick] (l);
\draw (t) edge[-<-,color=white] (l);
\draw (t) edge[double distance=1.5pt,thick] (r);
\draw (t) edge[-<-,color=white] (r);
\path (l) edge[thick] (r);
\node at (0,-1.5) {$HC_2^{(2)}$};
\end{tikzpicture}
\hspace{1.5cm}
\begin{tikzpicture}[scale=.6]
\node[circ,thick] (t) at (90:1) {};
\node[circ,thick] (l) at (210:1) {};
\node[circ,thick] (r) at (330:1) {};
\draw (t) edge[double distance=1.5pt,thick] (l);
\draw (t) edge[double distance=1.5pt,thick] (r);
\draw (t) edge[->-,color=white] (r);
\draw (t) edge[-<-,color=white] (l);
\path (l) edge[thick] (r);
\node at (0,-1.5) {$HBC_2^{(2)}$};
\end{tikzpicture}

\begin{tikzpicture}[scale=.6]
\node[circ,thick] (t) at (90:1) {};
\node[circ,thick] (l) at (210:1) {};
\node[circ,thick] (r) at (330:1) {};
\draw (t) edge[double distance=1.5pt,thick] (l);
\draw (t) edge[->-,color=white] (l);
\draw (t) edge[double distance=1.5pt,thick] (r);
\draw (t) edge[->-,color=white] (r);
\path (l) edge[double distance=1.5pt,thick] (r);
\path (l) edge[-<-,color=white] (r);
\node at (0,-1.5) {$HB_2^{(3)}$};
\end{tikzpicture}
\hspace{1.5cm}
\begin{tikzpicture}[scale=.6]
\node[circ,thick] (t) at (90:1) {};
\node[circ,thick] (l) at (210:1) {};
\node[circ,thick] (r) at (330:1) {};
\draw (t) edge[double distance=1.5pt,thick] (l);
\draw (t) edge[-<-,color=white] (l);
\draw (t) edge[double distance=1.5pt,thick] (r);
\draw (t) edge[->-,color=white] (r);
\path (l) edge[double distance=1.5pt,thick] (r);
\path (l) edge[-<-,color=white] (r);
\node at (0,-1.5) {$HBC_2^{(3)}$};
\end{tikzpicture}
\caption{Ten Dynkin diagrams of rank~$3$}
\label{fig:Dynkin}
\end{figure}

To each of those Dynkin diagrams and to every odd prime  $p$, we associate a finitely presented group, as follows.

\begin{align*}
\mathscr G_{\widetilde{B_2}}(p) = \langle a, b, c    \mid & a^p, b^p, c^p, [a, b], \\
 & [c, b, c], [c, b, b, c], [c, b, b, b ], [c, a, c], [c, a, a, c], [c, a, a, a]\rangle.
\end{align*}
\begin{align*}
\mathscr G_{\widetilde{C_2}}(p) = \langle a, b, c    \mid & a^p, b^p, c^p, [a, b],  \\
 & [b, c, b], [b, c, c, b], [b, c, c, c ], [a, c, a], [a, c, c, a], [a, c, c, c]\rangle.
\end{align*}
\begin{align*}
\mathscr G_{\widetilde{BC_2}}(p) = \langle a, b, c    \mid & a^p, b^p, c^p, [a, b], \\
 & [b, c, b], [b, c, c, b], [b, c, c, c ], [c, a, c], [c, a, a, c], [c, a, a, a]\rangle.
\end{align*}
\begin{align*}
\mathscr G_{ \widetilde{A_2}}(p) = \langle a, b, c  \mid  & a^p, b^p, c^p, [a, b, a], [a, b, b], \\
& [b, c, b], [b, c, c], [a, c, a], [a, c, c]\rangle.
\end{align*}
\begin{align*}
\mathscr G_{HC_2^{(1)}}(p) = \langle a, b, c  \mid  & a^p, b^p, c^p, [a, b, a], [a, b, b], \\
& [b, c, b], [b, c, c],  [a, c, a], [a, c, c, a], [a, c, c, c]\rangle.
\end{align*}
\begin{align*}
\mathscr G_{HB_2^{(2)}}(p) = \langle a, b, c    \mid & a^p, b^p, c^p, [a, b, a], [a, b, b], \\
 & [c, b, c], [c, b, b, c], [c, b, b, b ], [c, a, c], [c, a, a, c], [c, a, a, a]\rangle.
\end{align*}
\begin{align*}
\mathscr G_{HC_2^{(2)}}(p) = \langle a, b, c    \mid & a^p, b^p, c^p, [a, b, a], [a, b, b], \\
 & [b, c, b], [b, c, c, b], [b, c, c, c ], [a, c, a], [a, c, c, a], [a, c, c, c]\rangle.
\end{align*}
\begin{align*}
\mathscr G_{HBC_2^{(2)}}(p) = \langle a, b, c    \mid & a^p, b^p, c^p, [a, b, a], [a, b, b], \\
 & [b, c, b], [b, c, c, b], [b, c, c, c ], [c, a, c], [c, a, a, c], [c, a, a, a]\rangle.
\end{align*}
\begin{align*}
\mathscr G_{HB_2^{(3)}}(p) = \langle a, b, c    \mid & a^p, b^p, c^p, [a, b, a], [a, b, b, a], [a, b, b, b ], \\
 & [b, c, b], [b, c, c, b], [b, c, c, c ], [a, c, a], [a, c, c, a], [a, c, c, c]\rangle.
\end{align*}
\begin{align*}
\mathscr G_{HBC_2^{(3)}}(p) = \langle a, b, c    \mid & a^p, b^p, c^p,  [a, b, a], [a, b, b, a], [a, b, b, b ], \\
  & [b, c, b], [b, c, c, b], [b, c, c, c ], [c, a, c], [c, a, a, c], [c, a, a, a]\rangle.
\end{align*}

\begin{rmk}\label{rem:KMS-over-GF3}
Setting $p = 3$, we obtain $10$ trivalent triangle groups that belong to the sample from Section~\ref{sec:TriangleGroups} (see the presentations in Section~\ref{sec:PresentationList}). The isomorphisms are as follows:
$$\mathscr G_{\widetilde{B_2}}(3) \cong G^{6, 54, 54}_2,
\hspace{.75cm}
\mathscr G_{\widetilde{C_2}}(3) \cong G^{6, 54, 54}_8,
\hspace{.75cm}
\mathscr G_{\widetilde{BC_2}}(3)\cong G^{6, 54, 54}_0,$$
$$\mathscr G_{ \widetilde{A_2}}(3) \cong G^{18, 18, 18}_0,
\hspace{.75cm}
\mathscr G_{HC_2^{(1)}}(3) \cong G^{18, 18, 54}_0,$$
$$
\mathscr G_{HB_2^{(2)}}(3) \cong G^{18, 54, 54}_2,
\hspace{.75cm}
\mathscr G_{HC_2^{(2)}}(3) \cong G^{18, 54, 54}_8,
\hspace{.75cm}
\mathscr G_{HBC_2^{(2)}}(3) \cong G^{18, 54, 54}_0,$$
$$\mathscr G_{HB_2^{(3)}}(3) \cong G^{54, 54, 54}_2 
\hspace{.75cm} \text{and} \hspace{.75cm}
\mathscr G_{HBC_2^{(3)}}(3) \cong G^{54, 54, 54}_0.$$
\end{rmk}

The following result implies in particular that for a fixed $p$, the $10$ groups above are pairwise non-isomorphic. 

\begin{prop}\label{prop:Classif_KMS}
Let $p, p'$ be odd primes. Let   $G, G' \in \mathcal C$ be $p$- and $p'$-fold triangle groups of type $(r_0, r_1, r_2)$ and $(r'_0, r'_1, r'_2)$ respectively,  with $ r_0 \leq r_1 \leq r_2 $ and  $r'_0 \leq r'_1 \leq r'_2$. 

If $G \cong G'$, then $p=p'$ and $(r_0, r_1, r_2) = (r'_0, r'_1, r'_2)$. Moreover, exactly one of the following assertions holds:
\begin{itemize}
\item  $(r_0, r_1, r_2) = (2, 4, 4)$,  and 
$G \cong \mathscr G_{\widetilde{B_2}}(p)$ or $G \cong  \mathscr G_{\widetilde{C_2}}(p)$ or $G \cong \mathscr G_{\widetilde{BC_2}}(p)$.

\item  $(r_0, r_1, r_2) = (3, 3, 3)$ and 
$G \cong \mathscr G_{\widetilde{A_2}}(p)$.

\item  $(r_0, r_1, r_2) = (3, 3, 4)$ and 
$G \cong \mathscr G_{HC_2^{(1)}}(p)$.

\item  $(r_0, r_1, r_2) = (3, 4, 4)$,  and 
$G \cong \mathscr G_{HC_2^{(2)}}(p)$ or $G \cong  \mathscr G_{HB_2^{(2)}}(p)$ or $G \cong \mathscr G_{HBC_2^{(2)}}(p)$.

\item  $(r_0, r_1, r_2) = (4, 4, 4)$, and 
$G \cong \mathscr G_{HB_2^{(3)}}(p)$ or $G \cong \mathscr G_{HBC_2^{(3)}}(p)$.
\end{itemize}

\end{prop}
\begin{proof}
The first assertion is a straightforward consequence of Corollary~\ref{cor:isomorphism}. For the second, we use the following properties of $\mathscr U_3(p)$ and $\mathscr U_4(p)$:
\begin{itemize}
\item The subgroup of $\Aut(\mathscr U_3(p))$ which stabilizes the pair $\{U_1, U_3\}$ is isomorphic to the wreath product $C_{p-1} \wr C_2$. 
\item The subgroup of $\Aut(\mathscr U_4(p))$ which stabilizes the pair $\{U_1, U_4\}$ is isomorphic to the direct product $C_{p-1} \times C_{p-1}$. In particular no automorphism of $\mathscr U_4(p)$ swaps $U_1$ and $U_4$.
\end{itemize}
Indeed, it is easily verified    using Proposition~\ref{prop:PresentationsUnipotent} that for  any $r \in \{1, \dots, p-1\}$, the assignments $(a, b) \mapsto (a^r, b)$ and  $(a, b) \mapsto (a, b^r)$  both   extend to automorphisms of $\mathscr U_3(p)$ (resp. $\mathscr U_4(p)$). In particular, the subgroup of $\Aut(\mathscr U_3(p))$ (resp. $\mathscr U_4(p)$) which stabilizes the pair $\{U_1, U_3\}$ (resp. $\{U_1, U_4\}$) is contains $\Aut(U_1) \times \Aut(U_3)$  (resp. $\Aut(U_1) \times \Aut(U_4)$), which is isomorphic to $C_{p-1}^2$. Moreover, the assignment $(a, b) \mapsto (b, a)$ extends to an automorphism of $ \mathscr U_3(p)$, whereas no automorphism of $\mathscr U_4(p)$ swaps $\langle a \rangle = U_1$ and $\langle b \rangle = U_4$. The two properties listed above follow. 

In view of those assertions, the required conclusion is a consequence of  Corollary~\ref{cor:isomorphism}. 
\end{proof}

Following the  terminology in \cite{EJ}, we say that a $p$-fold triangle group  $G \in \mathcal C$ is a \textbf{Kac--Moody--Steinberg group}  or \textbf{KMS group} for short) over the field $\mathbf F_p$ of order $p$.

\begin{thm}\label{label:thm:KMS}
Let $G$ be a KMS group of half girth type $(r_0, r_1, r_2)$ over $\mathbf F_p$. Then $G$ is acylindircally hyperbolic, and if  $1/r_0+1/r_1+1/r_2<1$, then $G$ is hyperbolic. Moreover:
\begin{enumerate}[(i)]
\item If $(r_0, r_1, r_2) = (2, 4, 4)$, then $G$ has property (T) if and only if $p\geq 5$.

\item If $(r_0, r_1, r_2) = (3, 3, 3)$, then $G$ has property (T) if and only if $p\geq 5$.

\item If $(r_0, r_1, r_2) = (3, 3, 4)$, then $G$ has property (T) for all  $p\geq 7$.

\item If $(r_0, r_1, r_2) = (3, 4, 4)$, then $G$ has property (T) for all  $p\geq 7$.

\item If $(r_0, r_1, r_2) = (4, 4, 4)$, then $G$ has property (T) for all  $p\geq 11$.
\end{enumerate}

\end{thm}
\begin{proof}
In view of Theorem~\ref{thm:NPC} and Proposition~\ref{prop:Girth:Moufang}, we see that $G$ is infinite, and $G$ is hyperbolic if  $1/r_0 + 1/r_1 + 1/r_2 < 1$.  In all cases, $G$ is acylindrically hyperbolic by Theorem~\ref{thm:Acyl}.

For the other assertions, we invoke Theorem~\ref{thm:EJ} together with Proposition~\ref{prop:AngleUnipotent}.  If $(r_0, r_1, r_2)=(2, 4, 4)$, we find that $G$ has property (T) for all $p$ such that $p>4$. Moreover, for $p=3$, it follows from Remark~\ref{rem:KMS-over-GF3} and the results from Section~\ref{sec:tab_244} that $G$ has a finite index subgroup with infinite abelianization. Hence $G$ does not have (T). 

If $(r_0, r_1, r_2)=(3, 3, 3)$, we find that $G$ has property (T) for all $p$ such that $p^3 - 6p^2 + 9p -4 > 0$. In particular $G$ has (T) for $p \geq 5$. Moreover, for $p=3$, it follows from Remark~\ref{rem:KMS-over-GF3} and the results from Section~\ref{sec:tab_333} that $G$ has a finite index subgroup with infinite abelianization. Hence $G$ does not have (T).

If $(r_0, r_1, r_2)=(3, 3, 4)$, we find that $G$ has property (T) for all $p$ such that $p^3 - 8p^2 + 16p -8 > 0$. If $(r_0, r_1, r_2)=(3, 4, 4)$, we find that $G$ has property (T) for all $p$ such that $p^3 - 10p^2 + 25p -16 > 0$. If $(r_0, r_1, r_2)=(4, 4, 4)$, we find that $G$ has property (T) for all $p$ such that $p^3 - 12p^2 + 36p -32 > 0$.
\end{proof}

We shall see in Proposition~\ref{prop:ZxZ} below that if $(r_0, r_1, r_2) = (3, 3, 3)$, then $G$ is not hyperbolic.

The KMS groups $\mathscr G_{ \widetilde{A_2}}(p)$ and $\mathscr G_{HBC_2^{(3)}}(p)$ admit an automorphism of order~$3$ that permutes cyclically the generators. In the same vein as in Section~\ref{sec:CyclicExt_trivalent}, we therefore obtain overgroups of index~$3$ admitting the following presentations:
\begin{align*}
\widetilde{\mathscr G}_{ \widetilde{A_2}}(p) = \langle t, a, b  \mid  & t^3, a^p,  tat^{-1}b^{-1}, [a, b, a], [a, b, b]\rangle\\
\widetilde{\mathscr G}_{HBC_2^{(3)}}(p) = \langle t, a, b    \mid & t^3, a^p, tat^{-1}b^{-1},  [a, b, a], [a, b, b, a], [a, b, b, b ]\rangle.
\end{align*}
Clearly, the generator $b$ is redundant and the groups $\widetilde{\mathscr G}_{ \widetilde{A_2}}(p)$ and $\widetilde{\mathscr G}_{HBC_2^{(3)}}(p)$ are $2$-generated.

Combining Theorem~\ref{label:thm:KMS} with Proposition~\ref{prop:PresentationsUnipotent}, we obtain infinite families of infinite hyperbolic groups with property (T), each given by an explicit presentation.

\begin{cor}\label{cor:ExplicitKMS}
For every prime $p \geq 7$, the groups $\mathscr G_{HC_2^{(1)}}(p)$,  $\mathscr G_{HC_2^{(2)}}(p)$,  $\mathscr G_{HB_2^{(2)}}(p)$ and  $\mathscr G_{HBC_2^{(2)}}(p)$ are infinite hyperbolic groups with property (T). 

For every prime $p \geq 11$, the groups $\mathscr G_{HB_2^{(3)}}(p)$,  $\mathscr G_{HBC_2^{(3)}}(p)$ and $\widetilde{\mathscr G}_{HBC_2^{(3)}}(p)$ are infinite hyperbolic groups with property (T). 
\end{cor}

The number of relators in those presentations is uniformly bounded. Alternative presentations whose total length is a logarithmic function of $p$ can be obtained using the results from \cite[Section~8]{GHS}.

\subsection{Epimorphisms of KMS groups}\label{sec:KMS-epim}

The quotient of the group $\mathscr U_4(p)$ by its center is isomorphic to $\mathscr U_3(p)$. This can be used to construct various epimorphisms between the KMS groups over a fixed prime field $\mathbf F_p$,  sending each of the generator of the source KMS group to a generator of the target KMS group. The complete set of  epimorphisms constructed in this way is depicted in Figure~\ref{fig:Epimor-KMS}. 

\begin{figure}[h]
\begin{tikzcd}
& &  \mathscr G_{HC_2^{(1)}}(p) \arrow[r] & \mathscr G_{ \widetilde{A_2}}(p)\\
\mathscr G_{HB_2^{(3)}}(p) \arrow[r] \arrow[rd] \arrow[rdd] 
& \mathscr G_{HB_2^{(2)}}(p) \arrow[ru] \arrow[rr]  &  & \mathscr G_{ \widetilde{B_2}}(p)\\
& \mathscr G_{HC_2^{(2)}}(p) \arrow[rr] \arrow[ruu] & & \mathscr G_{ \widetilde{C_2}}(p)\\
\mathscr G_{HBC_2^{(3)}}(p) \arrow[r] & \mathscr G_{HBC_2^{(2)}}(p) \arrow[ruuu] \arrow[rr]
& & \mathscr G_{ \widetilde{BC_2}}(p)
\end{tikzcd}
\caption{Epimorphisms between KMS groups over $\mathbf F_p$}\label{fig:Epimor-KMS}
\end{figure}
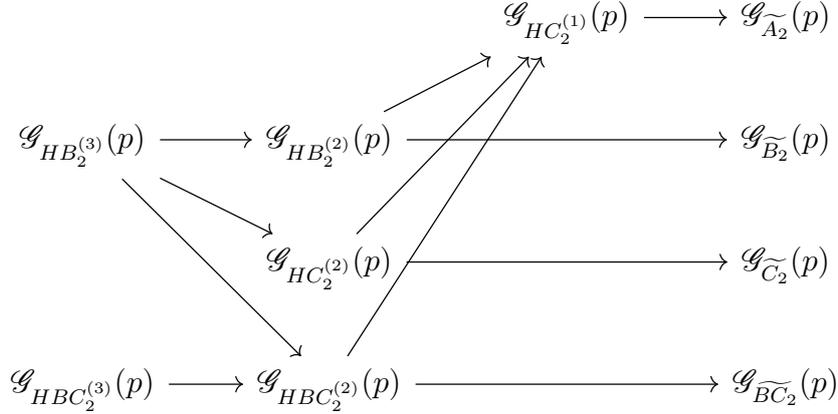

In particular, the group $\mathscr G_{ \widetilde{A_2}}(p)$ is a quotient of each hyperbolic KMS group over $\mathbf F_p$. This motivates further investigations of the former. 

\subsection{Further properties of $\mathscr G_{ \widetilde{A_2}}(p)$}\label{sec:KMS-A2t}

We start by observing that $\mathscr G_{ \widetilde{A_2}}(p)$ has a linear representation with infinite image. 

\begin{prop}\label{prop:Morphism-A2t}
The assignments
$$
a \mapsto \left(\begin{array}{ccc}
1 & 1 & 0\\
0 & 1 & 0\\
0 & 0 & 1
\end{array}\right),
\
b \mapsto \left(\begin{array}{ccc}
1 & 0 & 0\\
0 & 1 & 1\\
0 & 0 & 1
\end{array}\right),
\
c \mapsto \left(\begin{array}{ccc}
1 & 0 & 0\\
0 & 1 & 0\\
T & 0 & 1
\end{array}\right)
$$
extend to a homomorphism 
$$\rho \colon \mathscr G_{ \widetilde{A_2}}(p) \to \SL_3(\mathbf F_p[T]).$$
The image of $\rho$ is a residually-$p$ group  of index~$(p^2-1)(p^3-1)$ in $\SL_3(\mathbf F_p[T])$.
\end{prop}
\begin{proof}
The fact that $\rho$ is a well defined homomorphism $ \mathscr G_{ \widetilde{A_2}}(p) \to \SL_3(\mathbf F_p[T])$ is straightforward to verify from the presentation of  $ \mathscr G_{ \widetilde{A_2}}(p)$. The fact that the subgroup of $\SL_3(\mathbf F_p[T])$ generated by the matrices $\rho(a), \rho(b), \rho(c)$ have the asserted properties  can be deduced, for example, from \cite[Theorem~M]{AM97}.
\end{proof}
	
\begin{cor}\label{cor:SL(3,q)-quotients}
Let $p$ be an odd prime and $q = p^e$ for some $e \geq 3$. Let also $G$ be $\mathscr G_{ \widetilde{A_2}}(p)$, or any of the six hyperbolic KMS groups over $\mathbf F_p$. Then $G$ has an infinite pro-$p$ completion, and $G$ has a quotient isomorphic to $\SL_3(\mathbf F_q)$.
\end{cor}
\begin{proof}
In view of the epimorphisms from Figure~\ref{fig:Epimor-KMS}, it suffices to prove those assertions for $G = \mathscr G_{ \widetilde{A_2}}(p)$. To that end, consider the representation $\rho$ afforded by Proposition~\ref{prop:Morphism-A2t}. Since the image of $\rho$ is an infinite 
residually-$p$ group, it follows that $G$ has an infinite pro-$p$ completion. Observe moreover that the field $\mathbf F_q$ is a quotient of the ring $\mathbf F_p[T]$, since every finite dimensional separable field extension contains a primitive element (see \cite[\S4.14]{BAI}). It follows that $\SL_3(\mathbf F_q)$ is a quotient group of $\SL_3(\mathbf F_p[T])$. Since $q > p$, it follows that the group $\SL_3(\mathbf F_q)$ does not admit any proper subgroup of index $\leq (p^2-1)(p^3-1)$ since $q\geq p^3$. Therefore, the composite map 
\begin{tikzcd}
G \arrow[r, "\rho"] &\SL_3(\mathbf F_p[T]) \arrow[r] & \SL_3(\mathbf F_q)
\end{tikzcd}
must be surjective.
\end{proof}
	
Our next goal is to show that 		$\mathscr G_{ \widetilde{A_2}}(p)$ is not hyperbolic. In the following statement, the elements $a, b, c$ are the generators of $\mathscr G_{ \widetilde{A_2}}(p)$ as they appear in the presentation from Section~\ref{sec:KMS-presentations}.

\begin{prop}\label{prop:ZxZ}
Let $p$ be an odd prime and let $G = \mathscr G_{ \widetilde{A_2}}(p)$. Then the elements
$$x = aba^{\frac{p-1} 2} c \hspace{1cm} \text{and} \hspace{1cm} y = aca^{\frac{p-1} 2} b$$
generate a subgroup of $G$ isomorphic to   $\mathbf Z \times \mathbf Z$. 

In particular $G$ is an acylindrically hyperbolic group which is not hyperbolic.
\end{prop}
\begin{figure}
\begin{tikzpicture}[scale=1.25,rotate=240]
\draw (90:1) -- (210:1) -- (330:1) -- cycle;
\node at (0,0) {$1$};
\node[dot,fill=blue] at (210:1) {};
\begin{scope}[shift={($(30:1)$)}]
\draw (30:1) -- (150:1) -- (270:1) -- cycle;
\node at (0,0) {$a$};
\end{scope}
\begin{scope}[shift={($(30:1) + (330:1)$)}]
\draw (90:1) -- (210:1) -- (330:1) -- cycle;
\node at (0,0) {$ab$};
\end{scope}
\begin{scope}[shift={($(30:1) + (330:1) + (270:1)$)}]
\draw (30:1) -- (150:1) -- (270:1) -- cycle;
\node at (0,0) {$aba^{\frac{p-1}{2}}$};
\node[dot,fill=yellow] at (150:1) {};
\end{scope}
\begin{scope}[shift={($(30:1) + (330:1) + (270:1) + (330:1)$)}]
\draw (90:1) -- (210:1) -- (330:1) -- cycle;
\node at (0,0) {$aba^{\frac{p-1}{2}}c$};
\node[dot,fill=red] at (90:1) {};
\node[dot,fill=blue] at (210:1) {};
\node[dot,fill=yellow] at (330:1) {};
\end{scope}
\begin{scope}[shift={($(30:1) + (90:1)$)}]
\draw (90:1) -- (210:1) -- (330:1) -- cycle;
\node at (0,0) {$ac$};
\node[dot,fill=blue] at (330:1) {};
\end{scope}
\begin{scope}[shift={($(30:1) + (90:1) + (150:1)$)}]
\draw (30:1) -- (150:1) -- (270:1) -- cycle;
\node at (0,0) {$aca^{\frac{p-1}{2}}$};
\node[dot,fill=red] at (270:1) {};
\end{scope}
\begin{scope}[shift={($(30:1) + (90:1) + (150:1) + (90:1)$)}]
\draw (90:1) -- (210:1) -- (330:1) -- cycle;
\node at (0,0) {$aca^{\frac{p-1}{2}}b$};
\node[dot,fill=red] at (90:1) {};
\node[dot,fill=blue] at (210:1) {};
\node[dot,fill=yellow] at (330:1) {};
\end{scope}
\draw[dashed,gray] (0,0) -- ($(30:1) + (90:1) + (150:1) + (90:1)$);
\draw[dashed,gray] (0,0) -- ($(30:1) + (330:1) + (270:1) + (330:1)$);
\end{tikzpicture}
\caption{The elements in \protect{Proposition~\ref{prop:ZxZ}}.}
\label{fig:ZxZ}
\end{figure}
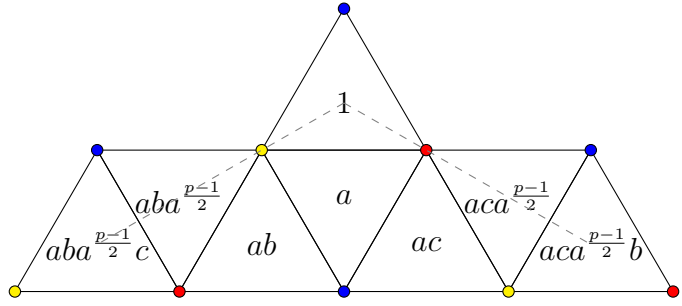
\begin{proof}
From the presentation of $G$, it follows that the commutator $[b, a]$ commutes with $a$ and $b$. Moreover, we have $[b, a^m] = [b, a]^m$ for all $m \geq 0$, and $[b, a]^p = 1$. Similar assertions hold for the pairs $\{a, c\}$ and $\{b, c\}$. 

Set $n = \frac{p-1} 2$. We have
\begin{align*}
xy & = aba^n c a c a^n b \\
& =  a^{n+1} b [b, a^n] a c [c, a] a^n c [c, a^n] b\\ 
&=  a^{n+1} ba [b, a]^n  ca^n c[c, a]^{n+1} b\\ 
& =  a^{n+2} b [b, a]^{n+1} a^n c [c, a]^n c [c, a]^{n+1} b\\ 
& =  a^{n+2} b  a^n  [b, a]^{n+1} c^2  [c, a]^{2n+1} b\\ 
& =  a^{2n+2}  b [b, a]^n  [b, a]^{n+1} c^2  [c, a]^{2n+1} b\\ 
& =  a^{p+1}   b  [b, a]^{p} c^2  [c, a]^{p} b\\ 
& =  a b c^2 b.
\end{align*}
Similar computations show that $yx = a cb^2 c$. Since $bc^2 b = b^2 c^2 [c^2, b] = 		 b^2 c^2 [c, b]^2$ and  $cb^2 c = b^2 c [c, b^2] c = b^2 c^2 [c, b]^2$, we infer that $xy = yx$. 

In order to show that $\langle x, y \rangle \cong  \mathbf Z \times \mathbf Z$, it remains to show that $x$ and $y$ are both of infinite order, and that the cyclic subgroups they generate are not commensurate in $G$. We establish this using the geometric action of $G$ on the $2$-dimensional CAT($0$) complex $Y$ afforded by Theorem~\ref{thm:NPC}.

Let $(v_0, v_1, v_2)$ be $2$-simplex of $Y$ such that $v_0$ is fixed by $\langle a, b\rangle$, $v_1$ is fixed by $
\langle b, c\rangle$ and $v_2$ is fixed by $\langle c, a \rangle$. Let $p$ be the center of $(v_0,v_1,v_2)$. From Lemma~\ref{lem:translation_length} we know that $x$ and $y$ are hyperbolic isometries and that $p$ lies on an axis of each. As in the proof of the lemma we see that $[p,x.p]$ contains $v_0$ while the geodesic segment $[p,y.p]$ contains $v_2$, see Figure~\ref{fig:ZxZ}. That is, the axes of $x$ and $y$ are not parallel and hence the cyclic groups $\langle x \rangle$ and $\langle y \rangle$ are not commensurate.

Clearly, this implies that $G$ is not hyperbolic. That $G$ is acylindrically hyperbolic follows from Theorem~\ref{thm:Acyl}.
\end{proof}
	 	
\begin{rmk}
Similar arguments show that in the group $\mathscr G_{\widetilde{C_2}}(p)$, the elements $x=acb^{-1}c^{-1}$ and $y=bca^{-1}c^{-1}$ generate a subgroup isomorphic to $\mathbf Z \times \mathbf Z$. In particular $\mathscr G_{\widetilde{C_2}}(p)$ is not hyperbolic. We omit the details.
\end{rmk}

We finish this section by reporting on a supplementary construction of finite quotients of the group $\mathscr G_{ \widetilde{A_2}}(p)$. In view of the epimorphisms from Figure~\ref{fig:Epimor-KMS}, it follows that the answer to Question~\ref{qu:KMS} can only be negative if $\mathscr G_{ \widetilde{A_2}}(p)$ fails to have quotients in $\mathscr S_d$ for all $d$. The following variation on Proposition~\ref{prop:Morphism-A2t} can be used to challenge this problem. It is inspired by the seminal work of M.~Kassabov \cite[\S4.1]{Kassabov2007}. 

\begin{prop}\label{prop:Morphism-A2t-block}
Let $q$ be a positive power of the odd prime $p$, and let $k \geq 1$ be an integer. Let also $M_a, M_b, M_c \in \mathrm{Mat}_{k \times k}(\mathbf F_q)$ be any three $k\times k$-matrices with coefficients in $\mathbf F_q$. Then the assignments
$$
a \mapsto U_a := \left(\begin{array}{ccc}
1 & M_a & 0\\
0 & 1 & 0\\
0 & 0 & 1
\end{array}\right),
\
b \mapsto U_b:=\left(\begin{array}{ccc}
1 & 0 & 0\\
0 & 1 & M_b\\
0 & 0 & 1
\end{array}\right),
\
c \mapsto U_c:=\left(\begin{array}{ccc}
1 & 0 & 0\\
0 & 1 & 0\\
M_c & 0 & 1
\end{array}\right), 
$$
where each entry represents a $k\times k$-block, 
extend to a homomorphism 
$$\rho_{M_a, M_b, M_c} \colon \mathscr G_{ \widetilde{A_2}}(p) \to \SL_{3k}(\mathbf F_q).$$
\end{prop}
\begin{proof}
Straightforward computation in view of  the presentation of  $ \mathscr G_{ \widetilde{A_2}}(p)$ from Section~\ref{sec:KMS-presentations}. 
\end{proof}

Thus, every subgroup of $\SL_{3k}(\mathbf F_q)$ generated by three matrices of the form $U_a, U_b, U_c$ is a finite quotient of $\mathscr G_{ \widetilde{A_2}}(p)$. It turns out that this construction does \textit{not} produce  finite simple quotients of large rank, in view of the following observation. 

\begin{prop}\label{prop:UaUbUc}
Retain the notation of Proposition~\ref{prop:Morphism-A2t-block} and let $\Gamma = \langle U_a, U_b, U_c \rangle$. Then there is a finite extension $F$ of $\mathbf F_q$ and a subnormal series of $\Gamma$ such that every subquotient from that series is either a $p$-group, or isomorphic to a subgroup of $\mathrm{GL}_3(F)$. 
\end{prop}

\begin{proof}
We work by induction on $k$. In the base case $k=1$, there is nothing to prove. Let now $k>1$. We distinguish two cases. 

Assume first that at least two of the matrices $M_a, M_b, M_c \in \mathrm{Mat}_{k \times k}(\mathbf F_q)$ have a non-zero determinant. Without loss of generality, we may assume that $\det(M_a) \neq 0 \neq \det(M_b)$. Let $x$ be the block-diagonal matrix in $\mathrm{GL}_{3k}(\mathbf F_q)$ defined by 
$$x=\left(\begin{array}{ccc}
M_a^{-1} & 0 & 0\\
0 & 1 & 0\\
0 & 0 & M_b
\end{array}\right), 
$$
and set $U'_a = xU_a x^{-1}$,  $U'_b = xU_b x^{-1}$ and  $U'_c = xU_c x^{-1}$. Then $U'_a = \left(\begin{array}{ccc}
1 & 1 & 0\\
0 & 1 & 0\\
0 & 0 & 1
\end{array}\right)$, 
$U'_b = \left(\begin{array}{ccc}
1 & 0 & 0\\
0 & 1 & 1\\
0 & 0 & 1
\end{array}\right)$, and 
$U'_c = \left(\begin{array}{ccc}
1 & 0 & 0\\
1 & 1 & 0\\
M_b M_c M_a & 0 & 1
\end{array}\right)$. In particular, the group $x \Gamma x^{-1}$ is a subgroup of  $\mathrm{SL}_3(\mathcal A)$, where $\mathcal A \leq  \mathrm{Mat}_{k \times k}(\mathbf F_q)$ is the $\mathbf F_q$-subalgebra generated by the single element $M_b M_c M_a$. Thus $\mathcal A$ is commutative, and is in fact a quotient of the polynomial ring $\mathbf F_q[T]$. Denoting by $\mathcal I$ a maximal ideal in $\mathcal A$, we obtain a congruence quotient $\mathrm{SL}_3(\mathcal A) \to \mathrm{SL}_3(\mathcal A/ \mathcal I)$ whose kernel is well-known to be a $p$-group (see \cite[\S 26.4]{DrutuKapovich}). Since $\mathcal A/ \mathcal I$ is a finite field extension of $\mathbf F_q$,  the required conclusion holds in this case. 

We now assume that at most one of the matrices $M_a, M_b, M_c \in \mathrm{Mat}_{k \times k}(\mathbf F_q)$ has a non-zero determinant. Then at least two of them have zero determinant. Without loss of generality, we may assume that $\det(M_b) =\det(M_c) =  0$. Then there exists $L_1, L_3 \in \mathrm{GL}_k(\mathbf F_q)$ such that the first column of $M_c L_1^{-1}$ and the first column of $M_b L_3^{-1}$ are both zero. If $\det(M_a) = 0$, we choose similarly a matrix $L_2 \in \mathrm{GL}_k(\mathbf F_q)$ such that the first column of $M_a L_2^{-1}$ is zero. If $\det(M_a) \neq 0$, then we set $L_2 = L_1M_a \in  \mathrm{GL}_k(\mathbf F_q)$. In either case, we see that the matrices $M'_a = L_1 M_a L_2^{-1}$, $M'_b = L_2 M_b L_3^{-1}$ and $M'_c = L_3 M_c L_1^{-1}$ all belong to 
$$\mathcal B = \{
\left(\begin{array}{ccc}
* & *  \\
0 & X
\end{array}\right)  \mid X \in \mathrm{Mat}_{(k-1) \times (k-1)}(\mathbf F_q)\}.
$$
Let $y$ be the block-diagonal matrix  in $\mathrm{GL}_{3k}(\mathbf F_q)$ defined by 
$y=\left(\begin{array}{ccc}
L_1 & 0 & 0\\
0 & L_2 & 0\\
0 & 0 & L_3
\end{array}\right).
$
Then the group $y\Gamma y^{-1}$ is generated by $U'_a = yU_a y^{-1} =  \left(\begin{array}{ccc}
1 & M'_a & 0\\
0 & 1 & 0\\
0 & 0 & 1
\end{array}\right)$,
 $U'_b = yU_b y^{-1} = \left(\begin{array}{ccc}
1 & 0 & 0\\
0 & 1 & M'_b\\
0 & 0 & 1
\end{array}\right)$,
and 
 $U'_c = yU_c y^{-1} =\left(\begin{array}{ccc}
1 & 0 & 0\\
0 & 1 & 0\\
M'_c & 0 & 1
\end{array}\right)$.
In particular $y\Gamma y^{-1}$ is a subgroup of   $\mathrm{GL}_3(\mathcal B)$. Clearly, the set $\mathcal B$ is a $\mathbf F_q$-subalgebra of $\mathrm{Mat}_{k \times k}(\mathbf F_q)$ that maps onto $\mathrm{Mat}_{(k-1) \times (k-1)}(\mathbf F_q)$. This yields a homomorphism $\mathrm{GL}_3(\mathcal B) \to \mathrm{GL}_{3(k-1)}(\mathbf F_q)$. It is easy to see that its kernel has a normal $p$-subgroup such that the quotient is isomorphic to $\mathrm{GL}_{3}(\mathbf F_q)$. Restricting to  $y\Gamma y^{-1}$, we obtain a homomorphism taking values in  $\mathrm{SL}_{3(k-1)}(\mathbf F_q)$, whose kernel decomposes as an extension of a $p$-group by a subgroup of $\mathrm{GL}_{3}(\mathbf F_q)$. Moreover, the image of the generators  $U'_a$, $U'_b$ and $U'_c$ under that homomorphism generate a subgroup of  $\mathrm{SL}_{3(k-1)}(\mathbf F_q)$ to which the induction hypothesis applies. The required conclusion follows. 
\end{proof}

\subsection{Further properties of  $\mathscr G_{HB_2^{(2)}}(p)$ and $\mathscr G_{\widetilde{C_2}}(p)$}\label{sec:KMS-HC2}

In view of Section~\ref{sec:KMS-epim}, the  quotients of 
 $\mathscr G_{ \widetilde{A_2}}(p)$  described in the previous section are also quotients of each hyperbolic KMS group over $\mathbf F_p$.  We now focus more specifically on $\mathscr G_{HC_2^{(2)}}(p)$ and $\mathscr G_{HB_2^{(2)}}(p)$, first establishing an analogue of Proposition~\ref{prop:Morphism-A2t-block}.

\begin{prop}\label{prop:Morphism-HC2-block}
Let $q$ be a positive power of the odd prime $p$, and let $k \geq 1$ be an integer. Let also $M_a, M_b, M_c \in \mathrm{Mat}_{k \times k}(\mathbf F_q)$ be any three $k\times k$-matrices with coefficients in $\mathbf F_q$. Set 
$$
 V_a = \left(\begin{array}{cccc}
1 & 0 & 0 & 0\\
0 & 1 & 0 & 0\\
M_a & 0 & 1 & 0\\
0 & 0 & 0 & 1
\end{array}\right),
\
V_b= \left(\begin{array}{cccc}
1 & 0 & 0 & 0\\
0 & 1 & 0 & 0\\
0 & 0 & 1 & 0\\
0 & M_b & 0 & 1
\end{array}\right)
\text{ and }
V_c= \left(\begin{array}{cccc}
1 & 0 & 0 & M_c\\
0 & 1 & M_c & 0\\
0 & 0 & 1 & 0\\
0 & 0 & 0 & 1
\end{array}\right),$$
where each entry represents a $k\times k$-block. Then  the assignments $(a, b, c) \mapsto (V_a, V_b, V_c)$ 
extend to a homomorphism 
$\sigma_{M_a, M_b, M_c} \colon 
\mathscr G_{\widetilde{C_2}}(p) \to \SL_{4k}(\mathbf F_q)$.

Similarly, the assignments $(a, b, c) \mapsto (V'_a, V'_b, V'_c)$, where 
$$
V'_a = \left(\begin{array}{cccc}
1 & 0 & 0 & M_a\\
0 & 1 & M_a & 0\\
0 & 0 & 1 & 0\\
0 & 0 & 0 & 1
\end{array}\right),
\
V'_b= \left(\begin{array}{cccc}
1 & 0 & 0 & 0\\
M_b & 1 & 0 & 0\\
0 & 0 & 1 & -M_b\\
0 & 0 & 0 & 1
\end{array}\right)
\text{ and }
V'_c=\left(\begin{array}{cccc}
1 & 0 & 0 & 0\\
0 & 1 & 0 & 0\\
0 & 0 & 1 & 0\\
0 & M_c & 0 & 1
\end{array}\right)
$$
extend to a homomorphism 
$\sigma'_{M_a, M_b, M_c} \colon \mathscr G_{HB_2^{(2)}}(p) \to \SL_{4k}(\mathbf F_q)$.
\end{prop}
\begin{proof}
Straightforward computation in view of  the presentation of the KMS groups  from Section~\ref{sec:KMS-presentations}. 
\end{proof}

It turns out that the outcome of the constructions from  Proposition~\ref{prop:Morphism-HC2-block} is very different for 
the groups $\mathscr G_{\widetilde{C_2}}(p)$ and 
 $\mathscr G_{HB_2^{(2)}}(p)$.  Indeed, for $\mathscr G_{\widetilde{C_2}}(p)$, the quotients from Proposition~\ref{prop:Morphism-HC2-block} are subjected to analogous restrictions as in Proposition~\ref{prop:UaUbUc}. This can easily be established by similar arguments. We omit the details. Instead, we focus on $\mathscr G_{HB_2^{(2)}}(p)$, for which the situation is  strikingly different, as illustrated by the following.

\begin{prop}\label{prop:Quotient-in-S_d}
Let $p$ be an odd prime. For any prime $k \neq p$,  there exist matrices $M_a, M_b, M_c \in \mathrm{Mat}_{k \times k}(\mathbf F_p)$ such that the subgroup of $\SL_{4k}(\mathbf F_p)$ generated by the three elements $V'_a, V'_b, V'_c$, defined as in Proposition~\ref{prop:Morphism-HC2-block}, has a quotient in $\mathscr S_{k-1}$.
\end{prop}
 
The proof of that proposition  requires some preparation. 
We start with a special case of a general result due to Shangzhi Li. 

\begin{thm}[{See \cite[Theorem~1]{Li}}]\label{thm:Li}
Let $F$ be a finite field of order $q$ and $K$ be a finite extension of $F$ of degree $k$. We assume that $k$ is prime and that $(q, k) \neq (2, 2)$.  Embed the group $N = \SL_2(K)$ as a subgroup of $G = \mathrm{SL}_{2k}(F)$ by identifying the natural $N$-module $K^2$ with the $G$-module $F^{2k}$. Let also $X$ be a subgroup of $G$ containing $N$. Then one of the following holds:
\begin{enumerate}[(i)]
\item  $X$ normalizes $N$.

\item $X$ contains a normal subgroup isomorphic to $\mathrm{Sp}_{2k}(F)$.

\item $X = G$. 

\end{enumerate}
	
\end{thm}

The group $N$ embedded as a subgroup of $G$ as in Theorem~\ref{thm:Li} is called a \textbf{field extension subgroup}. Specializing Theorem~\ref{thm:Li}, we obtain the following result describing subgroups of $  \SL_{2k}(\mathbf F_p)$ generated by three unipotent block matrices. 

\begin{prop}\label{prop:SL2-block}
Let $p, k$ be primes with $p >2$. Let also $M_1, M_2, M_3  \in \mathrm{Mat}_{k \times k}(\mathbf F_p)$ and define the elements $V_1, V_2, V_3 \in G = \SL_{2k}(\mathbf F_p)$ by 
$V_1 = \left(\begin{array}{cc}
1 &  M_1\\
0 & 1 
\end{array}\right)$,
$V_2 = \left(\begin{array}{cc}
1 &  0\\
M_2 & 1 
\end{array}\right)$ and 
$V_3 = \left(\begin{array}{cc}
1 &  M_3\\
0 & 1 
\end{array}\right)$,
where each entry represents a $k\times k$-block. We assume that $M_1 M_2$ is invertible, of multiplicative order~$p^k-1$. Then the following assertions hold. 
\begin{enumerate}[(i)]
\item 
$\langle V_1, V_2 \rangle$ is isomorphic to a field extension subgroup $\SL_2(\mathbf F_{p^k}) \leq G$. 

\item 
If in addition, we have $M_1 M_2 M_3 \neq M_3 M_2 M_1$ and $p \neq k$, then $\langle V_1, V_2, V_3 \rangle $ either contains a normal subgroup isomorphic to $\mathrm{Sp}_{2k}(\mathbf F_p)$, or is the whole group $G = \SL_{2k}(\mathbf F_p)$. 
\end{enumerate}
\end{prop}
	
\begin{proof}
An element $C$ of the group $\mathrm{GL}_k(\mathbf F_p)$ of  order $p^k$ is called a \textbf{Singer element}.  If $C$ is a Singer element, the cyclic group $\langle C \rangle$ acts irreducibly on $\mathbf F_p^k$. Therefore, it follows from Schur's lemma and Wedderburn's theorem that the subalgebra of $\mathrm{Mat}_{k \times k}(\mathbf F_p)$ generated by $C$ is a subfield, isomorphic to $\mathbf F_{p^k}$. It then follows from the work of Dickson (see \cite[\S2.1]{King}) that the subgroup of $G$ generated by $ \left(\begin{array}{cc}
1 &  1\\
0 & 1 
\end{array}\right)$
and 
$\left(\begin{array}{cc}
1 &  0\\
C & 1 
\end{array}\right)$ is isomorphic  to a field extension subgroup $\SL_2(\mathbf F_{p^k}) \leq G$. 

By hypothesis, the  matrix $C = M_1 M_2$ is  a Singer element of $\mathrm{GL}_k(\mathbf F_p)$. Moreover, the matrix $d= \left(\begin{array}{cc}
1 &  0\\
0 & M_1 
\end{array}\right)$
conjugates $V_1$ and $V_2$ respectively to  $ \left(\begin{array}{cc}
1 &  1\\
0 & 1 
\end{array}\right)$
and 
$\left(\begin{array}{cc}
1 &  0\\
C & 1 
\end{array}\right)$. The first assertion follows. 

Let now $M_3 \in \mathrm{Mat}_{k \times k}(\mathbf F_p)$ be such that $M_1 M_2 M_3 \neq M_3 M_2 M_1$, and assume that $p \neq k$. We claim that $V_3$ does not normalize the field extension subgroup $\langle V_1, V_2 \rangle$. In view of Theorem~\ref{thm:Li}, the required assertion follows from that claim. 

In order to prove the claim, it suffices to show that $d V_3 d^{-1}=\left(\begin{array}{cc}
1 &  M_3 M_1^{-1}\\
0 & 1 
\end{array}\right)$ 
does not normalize 
$H = \langle  \left(\begin{array}{cc}
1 &  1\\
0 & 1 
\end{array}\right),
\left(\begin{array}{cc}
1 &  0\\
C & 1 
\end{array}\right)
\rangle$. By \cite[Theorem~1]{Li}, the full normalizer of the field extension subgroup $H$ in $\mathrm{GL}_{2k}(\mathbf F_p)$ is isomorphic to  $\mathrm{GL}_2(\mathbf F_{p^k}) \rtimes \Aut(\mathbf F_{p^k})$. In particular, since $p \neq k$, the $p$-Sylow subgroups of that normalizer are all contained in $H$, and are conjugate to $\{\left(\begin{array}{cc}
1 &  X\\
0 & 1 
\end{array}\right) \mid X \in \mathcal C\}$, where $\mathcal C$ denotes the  subalgebra of $\mathrm{Mat}_{k \times k}(\mathbf F_p)$ generated by $C$. Since $d V_3 d^{-1}$ is of order~$p$ and centralizes the latter subgroup, we deduce that $d V_3 d^{-1}$ normalizes $H$ if and only if $M_3 M_1^{-1}$ belongs to $\mathcal C$. As mentionned above, the subalgebra $\mathcal C$ is isomorphic to $\mathbf F_{p^k}$, and is thus commutative. Since $M_1 M_2 M_3 \neq M_3 M_2 M_1$, it follows that $(M_1 M_2) (M_3 M_1^{-1}) \neq M_3  M_2 = (M_3 M_1^{-1})(M_1 M_2)$, so that $C = M_1 M_2$ and $M_3M_1^{-1}$ do not commute. This confirms that $d V_3 d^{-1}$ does not normalize $H$. The claim follows. 
\end{proof}

\begin{proof}[Proof of Proposition~\ref{prop:Quotient-in-S_d}]
Let $M_a$ be the matrix with $(M_a)_{i+1, i} = 1$ for all $i=1, \dots, k-1$ and with all other entries equal to $0$. Let $M_b$ be the matrix with $(M_b)_{i, i+1} = 1$ for all $i=1, \dots, k-1$ and with all other entries equal to $0$. Set $M_1 = M_a M_b + M_b M_a = \mathrm{diag}(1, 2, \dots, 2, 1)$. Choose a Singer element $C \in \mathrm{Mat}_{k \times k}(\mathbf F_p)$ whose last column is $(1, 0, \dots, 0)^\top$; clearly, such a Singer element exists since $\mathrm{GL}_k(\mathbf F_p)$ acts transitively on ordered pairs of linearly independent vectors in $\mathbf F_p^k$. Finally, set $M_c = M_1^{-1} C$. We claim that these matrices satisfy the required conditions. 

In order to verify this, we set $M_2 = M_c$ and $M_3 = M_a M_b M_c M_b M_a$. Observe that $M_1 M_2 M_3 M_1^{-1} = C  \mathrm{diag}(0, 2^{-1}, \dots, 2^{-1}, 1) C \mathrm{diag}(1, 2^{-1}, \dots, 2^{-1}, 0)$. In particular, the last column of  $M_1 M_2 M_3 M_1^{-1}$ is zero. On the other hand, we have $M_3 M_2 =   \mathrm{diag}(0, 2^{-1}, \dots, 2^{-1}, 1) C \mathrm{diag}(1, 2^{-1}, \dots, 2^{-1}, 0) C$. Applying that element to the $n$-th vector $e_k$ of the canonical basis of $\mathbf F_p^k$, we obtain  $M_3 M_2(e_k) = 
\mathrm{diag}(0, 2^{-1}, \dots, 2^{-1}, 1) C(e_1)$ since $C(e_k) = e_1$ by the definition of $C$. Since $C$ is a Singer element, the vector $C(e_1)$ is not collinear with $e_1$. Since the kernel of the linear map represented by the diagonal matrix $\mathrm{diag}(0, 2^{-1}, \dots, 2^{-1}, 1)$ is spanned by 
$e_1$, we deduce that $M_3 M_2(e_k) \neq 0$. It follows that $M_1 M_2 M_3 \neq M_3 M_2 M_1$.

Computations show that
$$[V'_a, V'_b] = 
\left(\begin{array}{cccc}
1 & 0 & 0 & 0\\
0 & 1 & 0 & -M_aM_b - M_bM_a\\
0 & 0 & 1 & 0\\
0 & 0 & 0 & 1
\end{array}\right) 
=
\left(\begin{array}{cccc}
1 & 0 & 0 & 0\\
0 & 1 & 0 & -M_1\\
0 & 0 & 1 & 0\\
0 & 0 & 0 & 1
\end{array}\right) 
$$
and 
$$[V'_c, V'_b, V'_b, V'_a, V'_a] = 
\left(\begin{array}{cccc}
1 & 0 & 0 & 0\\
0 & 1 & 0 & -4M_aM_bM_cM_bM_a\\
0 & 0 & 1 & 0\\
0 & 0 & 0 & 1
\end{array}\right)
=
\left(\begin{array}{cccc}
1 & 0 & 0 & 0\\
0 & 1 & 0 & -4M_3\\
0 & 0 & 1 & 0\\
0 & 0 & 0 & 1
\end{array}\right).
$$
Since $M_1 M_2 =C$ is a Singer element, it follows from Proposition~\ref{prop:SL2-block}(i) that $\langle V'_c, [V'_a, V'_b] \rangle$ is isomorphic to $\SL_2(\mathbf F_{p^k})$. Since $\SL_2(\mathbf F_{p^k})$ is perfect, this implies in particular that $V'_c$ belongs to the derived subgroup of $H = \langle V'_a, V'_b, V'_c \rangle$. 

Moreover, Proposition~\ref{prop:SL2-block}(i) implies that $\langle V'_c, [V'_a, V'_b] \rangle$ is a field extension subgroup of $\SL_{2k}(\mathbf F_p)$, which is itself embedded in $\SL_{4k}(\mathbf F_p)$ as the subgroup consisting of block matrices of the form $\left(\begin{array}{cccc}
1 & 0 & 0 & 0\\
0 & * & 0 & *\\
0 & 0 & 1 & 0\\
0 & * & 0 & *
\end{array}\right)$.
This implies that $\langle V'_c, [V'_a, V'_b] \rangle$ contains the matrix $d = \left(\begin{array}{cccc}
1 & 0 & 0 & 0\\
0 & 2^{-1} & 0 & 0\\
0 & 0 & 1 & 0\\
0 & 0 & 0 & 2
\end{array}\right)$ where, as before, each entry represents a $k \times k$-block.
Now, we compute that the commutator $[d, (V'_a)^{-1}]$ coincides with $V'_a$. Similarly, we have $[d, (V'_b)^{-1}] = V'_b$. It follows that the group $H = \langle V'_a, V'_b, V'_c \rangle$ is perfect.

On the other hand, we have seen above that $M_1 M_2 M_3 \neq M_3 M_2 M_1$. Therefore,  Proposition~\ref{prop:SL2-block}(ii) implies that $\langle V'_c, [V'_a, V'_b], [V'_c, V'_b, V'_b, V'_a, V'_a]\rangle$ either contains a normal subgroup isomorphic to $\mathrm{Sp}_{2k}(\mathbf F_p)$, or it is isomorphic to $\mathrm{SL}_{2k}(\mathbf F_p)$. 

Let finally $S$ be a smallest non-trivial quotient of $H$. Since $H$ is perfect, the group $S$ is a non-abelian finite simple group. The  image of $V'_c$ in $S$ is non-trivial, since otherwise $H$ would be a quotient of $\langle V'_a, V'_b\rangle$, which is a $p$-group. Therefore, the image of the subgroup $\langle V'_c, [V'_a, V'_b], [V'_c, V'_b, V'_b, V'_a, V'_a]\rangle$ must contain a normal subgroup isomorphic to $(P)\mathrm{Sp}_{2k}(\mathbf F_p)$, or   to $(P)\mathrm{SL}_{2k}(\mathbf F_p)$. Any of these groups contain a copy of $\Alt(k-1)$. The required conclusion follows. 
\end{proof}
	
\begin{cor}\label{cor:LargeRankQuotients}
Let $p$ be an odd prime. The groups  $\mathscr G_{HB_2^{(2)}}(p)$ and  $\mathscr G_{HB_2^{(3)}}(p)$ both admit a quotient in $\mathscr S_d$ for all $d$. 
\end{cor}
\begin{proof}
The statement for $\mathscr G_{HB_2^{(2)}}(p)$ is follows by combining Propositions~\ref{prop:Morphism-HC2-block} and~\ref{prop:Quotient-in-S_d}. In view of Section~\ref{sec:KMS-epim}, the group $\mathscr G_{HB_2^{(2)}}(p)$ is a quotient of   $\mathscr G_{HB_2^{(3)}}(p)$. The conclusion follows. 
\end{proof}

\begin{rmk}\label{rem:girth-type}
Propositions~\ref{prop:Morphism-A2t-block}  and~\ref{prop:Morphism-HC2-block} afford simple quotients of the groups $\mathscr G_{ \widetilde{A_2}}(p)$ and $\mathscr G_{HB_2^{(2)}}(p)$ respectively, based on closely related constructions. However,  the outcome of those constructions are in sharp contrast, as one can see by comparing Proposition~\ref{prop:UaUbUc} with Proposition~\ref{prop:Quotient-in-S_d}. Recall that the triangle group $\mathscr G_{ \widetilde{A_2}}(p)$ is of half-girth type $(3, 3, 3)$ whereas $\mathscr G_{HB_2^{(2)}}(p)$  is of half-girth type $(3, 4, 4)$. This somehow corroborates our experimental results on trivalent triangle groups, for which the most stringent restrictions on the existence of small finite simple quotients were observed for the groups of half-girth type $(3, 3, 3)$ as well. In some sense, the results of this study seem to indicate that for generalized triangle groups, the small size of the girth of the vertex links is more tightly related to the scarcity of finite simple quotients than the validity of Kazhdan's property (T).
\end{rmk}

Let $R = \mathbf F_q\langle x, y, z\rangle$ denote the free associative non-commutative ring in $3$~indeterminates over $\mathbf F_q$. 
The homomorphisms highlighted by Propositions~\ref{prop:Morphism-A2t-block}  and~\ref{prop:Morphism-HC2-block} can be factored as the composition of  representations $G \to \mathrm{GL}_d(R)$, where $G \in \big\{\mathscr G_{ \widetilde{A_2}}(p), \mathscr G_{ \widetilde{C_2}}(p), \mathscr G_{HB_2^{(2)}}(p)\big\}$ and $d \in \{3, 4\}$,  followed by the homomorphisms 
$$\mathrm{GL}_d(R) \to \mathrm{GL}_d\big(\mathrm{Mat}_{k\times k}(\mathbf F_q)\big)$$
induced by the  homomorphism of $\mathbf F_q$-algebras $R\to\mathrm{Mat}_{k\times k}(\mathbf F_q)$ defined by the assignments $(x, y, z) \mapsto (M_a, M_b, M_c)$. The following result provides analogous representations for the KMS groups $\mathscr G_{HC_2^{(2)}}(p)$ and $\mathscr G_{HBC_2^{(2)}}(p)$.

\begin{prop}\label{prop:FreeRing}
Let $R = \mathbf F_q\langle x, y, z\rangle$ denote the free associative non-commutative ring in $3$~indeterminates over $\mathbf F_q$, where $q$ is a power of the odd prime $p$. 
Then the assignments 
$$
a \mapsto  \left(\begin{array}{cccc}
1 & x & 0 & 0\\
0 & 1 & 0 & 0\\
0 & 0 & 1 & 0\\
0 & 0 & 0 & 1
\end{array}\right),
\
b\mapsto \left(\begin{array}{cccc}
1 & 0 & 0 & 0\\
0 & 1 & 0 & 0\\
0 & 0 & 1 & 0\\
y & 0 & 0 & 1
\end{array}\right),
\
c\mapsto \left(\begin{array}{cccc}
1 & 0 & 0 & 0\\
0 & 1 & z & 0\\
0 & 0 & 1 & z\\
0 & 0 & 0 & 1
\end{array}\right),$$
extend to a homomorphism 
$
\mathscr G_{HC_2^{(2)}}(p) \to \GL_4(R)$.

Similarly, the assignments 
$$
a \mapsto \left(\begin{array}{ccccc}
1 & 0 & 0 & 0 & 0\\
0 & 1 & x & \frac{x^2} 2 & 0\\
0 & 0 & 1 & x & 0\\
0 & 0 & 0 & 1 & 0\\
0 & 0 & 0 & 0 & 1
\end{array}\right),
\
b \mapsto  \left(\begin{array}{ccccc}
1 & 0 & 0 & 0 & 0\\
0 & 1 & 0 & 0 & 0\\
0 & 0 & 1 & 0 & 0\\
0 & 0 & 0 & 1 & 0\\
y & 0 & 0 & 0 & 1
\end{array}\right),
\
c \mapsto \left(\begin{array}{ccccc}
1 & z & 0 & 0 & 0\\
0 & 1 & 0 & 0 & 0\\
0 & 0 & 1 & 0 & 0\\
0 & 0 & 0 & 1 & z\\
0 & 0 & 0 & 0 & 1
\end{array}\right),
$$
extend to a homomorphism 
$\mathscr G_{HBC_2^{(2)}}(p) \to \GL_{5}(R)$.
\end{prop}
\begin{proof}
Straightforward computation in view of  the presentation of the KMS groups  from Section~\ref{sec:KMS-presentations}. 
\end{proof}

As before, numerous quotients of those KMS groups are obtained by postcomposing the representations from Proposition~\ref{prop:FreeRing} with congruence homomorphisms.

\appendix
\begin{landscape}
\section{Presentations}\label{sec:PresentationList}

Below we list presentations of the trivalent triangle groups that arise as fundamental groups of $\mathcal{Y} \in \mathbf{Y}$, see Section~\ref{sec:TriTriSetup}. The lists are sorted by half-girth type and within each half-girth type by link type.

\allowdisplaybreaks

{\small 
\subsection{Half-girth type $(2,4,4)$}
\begin{align*}
% 6 40 40 0
G^{6,40,40}_{0} = \langle a, b, c\mid{}& a^{3}, b^{3}, c^{3}, bab^{-1}a^{-1},\ (cb^{-1}cb)^{2}, (c^{-1}b^{-1}cb^{-1})^{2}, (ac^{-1}ac)^{2}, (a^{-1}c^{-1}ac^{-1})^{2}\rangle\\ 
% 6 40 48 0
G^{6,40,48}_{0} = \langle a, b, c\mid{}& a^{3}, b^{3}, c^{3}, bab^{-1}a^{-1},\ (cb^{-1}cb)^{2}, (c^{-1}b^{-1}cb^{-1})^{2}, (ac)^{2}(a^{-1}c^{-1})^{2}\rangle\\ 
% 6 40 54 0 2
G^{6,40,54}_{0} = \langle a, b, c\mid{}& a^{3}, b^{3}, c^{3}, bab^{-1}a^{-1},\ (cb^{-1}cb)^{2}, (c^{-1}b^{-1}cb^{-1})^{2}, aca^{-1}c^{-1}a^{-1}cac^{-1}, (aca^{-1}c)^{3}\rangle\\ 
G^{6,40,54}_{2} = \langle a, b, c\mid{}& a^{3}, b^{3}, c^{3}, bab^{-1}a^{-1},\ (cb^{-1}cb)^{2}, (c^{-1}b^{-1}cb^{-1})^{2}, cac^{-1}a^{-1}c^{-1}aca^{-1}, (cac^{-1}a)^{3}\rangle\\ 
% 6 48 48 0
G^{6,48,48}_{0} = \langle a, b, c\mid{}& a^{3}, b^{3}, c^{3}, bab^{-1}a^{-1},\ (cb)^{2}(c^{-1}b^{-1})^{2}, (ac)^{2}(a^{-1}c^{-1})^{2}\rangle\\ 
% 6 48 54 0 2
G^{6,48,54}_{0} = \langle a, b, c\mid{}& a^{3}, b^{3}, c^{3}, bab^{-1}a^{-1},\ (cb)^{2}(c^{-1}b^{-1})^{2}, aca^{-1}c^{-1}a^{-1}cac^{-1}, (aca^{-1}c)^{3}\rangle\\ 
G^{6,48,54}_{2} = \langle a, b, c\mid{}& a^{3}, b^{3}, c^{3}, bab^{-1}a^{-1},\ (cb)^{2}(c^{-1}b^{-1})^{2}, cac^{-1}a^{-1}c^{-1}aca^{-1}, (cac^{-1}a)^{3}\rangle\\ 
% 6 54 54 0 2 8
G^{6,54,54}_{0} = \langle a, b, c\mid{}& a^{3}, b^{3}, c^{3}, bab^{-1}a^{-1},\ cbc^{-1}b^{-1}c^{-1}bcb^{-1}, (cbc^{-1}b)^{3}, aca^{-1}c^{-1}a^{-1}cac^{-1}, (aca^{-1}c)^{3}\rangle\\ 
G^{6,54,54}_{2} = \langle a, b, c\mid{}& a^{3}, b^{3}, c^{3}, bab^{-1}a^{-1},\ cbc^{-1}b^{-1}c^{-1}bcb^{-1}, (cbc^{-1}b)^{3}, cac^{-1}a^{-1}c^{-1}aca^{-1}, (cac^{-1}a)^{3}\rangle\\ 
G^{6,54,54}_{8} = \langle a, b, c\mid{}& a^{3}, b^{3}, c^{3}, bab^{-1}a^{-1},\ bcb^{-1}c^{-1}b^{-1}cbc^{-1}, (bcb^{-1}c)^{3}, aca^{-1}c^{-1}a^{-1}cac^{-1}, (aca^{-1}c)^{3}\rangle\\ 
% 8 40 40 0
G^{8,40,40}_{0} = \langle a, b, c\mid{}& a^{3}, b^{3}, c^{3}, baba,\ (cb^{-1}cb)^{2}, (c^{-1}b^{-1}cb^{-1})^{2}, (ac^{-1}ac)^{2}, (a^{-1}c^{-1}ac^{-1})^{2}\rangle\\ 
% 8 40 48 0
G^{8,40,48}_{0} = \langle a, b, c\mid{}& a^{3}, b^{3}, c^{3}, baba,\ (cb^{-1}cb)^{2}, (c^{-1}b^{-1}cb^{-1})^{2}, (ac)^{2}(a^{-1}c^{-1})^{2}\rangle\\ 
% 8 40 54 0 2
G^{8,40,54}_{0} = \langle a, b, c\mid{}& a^{3}, b^{3}, c^{3}, baba,\ (cb^{-1}cb)^{2}, (c^{-1}b^{-1}cb^{-1})^{2}, aca^{-1}c^{-1}a^{-1}cac^{-1}, (aca^{-1}c)^{3}\rangle\\ 
G^{8,40,54}_{2} = \langle a, b, c\mid{}& a^{3}, b^{3}, c^{3}, baba,\ (cb^{-1}cb)^{2}, (c^{-1}b^{-1}cb^{-1})^{2}, cac^{-1}a^{-1}c^{-1}aca^{-1}, (cac^{-1}a)^{3}\rangle\\ 
% 8 48 48 0 1
G^{8,48,48}_{0} = \langle a, b, c\mid{}& a^{3}, b^{3}, c^{3}, baba,\ (cb)^{2}(c^{-1}b^{-1})^{2}, (ac)^{2}(a^{-1}c^{-1})^{2}\rangle\\ 
G^{8,48,48}_{1} = \langle a, b, c\mid{}& a^{3}, b^{3}, c^{3}, baba,\ (cb)^{2}(c^{-1}b^{-1})^{2}, (ac^{-1})^{2}(a^{-1}c)^{2}\rangle\\ 
% 8 48 54 0 2
G^{8,48,54}_{0} = \langle a, b, c\mid{}& a^{3}, b^{3}, c^{3}, baba,\ (cb)^{2}(c^{-1}b^{-1})^{2}, aca^{-1}c^{-1}a^{-1}cac^{-1}, (aca^{-1}c)^{3}\rangle\\ 
G^{8,48,54}_{2} = \langle a, b, c\mid{}& a^{3}, b^{3}, c^{3}, baba,\ (cb)^{2}(c^{-1}b^{-1})^{2}, cac^{-1}a^{-1}c^{-1}aca^{-1}, (cac^{-1}a)^{3}\rangle\\ 
% 8 54 54 0 2 8
G^{8,54,54}_{0} = \langle a, b, c\mid{}& a^{3}, b^{3}, c^{3}, baba,\ cbc^{-1}b^{-1}c^{-1}bcb^{-1}, (cbc^{-1}b)^{3}, aca^{-1}c^{-1}a^{-1}cac^{-1}, (aca^{-1}c)^{3}\rangle\\ 
G^{8,54,54}_{2} = \langle a, b, c\mid{}& a^{3}, b^{3}, c^{3}, baba,\ cbc^{-1}b^{-1}c^{-1}bcb^{-1}, (cbc^{-1}b)^{3}, cac^{-1}a^{-1}c^{-1}aca^{-1}, (cac^{-1}a)^{3}\rangle\\ 
G^{8,54,54}_{8} = \langle a, b, c\mid{}& a^{3}, b^{3}, c^{3}, baba,\ bcb^{-1}c^{-1}b^{-1}cbc^{-1}, (bcb^{-1}c)^{3}, aca^{-1}c^{-1}a^{-1}cac^{-1}, (aca^{-1}c)^{3}\rangle\\ 
\end{align*}

\subsection{Half-girth type $(3,3,3)$}
\begin{align*}
% 14 14 14 0 1 2 6
G^{14,14,14}_{0} = \langle a, b, c\mid{}& a^{3}, b^{3}, c^{3}, bab^{-1}a^{-1}ba, cbc^{-1}b^{-1}cb, aca^{-1}c^{-1}ac\rangle\\ 
G^{14,14,14}_{1} = \langle a, b, c\mid{}& a^{3}, b^{3}, c^{3}, bab^{-1}a^{-1}ba, cbc^{-1}b^{-1}cb, ac^{-1}a^{-1}cac^{-1}\rangle\\ 
G^{14,14,14}_{2} = \langle a, b, c\mid{}& a^{3}, b^{3}, c^{3}, bab^{-1}a^{-1}ba, cbc^{-1}b^{-1}cb, cac^{-1}a^{-1}ca\rangle\\ 
G^{14,14,14}_{6} = \langle a, b, c\mid{}& a^{3}, b^{3}, c^{3}, bab^{-1}a^{-1}ba, cb^{-1}c^{-1}bcb^{-1}, cac^{-1}a^{-1}ca\rangle\\ 
% 14 14 16 0 1 4 5
G^{14,14,16}_{0} = \langle a, b, c\mid{}& a^{3}, b^{3}, c^{3}, bab^{-1}a^{-1}ba, cbc^{-1}b^{-1}cb, acac^{-1}a^{-1}c^{-1}\rangle\\ 
G^{14,14,16}_{1} = \langle a, b, c\mid{}& a^{3}, b^{3}, c^{3}, bab^{-1}a^{-1}ba, cbc^{-1}b^{-1}cb, ac^{-1}aca^{-1}c\rangle\\ 
G^{14,14,16}_{4} = \langle a, b, c\mid{}& a^{3}, b^{3}, c^{3}, bab^{-1}a^{-1}ba, cb^{-1}c^{-1}bcb^{-1}, acac^{-1}a^{-1}c^{-1}\rangle\\ 
G^{14,14,16}_{5} = \langle a, b, c\mid{}& a^{3}, b^{3}, c^{3}, bab^{-1}a^{-1}ba, cb^{-1}c^{-1}bcb^{-1}, ac^{-1}aca^{-1}c\rangle\\ 
% 14 14 18 0 4
G^{14,14,18}_{0} = \langle a, b, c\mid{}& a^{3}, b^{3}, c^{3}, bab^{-1}a^{-1}ba, cbc^{-1}b^{-1}cb, (ac)^{3}, (ac^{-1})^{3}\rangle\\ 
G^{14,14,18}_{4} = \langle a, b, c\mid{}& a^{3}, b^{3}, c^{3}, bab^{-1}a^{-1}ba, cb^{-1}c^{-1}bcb^{-1}, (ac)^{3}, (ac^{-1})^{3}\rangle\\ 
% 14 14 24 0 1 4 5
G^{14,14,24}_{0} = \langle a, b, c\mid{}& a^{3}, b^{3}, c^{3}, bab^{-1}a^{-1}ba, cbc^{-1}b^{-1}cb, (ac)^{3}, aca^{-1}ca^{-1}c^{-1}ac^{-1}\rangle\\ 
G^{14,14,24}_{1} = \langle a, b, c\mid{}& a^{3}, b^{3}, c^{3}, bab^{-1}a^{-1}ba, cbc^{-1}b^{-1}cb, (ac^{-1})^{3}, ac^{-1}a^{-1}c^{-1}a^{-1}cac\rangle\\ 
G^{14,14,24}_{4} = \langle a, b, c\mid{}& a^{3}, b^{3}, c^{3}, bab^{-1}a^{-1}ba, cb^{-1}c^{-1}bcb^{-1}, (ac)^{3}, aca^{-1}ca^{-1}c^{-1}ac^{-1}\rangle\\ 
G^{14,14,24}_{5} = \langle a, b, c\mid{}& a^{3}, b^{3}, c^{3}, bab^{-1}a^{-1}ba, cb^{-1}c^{-1}bcb^{-1}, (ac^{-1})^{3}, ac^{-1}a^{-1}c^{-1}a^{-1}cac\rangle\\ 
% 14 14 26 0 1 3 4 5 7
G^{14,14,26}_{0} = \langle a, b, c\mid{}& a^{3}, b^{3}, c^{3}, bab^{-1}a^{-1}ba, cbc^{-1}b^{-1}cb, (ac)^{3}, aca^{-1}ca^{-1}ca^{-1}c^{-1}\rangle\\ 
G^{14,14,26}_{1} = \langle a, b, c\mid{}& a^{3}, b^{3}, c^{3}, bab^{-1}a^{-1}ba, cbc^{-1}b^{-1}cb, (ac^{-1})^{3}, ac^{-1}a^{-1}c^{-1}a^{-1}c^{-1}a^{-1}c\rangle\\ 
G^{14,14,26}_{3} = \langle a, b, c\mid{}& a^{3}, b^{3}, c^{3}, bab^{-1}a^{-1}ba, cbc^{-1}b^{-1}cb, (ca^{-1})^{3}, ca^{-1}c^{-1}a^{-1}c^{-1}a^{-1}c^{-1}a\rangle\\ 
G^{14,14,26}_{4} = \langle a, b, c\mid{}& a^{3}, b^{3}, c^{3}, bab^{-1}a^{-1}ba, cb^{-1}c^{-1}bcb^{-1}, (ac)^{3}, aca^{-1}ca^{-1}ca^{-1}c^{-1}\rangle\\ 
G^{14,14,26}_{5} = \langle a, b, c\mid{}& a^{3}, b^{3}, c^{3}, bab^{-1}a^{-1}ba, cb^{-1}c^{-1}bcb^{-1}, (ac^{-1})^{3}, ac^{-1}a^{-1}c^{-1}a^{-1}c^{-1}a^{-1}c\rangle\\ 
G^{14,14,26}_{7} = \langle a, b, c\mid{}& a^{3}, b^{3}, c^{3}, bab^{-1}a^{-1}ba, cb^{-1}c^{-1}bcb^{-1}, (ca^{-1})^{3}, ca^{-1}c^{-1}a^{-1}c^{-1}a^{-1}c^{-1}a\rangle\\ 
% 14 16 16 0 1
G^{14,16,16}_{0} = \langle a, b, c\mid{}& a^{3}, b^{3}, c^{3}, bab^{-1}a^{-1}ba, cbcb^{-1}c^{-1}b^{-1}, acac^{-1}a^{-1}c^{-1}\rangle\\ 
G^{14,16,16}_{1} = \langle a, b, c\mid{}& a^{3}, b^{3}, c^{3}, bab^{-1}a^{-1}ba, cbcb^{-1}c^{-1}b^{-1}, ac^{-1}aca^{-1}c\rangle\\ 
% 14 16 18 0
G^{14,16,18}_{0} = \langle a, b, c\mid{}& a^{3}, b^{3}, c^{3}, bab^{-1}a^{-1}ba, cbcb^{-1}c^{-1}b^{-1}, (ac)^{3}, (ac^{-1})^{3}\rangle\\ 
% 14 16 24 0 1
G^{14,16,24}_{0} = \langle a, b, c\mid{}& a^{3}, b^{3}, c^{3}, bab^{-1}a^{-1}ba, cbcb^{-1}c^{-1}b^{-1}, (ac)^{3}, aca^{-1}ca^{-1}c^{-1}ac^{-1}\rangle\\ 
G^{14,16,24}_{1} = \langle a, b, c\mid{}& a^{3}, b^{3}, c^{3}, bab^{-1}a^{-1}ba, cbcb^{-1}c^{-1}b^{-1}, (ac^{-1})^{3}, ac^{-1}a^{-1}c^{-1}a^{-1}cac\rangle\\ 
% 14 16 26 0 1 3 7
G^{14,16,26}_{0} = \langle a, b, c\mid{}& a^{3}, b^{3}, c^{3}, bab^{-1}a^{-1}ba, cbcb^{-1}c^{-1}b^{-1}, (ac)^{3}, aca^{-1}ca^{-1}ca^{-1}c^{-1}\rangle\\ 
G^{14,16,26}_{1} = \langle a, b, c\mid{}& a^{3}, b^{3}, c^{3}, bab^{-1}a^{-1}ba, cbcb^{-1}c^{-1}b^{-1}, (ac^{-1})^{3}, ac^{-1}a^{-1}c^{-1}a^{-1}c^{-1}a^{-1}c\rangle\\ 
G^{14,16,26}_{3} = \langle a, b, c\mid{}& a^{3}, b^{3}, c^{3}, bab^{-1}a^{-1}ba, cbcb^{-1}c^{-1}b^{-1}, (ca^{-1})^{3}, ca^{-1}c^{-1}a^{-1}c^{-1}a^{-1}c^{-1}a\rangle\\ 
G^{14,16,26}_{7} = \langle a, b, c\mid{}& a^{3}, b^{3}, c^{3}, bab^{-1}a^{-1}ba, cb^{-1}cbc^{-1}b, (ca^{-1})^{3}, ca^{-1}c^{-1}a^{-1}c^{-1}a^{-1}c^{-1}a\rangle\\ 
% 14 18 18 0
G^{14,18,18}_{0} = \langle a, b, c\mid{}& a^{3}, b^{3}, c^{3}, bab^{-1}a^{-1}ba, (cb)^{3}, (cb^{-1})^{3}, (ac)^{3}, (ac^{-1})^{3}\rangle\\ 
% 14 18 24 0
G^{14,18,24}_{0} = \langle a, b, c\mid{}& a^{3}, b^{3}, c^{3}, bab^{-1}a^{-1}ba, (cb)^{3}, (cb^{-1})^{3}, (ac)^{3}, aca^{-1}ca^{-1}c^{-1}ac^{-1}\rangle\\ 
% 14 18 26 0 3
G^{14,18,26}_{0} = \langle a, b, c\mid{}& a^{3}, b^{3}, c^{3}, bab^{-1}a^{-1}ba, (cb)^{3}, (cb^{-1})^{3}, (ac)^{3}, aca^{-1}ca^{-1}ca^{-1}c^{-1}\rangle\\ 
G^{14,18,26}_{3} = \langle a, b, c\mid{}& a^{3}, b^{3}, c^{3}, bab^{-1}a^{-1}ba, (cb)^{3}, (cb^{-1})^{3}, (ca^{-1})^{3}, ca^{-1}c^{-1}a^{-1}c^{-1}a^{-1}c^{-1}a\rangle\\ 
% 14 24 24 0 1
G^{14,24,24}_{0} = \langle a, b, c\mid{}& a^{3}, b^{3}, c^{3}, bab^{-1}a^{-1}ba, (cb)^{3}, cbc^{-1}bc^{-1}b^{-1}cb^{-1}, (ac)^{3}, aca^{-1}ca^{-1}c^{-1}ac^{-1}\rangle\\ 
G^{14,24,24}_{1} = \langle a, b, c\mid{}& a^{3}, b^{3}, c^{3}, bab^{-1}a^{-1}ba, (cb)^{3}, cbc^{-1}bc^{-1}b^{-1}cb^{-1}, (ac^{-1})^{3}, ac^{-1}a^{-1}c^{-1}a^{-1}cac\rangle\\ 
% 14 24 26 0 1 3 7
G^{14,24,26}_{0} = \langle a, b, c\mid{}& a^{3}, b^{3}, c^{3}, bab^{-1}a^{-1}ba, (cb)^{3}, cbc^{-1}bc^{-1}b^{-1}cb^{-1}, (ac)^{3}, aca^{-1}ca^{-1}ca^{-1}c^{-1}\rangle\\ 
G^{14,24,26}_{1} = \langle a, b, c\mid{}& a^{3}, b^{3}, c^{3}, bab^{-1}a^{-1}ba, (cb)^{3}, cbc^{-1}bc^{-1}b^{-1}cb^{-1}, (ac^{-1})^{3}, ac^{-1}a^{-1}c^{-1}a^{-1}c^{-1}a^{-1}c\rangle\\ 
G^{14,24,26}_{3} = \langle a, b, c\mid{}& a^{3}, b^{3}, c^{3}, bab^{-1}a^{-1}ba, (cb)^{3}, cbc^{-1}bc^{-1}b^{-1}cb^{-1}, (ca^{-1})^{3}, ca^{-1}c^{-1}a^{-1}c^{-1}a^{-1}c^{-1}a\rangle\\ 
G^{14,24,26}_{7} = \langle a, b, c\mid{}& a^{3}, b^{3}, c^{3}, bab^{-1}a^{-1}ba, (cb^{-1})^{3}, cb^{-1}c^{-1}b^{-1}c^{-1}bcb, (ca^{-1})^{3}, ca^{-1}c^{-1}a^{-1}c^{-1}a^{-1}c^{-1}a\rangle\\ 
% 14 26 26 0 1 3 4 5 15
G^{14,26,26}_{0} = \langle a, b, c\mid{}& a^{3}, b^{3}, c^{3}, bab^{-1}a^{-1}ba, (cb)^{3}, cbc^{-1}bc^{-1}bc^{-1}b^{-1}, (ac)^{3}, aca^{-1}ca^{-1}ca^{-1}c^{-1}\rangle\\ 
G^{14,26,26}_{1} = \langle a, b, c\mid{}& a^{3}, b^{3}, c^{3}, bab^{-1}a^{-1}ba, (cb)^{3}, cbc^{-1}bc^{-1}bc^{-1}b^{-1}, (ac^{-1})^{3}, ac^{-1}a^{-1}c^{-1}a^{-1}c^{-1}a^{-1}c\rangle\\ 
G^{14,26,26}_{3} = \langle a, b, c\mid{}& a^{3}, b^{3}, c^{3}, bab^{-1}a^{-1}ba, (cb)^{3}, cbc^{-1}bc^{-1}bc^{-1}b^{-1}, (ca^{-1})^{3}, ca^{-1}c^{-1}a^{-1}c^{-1}a^{-1}c^{-1}a\rangle\\ 
G^{14,26,26}_{4} = \langle a, b, c\mid{}& a^{3}, b^{3}, c^{3}, bab^{-1}a^{-1}ba, (cb^{-1})^{3}, cb^{-1}c^{-1}b^{-1}c^{-1}b^{-1}c^{-1}b, (ac)^{3}, aca^{-1}ca^{-1}ca^{-1}c^{-1}\rangle\\ 
G^{14,26,26}_{5} = \langle a, b, c\mid{}& a^{3}, b^{3}, c^{3}, bab^{-1}a^{-1}ba, (cb^{-1})^{3}, cb^{-1}c^{-1}b^{-1}c^{-1}b^{-1}c^{-1}b, (ac^{-1})^{3}, ac^{-1}a^{-1}c^{-1}a^{-1}c^{-1}a^{-1}c\rangle\\ 
G^{14,26,26}_{15} = \langle a, b, c\mid{}& a^{3}, b^{3}, c^{3}, bab^{-1}a^{-1}ba, (bc^{-1})^{3}, bc^{-1}b^{-1}c^{-1}b^{-1}c^{-1}b^{-1}c, (ca^{-1})^{3}, ca^{-1}c^{-1}a^{-1}c^{-1}a^{-1}c^{-1}a\rangle\\ 
% 16 16 16 0 1
G^{16,16,16}_{0} = \langle a, b, c\mid{}& a^{3}, b^{3}, c^{3}, baba^{-1}b^{-1}a^{-1}, cbcb^{-1}c^{-1}b^{-1}, acac^{-1}a^{-1}c^{-1}\rangle\\ 
G^{16,16,16}_{1} = \langle a, b, c\mid{}& a^{3}, b^{3}, c^{3}, baba^{-1}b^{-1}a^{-1}, cbcb^{-1}c^{-1}b^{-1}, ac^{-1}aca^{-1}c\rangle\\ 
% 16 16 18 0
G^{16,16,18}_{0} = \langle a, b, c\mid{}& a^{3}, b^{3}, c^{3}, baba^{-1}b^{-1}a^{-1}, cbcb^{-1}c^{-1}b^{-1}, (ac)^{3}, (ac^{-1})^{3}\rangle\\ 
% 16 16 24 0 1
G^{16,16,24}_{0} = \langle a, b, c\mid{}& a^{3}, b^{3}, c^{3}, baba^{-1}b^{-1}a^{-1}, cbcb^{-1}c^{-1}b^{-1}, (ac)^{3}, aca^{-1}ca^{-1}c^{-1}ac^{-1}\rangle\\ 
G^{16,16,24}_{1} = \langle a, b, c\mid{}& a^{3}, b^{3}, c^{3}, baba^{-1}b^{-1}a^{-1}, cbcb^{-1}c^{-1}b^{-1}, (ac^{-1})^{3}, ac^{-1}a^{-1}c^{-1}a^{-1}cac\rangle\\ 
% 16 16 26 0 1
G^{16,16,26}_{0} = \langle a, b, c\mid{}& a^{3}, b^{3}, c^{3}, baba^{-1}b^{-1}a^{-1}, cbcb^{-1}c^{-1}b^{-1}, (ac)^{3}, aca^{-1}ca^{-1}ca^{-1}c^{-1}\rangle\\ 
G^{16,16,26}_{1} = \langle a, b, c\mid{}& a^{3}, b^{3}, c^{3}, baba^{-1}b^{-1}a^{-1}, cbcb^{-1}c^{-1}b^{-1}, (ac^{-1})^{3}, ac^{-1}a^{-1}c^{-1}a^{-1}c^{-1}a^{-1}c\rangle\\ 
% 16 18 18 0
G^{16,18,18}_{0} = \langle a, b, c\mid{}& a^{3}, b^{3}, c^{3}, baba^{-1}b^{-1}a^{-1}, (cb)^{3}, (cb^{-1})^{3}, (ac)^{3}, (ac^{-1})^{3}\rangle\\ 
% 16 18 24 0
G^{16,18,24}_{0} = \langle a, b, c\mid{}& a^{3}, b^{3}, c^{3}, baba^{-1}b^{-1}a^{-1}, (cb)^{3}, (cb^{-1})^{3}, (ac)^{3}, aca^{-1}ca^{-1}c^{-1}ac^{-1}\rangle\\ 
% 16 18 26 0
G^{16,18,26}_{0} = \langle a, b, c\mid{}& a^{3}, b^{3}, c^{3}, baba^{-1}b^{-1}a^{-1}, (cb)^{3}, (cb^{-1})^{3}, (ac)^{3}, aca^{-1}ca^{-1}ca^{-1}c^{-1}\rangle\\ 
% 16 24 24 0 1
G^{16,24,24}_{0} = \langle a, b, c\mid{}& a^{3}, b^{3}, c^{3}, baba^{-1}b^{-1}a^{-1}, (cb)^{3}, cbc^{-1}bc^{-1}b^{-1}cb^{-1}, (ac)^{3}, aca^{-1}ca^{-1}c^{-1}ac^{-1}\rangle\\ 
G^{16,24,24}_{1} = \langle a, b, c\mid{}& a^{3}, b^{3}, c^{3}, baba^{-1}b^{-1}a^{-1}, (cb)^{3}, cbc^{-1}bc^{-1}b^{-1}cb^{-1}, (ac^{-1})^{3}, ac^{-1}a^{-1}c^{-1}a^{-1}cac\rangle\\ 
% 16 24 26 0 1
G^{16,24,26}_{0} = \langle a, b, c\mid{}& a^{3}, b^{3}, c^{3}, baba^{-1}b^{-1}a^{-1}, (cb)^{3}, cbc^{-1}bc^{-1}b^{-1}cb^{-1}, (ac)^{3}, aca^{-1}ca^{-1}ca^{-1}c^{-1}\rangle\\ 
G^{16,24,26}_{1} = \langle a, b, c\mid{}& a^{3}, b^{3}, c^{3}, baba^{-1}b^{-1}a^{-1}, (cb)^{3}, cbc^{-1}bc^{-1}b^{-1}cb^{-1}, (ac^{-1})^{3}, ac^{-1}a^{-1}c^{-1}a^{-1}c^{-1}a^{-1}c\rangle\\ 
% 16 26 26 0 1 3 5
G^{16,26,26}_{0} = \langle a, b, c\mid{}& a^{3}, b^{3}, c^{3}, baba^{-1}b^{-1}a^{-1}, (cb)^{3}, cbc^{-1}bc^{-1}bc^{-1}b^{-1}, (ac)^{3}, aca^{-1}ca^{-1}ca^{-1}c^{-1}\rangle\\ 
G^{16,26,26}_{1} = \langle a, b, c\mid{}& a^{3}, b^{3}, c^{3}, baba^{-1}b^{-1}a^{-1}, (cb)^{3}, cbc^{-1}bc^{-1}bc^{-1}b^{-1}, (ac^{-1})^{3}, ac^{-1}a^{-1}c^{-1}a^{-1}c^{-1}a^{-1}c\rangle\\ 
G^{16,26,26}_{3} = \langle a, b, c\mid{}& a^{3}, b^{3}, c^{3}, baba^{-1}b^{-1}a^{-1}, (cb)^{3}, cbc^{-1}bc^{-1}bc^{-1}b^{-1}, (ca^{-1})^{3}, ca^{-1}c^{-1}a^{-1}c^{-1}a^{-1}c^{-1}a\rangle\\ 
G^{16,26,26}_{5} = \langle a, b, c\mid{}& a^{3}, b^{3}, c^{3}, baba^{-1}b^{-1}a^{-1}, (cb^{-1})^{3}, cb^{-1}c^{-1}b^{-1}c^{-1}b^{-1}c^{-1}b, (ac^{-1})^{3}, ac^{-1}a^{-1}c^{-1}a^{-1}c^{-1}a^{-1}c\rangle\\ 
% 18 18 18 0
G^{18,18,18}_{0} = \langle a, b, c\mid{}& a^{3}, b^{3}, c^{3}, (ba)^{3}, (ba^{-1})^{3}, (cb)^{3}, (cb^{-1})^{3}, (ac)^{3}, (ac^{-1})^{3}\rangle\\ 
% 18 18 24 0
G^{18,18,24}_{0} = \langle a, b, c\mid{}& a^{3}, b^{3}, c^{3}, (ba)^{3}, (ba^{-1})^{3}, (cb)^{3}, (cb^{-1})^{3}, (ac)^{3}, aca^{-1}ca^{-1}c^{-1}ac^{-1}\rangle\\ 
% 18 18 26 0
G^{18,18,26}_{0} = \langle a, b, c\mid{}& a^{3}, b^{3}, c^{3}, (ba)^{3}, (ba^{-1})^{3}, (cb)^{3}, (cb^{-1})^{3}, (ac)^{3}, aca^{-1}ca^{-1}ca^{-1}c^{-1}\rangle\\ 
% 18 24 24 0
G^{18,24,24}_{0} = \langle a, b, c\mid{}& a^{3}, b^{3}, c^{3}, (ba)^{3}, (ba^{-1})^{3}, (cb)^{3}, cbc^{-1}bc^{-1}b^{-1}cb^{-1}, (ac)^{3}, aca^{-1}ca^{-1}c^{-1}ac^{-1}\rangle\\ 
% 18 24 26 0
G^{18,24,26}_{0} = \langle a, b, c\mid{}& a^{3}, b^{3}, c^{3}, (ba)^{3}, (ba^{-1})^{3}, (cb)^{3}, cbc^{-1}bc^{-1}b^{-1}cb^{-1}, (ac)^{3}, aca^{-1}ca^{-1}ca^{-1}c^{-1}\rangle\\ 
% 18 26 26 0 1
G^{18,26,26}_{0} = \langle a, b, c\mid{}& a^{3}, b^{3}, c^{3}, (ba)^{3}, (ba^{-1})^{3}, (cb)^{3}, cbc^{-1}bc^{-1}bc^{-1}b^{-1}, (ac)^{3}, aca^{-1}ca^{-1}ca^{-1}c^{-1}\rangle\\ 
G^{18,26,26}_{1} = \langle a, b, c\mid{}& a^{3}, b^{3}, c^{3}, (ba)^{3}, (ba^{-1})^{3}, (cb)^{3}, cbc^{-1}bc^{-1}bc^{-1}b^{-1}, (ac^{-1})^{3}, ac^{-1}a^{-1}c^{-1}a^{-1}c^{-1}a^{-1}c\rangle\\ 
% 24 24 24 0 1
G^{24,24,24}_{0} = \langle a, b, c\mid{}& a^{3}, b^{3}, c^{3}, (ba)^{3}, bab^{-1}ab^{-1}a^{-1}ba^{-1}, (cb)^{3}, cbc^{-1}bc^{-1}b^{-1}cb^{-1}, (ac)^{3}, aca^{-1}ca^{-1}c^{-1}ac^{-1}\rangle\\ 
G^{24,24,24}_{1} = \langle a, b, c\mid{}& a^{3}, b^{3}, c^{3}, (ba)^{3}, bab^{-1}ab^{-1}a^{-1}ba^{-1}, (cb)^{3}, cbc^{-1}bc^{-1}b^{-1}cb^{-1}, (ac^{-1})^{3}, ac^{-1}a^{-1}c^{-1}a^{-1}cac\rangle\\ 
% 24 24 26 0 1
G^{24,24,26}_{0} = \langle a, b, c\mid{}& a^{3}, b^{3}, c^{3}, (ba)^{3}, bab^{-1}ab^{-1}a^{-1}ba^{-1}, (cb)^{3}, cbc^{-1}bc^{-1}b^{-1}cb^{-1}, (ac)^{3}, aca^{-1}ca^{-1}ca^{-1}c^{-1}\rangle\\ 
G^{24,24,26}_{1} = \langle a, b, c\mid{}& a^{3}, b^{3}, c^{3}, (ba)^{3}, bab^{-1}ab^{-1}a^{-1}ba^{-1}, (cb)^{3}, cbc^{-1}bc^{-1}b^{-1}cb^{-1}, (ac^{-1})^{3}, ac^{-1}a^{-1}c^{-1}a^{-1}c^{-1}a^{-1}c\rangle\\ 
% 24 26 26 0 1 3 5
G^{24,26,26}_{0} = \langle a, b, c\mid{}& a^{3}, b^{3}, c^{3}, (ba)^{3}, bab^{-1}ab^{-1}a^{-1}ba^{-1}, (cb)^{3}, cbc^{-1}bc^{-1}bc^{-1}b^{-1}, (ac)^{3}, aca^{-1}ca^{-1}ca^{-1}c^{-1}\rangle\\ 
G^{24,26,26}_{1} = \langle a, b, c\mid{}& a^{3}, b^{3}, c^{3}, (ba)^{3}, bab^{-1}ab^{-1}a^{-1}ba^{-1}, (cb)^{3}, cbc^{-1}bc^{-1}bc^{-1}b^{-1}, (ac^{-1})^{3}, ac^{-1}a^{-1}c^{-1}a^{-1}c^{-1}a^{-1}c\rangle\\ 
G^{24,26,26}_{3} = \langle a, b, c\mid{}& a^{3}, b^{3}, c^{3}, (ba)^{3}, bab^{-1}ab^{-1}a^{-1}ba^{-1}, (cb)^{3}, cbc^{-1}bc^{-1}bc^{-1}b^{-1}, (ca^{-1})^{3}, ca^{-1}c^{-1}a^{-1}c^{-1}a^{-1}c^{-1}a\rangle\\ 
G^{24,26,26}_{5} = \langle a, b, c\mid{}& a^{3}, b^{3}, c^{3}, (ba)^{3}, bab^{-1}ab^{-1}a^{-1}ba^{-1}, (cb^{-1})^{3}, cb^{-1}c^{-1}b^{-1}c^{-1}b^{-1}c^{-1}b, (ac^{-1})^{3}, ac^{-1}a^{-1}c^{-1}a^{-1}c^{-1}a^{-1}c\rangle\\ 
% 26 26 26 0 1 5 21
G^{26,26,26}_{0} = \langle a, b, c\mid{}& a^{3}, b^{3}, c^{3}, (ba)^{3}, bab^{-1}ab^{-1}ab^{-1}a^{-1}, (cb)^{3}, cbc^{-1}bc^{-1}bc^{-1}b^{-1}, (ac)^{3}, aca^{-1}ca^{-1}ca^{-1}c^{-1}\rangle\\ 
G^{26,26,26}_{1} = \langle a, b, c\mid{}& a^{3}, b^{3}, c^{3}, (ba)^{3}, bab^{-1}ab^{-1}ab^{-1}a^{-1}, (cb)^{3}, cbc^{-1}bc^{-1}bc^{-1}b^{-1}, (ac^{-1})^{3}, ac^{-1}a^{-1}c^{-1}a^{-1}c^{-1}a^{-1}c\rangle\\ 
G^{26,26,26}_{5} = \langle a, b, c\mid{}& a^{3}, b^{3}, c^{3}, (ba)^{3}, bab^{-1}ab^{-1}ab^{-1}a^{-1}, (cb^{-1})^{3}, cb^{-1}c^{-1}b^{-1}c^{-1}b^{-1}c^{-1}b, (ac^{-1})^{3}, ac^{-1}a^{-1}c^{-1}a^{-1}c^{-1}a^{-1}c\rangle\\ 
G^{26,26,26}_{21} = \langle a, b, c\mid{}& a^{3}, b^{3}, c^{3}, (ba^{-1})^{3}, ba^{-1}b^{-1}a^{-1}b^{-1}a^{-1}b^{-1}a, (cb^{-1})^{3}, cb^{-1}c^{-1}b^{-1}c^{-1}b^{-1}c^{-1}b, (ac^{-1})^{3}, ac^{-1}a^{-1}c^{-1}a^{-1}c^{-1}a^{-1}c\rangle\\ 
\end{align*}

\subsection{Half-girth type $(3,3,4)$}
\begin{align*}
% 14 14 40 0 4
G^{14,14,40}_{0} = \langle a, b, c\mid{}& a^{3}, b^{3}, c^{3}, bab^{-1}a^{-1}ba, cbc^{-1}b^{-1}cb, (ac^{-1}ac)^{2}, (a^{-1}c^{-1}ac^{-1})^{2}\rangle\\ 
G^{14,14,40}_{4} = \langle a, b, c\mid{}& a^{3}, b^{3}, c^{3}, bab^{-1}a^{-1}ba, cb^{-1}c^{-1}bcb^{-1}, (ac^{-1}ac)^{2}, (a^{-1}c^{-1}ac^{-1})^{2}\rangle\\ 
% 14 14 48 0 1 4 5
G^{14,14,48}_{0} = \langle a, b, c\mid{}& a^{3}, b^{3}, c^{3}, bab^{-1}a^{-1}ba, cbc^{-1}b^{-1}cb, (ac)^{2}(a^{-1}c^{-1})^{2}\rangle\\ 
G^{14,14,48}_{1} = \langle a, b, c\mid{}& a^{3}, b^{3}, c^{3}, bab^{-1}a^{-1}ba, cbc^{-1}b^{-1}cb, (ac^{-1})^{2}(a^{-1}c)^{2}\rangle\\ 
G^{14,14,48}_{4} = \langle a, b, c\mid{}& a^{3}, b^{3}, c^{3}, bab^{-1}a^{-1}ba, cb^{-1}c^{-1}bcb^{-1}, (ac)^{2}(a^{-1}c^{-1})^{2}\rangle\\ 
G^{14,14,48}_{5} = \langle a, b, c\mid{}& a^{3}, b^{3}, c^{3}, bab^{-1}a^{-1}ba, cb^{-1}c^{-1}bcb^{-1}, (ac^{-1})^{2}(a^{-1}c)^{2}\rangle\\ 
% 14 14 54 0 4
G^{14,14,54}_{0} = \langle a, b, c\mid{}& a^{3}, b^{3}, c^{3}, bab^{-1}a^{-1}ba, cbc^{-1}b^{-1}cb, aca^{-1}c^{-1}a^{-1}cac^{-1}, (aca^{-1}c)^{3}\rangle\\ 
G^{14,14,54}_{4} = \langle a, b, c\mid{}& a^{3}, b^{3}, c^{3}, bab^{-1}a^{-1}ba, cb^{-1}c^{-1}bcb^{-1}, aca^{-1}c^{-1}a^{-1}cac^{-1}, (aca^{-1}c)^{3}\rangle\\ 
% 14 16 40 0
G^{14,16,40}_{0} = \langle a, b, c\mid{}& a^{3}, b^{3}, c^{3}, bab^{-1}a^{-1}ba, cbcb^{-1}c^{-1}b^{-1}, (ac^{-1}ac)^{2}, (a^{-1}c^{-1}ac^{-1})^{2}\rangle\\ 
% 14 16 48 0 1
G^{14,16,48}_{0} = \langle a, b, c\mid{}& a^{3}, b^{3}, c^{3}, bab^{-1}a^{-1}ba, cbcb^{-1}c^{-1}b^{-1}, (ac)^{2}(a^{-1}c^{-1})^{2}\rangle\\ 
G^{14,16,48}_{1} = \langle a, b, c\mid{}& a^{3}, b^{3}, c^{3}, bab^{-1}a^{-1}ba, cbcb^{-1}c^{-1}b^{-1}, (ac^{-1})^{2}(a^{-1}c)^{2}\rangle\\ 
% 14 16 54 0 2
G^{14,16,54}_{0} = \langle a, b, c\mid{}& a^{3}, b^{3}, c^{3}, bab^{-1}a^{-1}ba, cbcb^{-1}c^{-1}b^{-1}, aca^{-1}c^{-1}a^{-1}cac^{-1}, (aca^{-1}c)^{3}\rangle\\ 
G^{14,16,54}_{2} = \langle a, b, c\mid{}& a^{3}, b^{3}, c^{3}, bab^{-1}a^{-1}ba, cbcb^{-1}c^{-1}b^{-1}, cac^{-1}a^{-1}c^{-1}aca^{-1}, (cac^{-1}a)^{3}\rangle\\ 
% 14 18 40 0
G^{14,18,40}_{0} = \langle a, b, c\mid{}& a^{3}, b^{3}, c^{3}, bab^{-1}a^{-1}ba, (cb)^{3}, (cb^{-1})^{3}, (ac^{-1}ac)^{2}, (a^{-1}c^{-1}ac^{-1})^{2}\rangle\\ 
% 14 18 48 0
G^{14,18,48}_{0} = \langle a, b, c\mid{}& a^{3}, b^{3}, c^{3}, bab^{-1}a^{-1}ba, (cb)^{3}, (cb^{-1})^{3}, (ac)^{2}(a^{-1}c^{-1})^{2}\rangle\\ 
% 14 18 54 0 2
G^{14,18,54}_{0} = \langle a, b, c\mid{}& a^{3}, b^{3}, c^{3}, bab^{-1}a^{-1}ba, (cb)^{3}, (cb^{-1})^{3}, aca^{-1}c^{-1}a^{-1}cac^{-1}, (aca^{-1}c)^{3}\rangle\\ 
G^{14,18,54}_{2} = \langle a, b, c\mid{}& a^{3}, b^{3}, c^{3}, bab^{-1}a^{-1}ba, (cb)^{3}, (cb^{-1})^{3}, cac^{-1}a^{-1}c^{-1}aca^{-1}, (cac^{-1}a)^{3}\rangle\\ 
% 14 24 40 0
G^{14,24,40}_{0} = \langle a, b, c\mid{}& a^{3}, b^{3}, c^{3}, bab^{-1}a^{-1}ba, (cb)^{3}, cbc^{-1}bc^{-1}b^{-1}cb^{-1}, (ac^{-1}ac)^{2}, (a^{-1}c^{-1}ac^{-1})^{2}\rangle\\ 
% 14 24 48 0 1
G^{14,24,48}_{0} = \langle a, b, c\mid{}& a^{3}, b^{3}, c^{3}, bab^{-1}a^{-1}ba, (cb)^{3}, cbc^{-1}bc^{-1}b^{-1}cb^{-1}, (ac)^{2}(a^{-1}c^{-1})^{2}\rangle\\ 
G^{14,24,48}_{1} = \langle a, b, c\mid{}& a^{3}, b^{3}, c^{3}, bab^{-1}a^{-1}ba, (cb)^{3}, cbc^{-1}bc^{-1}b^{-1}cb^{-1}, (ac^{-1})^{2}(a^{-1}c)^{2}\rangle\\ 
% 14 24 54 0 2
G^{14,24,54}_{0} = \langle a, b, c\mid{}& a^{3}, b^{3}, c^{3}, bab^{-1}a^{-1}ba, (cb)^{3}, cbc^{-1}bc^{-1}b^{-1}cb^{-1}, aca^{-1}c^{-1}a^{-1}cac^{-1}, (aca^{-1}c)^{3}\rangle\\ 
G^{14,24,54}_{2} = \langle a, b, c\mid{}& a^{3}, b^{3}, c^{3}, bab^{-1}a^{-1}ba, (cb)^{3}, cbc^{-1}bc^{-1}b^{-1}cb^{-1}, cac^{-1}a^{-1}c^{-1}aca^{-1}, (cac^{-1}a)^{3}\rangle\\ 
% 14 26 40 0 4
G^{14,26,40}_{0} = \langle a, b, c\mid{}& a^{3}, b^{3}, c^{3}, bab^{-1}a^{-1}ba, (cb)^{3}, cbc^{-1}bc^{-1}bc^{-1}b^{-1}, (ac^{-1}ac)^{2}, (a^{-1}c^{-1}ac^{-1})^{2}\rangle\\ 
G^{14,26,40}_{4} = \langle a, b, c\mid{}& a^{3}, b^{3}, c^{3}, bab^{-1}a^{-1}ba, (cb^{-1})^{3}, cb^{-1}c^{-1}b^{-1}c^{-1}b^{-1}c^{-1}b, (ac^{-1}ac)^{2}, (a^{-1}c^{-1}ac^{-1})^{2}\rangle\\ 
% 14 26 48 0 1 4 5
G^{14,26,48}_{0} = \langle a, b, c\mid{}& a^{3}, b^{3}, c^{3}, bab^{-1}a^{-1}ba, (cb)^{3}, cbc^{-1}bc^{-1}bc^{-1}b^{-1}, (ac)^{2}(a^{-1}c^{-1})^{2}\rangle\\ 
G^{14,26,48}_{1} = \langle a, b, c\mid{}& a^{3}, b^{3}, c^{3}, bab^{-1}a^{-1}ba, (cb)^{3}, cbc^{-1}bc^{-1}bc^{-1}b^{-1}, (ac^{-1})^{2}(a^{-1}c)^{2}\rangle\\ 
G^{14,26,48}_{4} = \langle a, b, c\mid{}& a^{3}, b^{3}, c^{3}, bab^{-1}a^{-1}ba, (cb^{-1})^{3}, cb^{-1}c^{-1}b^{-1}c^{-1}b^{-1}c^{-1}b, (ac)^{2}(a^{-1}c^{-1})^{2}\rangle\\ 
G^{14,26,48}_{5} = \langle a, b, c\mid{}& a^{3}, b^{3}, c^{3}, bab^{-1}a^{-1}ba, (cb^{-1})^{3}, cb^{-1}c^{-1}b^{-1}c^{-1}b^{-1}c^{-1}b, (ac^{-1})^{2}(a^{-1}c)^{2}\rangle\\ 
% 14 26 54 0 2 4 6
G^{14,26,54}_{0} = \langle a, b, c\mid{}& a^{3}, b^{3}, c^{3}, bab^{-1}a^{-1}ba, (cb)^{3}, cbc^{-1}bc^{-1}bc^{-1}b^{-1}, aca^{-1}c^{-1}a^{-1}cac^{-1}, (aca^{-1}c)^{3}\rangle\\ 
G^{14,26,54}_{2} = \langle a, b, c\mid{}& a^{3}, b^{3}, c^{3}, bab^{-1}a^{-1}ba, (cb)^{3}, cbc^{-1}bc^{-1}bc^{-1}b^{-1}, cac^{-1}a^{-1}c^{-1}aca^{-1}, (cac^{-1}a)^{3}\rangle\\ 
G^{14,26,54}_{4} = \langle a, b, c\mid{}& a^{3}, b^{3}, c^{3}, bab^{-1}a^{-1}ba, (cb^{-1})^{3}, cb^{-1}c^{-1}b^{-1}c^{-1}b^{-1}c^{-1}b, aca^{-1}c^{-1}a^{-1}cac^{-1}, (aca^{-1}c)^{3}\rangle\\ 
G^{14,26,54}_{6} = \langle a, b, c\mid{}& a^{3}, b^{3}, c^{3}, bab^{-1}a^{-1}ba, (cb^{-1})^{3}, cb^{-1}c^{-1}b^{-1}c^{-1}b^{-1}c^{-1}b, cac^{-1}a^{-1}c^{-1}aca^{-1}, (cac^{-1}a)^{3}\rangle\\ 
% 16 16 40 0
G^{16,16,40}_{0} = \langle a, b, c\mid{}& a^{3}, b^{3}, c^{3}, baba^{-1}b^{-1}a^{-1}, cbcb^{-1}c^{-1}b^{-1}, (ac^{-1}ac)^{2}, (a^{-1}c^{-1}ac^{-1})^{2}\rangle\\ 
% 16 16 48 0 1
G^{16,16,48}_{0} = \langle a, b, c\mid{}& a^{3}, b^{3}, c^{3}, baba^{-1}b^{-1}a^{-1}, cbcb^{-1}c^{-1}b^{-1}, (ac)^{2}(a^{-1}c^{-1})^{2}\rangle\\ 
G^{16,16,48}_{1} = \langle a, b, c\mid{}& a^{3}, b^{3}, c^{3}, baba^{-1}b^{-1}a^{-1}, cbcb^{-1}c^{-1}b^{-1}, (ac^{-1})^{2}(a^{-1}c)^{2}\rangle\\ 
% 16 16 54 0
G^{16,16,54}_{0} = \langle a, b, c\mid{}& a^{3}, b^{3}, c^{3}, baba^{-1}b^{-1}a^{-1}, cbcb^{-1}c^{-1}b^{-1}, aca^{-1}c^{-1}a^{-1}cac^{-1}, (aca^{-1}c)^{3}\rangle\\ 
% 16 18 40 0
G^{16,18,40}_{0} = \langle a, b, c\mid{}& a^{3}, b^{3}, c^{3}, baba^{-1}b^{-1}a^{-1}, (cb)^{3}, (cb^{-1})^{3}, (ac^{-1}ac)^{2}, (a^{-1}c^{-1}ac^{-1})^{2}\rangle\\ 
% 16 18 48 0
G^{16,18,48}_{0} = \langle a, b, c\mid{}& a^{3}, b^{3}, c^{3}, baba^{-1}b^{-1}a^{-1}, (cb)^{3}, (cb^{-1})^{3}, (ac)^{2}(a^{-1}c^{-1})^{2}\rangle\\ 
% 16 18 54 0 2
G^{16,18,54}_{0} = \langle a, b, c\mid{}& a^{3}, b^{3}, c^{3}, baba^{-1}b^{-1}a^{-1}, (cb)^{3}, (cb^{-1})^{3}, aca^{-1}c^{-1}a^{-1}cac^{-1}, (aca^{-1}c)^{3}\rangle\\ 
G^{16,18,54}_{2} = \langle a, b, c\mid{}& a^{3}, b^{3}, c^{3}, baba^{-1}b^{-1}a^{-1}, (cb)^{3}, (cb^{-1})^{3}, cac^{-1}a^{-1}c^{-1}aca^{-1}, (cac^{-1}a)^{3}\rangle\\ 
% 16 24 40 0
G^{16,24,40}_{0} = \langle a, b, c\mid{}& a^{3}, b^{3}, c^{3}, baba^{-1}b^{-1}a^{-1}, (cb)^{3}, cbc^{-1}bc^{-1}b^{-1}cb^{-1}, (ac^{-1}ac)^{2}, (a^{-1}c^{-1}ac^{-1})^{2}\rangle\\ 
% 16 24 48 0 1
G^{16,24,48}_{0} = \langle a, b, c\mid{}& a^{3}, b^{3}, c^{3}, baba^{-1}b^{-1}a^{-1}, (cb)^{3}, cbc^{-1}bc^{-1}b^{-1}cb^{-1}, (ac)^{2}(a^{-1}c^{-1})^{2}\rangle\\ 
G^{16,24,48}_{1} = \langle a, b, c\mid{}& a^{3}, b^{3}, c^{3}, baba^{-1}b^{-1}a^{-1}, (cb)^{3}, cbc^{-1}bc^{-1}b^{-1}cb^{-1}, (ac^{-1})^{2}(a^{-1}c)^{2}\rangle\\ 
% 16 24 54 0 2
G^{16,24,54}_{0} = \langle a, b, c\mid{}& a^{3}, b^{3}, c^{3}, baba^{-1}b^{-1}a^{-1}, (cb)^{3}, cbc^{-1}bc^{-1}b^{-1}cb^{-1}, aca^{-1}c^{-1}a^{-1}cac^{-1}, (aca^{-1}c)^{3}\rangle\\ 
G^{16,24,54}_{2} = \langle a, b, c\mid{}& a^{3}, b^{3}, c^{3}, baba^{-1}b^{-1}a^{-1}, (cb)^{3}, cbc^{-1}bc^{-1}b^{-1}cb^{-1}, cac^{-1}a^{-1}c^{-1}aca^{-1}, (cac^{-1}a)^{3}\rangle\\ 
% 16 26 40 0
G^{16,26,40}_{0} = \langle a, b, c\mid{}& a^{3}, b^{3}, c^{3}, baba^{-1}b^{-1}a^{-1}, (cb)^{3}, cbc^{-1}bc^{-1}bc^{-1}b^{-1}, (ac^{-1}ac)^{2}, (a^{-1}c^{-1}ac^{-1})^{2}\rangle\\ 
% 16 26 48 0 1
G^{16,26,48}_{0} = \langle a, b, c\mid{}& a^{3}, b^{3}, c^{3}, baba^{-1}b^{-1}a^{-1}, (cb)^{3}, cbc^{-1}bc^{-1}bc^{-1}b^{-1}, (ac)^{2}(a^{-1}c^{-1})^{2}\rangle\\ 
G^{16,26,48}_{1} = \langle a, b, c\mid{}& a^{3}, b^{3}, c^{3}, baba^{-1}b^{-1}a^{-1}, (cb)^{3}, cbc^{-1}bc^{-1}bc^{-1}b^{-1}, (ac^{-1})^{2}(a^{-1}c)^{2}\rangle\\ 
% 16 26 54 0 2
G^{16,26,54}_{0} = \langle a, b, c\mid{}& a^{3}, b^{3}, c^{3}, baba^{-1}b^{-1}a^{-1}, (cb)^{3}, cbc^{-1}bc^{-1}bc^{-1}b^{-1}, aca^{-1}c^{-1}a^{-1}cac^{-1}, (aca^{-1}c)^{3}\rangle\\ 
G^{16,26,54}_{2} = \langle a, b, c\mid{}& a^{3}, b^{3}, c^{3}, baba^{-1}b^{-1}a^{-1}, (cb)^{3}, cbc^{-1}bc^{-1}bc^{-1}b^{-1}, cac^{-1}a^{-1}c^{-1}aca^{-1}, (cac^{-1}a)^{3}\rangle\\ 
% 18 18 40 0
G^{18,18,40}_{0} = \langle a, b, c\mid{}& a^{3}, b^{3}, c^{3}, (ba)^{3}, (ba^{-1})^{3}, (cb)^{3}, (cb^{-1})^{3}, (ac^{-1}ac)^{2}, (a^{-1}c^{-1}ac^{-1})^{2}\rangle\\ 
% 18 18 48 0
G^{18,18,48}_{0} = \langle a, b, c\mid{}& a^{3}, b^{3}, c^{3}, (ba)^{3}, (ba^{-1})^{3}, (cb)^{3}, (cb^{-1})^{3}, (ac)^{2}(a^{-1}c^{-1})^{2}\rangle\\ 
% 18 18 54 0
G^{18,18,54}_{0} = \langle a, b, c\mid{}& a^{3}, b^{3}, c^{3}, (ba)^{3}, (ba^{-1})^{3}, (cb)^{3}, (cb^{-1})^{3}, aca^{-1}c^{-1}a^{-1}cac^{-1}, (aca^{-1}c)^{3}\rangle\\ 
% 18 24 40 0
G^{18,24,40}_{0} = \langle a, b, c\mid{}& a^{3}, b^{3}, c^{3}, (ba)^{3}, (ba^{-1})^{3}, (cb)^{3}, cbc^{-1}bc^{-1}b^{-1}cb^{-1}, (ac^{-1}ac)^{2}, (a^{-1}c^{-1}ac^{-1})^{2}\rangle\\ 
% 18 24 48 0
G^{18,24,48}_{0} = \langle a, b, c\mid{}& a^{3}, b^{3}, c^{3}, (ba)^{3}, (ba^{-1})^{3}, (cb)^{3}, cbc^{-1}bc^{-1}b^{-1}cb^{-1}, (ac)^{2}(a^{-1}c^{-1})^{2}\rangle\\ 
% 18 24 54 0 2
G^{18,24,54}_{0} = \langle a, b, c\mid{}& a^{3}, b^{3}, c^{3}, (ba)^{3}, (ba^{-1})^{3}, (cb)^{3}, cbc^{-1}bc^{-1}b^{-1}cb^{-1}, aca^{-1}c^{-1}a^{-1}cac^{-1}, (aca^{-1}c)^{3}\rangle\\ 
G^{18,24,54}_{2} = \langle a, b, c\mid{}& a^{3}, b^{3}, c^{3}, (ba)^{3}, (ba^{-1})^{3}, (cb)^{3}, cbc^{-1}bc^{-1}b^{-1}cb^{-1}, cac^{-1}a^{-1}c^{-1}aca^{-1}, (cac^{-1}a)^{3}\rangle\\ 
% 18 26 40 0
G^{18,26,40}_{0} = \langle a, b, c\mid{}& a^{3}, b^{3}, c^{3}, (ba)^{3}, (ba^{-1})^{3}, (cb)^{3}, cbc^{-1}bc^{-1}bc^{-1}b^{-1}, (ac^{-1}ac)^{2}, (a^{-1}c^{-1}ac^{-1})^{2}\rangle\\ 
% 18 26 48 0
G^{18,26,48}_{0} = \langle a, b, c\mid{}& a^{3}, b^{3}, c^{3}, (ba)^{3}, (ba^{-1})^{3}, (cb)^{3}, cbc^{-1}bc^{-1}bc^{-1}b^{-1}, (ac)^{2}(a^{-1}c^{-1})^{2}\rangle\\ 
% 18 26 54 0 2
G^{18,26,54}_{0} = \langle a, b, c\mid{}& a^{3}, b^{3}, c^{3}, (ba)^{3}, (ba^{-1})^{3}, (cb)^{3}, cbc^{-1}bc^{-1}bc^{-1}b^{-1}, aca^{-1}c^{-1}a^{-1}cac^{-1}, (aca^{-1}c)^{3}\rangle\\ 
G^{18,26,54}_{2} = \langle a, b, c\mid{}& a^{3}, b^{3}, c^{3}, (ba)^{3}, (ba^{-1})^{3}, (cb)^{3}, cbc^{-1}bc^{-1}bc^{-1}b^{-1}, cac^{-1}a^{-1}c^{-1}aca^{-1}, (cac^{-1}a)^{3}\rangle\\ 
% 24 24 40 0
G^{24,24,40}_{0} = \langle a, b, c\mid{}& a^{3}, b^{3}, c^{3}, (ba)^{3}, bab^{-1}ab^{-1}a^{-1}ba^{-1}, (cb)^{3}, cbc^{-1}bc^{-1}b^{-1}cb^{-1}, (ac^{-1}ac)^{2}, (a^{-1}c^{-1}ac^{-1})^{2}\rangle\\ 
% 24 24 48 0 1
G^{24,24,48}_{0} = \langle a, b, c\mid{}& a^{3}, b^{3}, c^{3}, (ba)^{3}, bab^{-1}ab^{-1}a^{-1}ba^{-1}, (cb)^{3}, cbc^{-1}bc^{-1}b^{-1}cb^{-1}, (ac)^{2}(a^{-1}c^{-1})^{2}\rangle\\ 
G^{24,24,48}_{1} = \langle a, b, c\mid{}& a^{3}, b^{3}, c^{3}, (ba)^{3}, bab^{-1}ab^{-1}a^{-1}ba^{-1}, (cb)^{3}, cbc^{-1}bc^{-1}b^{-1}cb^{-1}, (ac^{-1})^{2}(a^{-1}c)^{2}\rangle\\ 
% 24 24 54 0
G^{24,24,54}_{0} = \langle a, b, c\mid{}& a^{3}, b^{3}, c^{3}, (ba)^{3}, bab^{-1}ab^{-1}a^{-1}ba^{-1}, (cb)^{3}, cbc^{-1}bc^{-1}b^{-1}cb^{-1}, aca^{-1}c^{-1}a^{-1}cac^{-1}, (aca^{-1}c)^{3}\rangle\\ 
% 24 26 40 0
G^{24,26,40}_{0} = \langle a, b, c\mid{}& a^{3}, b^{3}, c^{3}, (ba)^{3}, bab^{-1}ab^{-1}a^{-1}ba^{-1}, (cb)^{3}, cbc^{-1}bc^{-1}bc^{-1}b^{-1}, (ac^{-1}ac)^{2}, (a^{-1}c^{-1}ac^{-1})^{2}\rangle\\ 
% 24 26 48 0 1
G^{24,26,48}_{0} = \langle a, b, c\mid{}& a^{3}, b^{3}, c^{3}, (ba)^{3}, bab^{-1}ab^{-1}a^{-1}ba^{-1}, (cb)^{3}, cbc^{-1}bc^{-1}bc^{-1}b^{-1}, (ac)^{2}(a^{-1}c^{-1})^{2}\rangle\\ 
G^{24,26,48}_{1} = \langle a, b, c\mid{}& a^{3}, b^{3}, c^{3}, (ba)^{3}, bab^{-1}ab^{-1}a^{-1}ba^{-1}, (cb)^{3}, cbc^{-1}bc^{-1}bc^{-1}b^{-1}, (ac^{-1})^{2}(a^{-1}c)^{2}\rangle\\ 
% 24 26 54 0 2
G^{24,26,54}_{0} = \langle a, b, c\mid{}& a^{3}, b^{3}, c^{3}, (ba)^{3}, bab^{-1}ab^{-1}a^{-1}ba^{-1}, (cb)^{3}, cbc^{-1}bc^{-1}bc^{-1}b^{-1}, aca^{-1}c^{-1}a^{-1}cac^{-1}, (aca^{-1}c)^{3}\rangle\\ 
G^{24,26,54}_{2} = \langle a, b, c\mid{}& a^{3}, b^{3}, c^{3}, (ba)^{3}, bab^{-1}ab^{-1}a^{-1}ba^{-1}, (cb)^{3}, cbc^{-1}bc^{-1}bc^{-1}b^{-1}, cac^{-1}a^{-1}c^{-1}aca^{-1}, (cac^{-1}a)^{3}\rangle\\ 
% 26 26 40 0 4
G^{26,26,40}_{0} = \langle a, b, c\mid{}& a^{3}, b^{3}, c^{3}, (ba)^{3}, bab^{-1}ab^{-1}ab^{-1}a^{-1}, (cb)^{3}, cbc^{-1}bc^{-1}bc^{-1}b^{-1}, (ac^{-1}ac)^{2}, (a^{-1}c^{-1}ac^{-1})^{2}\rangle\\ 
G^{26,26,40}_{4} = \langle a, b, c\mid{}& a^{3}, b^{3}, c^{3}, (ba)^{3}, bab^{-1}ab^{-1}ab^{-1}a^{-1}, (cb^{-1})^{3}, cb^{-1}c^{-1}b^{-1}c^{-1}b^{-1}c^{-1}b, (ac^{-1}ac)^{2}, (a^{-1}c^{-1}ac^{-1})^{2}\rangle\\ 
% 26 26 48 0 1 4 5
G^{26,26,48}_{0} = \langle a, b, c\mid{}& a^{3}, b^{3}, c^{3}, (ba)^{3}, bab^{-1}ab^{-1}ab^{-1}a^{-1}, (cb)^{3}, cbc^{-1}bc^{-1}bc^{-1}b^{-1}, (ac)^{2}(a^{-1}c^{-1})^{2}\rangle\\ 
G^{26,26,48}_{1} = \langle a, b, c\mid{}& a^{3}, b^{3}, c^{3}, (ba)^{3}, bab^{-1}ab^{-1}ab^{-1}a^{-1}, (cb)^{3}, cbc^{-1}bc^{-1}bc^{-1}b^{-1}, (ac^{-1})^{2}(a^{-1}c)^{2}\rangle\\ 
G^{26,26,48}_{4} = \langle a, b, c\mid{}& a^{3}, b^{3}, c^{3}, (ba)^{3}, bab^{-1}ab^{-1}ab^{-1}a^{-1}, (cb^{-1})^{3}, cb^{-1}c^{-1}b^{-1}c^{-1}b^{-1}c^{-1}b, (ac)^{2}(a^{-1}c^{-1})^{2}\rangle\\ 
G^{26,26,48}_{5} = \langle a, b, c\mid{}& a^{3}, b^{3}, c^{3}, (ba)^{3}, bab^{-1}ab^{-1}ab^{-1}a^{-1}, (cb^{-1})^{3}, cb^{-1}c^{-1}b^{-1}c^{-1}b^{-1}c^{-1}b, (ac^{-1})^{2}(a^{-1}c)^{2}\rangle\\ 
% 26 26 54 0 4
G^{26,26,54}_{0} = \langle a, b, c\mid{}& a^{3}, b^{3}, c^{3}, (ba)^{3}, bab^{-1}ab^{-1}ab^{-1}a^{-1}, (cb)^{3}, cbc^{-1}bc^{-1}bc^{-1}b^{-1}, aca^{-1}c^{-1}a^{-1}cac^{-1}, (aca^{-1}c)^{3}\rangle\\ 
G^{26,26,54}_{4} = \langle a, b, c\mid{}& a^{3}, b^{3}, c^{3}, (ba)^{3}, bab^{-1}ab^{-1}ab^{-1}a^{-1}, (cb^{-1})^{3}, cb^{-1}c^{-1}b^{-1}c^{-1}b^{-1}c^{-1}b, aca^{-1}c^{-1}a^{-1}cac^{-1}, (aca^{-1}c)^{3}\rangle\\ 
\end{align*}

\subsection{Half-girth type $(3,4,4)$}
\begin{align*}
% 14 40 40 0
G^{14,40,40}_{0} = \langle a, b, c\mid{}& a^{3}, b^{3}, c^{3}, bab^{-1}a^{-1}ba, (cb^{-1}cb)^{2}, (c^{-1}b^{-1}cb^{-1})^{2}, (ac^{-1}ac)^{2}, (a^{-1}c^{-1}ac^{-1})^{2}\rangle\\ 
% 14 40 48 0
G^{14,40,48}_{0} = \langle a, b, c\mid{}& a^{3}, b^{3}, c^{3}, bab^{-1}a^{-1}ba, (cb^{-1}cb)^{2}, (c^{-1}b^{-1}cb^{-1})^{2}, (ac)^{2}(a^{-1}c^{-1})^{2}\rangle\\ 
% 14 40 54 0 2
G^{14,40,54}_{0} = \langle a, b, c\mid{}& a^{3}, b^{3}, c^{3}, bab^{-1}a^{-1}ba, (cb^{-1}cb)^{2}, (c^{-1}b^{-1}cb^{-1})^{2}, aca^{-1}c^{-1}a^{-1}cac^{-1}, (aca^{-1}c)^{3}\rangle\\ 
G^{14,40,54}_{2} = \langle a, b, c\mid{}& a^{3}, b^{3}, c^{3}, bab^{-1}a^{-1}ba, (cb^{-1}cb)^{2}, (c^{-1}b^{-1}cb^{-1})^{2}, cac^{-1}a^{-1}c^{-1}aca^{-1}, (cac^{-1}a)^{3}\rangle\\ 
% 14 48 48 0 1
G^{14,48,48}_{0} = \langle a, b, c\mid{}& a^{3}, b^{3}, c^{3}, bab^{-1}a^{-1}ba, (cb)^{2}(c^{-1}b^{-1})^{2}, (ac)^{2}(a^{-1}c^{-1})^{2}\rangle\\ 
G^{14,48,48}_{1} = \langle a, b, c\mid{}& a^{3}, b^{3}, c^{3}, bab^{-1}a^{-1}ba, (cb)^{2}(c^{-1}b^{-1})^{2}, (ac^{-1})^{2}(a^{-1}c)^{2}\rangle\\ 
% 14 48 54 0 2
G^{14,48,54}_{0} = \langle a, b, c\mid{}& a^{3}, b^{3}, c^{3}, bab^{-1}a^{-1}ba, (cb)^{2}(c^{-1}b^{-1})^{2}, aca^{-1}c^{-1}a^{-1}cac^{-1}, (aca^{-1}c)^{3}\rangle\\ 
G^{14,48,54}_{2} = \langle a, b, c\mid{}& a^{3}, b^{3}, c^{3}, bab^{-1}a^{-1}ba, (cb)^{2}(c^{-1}b^{-1})^{2}, cac^{-1}a^{-1}c^{-1}aca^{-1}, (cac^{-1}a)^{3}\rangle\\ 
% 14 54 54 0 2 8
G^{14,54,54}_{0} = \langle a, b, c\mid{}& a^{3}, b^{3}, c^{3}, bab^{-1}a^{-1}ba, cbc^{-1}b^{-1}c^{-1}bcb^{-1}, (cbc^{-1}b)^{3}, aca^{-1}c^{-1}a^{-1}cac^{-1}, (aca^{-1}c)^{3}\rangle\\ 
G^{14,54,54}_{2} = \langle a, b, c\mid{}& a^{3}, b^{3}, c^{3}, bab^{-1}a^{-1}ba, cbc^{-1}b^{-1}c^{-1}bcb^{-1}, (cbc^{-1}b)^{3}, cac^{-1}a^{-1}c^{-1}aca^{-1}, (cac^{-1}a)^{3}\rangle\\ 
G^{14,54,54}_{8} = \langle a, b, c\mid{}& a^{3}, b^{3}, c^{3}, bab^{-1}a^{-1}ba, bcb^{-1}c^{-1}b^{-1}cbc^{-1}, (bcb^{-1}c)^{3}, aca^{-1}c^{-1}a^{-1}cac^{-1}, (aca^{-1}c)^{3}\rangle\\ 
% 16 40 40 0
G^{16,40,40}_{0} = \langle a, b, c\mid{}& a^{3}, b^{3}, c^{3}, baba^{-1}b^{-1}a^{-1}, (cb^{-1}cb)^{2}, (c^{-1}b^{-1}cb^{-1})^{2}, (ac^{-1}ac)^{2}, (a^{-1}c^{-1}ac^{-1})^{2}\rangle\\ 
% 16 40 48 0
G^{16,40,48}_{0} = \langle a, b, c\mid{}& a^{3}, b^{3}, c^{3}, baba^{-1}b^{-1}a^{-1}, (cb^{-1}cb)^{2}, (c^{-1}b^{-1}cb^{-1})^{2}, (ac)^{2}(a^{-1}c^{-1})^{2}\rangle\\ 
% 16 40 54 0 2
G^{16,40,54}_{0} = \langle a, b, c\mid{}& a^{3}, b^{3}, c^{3}, baba^{-1}b^{-1}a^{-1}, (cb^{-1}cb)^{2}, (c^{-1}b^{-1}cb^{-1})^{2}, aca^{-1}c^{-1}a^{-1}cac^{-1}, (aca^{-1}c)^{3}\rangle\\ 
G^{16,40,54}_{2} = \langle a, b, c\mid{}& a^{3}, b^{3}, c^{3}, baba^{-1}b^{-1}a^{-1}, (cb^{-1}cb)^{2}, (c^{-1}b^{-1}cb^{-1})^{2}, cac^{-1}a^{-1}c^{-1}aca^{-1}, (cac^{-1}a)^{3}\rangle\\ 
% 16 48 48 0 1
G^{16,48,48}_{0} = \langle a, b, c\mid{}& a^{3}, b^{3}, c^{3}, baba^{-1}b^{-1}a^{-1}, (cb)^{2}(c^{-1}b^{-1})^{2}, (ac)^{2}(a^{-1}c^{-1})^{2}\rangle\\ 
G^{16,48,48}_{1} = \langle a, b, c\mid{}& a^{3}, b^{3}, c^{3}, baba^{-1}b^{-1}a^{-1}, (cb)^{2}(c^{-1}b^{-1})^{2}, (ac^{-1})^{2}(a^{-1}c)^{2}\rangle\\ 
% 16 48 54 0 2
G^{16,48,54}_{0} = \langle a, b, c\mid{}& a^{3}, b^{3}, c^{3}, baba^{-1}b^{-1}a^{-1}, (cb)^{2}(c^{-1}b^{-1})^{2}, aca^{-1}c^{-1}a^{-1}cac^{-1}, (aca^{-1}c)^{3}\rangle\\ 
G^{16,48,54}_{2} = \langle a, b, c\mid{}& a^{3}, b^{3}, c^{3}, baba^{-1}b^{-1}a^{-1}, (cb)^{2}(c^{-1}b^{-1})^{2}, cac^{-1}a^{-1}c^{-1}aca^{-1}, (cac^{-1}a)^{3}\rangle\\ 
% 16 54 54 0 2 8
G^{16,54,54}_{0} = \langle a, b, c\mid{}& a^{3}, b^{3}, c^{3}, baba^{-1}b^{-1}a^{-1}, cbc^{-1}b^{-1}c^{-1}bcb^{-1}, (cbc^{-1}b)^{3}, aca^{-1}c^{-1}a^{-1}cac^{-1}, (aca^{-1}c)^{3}\rangle\\ 
G^{16,54,54}_{2} = \langle a, b, c\mid{}& a^{3}, b^{3}, c^{3}, baba^{-1}b^{-1}a^{-1}, cbc^{-1}b^{-1}c^{-1}bcb^{-1}, (cbc^{-1}b)^{3}, cac^{-1}a^{-1}c^{-1}aca^{-1}, (cac^{-1}a)^{3}\rangle\\ 
G^{16,54,54}_{8} = \langle a, b, c\mid{}& a^{3}, b^{3}, c^{3}, baba^{-1}b^{-1}a^{-1}, bcb^{-1}c^{-1}b^{-1}cbc^{-1}, (bcb^{-1}c)^{3}, aca^{-1}c^{-1}a^{-1}cac^{-1}, (aca^{-1}c)^{3}\rangle\\ 
% 18 40 40 0
G^{18,40,40}_{0} = \langle a, b, c\mid{}& a^{3}, b^{3}, c^{3}, (ba)^{3}, (ba^{-1})^{3}, (cb^{-1}cb)^{2}, (c^{-1}b^{-1}cb^{-1})^{2}, (ac^{-1}ac)^{2}, (a^{-1}c^{-1}ac^{-1})^{2}\rangle\\ 
% 18 40 48 0
G^{18,40,48}_{0} = \langle a, b, c\mid{}& a^{3}, b^{3}, c^{3}, (ba)^{3}, (ba^{-1})^{3}, (cb^{-1}cb)^{2}, (c^{-1}b^{-1}cb^{-1})^{2}, (ac)^{2}(a^{-1}c^{-1})^{2}\rangle\\ 
% 18 40 54 0 2
G^{18,40,54}_{0} = \langle a, b, c\mid{}& a^{3}, b^{3}, c^{3}, (ba)^{3}, (ba^{-1})^{3}, (cb^{-1}cb)^{2}, (c^{-1}b^{-1}cb^{-1})^{2}, aca^{-1}c^{-1}a^{-1}cac^{-1}, (aca^{-1}c)^{3}\rangle\\ 
G^{18,40,54}_{2} = \langle a, b, c\mid{}& a^{3}, b^{3}, c^{3}, (ba)^{3}, (ba^{-1})^{3}, (cb^{-1}cb)^{2}, (c^{-1}b^{-1}cb^{-1})^{2}, cac^{-1}a^{-1}c^{-1}aca^{-1}, (cac^{-1}a)^{3}\rangle\\ 
% 18 48 48 0
G^{18,48,48}_{0} = \langle a, b, c\mid{}& a^{3}, b^{3}, c^{3}, (ba)^{3}, (ba^{-1})^{3}, (cb)^{2}(c^{-1}b^{-1})^{2}, (ac)^{2}(a^{-1}c^{-1})^{2}\rangle\\ 
% 18 48 54 0 2
G^{18,48,54}_{0} = \langle a, b, c\mid{}& a^{3}, b^{3}, c^{3}, (ba)^{3}, (ba^{-1})^{3}, (cb)^{2}(c^{-1}b^{-1})^{2}, aca^{-1}c^{-1}a^{-1}cac^{-1}, (aca^{-1}c)^{3}\rangle\\ 
G^{18,48,54}_{2} = \langle a, b, c\mid{}& a^{3}, b^{3}, c^{3}, (ba)^{3}, (ba^{-1})^{3}, (cb)^{2}(c^{-1}b^{-1})^{2}, cac^{-1}a^{-1}c^{-1}aca^{-1}, (cac^{-1}a)^{3}\rangle\\ 
% 18 54 54 0 2 8
G^{18,54,54}_{0} = \langle a, b, c\mid{}& a^{3}, b^{3}, c^{3}, (ba)^{3}, (ba^{-1})^{3}, cbc^{-1}b^{-1}c^{-1}bcb^{-1}, (cbc^{-1}b)^{3}, aca^{-1}c^{-1}a^{-1}cac^{-1}, (aca^{-1}c)^{3}\rangle\\ 
G^{18,54,54}_{2} = \langle a, b, c\mid{}& a^{3}, b^{3}, c^{3}, (ba)^{3}, (ba^{-1})^{3}, cbc^{-1}b^{-1}c^{-1}bcb^{-1}, (cbc^{-1}b)^{3}, cac^{-1}a^{-1}c^{-1}aca^{-1}, (cac^{-1}a)^{3}\rangle\\ 
G^{18,54,54}_{8} = \langle a, b, c\mid{}& a^{3}, b^{3}, c^{3}, (ba)^{3}, (ba^{-1})^{3}, bcb^{-1}c^{-1}b^{-1}cbc^{-1}, (bcb^{-1}c)^{3}, aca^{-1}c^{-1}a^{-1}cac^{-1}, (aca^{-1}c)^{3}\rangle\\ 
% 24 40 40 0
G^{24,40,40}_{0} = \langle a, b, c\mid{}& a^{3}, b^{3}, c^{3}, (ba)^{3}, bab^{-1}ab^{-1}a^{-1}ba^{-1}, (cb^{-1}cb)^{2}, (c^{-1}b^{-1}cb^{-1})^{2}, (ac^{-1}ac)^{2}, (a^{-1}c^{-1}ac^{-1})^{2}\rangle\\ 
% 24 40 48 0
G^{24,40,48}_{0} = \langle a, b, c\mid{}& a^{3}, b^{3}, c^{3}, (ba)^{3}, bab^{-1}ab^{-1}a^{-1}ba^{-1}, (cb^{-1}cb)^{2}, (c^{-1}b^{-1}cb^{-1})^{2}, (ac)^{2}(a^{-1}c^{-1})^{2}\rangle\\ 
% 24 40 54 0 2
G^{24,40,54}_{0} = \langle a, b, c\mid{}& a^{3}, b^{3}, c^{3}, (ba)^{3}, bab^{-1}ab^{-1}a^{-1}ba^{-1}, (cb^{-1}cb)^{2}, (c^{-1}b^{-1}cb^{-1})^{2}, aca^{-1}c^{-1}a^{-1}cac^{-1}, (aca^{-1}c)^{3}\rangle\\ 
G^{24,40,54}_{2} = \langle a, b, c\mid{}& a^{3}, b^{3}, c^{3}, (ba)^{3}, bab^{-1}ab^{-1}a^{-1}ba^{-1}, (cb^{-1}cb)^{2}, (c^{-1}b^{-1}cb^{-1})^{2}, cac^{-1}a^{-1}c^{-1}aca^{-1}, (cac^{-1}a)^{3}\rangle\\ 
% 24 48 48 0 1
G^{24,48,48}_{0} = \langle a, b, c\mid{}& a^{3}, b^{3}, c^{3}, (ba)^{3}, bab^{-1}ab^{-1}a^{-1}ba^{-1}, (cb)^{2}(c^{-1}b^{-1})^{2}, (ac)^{2}(a^{-1}c^{-1})^{2}\rangle\\ 
G^{24,48,48}_{1} = \langle a, b, c\mid{}& a^{3}, b^{3}, c^{3}, (ba)^{3}, bab^{-1}ab^{-1}a^{-1}ba^{-1}, (cb)^{2}(c^{-1}b^{-1})^{2}, (ac^{-1})^{2}(a^{-1}c)^{2}\rangle\\ 
% 24 48 54 0 2
G^{24,48,54}_{0} = \langle a, b, c\mid{}& a^{3}, b^{3}, c^{3}, (ba)^{3}, bab^{-1}ab^{-1}a^{-1}ba^{-1}, (cb)^{2}(c^{-1}b^{-1})^{2}, aca^{-1}c^{-1}a^{-1}cac^{-1}, (aca^{-1}c)^{3}\rangle\\ 
G^{24,48,54}_{2} = \langle a, b, c\mid{}& a^{3}, b^{3}, c^{3}, (ba)^{3}, bab^{-1}ab^{-1}a^{-1}ba^{-1}, (cb)^{2}(c^{-1}b^{-1})^{2}, cac^{-1}a^{-1}c^{-1}aca^{-1}, (cac^{-1}a)^{3}\rangle\\ 
% 24 54 54 0 2 8
G^{24,54,54}_{0} = \langle a, b, c\mid{}& a^{3}, b^{3}, c^{3}, (ba)^{3}, bab^{-1}ab^{-1}a^{-1}ba^{-1}, cbc^{-1}b^{-1}c^{-1}bcb^{-1}, (cbc^{-1}b)^{3}, aca^{-1}c^{-1}a^{-1}cac^{-1}, (aca^{-1}c)^{3}\rangle\\ 
G^{24,54,54}_{2} = \langle a, b, c\mid{}& a^{3}, b^{3}, c^{3}, (ba)^{3}, bab^{-1}ab^{-1}a^{-1}ba^{-1}, cbc^{-1}b^{-1}c^{-1}bcb^{-1}, (cbc^{-1}b)^{3}, cac^{-1}a^{-1}c^{-1}aca^{-1}, (cac^{-1}a)^{3}\rangle\\ 
G^{24,54,54}_{8} = \langle a, b, c\mid{}& a^{3}, b^{3}, c^{3}, (ba)^{3}, bab^{-1}ab^{-1}a^{-1}ba^{-1}, bcb^{-1}c^{-1}b^{-1}cbc^{-1}, (bcb^{-1}c)^{3}, aca^{-1}c^{-1}a^{-1}cac^{-1}, (aca^{-1}c)^{3}\rangle\\ 
% 26 40 40 0
G^{26,40,40}_{0} = \langle a, b, c\mid{}& a^{3}, b^{3}, c^{3}, (ba)^{3}, bab^{-1}ab^{-1}ab^{-1}a^{-1}, (cb^{-1}cb)^{2}, (c^{-1}b^{-1}cb^{-1})^{2}, (ac^{-1}ac)^{2}, (a^{-1}c^{-1}ac^{-1})^{2}\rangle\\ 
% 26 40 48 0
G^{26,40,48}_{0} = \langle a, b, c\mid{}& a^{3}, b^{3}, c^{3}, (ba)^{3}, bab^{-1}ab^{-1}ab^{-1}a^{-1}, (cb^{-1}cb)^{2}, (c^{-1}b^{-1}cb^{-1})^{2}, (ac)^{2}(a^{-1}c^{-1})^{2}\rangle\\ 
% 26 40 54 0 2
G^{26,40,54}_{0} = \langle a, b, c\mid{}& a^{3}, b^{3}, c^{3}, (ba)^{3}, bab^{-1}ab^{-1}ab^{-1}a^{-1}, (cb^{-1}cb)^{2}, (c^{-1}b^{-1}cb^{-1})^{2}, aca^{-1}c^{-1}a^{-1}cac^{-1}, (aca^{-1}c)^{3}\rangle\\ 
G^{26,40,54}_{2} = \langle a, b, c\mid{}& a^{3}, b^{3}, c^{3}, (ba)^{3}, bab^{-1}ab^{-1}ab^{-1}a^{-1}, (cb^{-1}cb)^{2}, (c^{-1}b^{-1}cb^{-1})^{2}, cac^{-1}a^{-1}c^{-1}aca^{-1}, (cac^{-1}a)^{3}\rangle\\ 
% 26 48 48 0 1
G^{26,48,48}_{0} = \langle a, b, c\mid{}& a^{3}, b^{3}, c^{3}, (ba)^{3}, bab^{-1}ab^{-1}ab^{-1}a^{-1}, (cb)^{2}(c^{-1}b^{-1})^{2}, (ac)^{2}(a^{-1}c^{-1})^{2}\rangle\\ 
G^{26,48,48}_{1} = \langle a, b, c\mid{}& a^{3}, b^{3}, c^{3}, (ba)^{3}, bab^{-1}ab^{-1}ab^{-1}a^{-1}, (cb)^{2}(c^{-1}b^{-1})^{2}, (ac^{-1})^{2}(a^{-1}c)^{2}\rangle\\ 
% 26 48 54 0 2
G^{26,48,54}_{0} = \langle a, b, c\mid{}& a^{3}, b^{3}, c^{3}, (ba)^{3}, bab^{-1}ab^{-1}ab^{-1}a^{-1}, (cb)^{2}(c^{-1}b^{-1})^{2}, aca^{-1}c^{-1}a^{-1}cac^{-1}, (aca^{-1}c)^{3}\rangle\\ 
G^{26,48,54}_{2} = \langle a, b, c\mid{}& a^{3}, b^{3}, c^{3}, (ba)^{3}, bab^{-1}ab^{-1}ab^{-1}a^{-1}, (cb)^{2}(c^{-1}b^{-1})^{2}, cac^{-1}a^{-1}c^{-1}aca^{-1}, (cac^{-1}a)^{3}\rangle\\ 
% 26 54 54 0 2 8
G^{26,54,54}_{0} = \langle a, b, c\mid{}& a^{3}, b^{3}, c^{3}, (ba)^{3}, bab^{-1}ab^{-1}ab^{-1}a^{-1}, cbc^{-1}b^{-1}c^{-1}bcb^{-1}, (cbc^{-1}b)^{3}, aca^{-1}c^{-1}a^{-1}cac^{-1}, (aca^{-1}c)^{3}\rangle\\ 
G^{26,54,54}_{2} = \langle a, b, c\mid{}& a^{3}, b^{3}, c^{3}, (ba)^{3}, bab^{-1}ab^{-1}ab^{-1}a^{-1}, cbc^{-1}b^{-1}c^{-1}bcb^{-1}, (cbc^{-1}b)^{3}, cac^{-1}a^{-1}c^{-1}aca^{-1}, (cac^{-1}a)^{3}\rangle\\ 
G^{26,54,54}_{8} = \langle a, b, c\mid{}& a^{3}, b^{3}, c^{3}, (ba)^{3}, bab^{-1}ab^{-1}ab^{-1}a^{-1}, bcb^{-1}c^{-1}b^{-1}cbc^{-1}, (bcb^{-1}c)^{3}, aca^{-1}c^{-1}a^{-1}cac^{-1}, (aca^{-1}c)^{3}\rangle\\ 
\end{align*}

\subsection{Half-girth type $(4,4,4)$}
\begin{align*}
% 40 40 40 0
G^{40,40,40}_{0} = \langle a, b, c\mid{}& a^{3}, b^{3}, c^{3}, (ba^{-1}ba)^{2}, (b^{-1}a^{-1}ba^{-1})^{2}, (cb^{-1}cb)^{2}, (c^{-1}b^{-1}cb^{-1})^{2}, (ac^{-1}ac)^{2}, (a^{-1}c^{-1}ac^{-1})^{2}\rangle\\ 
% 40 40 48 0
G^{40,40,48}_{0} = \langle a, b, c\mid{}& a^{3}, b^{3}, c^{3}, (ba^{-1}ba)^{2}, (b^{-1}a^{-1}ba^{-1})^{2}, (cb^{-1}cb)^{2}, (c^{-1}b^{-1}cb^{-1})^{2}, (ac)^{2}(a^{-1}c^{-1})^{2}\rangle\\ 
% 40 40 54 0
G^{40,40,54}_{0} = \langle a, b, c\mid{}& a^{3}, b^{3}, c^{3}, (ba^{-1}ba)^{2}, (b^{-1}a^{-1}ba^{-1})^{2}, (cb^{-1}cb)^{2}, (c^{-1}b^{-1}cb^{-1})^{2}, aca^{-1}c^{-1}a^{-1}cac^{-1}, (aca^{-1}c)^{3}\rangle\\ 
% 40 48 48 0
G^{40,48,48}_{0} = \langle a, b, c\mid{}& a^{3}, b^{3}, c^{3}, (ba^{-1}ba)^{2}, (b^{-1}a^{-1}ba^{-1})^{2}, (cb)^{2}(c^{-1}b^{-1})^{2}, (ac)^{2}(a^{-1}c^{-1})^{2}\rangle\\ 
% 40 48 54 0 2
G^{40,48,54}_{0} = \langle a, b, c\mid{}& a^{3}, b^{3}, c^{3}, (ba^{-1}ba)^{2}, (b^{-1}a^{-1}ba^{-1})^{2}, (cb)^{2}(c^{-1}b^{-1})^{2}, aca^{-1}c^{-1}a^{-1}cac^{-1}, (aca^{-1}c)^{3}\rangle\\ 
G^{40,48,54}_{2} = \langle a, b, c\mid{}& a^{3}, b^{3}, c^{3}, (ba^{-1}ba)^{2}, (b^{-1}a^{-1}ba^{-1})^{2}, (cb)^{2}(c^{-1}b^{-1})^{2}, cac^{-1}a^{-1}c^{-1}aca^{-1}, (cac^{-1}a)^{3}\rangle\\ 
% 40 54 54 0 2 8
G^{40,54,54}_{0} = \langle a, b, c\mid{}& a^{3}, b^{3}, c^{3}, (ba^{-1}ba)^{2}, (b^{-1}a^{-1}ba^{-1})^{2}, cbc^{-1}b^{-1}c^{-1}bcb^{-1}, (cbc^{-1}b)^{3}, aca^{-1}c^{-1}a^{-1}cac^{-1}, (aca^{-1}c)^{3}\rangle\\ 
G^{40,54,54}_{2} = \langle a, b, c\mid{}& a^{3}, b^{3}, c^{3}, (ba^{-1}ba)^{2}, (b^{-1}a^{-1}ba^{-1})^{2}, cbc^{-1}b^{-1}c^{-1}bcb^{-1}, (cbc^{-1}b)^{3}, cac^{-1}a^{-1}c^{-1}aca^{-1}, (cac^{-1}a)^{3}\rangle\\ 
G^{40,54,54}_{8} = \langle a, b, c\mid{}& a^{3}, b^{3}, c^{3}, (ba^{-1}ba)^{2}, (b^{-1}a^{-1}ba^{-1})^{2}, bcb^{-1}c^{-1}b^{-1}cbc^{-1}, (bcb^{-1}c)^{3}, aca^{-1}c^{-1}a^{-1}cac^{-1}, (aca^{-1}c)^{3}\rangle\\ 
% 48 48 48 0 1
G^{48,48,48}_{0} = \langle a, b, c\mid{}& a^{3}, b^{3}, c^{3}, (ba)^{2}(b^{-1}a^{-1})^{2}, (cb)^{2}(c^{-1}b^{-1})^{2}, (ac)^{2}(a^{-1}c^{-1})^{2}\rangle\\ 
G^{48,48,48}_{1} = \langle a, b, c\mid{}& a^{3}, b^{3}, c^{3}, (ba)^{2}(b^{-1}a^{-1})^{2}, (cb)^{2}(c^{-1}b^{-1})^{2}, (ac^{-1})^{2}(a^{-1}c)^{2}\rangle\\ 
% 48 48 54 0
G^{48,48,54}_{0} = \langle a, b, c\mid{}& a^{3}, b^{3}, c^{3}, (ba)^{2}(b^{-1}a^{-1})^{2}, (cb)^{2}(c^{-1}b^{-1})^{2}, aca^{-1}c^{-1}a^{-1}cac^{-1}, (aca^{-1}c)^{3}\rangle\\ 
% 48 54 54 0 2 8
G^{48,54,54}_{0} = \langle a, b, c\mid{}& a^{3}, b^{3}, c^{3}, (ba)^{2}(b^{-1}a^{-1})^{2}, cbc^{-1}b^{-1}c^{-1}bcb^{-1}, (cbc^{-1}b)^{3}, aca^{-1}c^{-1}a^{-1}cac^{-1}, (aca^{-1}c)^{3}\rangle\\ 
G^{48,54,54}_{2} = \langle a, b, c\mid{}& a^{3}, b^{3}, c^{3}, (ba)^{2}(b^{-1}a^{-1})^{2}, cbc^{-1}b^{-1}c^{-1}bcb^{-1}, (cbc^{-1}b)^{3}, cac^{-1}a^{-1}c^{-1}aca^{-1}, (cac^{-1}a)^{3}\rangle\\ 
G^{48,54,54}_{8} = \langle a, b, c\mid{}& a^{3}, b^{3}, c^{3}, (ba)^{2}(b^{-1}a^{-1})^{2}, bcb^{-1}c^{-1}b^{-1}cbc^{-1}, (bcb^{-1}c)^{3}, aca^{-1}c^{-1}a^{-1}cac^{-1}, (aca^{-1}c)^{3}\rangle\\ 
% 54 54 54 0 2
G^{54,54,54}_{0} = \langle a, b, c\mid{}& a^{3}, b^{3}, c^{3}, bab^{-1}a^{-1}b^{-1}aba^{-1}, (bab^{-1}a)^{3}, cbc^{-1}b^{-1}c^{-1}bcb^{-1}, (cbc^{-1}b)^{3}, aca^{-1}c^{-1}a^{-1}cac^{-1}, (aca^{-1}c)^{3}\rangle\\ 
G^{54,54,54}_{2} = \langle a, b, c\mid{}& a^{3}, b^{3}, c^{3}, bab^{-1}a^{-1}b^{-1}aba^{-1}, (bab^{-1}a)^{3}, cbc^{-1}b^{-1}c^{-1}bcb^{-1}, (cbc^{-1}b)^{3}, cac^{-1}a^{-1}c^{-1}aca^{-1}, (cac^{-1}a)^{3}\rangle\\ 
\end{align*}

\subsection{Cyclic extensions}\label{sec:Presentation_extended}
\begin{align*}
\widetilde G_0^{14, 14, 14}= \langle a, t \mid{}&
a^3, t^3, a^{-1}  t^{-1}  a^{-1}  t  a  t^{-1}  a  t  a 
 t^{-1}  a  t \rangle\\
\widetilde G_0^{16, 16, 16}=\langle a, t \mid{}& 
 a^3, t^3, a^{-1}  t  a  t^{-1}  a  t  a  t^{-1}  a^{-1}  t  a^{-1}  t^{-1}
 \rangle\\
\widetilde G_0^{18, 18, 18}=\langle a, t \mid{}& 
 a^3, t^3, (a  t^{-1}  a  t)^3, [a^{-1}, t]^3 \rangle\\
\widetilde G_0^{24, 24, 24}=\langle a, t \mid{}& 
 a^3, t^3, (a  t^{-1}  a  t)^3, a  t^{-1}  a  t  a^{-1} 
 t^{-1}  a  t  a^{-1}  t^{-1}  a^{-1}  t  a  t^{-1}  a^{-1}  
t \rangle\\
\widetilde G_0^{26, 26, 26}=\langle a, t \mid{}& 
 a^3, t^3, (a  t^{-1}  a  t)^3, a  t^{-1}  a  t  a^{-1} 
 t^{-1}  a  t  a^{-1}  t^{-1}  a  t  a^{-1}  t^{-1}  a^{-1}  
t \rangle\\
\widetilde G_0^{40, 40, 40}=\langle a, t \mid{}& 
 a^3, t^3, (a  t^{-1}  a  t  a  t^{-1}  a^{-1}  t)^2, (a
 t^{-1}  a^{-1}  t  a^{-1}  t^{-1}  a^{-1}  t)^2 \rangle\\
\widetilde G_0^{48, 48, 48}=\langle a, t \mid{}& 
 a^3, t^3, a  t^{-1}  a  t  a  t^{-1}  a  t  a^{-1}  
t^{-1}  a^{-1}  t  a^{-1}  t^{-1}  a^{-1}  t \rangle\\
\widetilde G_0^{54, 54, 54}=\langle a, t \mid{}& 
 a^3, t^3, a  t^{-1}  a  t  a^{-1}  t^{-1}  a^{-1}  t  
a^{-1}  t^{-1}  a  t  a  t^{-1}  a^{-1}  t, (a  t^{-1}  a 
t  a^{-1}  t^{-1}  a  t)^3 \rangle
\end{align*}

}

\section{Tables}\label{sec:Tables}
\noindent
The following tables contain information about the trivalent triangle groups that arise as fundamental groups of $\mathcal{Y} \in \mathbf{Y}$, see Section~\ref{sec:TriTriSetup}. Each table contains the following entries: the name of the group, as presented in Section~\ref{sec:PresentationList} above; the presentation length of the presentation, which can be read off \textsc{Magma}'s \texttt{PresentationLength($G$)}; whether the group has Kazhdan's property (T) if this is known, the only cases known to have (T) are Ronan's groups, in all the cases where it is known not to have (T) this was shown by finding a subgroup of finite index that maps onto $\Zbb$; the dimension of its abelianization as an $\Fbb_3$-vector space, which \textsc{Magma} returns using \texttt{AbelianQuotient($G$)}; the $L_2$-quotients using \textsc{Magma}'s notation for infinite occurrences, obtained via \texttt{L2Quotients($G$)}; the finite simple quotients that are not $L_2$-quotients up to order $5 \cdot 10^7$, including multiplicities, obtained via \texttt{SimpleQuotients($G$,$5 \cdot 10^7$ :\ Family:="notPSL2", Limit:=$0$)}; and the degrees of alternating quotients up to degree $30$ (up to degree $36$ for half-girth type $(3,3,3)$, up to degree $40$ for half-girth type $(4,4,4)$).

\subsection{Half-girth type  $(2,4,4)$}\label{sec:tab_244}

The table below also shows which of the groups are hyperbolic. For the non-hyperbolic groups, two elements that generate $\Zbb \times \Zbb$ are shown. That they generate $\Zbb \times \Zbb$ follows from Corollary~\ref{cor:z_squared} in each case.  The following words are used in the table 
$g_{26}=a^{-1}c^{-1}bc$, $g_{27}=bcac$, $g_{28}=b^{-1}c^{-1}ac$, $g_{29}=bc^{-1}a^{-1}c$, $g_{30}=a^{-1}c^{-1}b^{-1}c$, $g_{31}=a^{-1}cbc$, $g_{32}=a^{-1}cbc^{-1}$, $g_{33}=bca^{-1}c$, $g_{34}=bca^{-1}c^{-1}$, $g_{35}=ac^{-1}b^{-1}c^{-1}$, $g_{36}=acbc$, $g_{37}=b^{-1}ca^{-1}c^{-1}$, $g_{38}=bc^{-1}ac$, $g_{39}=bcac^{-1}$, $g_{40}=ac^{-1}b^{-1}c$, $g_{41}=a^{-1}cb^{-1}c$, $g_{42}=a^{-1}c^{-1}b^{-1}c^{-1}$, $g_{43}=b^{-1}c^{-1}a^{-1}c$, $g_{44}=ac^{-1}bc$, $g_{45}=b^{-1}ca^{-1}c$, $g_{46}=a^{-1}cb^{-1}c^{-1}$, $g_{47}=acb^{-1}c^{-1}$, $g_{48}=b^{-1}cac^{-1}$, $g_{49}=bc^{-1}a^{-1}c^{-1}$, $g_{50}=acbc^{-1}$.

\leavevmode
\begin{longtable}{lclcccp{2cm}p{5.5cm}p{5cm}} $G$ & p.len. & hyp. & VTF & (T) & $\dim(H_1)$ & $L_2$-quot's & Small quotients & Degrees of small alternating quotients\\
\hline
$G_0^{6,40,40}$ & $45$  & No $g_{31}g_{32}, g_{33}g_{34}$ & Yes & No & 0 &  & $\textrm{Alt}_{7}$ ($\times 2$), $B_{2}(3)$  &$ 5$, $7 $ ($\le 28$)  \\
\hline
$G_0^{6,40,48}$ & $37$  & No $g_{39}g_{49}, g_{31}g_{50}$ & Yes & No & 1 & $L_2(3^2)$ & $B_{2}(3)$ ($\times 3$), $A_{3}(3)$  &$ 3$, $ 5$, $6 $ ($\le 28$)  \\
\hline
$G_0^{6,40,54}$ & $49$  & No $g_{35}g_{36}, g_{37}g_{38}$ & Yes & No & 1 &  & $B_{2}(3)$ ($\times 2$), $\textrm{Alt}_{10}$ ($\times 4$), ${}^2A_{4}(4)$  &$ 3$, $ 5$, $ 10$, $ 15$, $ 20$, $ 25$ ($\le 28$)  \\
$G_{2}^{6,40,54}$ & $49$  & No $g_{39}g_{33}, g_{40}g_{41}$ & Yes & No & 1 &  & $\textrm{Alt}_{9}$ ($\times 2$), ${}^2A_{3}(9)$, $A_{3}(3)$  &$ 3$, $ 5$, $ 9$ ($\le 28$)  \\
\hline
$G_0^{6,48,48}$ & $29$  & No $g_{26}, g_{27}$ & Yes & No & 3 &  & $B_{2}(3)$, ${}^2A_{3}(9)$, $A_{3}(3)$  &$ 3$, $ 4$ ($\le 28$)  \\
\hline
$G_0^{6,48,54}$ & $41$  & No $g_{39}g_{39}, g_{31}g_{42}$ & Yes & No & 3 &  & $B_{2}(3)$ ($\times 2$), $\textrm{Alt}_{10}$, $\textrm{Alt}_{11}$ ($\times 2$)  &$ 3$, $ 4$, $ 10$, $ 11$, $ 14$, $ 15$, $ 19$, $ 20$, $ 21$, $ 22$, $ 23$, $ 24$, $ 25$, $ 26$, $ 27$, $ 28$ ($\le 28$)  \\
$G_{2}^{6,48,54}$ & $41$  & No $g_{39}g_{33}, g_{35}g_{30}$ & Yes & No & 3 &  & ${}^2A_{3}(9)$, $A_{3}(3)$, ${}^2A_{4}(4)$  &$ 3$, $ 4$ ($\le 28$)  \\
\hline
$G_0^{6,54,54}$ & $53$  & No $g_{35}g_{36}, g_{43}g_{38}$ & Yes & No & 3 &  & $\textrm{Alt}_{9}$ ($\times 2$), ${}^2A_{4}(4)$ ($\times 2$)  &$ 3$, $ 9$, $ 27$ ($\le 28$)  \\
$G_{2}^{6,54,54}$ & $53$  & No $g_{41}g_{44}, g_{45}g_{38}$ & Yes & No & 3 &  & $\textrm{Alt}_{9}$ ($\times 2$), ${}^2A_{3}(9)$, $A_{3}(3)$, ${}^2A_{4}(4)$  &$ 3$, $ 9$, $ 12$, $ 15$, $ 18$, $ 21$, $ 24$, $ 27$ ($\le 28$)  \\
$G_{8}^{6,54,54}$ & $53$  & No $g_{26}, g_{28}$ & Yes & No & 3 &  & $B_{2}(3)$ ($\times 2$), $\textrm{Alt}_{9}$ ($\times 4$)  &$ 3$, $ 9$, $ 12$, $ 18$, $ 21$, $ 24$, $ 27$ ($\le 28$)  \\
\hline
$G_0^{8,40,40}$ & $45$  & No $g_{26}, g_{29}$ & Yes & No & 0 & $L_2(\infty^4)$ & $B_{2}(3)$, $C_{2}(4)$ ($\times 2$), $\textrm{Alt}_{10}$ ($\times 2$), $B_{2}(5)$ ($\times 5$), $\textrm{Alt}_{11}$ ($\times 2$)  &$ 5$, $ 6$, $ 10$, $ 11$, $ 15$, $ 20$, $ 21$, $ 25$, $ 26$ ($\le 28$)  \\
\hline
$G_0^{8,40,48}$ & $37$  & Yes & ? & No & 0 & $L_2(3^2)$ & $B_{2}(5)$ ($\times 4$)  &$ 5$, $ 6$ ($\le 28$)  \\
\hline
$G_0^{8,40,54}$ & $49$  & Yes & Yes & No & 0 & $L_2(3^2)$ & $B_{2}(3)$ ($\times 2$), $\textrm{M}_{12}$ ($\times 4$)  &$ 6$ ($\le 28$)  \\
$G_{2}^{8,40,54}$ & $49$  & No $g_{39}g_{38}, g_{41}g_{41}$ & Yes & No & 0 & $L_2(3^2)$ & $B_{2}(3)$ ($\times 2$), $\textrm{M}_{12}$ ($\times 4$), $\textrm{Alt}_{10}$ ($\times 3$), $A_{3}(3)$ ($\times 2$), ${}^2A_{4}(4)$  &$ 6$, $ 10$, $ 12$, $ 15$, $ 16$, $ 21$, $ 22$, $ 27$, $ 28$ ($\le 28$)  \\
\hline
$G_0^{8,48,48}$ & $29$  & Yes & Yes & No & 2 &  & $B_{2}(3)$ ($\times 3$), $C_{3}(2)$ ($\times 4$), $\textrm{Alt}_{11}$  &$ 3$, $ 4$, $ 5$, $ 11$, $ 19$, $ 25$, $ 28$ ($\le 28$)  \\
$G_{1}^{8,48,48}$ & $29$  & No $g_{43}g_{39}, g_{44}g_{46}$ & Yes & No & 2 &  & $\textrm{Alt}_{7}$, $B_{2}(3)$ ($\times 2$), $C_{3}(2)$, $B_{2}(5)$ ($\times 3$), $\textrm{Alt}_{11}$  &$ 3$, $ 4$, $ 7$, $ 11$, $ 15$, $ 19$, $ 22$, $ 23$, $ 24$, $ 25$, $ 26$, $ 27$, $ 28$ ($\le 28$)  \\
\hline
$G_0^{8,48,54}$ & $41$  & Yes & Yes & No & 2 &  & $B_{2}(3)$ ($\times 2$), $\textrm{Alt}_{9}$  &$ 3$, $ 4$, $ 9$ ($\le 28$)  \\
$G_{2}^{8,48,54}$ & $41$  & Yes & Yes & No & 2 &  & $B_{2}(3)$ ($\times 2$), $C_{3}(2)$, $\textrm{Alt}_{10}$ ($\times 2$), ${}^2A_{4}(4)$  &$ 3$, $ 4$, $ 10$, $ 13$, $ 20$, $ 26$, $ 28$ ($\le 28$)  \\
\hline
$G_0^{8,54,54}$ & $53$  & Yes & ? & No & 2 &  &   &$ 3$, $ 4$ ($\le 28$)  \\
$G_{2}^{8,54,54}$ & $53$  & No $g_{30}, g_{27}$ & Yes & No & 2 &  & $B_{2}(3)$ ($\times 2$), $\textrm{Alt}_{9}$ ($\times 2$), $C_{3}(2)$ ($\times 4$), $\textrm{Alt}_{10}$ ($\times 12$), ${}^2A_{3}(9)$, $A_{3}(3)$ ($\times 5$), ${}^2A_{4}(4)$, $\textrm{Alt}_{11}$ ($\times 4$)  &$ 3$, $ 4$, $ 9$, $ 10$, $ 11$, $ 12$, $ 13$, $ 14$, $ 15$, $ 16$, $ 18$, $ 19$, $ 20$, $ 21$, $ 22$, $ 23$, $ 24$, $ 25$, $ 26$, $ 27$, $ 28$ ($\le 28$)  \\
$G_{8}^{8,54,54}$ & $53$  & No $g_{47}g_{32}, g_{48}g_{34}$ & Yes & No & 2 &  & $B_{2}(3)$ ($\times 2$), $\textrm{Alt}_{9}$ ($\times 2$)  &$ 3$, $ 4$, $ 9$, $ 18$, $ 27$, $ 28$ ($\le 28$)  \\
\end{longtable}

\subsection{Half-girth type  $(3,3,3)$}\label{sec:tab_333}

The table below also shows which of the groups are hyperbolic. For the non-hyperbolic groups, two elements that generate $\Zbb \times \Zbb$ are shown. That they generate $\Zbb \times \Zbb$ follows from Corollary~\ref{cor:z_squared} in each case. The following words are used in the table $g_1=a^{-1}bc^{-1}b$, $g_2=a^{-1}bcb$, $g_3=c^{-1}ab^{-1}a$, $g_4=cab^{-1}a$, $g_5=caba$, $g_6=baca$, $g_7=bac^{-1}a$, $g_8=a^{-1}b^{-1}c^{-1}b^{-1}$, $g_9=a^{-1}b^{-1}c^{-1}b$, $g_{10}=c^{-1}aba$, $g_{11}=b^{-1}ac^{-1}a$, $g_{12}=b^{-1}aca^{-1}$, $g_{13} = c^{-1}a^{-1}b^{-1}a^{-1}$, $g_{14}=c^{-1}a^{-1}b^{-1}a$, $g_{15}=b^{-1}a^{-1}c^{-1}a^{-1}$, $g_{16}=c^{-1}aba^{-1}$, $g_{17}=ca^{-1}b^{-1}a$, $g_{18}=ba^{-1}ca^{-1}$, $g_{19}=bac^{-1}a^{-1}$, $g_{20}=ba^{-1}c^{-1}a$, $g_{21}=b^{-1}a^{-1}ca$, $g_{22}=a^{-1}bc^{-1}b^{-1}$, $g_{23}=abcb^{-1}$, $g_{24}=ab^{-1}c^{-1}b^{-1}$, $g_{25}=ca^{-1}ba$, $h_1=c^{-1}bac^{-1}ba^{-1}$, $h_{2}=c^{-1}a^{-1}b^{-1}c^{-1}ab$, $h_{3}=b^{-1}c^{-1}abca^{-1}$, $h_{4}=b^{-1}cab^{-1}c^{-1}a$, $h_{5}=b^{-1}acba^{-1}c^{-1}$, $h_{6}=b^{-1}ac^{-1}bac^{-1}$ $h_{7}=acb^{-1}a^{-1}cb$, $h_{8}=ac^{-1}b^{-1}ac^{-1}b$.

\leavevmode
\begin{longtable}{lclcccp{2cm}p{5.5cm}p{5cm}} $G$ & p.len. & hyp. & VTF & (T) & $\dim(H_1)$ & $L_2$-quot's & Small quotients & Degrees of small alternating quotients\\
\hline
$G_0^{14,14,14}$ & $27$  & No $g_3,g_1$ & Yes & Yes & 0 & $L_2(7)$ & ${}^2A_{2}(9)$, ${}^2A_{2}(25)$  &$ $ ($\le 36$)  \\
$G_{1}^{14,14,14}$ & $27$  & No $g_3g_4, g_2g_1$ & ? & Yes & 1 &  &   &$ 3 $ ($\le 36$)  \\
$G_{2}^{14,14,14}$ & $27$  & No $g_3, g_1$ & Yes & Yes & 0 &  & $\textrm{Alt}_{7}$  &$ 7 $ ($\le 36$)  \\
$G_{6}^{14,14,14}$ & $27$  & No $g_5, g_{21}$ & Yes & Yes & 1 &  & $A_{2}(8)$ ($\times 2$)  &$ 3 $ ($\le 36$)  \\
\hline
$G_0^{14,14,16}$ & $27$  & No $g_5, g_9$ & Yes & No & 1 & $L_2(7)$ & $\textrm{Alt}_{8}$ or $A_{2}(4)$  &$ 3$, $8 $ ($\le 36$)  \\
$G_{1}^{14,14,16}$ & $27$  & No $g_5g_{13}, g_{12}g_{12}$ & ? & ? & 0 & $L_2(7)$ &   &$ $ ($\le 36$)  \\
$G_{4}^{14,14,16}$ & $27$  & No $g_5g_{14}, g_{22}g_{23}$ & ? & ? & 0 &  &   &$ $ ($\le 36$)  \\
$G_{5}^{14,14,16}$ & $27$  & No $g_5h_1, g_{19}h_{3}$ & ? & ? & 1 &  &   &$ 3 $ ($\le 36$)  \\
\hline
$G_0^{14,14,18}$ & $33$  & No $g_3, g_1$ & Yes & ? & 1 &  & ${}^2A_{2}(9)$  &$ 3 $ ($\le 36$)  \\
$G_{4}^{14,14,18}$ & $33$  & No $g_2, g_4$ & ? & ? & 1 &  &   &$ 3 $ ($\le 36$)  \\
\hline
$G_0^{14,14,24}$ & $35$  & Yes & ? & ? & 1 & $L_2(7)$ &   &$ 3 $ ($\le 36$)  \\
$G_{1}^{14,14,24}$ & $35$  & No $g_3, g_1$ & Yes & No & 1 & $L_2(7)$ & $\textrm{Alt}_{7}$, ${}^2A_{2}(25)$  &$ 3$, $7 $ ($\le 36$)  \\
$G_{4}^{14,14,24}$ & $35$  & No $g_2, g_4$ & Yes & No & 1 &  & $\textrm{Alt}_{8}$ or $A_{2}(4)$, $\textrm{M}_{22}$  &$ 3$, $8 $ ($\le 36$)  \\
$G_{5}^{14,14,24}$ & $35$  & No $g_{18}g_{11}, g_{16}g_{17}$ & Yes & ? & 1 &  & $\textrm{Alt}_{7}$  &$ 3$, $7 $ ($\le 36$)  \\
\hline
$G_0^{14,14,26}$ & $35$  & Yes & ? & ? & 1 &  &   &$ 3 $ ($\le 36$)  \\
$G_{1}^{14,14,26}$ & $35$  & No $g_3, g_1$ & Yes & ? & 0 &  & $A_{2}(9)$  &$ 14 $ ($\le 36$)  \\
$G_{3}^{14,14,26}$ & $35$  & No $g_3, g_1$ & ? & ? & 0 &  &   &$ $ ($\le 36$)  \\
$G_{4}^{14,14,26}$ & $35$  & No $g_2, g_4$ & ? & ? & 0 &  &   &$ $ ($\le 36$)  \\
$G_{5}^{14,14,26}$ & $35$  & No $g_{18}g_{11}, g_{16}g_{17}$ & ? & ? & 1 &  &   &$ 3 $ ($\le 36$)  \\
$G_{7}^{14,14,26}$ & $35$  & No $g_{18}g_{11}, g_{16}g_{17}$ & ? & ? & 1 &  &   &$ 3 $ ($\le 36$)  \\
\hline
$G_0^{14,16,16}$ & $27$  & No $h_{5}, h_{4}$ & ? & No & 0 & $L_2(7)$ &   &$ $ ($\le 36$)  \\
$G_{1}^{14,16,16}$ & $27$  & No $g_2g_9g_{24}, g_{25}g_{17}g_{14}$ & ? & ? & 1 &  &   &$ 3$, $4 $ ($\le 36$)  \\
\hline
$G_0^{14,16,18}$ & $33$  & No $h_{7}, h_{2}$ & ? & ? & 1 &  &   &$ 3 $ ($\le 36$)  \\
\hline
$G_0^{14,16,24}$ & $35$  & Yes & ? & No & 1 & $L_2(7)$ &   &$ 3 $ ($\le 36$)  \\
$G_{1}^{14,16,24}$ & $35$  & Yes & ? & ? & 1 &  &   &$ 3$, $4 $ ($\le 36$)  \\
\hline
$G_0^{14,16,26}$ & $35$  & Yes & ? & ? & 0 &  &   &$ $ ($\le 36$)  \\
$G_{1}^{14,16,26}$ & $35$  & Yes & ? & No & 1 &  &   &$ 3 $ ($\le 36$)  \\
$G_{3}^{14,16,26}$ & $35$  & Yes & ? & ? & 1 &  &   &$ 3 $ ($\le 36$)  \\
$G_{7}^{14,16,26}$ & $35$  & Yes & ? & ? & 0 &  &   &$ $ ($\le 36$)  \\
\hline
$G_0^{14,18,18}$ & $39$  & No $g_2, g_4$ & ? & No & 2 &  &   &$ 3 $ ($\le 36$)  \\
\hline
$G_0^{14,18,24}$ & $41$  & No $g_2, g_4$ & ? & ? & 2 &  &   &$ 3 $ ($\le 36$)  \\
\hline
$G_0^{14,18,26}$ & $41$  & No $g_2, g_4$ & ? & ? & 1 &  &   &$ 3 $ ($\le 36$)  \\
$G_{3}^{14,18,26}$ & $41$  & No $g_3, g_1$ & ? & ? & 1 &  &   &$ 3 $ ($\le 36$)  \\
\hline
$G_0^{14,24,24}$ & $43$  & No $g_2, g_4$ & Yes & No & 2 & $L_2(7)$ & $\textrm{Alt}_{7}$, $\textrm{Alt}_{8}$ or $A_{2}(4)$, $\textrm{J}_{2}$, ${}^2A_{3}(9)$  &$ 3$, $7$, $8$, $22$, $28$, $29$, $31$, $35$, $36 $ ($\le 36$)  \\
$G_{1}^{14,24,24}$ & $43$  & Yes & ? & No & 2 &  &   &$ 3$, $4 $ ($\le 36$)  \\
\hline
$G_0^{14,24,26}$ & $43$  & No $g_2, g_4$ & ? & ? & 1 &  &   &$ 3 $ ($\le 36$)  \\
$G_{1}^{14,24,26}$ & $43$  & Yes & ? & ? & 1 &  &   &$ 3 $ ($\le 36$)  \\
$G_{3}^{14,24,26}$ & $43$  & Yes & ? & ? & 1 &  &   &$ 3 $ ($\le 36$)  \\
$G_{7}^{14,24,26}$ & $43$  & No $g_3, g_1$ & ? & ? & 1 &  &   &$ 3 $ ($\le 36$)  \\
\hline
$G_0^{14,26,26}$ & $43$  & No $g_2, g_4$ & ? & ? & 0 &  &   &$ $ ($\le 36$)  \\
$G_{1}^{14,26,26}$ & $43$  & Yes & ? & ? & 1 &  &   &$ 3 $ ($\le 36$)  \\
$G_{3}^{14,26,26}$ & $43$  & Yes & ? & ? & 1 &  &   &$ 3 $ ($\le 36$)  \\
$G_{4}^{14,26,26}$ & $43$  & Yes & ? & ? & 1 &  &   &$ 3 $ ($\le 36$)  \\
$G_{5}^{14,26,26}$ & $43$  & No $g_3, g_1$ & ? & ? & 0 &  &   &$ 14 $ ($\le 36$)  \\
$G_{15}^{14,26,26}$ & $43$  & No $g_3, g_1$ & ? & ? & 0 &  &   &$ 13 $ ($\le 36$)  \\
\hline
$G_0^{16,16,16}$ & $27$  & No $g_5, g_9$ & Yes & No & 1 &  & ${}^2A_{2}(9)$, $\textrm{J}_{2}$, ${}^2A_{2}(64)$ ($\times 2$), $A_{2}(9)$, ${}^2A_{2}(81)$ ($\times 2$)  &$ 3$, $4 $ ($\le 36$)  \\
$G_{1}^{16,16,16}$ & $27$  & No $g_5g_{13}, g_{19}g_{12}$ & Yes & No & 0 &  & $A_{2}(3)$, ${}^2A_{2}(9)$ ($\times 2$), ${}^2A_{2}(81)$ ($\times 2$)  &$ 5$, $29 $ ($\le 36$)  \\
\hline
$G_0^{16,16,18}$ & $33$  & No $h_{6}, h_{8}$ & Yes & No & 1 &  & $A_{2}(3)$ ($\times 2$), $A_{2}(9)$ ($\times 3$)  &$ 3$, $4 $ ($\le 36$)  \\
\hline
$G_0^{16,16,24}$ & $35$  & Yes & Yes & No & 1 &  & $\textrm{Alt}_{10}$, $A_{4}(2)$  &$ 3$, $4$, $10$, $34$, $36 $ ($\le 36$)  \\
$G_{1}^{16,16,24}$ & $35$  & No $h_{6}, h_{8}$ & Yes & No & 1 &  & $\textrm{Alt}_{9}$, $\textrm{HS}_{}$  &$ 3$, $4$, $5$, $9$, $21$, $29$, $33$, $34 $ ($\le 36$)  \\
\hline
$G_0^{16,16,26}$ & $35$  & Yes & ? & ? & 1 &  &   &$ 3$, $4 $ ($\le 36$)  \\
$G_{1}^{16,16,26}$ & $35$  & No $h_{6}, h_{8}$ & ? & No & 0 & $L_2(13)$ &   &$ 16$, $30 $ ($\le 36$)  \\
\hline
$G_0^{16,18,18}$ & $39$  & No $g_{15}g_7, g_{10}$ & Yes & No & 2 &  & $A_{2}(3)$ ($\times 2$), ${}^2A_{2}(64)$ ($\times 2$), $A_{2}(9)$ ($\times 3$)  &$ 3$, $4 $ ($\le 36$)  \\
\hline
$G_0^{16,18,24}$ & $41$  & Yes & Yes & No & 2 &  & $\textrm{Alt}_{10}$  &$ 3$, $4$, $10$, $19$, $34 $ ($\le 36$)  \\
\hline
$G_0^{16,18,26}$ & $41$  & Yes & ? & ? & 1 &  &   &$ 3 $ ($\le 36$)  \\
\hline
$G_0^{16,24,24}$ & $43$  & Yes & Yes & No & 2 &  & $\textrm{Alt}_{7}$, $\textrm{Alt}_{8}$ or $A_{2}(4)$ ($\times 2$), ${}^2A_{2}(25)$, $\textrm{J}_{2}$, $C_{3}(2)$, ${}^2A_{3}(9)$, $B_{2}(5)$, $\textrm{HS}_{}$  &$ 3$, $4$, $7$, $8$, $15$, $18$, $19$, $20$, $22$, $23$, $24$, $25$, $27$, $28$, $30$, $31$, $32$, $33$, $34$, $35$, $36 $ ($\le 36$)  \\
$G_{1}^{16,24,24}$ & $43$  & Yes & Yes & No & 2 &  & $C_{3}(2)$ ($\times 2$)  &$ 3$, $4$, $5$, $17$, $18$, $19$, $21$, $22$, $27$, $29$, $30$, $31$, $32$, $33$, $34$, $35$, $36 $ ($\le 36$)  \\
\hline
$G_0^{16,24,26}$ & $43$  & Yes & ? & ? & 1 &  &   &$ 3$, $4 $ ($\le 36$)  \\
$G_{1}^{16,24,26}$ & $43$  & Yes & ? & ? & 1 & $L_2(13)$ &   &$ 3 $ ($\le 36$)  \\
\hline
$G_0^{16,26,26}$ & $43$  & Yes & ? & No & 1 &  &   &$ 3$, $26 $ ($\le 36$)  \\
$G_{1}^{16,26,26}$ & $43$  & Yes & ? & ? & 0 & $L_2(13)$ & $A_{2}(3)$  &$ $ ($\le 36$)  \\
$G_{3}^{16,26,26}$ & $43$  & Yes & Yes & No & 0 &  & $A_{2}(3)$ ($\times 2$), $G_{2}(3)$  &$ 26 $ ($\le 36$)  \\
$G_{5}^{16,26,26}$ & $43$  & Yes & Yes & ? & 1 & $L_2(13)$ & $A_{2}(3)$, $G_{2}(3)$, $A_{3}(3)$  &$ 3$, $14$, $26$, $28$, $29 $ ($\le 36$)  \\
\hline
$G_0^{18,18,18}$ & $45$  & No $g_5, g_6$ & Yes & No & 3 &  & $A_{2}(3)$ ($\times 2$), $A_{2}(9)$ ($\times 3$)  &$ 3$, $27$, $36 $ ($\le 36$)  \\
\hline
$G_0^{18,18,24}$ & $47$  & No $g_5, g_6$ & Yes & No & 3 &  & $\textrm{Alt}_{10}$, $\textrm{Alt}_{11}$  &$ 3$, $4$, $10$, $11$, $12$, $15$, $19$, $20$, $21$, $22$, $23$, $24$, $25$, $26$, $27$, $28$, $30$, $31$, $32$, $33$, $34$, $35$, $36 $ ($\le 36$)  \\
\hline
$G_0^{18,18,26}$ & $47$  & No $g_5, g_6$ & Yes & No & 2 &  & $A_{2}(3)$ ($\times 2$), $A_{2}(9)$ ($\times 3$)  &$ 3$, $13 $ ($\le 36$)  \\
\hline
$G_0^{18,24,24}$ & $49$  & No $g_2, g_4$ & Yes & No & 3 &  & $\textrm{Alt}_{10}$, ${}^2A_{3}(9)$, $\textrm{Alt}_{11}$  &$ 3$, $4$, $10$, $11$, $15$, $19$, $20$, $21$, $22$, $23$, $24$, $25$, $26$, $27$, $28$, $30$, $31$, $32$, $33$, $34$, $35$, $36 $ ($\le 36$)  \\
\hline
$G_0^{18,24,26}$ & $49$  & No $g_2, g_4$ & ? & ? & 2 &  &   &$ 3$, $27 $ ($\le 36$)  \\
\hline
$G_0^{18,26,26}$ & $49$  & No $g_2, g_4$ & Yes & No & 1 &  & $G_{2}(3)$ ($\times 2$)  &$ 3$, $13 $ ($\le 36$)  \\
$G_{1}^{18,26,26}$ & $49$  & No $g_8, g_7$ & Yes & No & 1 &  & $A_{2}(3)$ ($\times 2$), $G_{2}(3)$  &$ 3$, $13$, $27 $ ($\le 36$)  \\
\hline
$G_0^{24,24,24}$ & $51$  & Yes & Yes & No & 3 &  & $\textrm{Alt}_{7}$ ($\times 3$), $\textrm{M}_{12}$, $A_{2}(7)$, $B_{2}(5)$ ($\times 3$), $A_{4}(2)$  &$ 3$, $4$, $7$, $13$, $15$, $18$, $19$, $20$, $22$, $23$, $24$, $25$, $27$, $28$, $29$, $30$, $31$, $32$, $33$, $34$, $35$, $36 $ ($\le 36$)  \\
$G_{1}^{24,24,24}$ & $51$  & No $g_8, g_7$ & Yes & No & 3 &  & $\textrm{M}_{22}$, ${}^2A_{3}(9)$ ($\times 3$)  &$ 3$, $4$, $5$, $12$, $13$, $14$, $15$, $17$, $18$, $19$, $20$, $21$, $22$, $23$, $24$, $25$, $26$, $27$, $28$, $29$, $30$, $31$, $32$, $33$, $34$, $35$, $36 $ ($\le 36$)  \\
\hline
$G_0^{24,24,26}$ & $51$  & Yes & ? & No & 2 &  &   &$ 3$, $4 $ ($\le 36$)  \\
$G_{1}^{24,24,26}$ & $51$  & No $g_8, g_7$ & Yes & No & 2 & $L_2(13)$ & $A_{3}(3)$  &$ 3$, $13$, $14$, $15$, $16$, $26$, $27$, $28 $ ($\le 36$)  \\
\hline
$G_0^{24,26,26}$ & $51$  & Yes & ? & No & 1 &  &   &$ 3$, $26$, $28 $ ($\le 36$)  \\
$G_{1}^{24,26,26}$ & $51$  & No $g_8, g_7$ & Yes & No & 1 & $L_2(13)$ & $A_{2}(3)$ ($\times 2$), $A_{3}(3)$  &$ 3$, $13$, $14$, $27 $ ($\le 36$)  \\
$G_{3}^{24,26,26}$ & $51$  & No $g_8, g_7$ & ? & No & 1 &  & $A_{2}(3)$ ($\times 2$)  &$ 3$, $13$, $14$, $16$, $26 $ ($\le 36$)  \\
$G_{5}^{24,26,26}$ & $51$  & Yes & ? & No & 1 & $L_2(13)$ & $A_{2}(3)$ ($\times 2$), $A_{3}(3)$  &$ 3$, $13$, $14$, $26$, $27$, $28 $ ($\le 36$)  \\
\hline
$G_0^{26,26,26}$ & $51$  & Yes & Yes & No & 1 &  & $A_{2}(3)$, $A_{2}(9)$ ($\times 3$)  &$ 3$, $26 $ ($\le 36$)  \\
$G_{1}^{26,26,26}$ & $51$  & No $g_8, g_7$ & Yes & No & 0 &  & $A_{2}(3)$ ($\times 2$), ${}^2A_{2}(16)$ ($\times 2$), $G_{2}(3)$ ($\times 6$)  &$ 13$, $26 $ ($\le 36$)  \\
$G_{5}^{26,26,26}$ & $51$  & Yes & Yes & No & 1 &  & $A_{2}(3)$ ($\times 2$), ${}^2A_{2}(16)$  &$ 3 $ ($\le 36$)  \\
$G_{21}^{26,26,26}$ & $51$  & No $g_{11}, g_1$ & Yes & No & 0 & $L_2(13)$ & $A_{2}(3)$ ($\times 5$), ${}^2A_{2}(16)$ ($\times 3$), $G_{2}(3)$, ${}^2F_4(2)'$  &$ 13$, $30 $ ($\le 36$)  \\
\end{longtable}

\subsection{Half-girth type $(3,3,4)$}

\leavevmode
\begin{longtable}{lccccp{2cm}p{8cm}p{7cm}} $G$ & p.len. & VTF & (T) & $\dim(H_1)$ & $L_2$-quot's & Small quotients & Degrees of small alternating quotients\\
\hline
$G_0^{14,14,40}$ & $37$  & Yes & No & 0 & $L_2(7^2)$ & $\textrm{Alt}_{7}$, $\textrm{J}_{1}$ ($\times 2$), ${}^2A_{3}(9)$  &$ 7 $ ($\le 30$)  \\
$G_{4}^{14,14,40}$ & $37$  & Yes & ? & 0 &  & $\textrm{Alt}_{7}$ ($\times 2$), $\textrm{M}_{22}$  &$ 7$, $28 $ ($\le 30$)  \\
\hline
$G_0^{14,14,48}$ & $29$  & ? & No & 1 & $L_2(7)$ & $\textrm{Alt}_{7}$, ${}^2A_{2}(25)$  &$ 3$, $7 $ ($\le 30$)  \\
$G_{1}^{14,14,48}$ & $29$  & ? & No & 1 & $L_2(7)$ & $\textrm{Alt}_{8}$ or $A_{2}(4)$  &$ 3$, $8 $ ($\le 30$)  \\
$G_{4}^{14,14,48}$ & $29$  & ? & ? & 1 &  & $\textrm{Alt}_{7}$  &$ 3$, $7 $ ($\le 30$)  \\
$G_{5}^{14,14,48}$ & $29$  & ? & No & 1 &  & $\textrm{Alt}_{8}$ or $A_{2}(4)$, $\textrm{M}_{22}$  &$ 3$, $8$, $21 $ ($\le 30$)  \\
\hline
$G_0^{14,14,54}$ & $41$  & ? & ? & 1 &  & ${}^2A_{2}(9)$  &$ 3 $ ($\le 30$)  \\
$G_{4}^{14,14,54}$ & $41$  & ? & ? & 1 &  &   &$ 3 $ ($\le 30$)  \\
\hline
$G_0^{14,16,40}$ & $37$  & ? & ? & 0 & $L_2(7^2)$ &   &$ $ ($\le 30$)  \\
\hline
$G_0^{14,16,48}$ & $29$  & ? & ? & 1 &  &   &$ 3$, $4 $ ($\le 30$)  \\
$G_{1}^{14,16,48}$ & $29$  & ? & No & 1 & $L_2(7)$ &   &$ 3 $ ($\le 30$)  \\
\hline
$G_0^{14,16,54}$ & $41$  & ? & ? & 1 &  &   &$ 3 $ ($\le 30$)  \\
$G_{2}^{14,16,54}$ & $41$  & ? & ? & 1 &  &   &$ 3 $ ($\le 30$)  \\
\hline
$G_0^{14,18,40}$ & $43$  & Yes & ? & 0 &  & $\textrm{J}_{2}$  &$ 21$, $25 $ ($\le 30$)  \\
\hline
$G_0^{14,18,48}$ & $35$  & Yes & ? & 2 &  & $G_{2}(3)$  &$ 3 $ ($\le 30$)  \\
\hline
$G_0^{14,18,54}$ & $47$  & ? & No & 2 &  &   &$ 3 $ ($\le 30$)  \\
$G_{2}^{14,18,54}$ & $47$  & ? & No & 2 &  &   &$ 3$, $21$, $28$, $29 $ ($\le 30$)  \\
\hline
$G_0^{14,24,40}$ & $45$  & Yes & ? & 0 & $L_2(7^2)$ & $\textrm{Alt}_{7}$, $\textrm{Alt}_{10}$, $A_{4}(2)$  &$ 7$, $10 $ ($\le 30$)  \\
\hline
$G_0^{14,24,48}$ & $37$  & ? & No & 2 &  &   &$ 3$, $4 $ ($\le 30$)  \\
$G_{1}^{14,24,48}$ & $37$  & Yes & No & 2 & $L_2(7)$ & $\textrm{Alt}_{7}$, $\textrm{Alt}_{8}$ or $A_{2}(4)$, $\textrm{J}_{2}$, $C_{3}(2)$, ${}^2A_{3}(9)$  &$ 3$, $7$, $8$, $15$, $22$, $28$, $29 $ ($\le 30$)  \\
\hline
$G_0^{14,24,54}$ & $49$  & ? & ? & 2 &  &   &$ 3$, $18 $ ($\le 30$)  \\
$G_{2}^{14,24,54}$ & $49$  & Yes & No & 2 &  & $C_{3}(2)$, ${}^2A_{3}(9)$  &$ 3$, $14$, $21$, $28 $ ($\le 30$)  \\
\hline
$G_0^{14,26,40}$ & $45$  & ? & ? & 0 &  &   &$ $ ($\le 30$)  \\
$G_{4}^{14,26,40}$ & $45$  & ? & ? & 0 &  &   &$ $ ($\le 30$)  \\
\hline
$G_0^{14,26,48}$ & $37$  & ? & ? & 1 &  &   &$ 3 $ ($\le 30$)  \\
$G_{1}^{14,26,48}$ & $37$  & ? & ? & 1 &  &   &$ 3 $ ($\le 30$)  \\
$G_{4}^{14,26,48}$ & $37$  & ? & ? & 1 &  &   &$ 3 $ ($\le 30$)  \\
$G_{5}^{14,26,48}$ & $37$  & ? & ? & 1 &  &   &$ 3 $ ($\le 30$)  \\
\hline
$G_0^{14,26,54}$ & $49$  & ? & ? & 1 &  &   &$ 3 $ ($\le 30$)  \\
$G_{2}^{14,26,54}$ & $49$  & ? & ? & 1 &  &   &$ 3 $ ($\le 30$)  \\
$G_{4}^{14,26,54}$ & $49$  & ? & ? & 1 &  &   &$ 3 $ ($\le 30$)  \\
$G_{6}^{14,26,54}$ & $49$  & ? & ? & 1 &  &   &$ 3 $ ($\le 30$)  \\
\hline
$G_0^{16,16,40}$ & $37$  & Yes & No & 0 &  & $\textrm{M}_{11}$, $B_{2}(3)$, $\textrm{J}_{2}$ ($\times 2$), ${}^2A_{3}(9)$, $B_{2}(5)$, $A_{3}(3)$ ($\times 2$)  &$ 5$, $21$, $26$, $28 $ ($\le 30$)  \\
\hline
$G_0^{16,16,48}$ & $29$  & ? & No & 1 &  & $A_{2}(3)$, ${}^2A_{2}(9)$ ($\times 2$), $\textrm{Alt}_{9}$, ${}^2A_{2}(81)$ ($\times 2$), $\textrm{HS}_{}$  &$ 3$, $4$, $5$, $9$, $21$, $26$, $29$, $30 $ ($\le 30$)  \\
$G_{1}^{16,16,48}$ & $29$  & Yes & No & 1 &  & ${}^2A_{2}(9)$, $\textrm{J}_{2}$, $\textrm{Alt}_{10}$, $B_{2}(5)$, ${}^2A_{2}(64)$ ($\times 2$), $A_{4}(2)$, $A_{2}(9)$, ${}^2A_{2}(81)$ ($\times 2$)  &$ 3$, $4$, $10 $ ($\le 30$)  \\
\hline
$G_0^{16,16,54}$ & $41$  & ? & No & 1 &  & $A_{2}(3)$ ($\times 2$), $B_{2}(3)$, $A_{2}(9)$ ($\times 3$)  &$ 3$, $4$, $18$, $22$, $25$, $26$, $27 $ ($\le 30$)  \\
\hline
$G_0^{16,18,40}$ & $43$  & Yes & No & 0 & $L_2(3^2)$ & $B_{2}(3)$ ($\times 2$), $\textrm{M}_{12}$ ($\times 5$)  &$ 6$, $18$, $24$, $27$, $30 $ ($\le 30$)  \\
\hline
$G_0^{16,18,48}$ & $35$  & ? & No & 2 &  & $A_{2}(3)$ ($\times 2$), $\textrm{Alt}_{10}$, $A_{2}(9)$ ($\times 3$)  &$ 3$, $4$, $10$, $17$, $19$, $30 $ ($\le 30$)  \\
\hline
$G_0^{16,18,54}$ & $47$  & ? & No & 2 &  & $A_{2}(3)$ ($\times 2$), ${}^2A_{2}(64)$ ($\times 2$), $A_{2}(9)$ ($\times 3$)  &$ 3$, $4$, $25$, $26$, $27 $ ($\le 30$)  \\
$G_{2}^{16,18,54}$ & $47$  & ? & No & 2 &  & $A_{2}(3)$ ($\times 2$), ${}^2A_{2}(64)$ ($\times 2$), $A_{2}(9)$ ($\times 3$)  &$ 3$, $4$, $20$, $21$, $22$, $24$, $25$, $26$, $27$, $29$, $30 $ ($\le 30$)  \\
\hline
$G_0^{16,24,40}$ & $45$  & Yes & No & 0 & $L_2(3^2)$ & $B_{2}(5)$ ($\times 2$), $A_{4}(2)$ ($\times 3$), $\textrm{Alt}_{11}$ ($\times 2$)  &$ 5$, $6$, $11$, $21$, $22 $ ($\le 30$)  \\
\hline
$G_0^{16,24,48}$ & $37$  & ? & No & 2 &  & $\textrm{Alt}_{9}$, $C_{3}(2)$ ($\times 5$), $\textrm{HS}_{}$  &$ 3$, $4$, $5$, $9$, $14$, $17$, $18$, $19$, $20$, $21$, $22$, $23$, $24$, $25$, $26$, $27$, $28$, $29$, $30 $ ($\le 30$)  \\
$G_{1}^{16,24,48}$ & $37$  & Yes & No & 2 &  & $\textrm{Alt}_{7}$, $\textrm{Alt}_{8}$ or $A_{2}(4)$ ($\times 2$), ${}^2A_{2}(25)$, $\textrm{J}_{2}$, $C_{3}(2)$ ($\times 2$), $\textrm{Alt}_{10}$, ${}^2A_{3}(9)$, $B_{2}(5)$, $A_{4}(2)$, $\textrm{HS}_{}$  &$ 3$, $4$, $7$, $8$, $10$, $12$, $15$, $16$, $18$, $19$, $20$, $22$, $23$, $24$, $25$, $26$, $27$, $28$, $29$, $30 $ ($\le 30$)  \\
\hline
$G_0^{16,24,54}$ & $49$  & Yes & No & 2 &  & $\textrm{Alt}_{9}$, $C_{3}(2)$, $\textrm{Alt}_{10}$  &$ 3$, $4$, $9$, $10$, $12$, $18$, $19$, $21$, $25$, $27$, $28$, $29$, $30 $ ($\le 30$)  \\
$G_{2}^{16,24,54}$ & $49$  & ? & No & 2 &  & $B_{2}(3)$, $\textrm{Alt}_{10}$ ($\times 3$)  &$ 3$, $4$, $10$, $12$, $14$, $16$, $19$, $20$, $22$, $23$, $24$, $26$, $27$, $28$, $30 $ ($\le 30$)  \\
\hline
$G_0^{16,26,40}$ & $45$  & ? & ? & 0 & $L_2(13^2)$ & ${}^2F_4(2)'$  &$ $ ($\le 30$)  \\
\hline
$G_0^{16,26,48}$ & $37$  & ? & No & 1 & $L_2(13)$ &   &$ 3$, $16$, $30 $ ($\le 30$)  \\
$G_{1}^{16,26,48}$ & $37$  & ? & ? & 1 &  &   &$ 3$, $4 $ ($\le 30$)  \\
\hline
$G_0^{16,26,54}$ & $49$  & ? & ? & 1 &  &   &$ 3 $ ($\le 30$)  \\
$G_{2}^{16,26,54}$ & $49$  & ? & ? & 1 &  &   &$ 3$, $28 $ ($\le 30$)  \\
\hline
$G_0^{18,18,40}$ & $49$  & Yes & No & 1 &  & $\textrm{M}_{12}$ ($\times 2$), $A_{3}(3)$ ($\times 4$)  &$ 3$, $5$, $12$, $17$, $18$, $19$, $20$, $21$, $22$, $24$, $26$, $27$, $29$, $30 $ ($\le 30$)  \\
\hline
$G_0^{18,18,48}$ & $41$  & ? & No & 3 &  & $A_{2}(3)$ ($\times 2$), $\textrm{Alt}_{10}$, ${}^2A_{2}(64)$ ($\times 2$), $\textrm{Alt}_{11}$, $A_{2}(9)$ ($\times 3$)  &$ 3$, $4$, $10$, $11$, $12$, $15$, $19$, $20$, $21$, $22$, $23$, $24$, $25$, $26$, $27$, $28$, $30 $ ($\le 30$)  \\
\hline
$G_0^{18,18,54}$ & $53$  & ? & No & 3 &  & $A_{2}(3)$ ($\times 2$), $A_{2}(9)$ ($\times 3$)  &$ 3$, $19$, $22$, $24$, $25$, $26$, $27$, $28$, $29$, $30 $ ($\le 30$)  \\
\hline
$G_0^{18,24,40}$ & $51$  & Yes & No & 1 & $L_2(3^2)$ & $\textrm{M}_{12}$ ($\times 6$), $\textrm{Alt}_{10}$ ($\times 2$), ${}^2A_{3}(9)$ ($\times 2$), $A_{3}(3)$ ($\times 3$), $\textrm{Alt}_{11}$ ($\times 4$)  &$ 3$, $5$, $6$, $10$, $11$, $12$, $15$, $16$, $17$, $18$, $21$, $22$, $23$, $24$, $25$, $26$, $27$, $28$, $30 $ ($\le 30$)  \\
\hline
$G_0^{18,24,48}$ & $43$  & Yes & No & 3 &  & $\textrm{Alt}_{10}$ ($\times 2$), ${}^2A_{3}(9)$, $A_{3}(3)$ ($\times 2$), $\textrm{Alt}_{11}$  &$ 3$, $4$, $10$, $11$, $12$, $13$, $15$, $18$, $19$, $20$, $21$, $22$, $23$, $24$, $25$, $26$, $27$, $28$, $29$, $30 $ ($\le 30$)  \\
\hline
$G_0^{18,24,54}$ & $55$  & Yes & No & 3 &  & $\textrm{Alt}_{10}$ ($\times 2$), $A_{3}(3)$ ($\times 4$), $\textrm{Alt}_{11}$ ($\times 2$)  &$ 3$, $4$, $10$, $11$, $12$, $15$, $18$, $19$, $20$, $21$, $22$, $23$, $24$, $25$, $26$, $27$, $28$, $29$, $30 $ ($\le 30$)  \\
$G_{2}^{18,24,54}$ & $55$  & Yes & No & 3 &  & $\textrm{Alt}_{9}$ ($\times 2$), $\textrm{Alt}_{10}$, $\textrm{Alt}_{11}$  &$ 3$, $4$, $9$, $10$, $11$, $12$, $15$, $18$, $19$, $20$, $21$, $22$, $23$, $24$, $25$, $26$, $27$, $28$, $30 $ ($\le 30$)  \\
\hline
$G_0^{18,26,40}$ & $51$  & ? & ? & 0 &  &   &$ $ ($\le 30$)  \\
\hline
$G_0^{18,26,48}$ & $43$  & Yes & ? & 2 &  & $G_{2}(3)$  &$ 3$, $27 $ ($\le 30$)  \\
\hline
$G_0^{18,26,54}$ & $55$  & ? & No & 2 &  & $A_{2}(3)$ ($\times 2$), $A_{2}(9)$ ($\times 3$)  &$ 3$, $13$, $26$, $27 $ ($\le 30$)  \\
$G_{2}^{18,26,54}$ & $55$  & ? & No & 2 &  & $A_{2}(3)$ ($\times 2$), $A_{2}(9)$ ($\times 3$)  &$ 3$, $13 $ ($\le 30$)  \\
\hline
$G_0^{24,24,40}$ & $53$  & Yes & No & 1 & $L_2(3^2)$, $L_2(3^2)$ & $\textrm{Alt}_{7}$ ($\times 2$), $\textrm{M}_{22}$ ($\times 2$), $\textrm{J}_{2}$ ($\times 4$), $C_{2}(4)$ ($\times 4$), $C_{3}(2)$, $B_{2}(5)$ ($\times 8$), $A_{3}(3)$, $A_{4}(2)$ ($\times 2$)  &$ 3$, $5$, $6$, $7$, $12$, $13$, $15$, $16$, $17$, $21$, $22$, $23$, $24$, $25$, $26$, $27$, $28$, $29$, $30 $ ($\le 30$)  \\
\hline
$G_0^{24,24,48}$ & $45$  & Yes & No & 3 &  & $\textrm{M}_{22}$, $C_{3}(2)$ ($\times 6$), ${}^2A_{3}(9)$ ($\times 5$), $B_{2}(5)$ ($\times 2$), $A_{3}(3)$  &$ 3$, $4$, $5$, $12$, $13$, $14$, $15$, $16$, $17$, $18$, $19$, $20$, $21$, $22$, $23$, $24$, $25$, $26$, $27$, $28$, $29$, $30 $ ($\le 30$)  \\
$G_{1}^{24,24,48}$ & $45$  & Yes & No & 3 &  & $\textrm{Alt}_{7}$ ($\times 3$), $\textrm{Alt}_{8}$ or $A_{2}(4)$ ($\times 2$), $\textrm{M}_{12}$, ${}^2A_{2}(25)$, $\textrm{J}_{2}$, $C_{3}(2)$ ($\times 3$), $A_{2}(7)$, ${}^2A_{3}(9)$, $B_{2}(5)$ ($\times 3$), $A_{4}(2)$, ${}^2A_{4}(4)$ ($\times 2$), $\textrm{HS}_{}$  &$ 3$, $4$, $7$, $8$, $13$, $14$, $15$, $18$, $19$, $20$, $21$, $22$, $23$, $24$, $25$, $26$, $27$, $28$, $29$, $30 $ ($\le 30$)  \\
\hline
$G_0^{24,24,54}$ & $57$  & Yes & No & 3 &  & $\textrm{Alt}_{9}$ ($\times 3$), $\textrm{Alt}_{10}$ ($\times 4$), ${}^2A_{3}(9)$, $\textrm{Alt}_{11}$ ($\times 2$)  &$ 3$, $4$, $9$, $10$, $11$, $12$, $13$, $15$, $16$, $18$, $19$, $20$, $21$, $22$, $23$, $24$, $25$, $26$, $27$, $28$, $29$, $30 $ ($\le 30$)  \\
\hline
$G_0^{24,26,40}$ & $53$  & ? & ? & 0 & $L_2(13^2)$ &   &$ $ ($\le 30$)  \\
\hline
$G_0^{24,26,48}$ & $45$  & ? & No & 2 & $L_2(13)$ & $A_{3}(3)$  &$ 3$, $13$, $14$, $15$, $16$, $26$, $27$, $28$, $29$, $30 $ ($\le 30$)  \\
$G_{1}^{24,26,48}$ & $45$  & ? & No & 2 &  &   &$ 3$, $4$, $14$, $28 $ ($\le 30$)  \\
\hline
$G_0^{24,26,54}$ & $57$  & Yes & No & 2 &  & $A_{3}(3)$  &$ 3$, $13$, $26$, $27$, $28 $ ($\le 30$)  \\
$G_{2}^{24,26,54}$ & $57$  & ? & ? & 2 &  &   &$ 3$, $13$, $27 $ ($\le 30$)  \\
\hline
$G_0^{26,26,40}$ & $53$  & ? & ? & 0 &  &   &$ 13 $ ($\le 30$)  \\
$G_{4}^{26,26,40}$ & $53$  & Yes & ? & 0 & $L_2(13^2)$ & ${}^2A_{2}(16)$, $A_{3}(3)$  &$ 13$, $26 $ ($\le 30$)  \\
\hline
$G_0^{26,26,48}$ & $45$  & ? & No & 1 &  & $A_{2}(3)$ ($\times 2$), $G_{2}(3)$  &$ 3$, $13$, $14$, $16$, $26 $ ($\le 30$)  \\
$G_{1}^{26,26,48}$ & $45$  & ? & No & 1 &  &   &$ 3$, $26$, $28 $ ($\le 30$)  \\
$G_{4}^{26,26,48}$ & $45$  & ? & No & 1 & $L_2(13)$ & $A_{2}(3)$ ($\times 2$), $G_{2}(3)$, $A_{3}(3)$  &$ 3$, $13$, $14$, $26$, $27$, $28$, $29 $ ($\le 30$)  \\
$G_{5}^{26,26,48}$ & $45$  & ? & No & 1 & $L_2(13)$ & $A_{2}(3)$ ($\times 2$), $A_{3}(3)$  &$ 3$, $13$, $14$, $27 $ ($\le 30$)  \\
\hline
$G_0^{26,26,54}$ & $57$  & ? & No & 1 &  & $G_{2}(3)$ ($\times 2$)  &$ 3$, $13 $ ($\le 30$)  \\
$G_{4}^{26,26,54}$ & $57$  & ? & No & 1 &  & $A_{2}(3)$ ($\times 2$), $G_{2}(3)$  &$ 3$, $13$, $27 $ ($\le 30$)  \\
\end{longtable}

\subsection{Half-girth type $(3,4,4)$}

\leavevmode
\begin{longtable}{lccccp{2cm}p{8cm}p{7cm}} $G$ & p.len. & VTF & (T) & $\dim(H_1)$ & $L_2$-quot's & Small quotients & Degrees of small alternating quotients\\
\hline
$G_0^{14,40,40}$ & $47$  & Yes & No & 0 & $L_2(7^2)$ & $\textrm{Alt}_{8}$ or $A_{2}(4)$ ($\times 5$), $C_{3}(2)$ ($\times 2$), $\textrm{Alt}_{10}$ ($\times 4$), ${}^2A_{3}(9)$ ($\times 2$), $A_{4}(2)$ ($\times 3$), $\textrm{Alt}_{11}$ ($\times 3$), $A_{2}(9)$  &$ 5$, $10$, $11$, $20$, $21$, $30 $ ($\le 30$)  \\
\hline
$G_0^{14,40,48}$ & $39$  & ? & ? & 0 & $L_2(7^2)$ & $\textrm{Alt}_{7}$, $\textrm{Alt}_{10}$, $A_{4}(2)$  &$ 7$, $10 $ ($\le 30$)  \\
\hline
$G_0^{14,40,54}$ & $51$  & Yes & ? & 0 &  & $\textrm{J}_{2}$, $C_{3}(2)$ ($\times 2$)  &$ 21$, $25 $ ($\le 30$)  \\
$G_{2}^{14,40,54}$ & $51$  & Yes & ? & 0 &  & $\textrm{J}_{2}$, $C_{3}(2)$ ($\times 2$)  &$ 20$, $21$, $22$, $25$, $27$, $30 $ ($\le 30$)  \\
\hline
$G_0^{14,48,48}$ & $31$  & Yes & No & 2 & $L_2(7)$ & $\textrm{Alt}_{7}$, $\textrm{Alt}_{8}$ or $A_{2}(4)$, $\textrm{J}_{2}$, $C_{3}(2)$ ($\times 2$), ${}^2A_{3}(9)$, $G_{2}(3)$ ($\times 2$)  &$ 3$, $7$, $8$, $15$, $16$, $22$, $23$, $24$, $27$, $28$, $29$, $30 $ ($\le 30$)  \\
$G_{1}^{14,48,48}$ & $31$  & ? & No & 2 &  &   &$ 3$, $4 $ ($\le 30$)  \\
\hline
$G_0^{14,48,54}$ & $43$  & ? & ? & 2 &  & $G_{2}(3)$  &$ 3$, $18 $ ($\le 30$)  \\
$G_{2}^{14,48,54}$ & $43$  & Yes & No & 2 &  & $C_{3}(2)$ ($\times 3$), ${}^2A_{3}(9)$, $G_{2}(3)$  &$ 3$, $14$, $15$, $21$, $22$, $28$, $29$, $30 $ ($\le 30$)  \\
\hline
$G_0^{14,54,54}$ & $55$  & ? & No & 2 &  &   &$ 3$, $21$, $28$, $29 $ ($\le 30$)  \\
$G_{2}^{14,54,54}$ & $55$  & Yes & No & 2 &  & $\textrm{Alt}_{10}$ ($\times 6$), ${}^2A_{3}(9)$ ($\times 2$)  &$ 3$, $10$, $13$, $14$, $17$, $19$, $20$, $21$, $23$, $24$, $27$, $28$, $29$, $30 $ ($\le 30$)  \\
$G_{8}^{14,54,54}$ & $55$  & ? & No & 2 &  &   &$ 3$, $18$, $21$, $27$, $30 $ ($\le 30$)  \\
\hline
$G_0^{16,40,40}$ & $47$  & Yes & No & 0 & $L_2(\infty^4)$ & $\textrm{M}_{11}$ ($\times 4$), $B_{2}(3)$ ($\times 7$), ${}^2A_{2}(25)$, $\textrm{J}_{2}$ ($\times 2$), $C_{2}(4)$ ($\times 2$), $\textrm{Alt}_{10}$ ($\times 4$), ${}^2A_{3}(9)$ ($\times 4$), $B_{2}(5)$ ($\times 11$), $A_{3}(3)$ ($\times 2$), $\textrm{Alt}_{11}$ ($\times 6$)  &$ 5$, $6$, $10$, $11$, $15$, $16$, $17$, $20$, $21$, $22$, $24$, $25$, $26$, $27$, $28$, $29$, $30 $ ($\le 30$)  \\
\hline
$G_0^{16,40,48}$ & $39$  & ? & No & 0 & $L_2(3^2)$ & $\textrm{M}_{11}$, $B_{2}(3)$, $\textrm{J}_{2}$ ($\times 2$), ${}^2A_{3}(9)$, $B_{2}(5)$ ($\times 5$), $A_{3}(3)$ ($\times 2$), $A_{4}(2)$ ($\times 3$), $\textrm{Alt}_{11}$ ($\times 2$)  &$ 5$, $6$, $11$, $16$, $18$, $21$, $22$, $23$, $24$, $26$, $27$, $28$, $29$, $30 $ ($\le 30$)  \\
\hline
$G_0^{16,40,54}$ & $51$  & Yes & No & 0 & $L_2(3^2)$ & $B_{2}(3)$ ($\times 5$), $\textrm{M}_{12}$ ($\times 5$), $C_{3}(2)$, ${}^2A_{3}(9)$ ($\times 2$), $A_{3}(3)$ ($\times 3$), ${}^2A_{4}(4)$  &$ 6$, $12$, $17$, $18$, $21$, $23$, $24$, $26$, $27$, $28$, $29$, $30 $ ($\le 30$)  \\
$G_{2}^{16,40,54}$ & $51$  & Yes & No & 0 & $L_2(3^2)$ & $B_{2}(3)$ ($\times 4$), $\textrm{M}_{12}$ ($\times 5$), $\textrm{Alt}_{10}$ ($\times 3$), ${}^2A_{3}(9)$ ($\times 4$), $A_{3}(3)$ ($\times 4$), ${}^2A_{4}(4)$  &$ 6$, $10$, $12$, $15$, $16$, $18$, $21$, $22$, $23$, $24$, $25$, $26$, $27$, $28$, $29$, $30 $ ($\le 30$)  \\
\hline
$G_0^{16,48,48}$ & $31$  & Yes & No & 2 &  & $\textrm{Alt}_{7}$, ${}^2A_{2}(9)$, $\textrm{Alt}_{8}$ or $A_{2}(4)$ ($\times 2$), $B_{2}(3)$ ($\times 5$), ${}^2A_{2}(25)$, $\textrm{J}_{2}$ ($\times 2$), $C_{3}(2)$ ($\times 5$), $\textrm{Alt}_{10}$ ($\times 2$), ${}^2A_{3}(9)$ ($\times 4$), $B_{2}(5)$ ($\times 5$), ${}^2A_{2}(64)$ ($\times 2$), $A_{3}(3)$ ($\times 5$), $A_{4}(2)$ ($\times 2$), $\textrm{Alt}_{11}$, $A_{2}(9)$, ${}^2A_{2}(81)$ ($\times 2$), $\textrm{HS}_{}$  &$ 3$, $4$, $7$, $8$, $10$, $11$, $12$, $15$, $16$, $18$, $19$, $20$, $22$, $23$, $24$, $25$, $26$, $27$, $28$, $29$, $30 $ ($\le 30$)  \\
$G_{1}^{16,48,48}$ & $31$  & Yes & No & 2 &  & $A_{2}(3)$, ${}^2A_{2}(9)$ ($\times 2$), $B_{2}(3)$ ($\times 8$), $\textrm{Alt}_{9}$ ($\times 2$), $C_{3}(2)$ ($\times 10$), ${}^2A_{3}(9)$, $A_{3}(3)$ ($\times 6$), $\textrm{Alt}_{11}$, ${}^2A_{2}(81)$ ($\times 2$), $\textrm{HS}_{}$ ($\times 2$)  &$ 3$, $4$, $5$, $9$, $11$, $14$, $15$, $17$, $18$, $19$, $20$, $21$, $22$, $23$, $24$, $25$, $26$, $27$, $28$, $29$, $30 $ ($\le 30$)  \\
\hline
$G_0^{16,48,54}$ & $43$  & Yes & No & 2 &  & $A_{2}(3)$ ($\times 2$), $B_{2}(3)$ ($\times 4$), $\textrm{Alt}_{9}$, $C_{3}(2)$, $\textrm{Alt}_{10}$, ${}^2A_{3}(9)$ ($\times 3$), $A_{3}(3)$, ${}^2A_{4}(4)$, $A_{2}(9)$ ($\times 3$)  &$ 3$, $4$, $9$, $10$, $12$, $17$, $18$, $19$, $21$, $22$, $24$, $25$, $26$, $27$, $28$, $29$, $30 $ ($\le 30$)  \\
$G_{2}^{16,48,54}$ & $43$  & Yes & No & 2 &  & $A_{2}(3)$ ($\times 2$), $B_{2}(3)$ ($\times 4$), $C_{3}(2)$, $\textrm{Alt}_{10}$ ($\times 3$), ${}^2A_{3}(9)$ ($\times 2$), $A_{3}(3)$ ($\times 3$), ${}^2A_{4}(4)$ ($\times 5$), $A_{2}(9)$ ($\times 3$)  &$ 3$, $4$, $10$, $12$, $13$, $14$, $15$, $16$, $17$, $18$, $19$, $20$, $21$, $22$, $23$, $24$, $25$, $26$, $27$, $28$, $29$, $30 $ ($\le 30$)  \\
\hline
$G_0^{16,54,54}$ & $55$  & Yes & No & 2 &  & $A_{2}(3)$ ($\times 2$), ${}^2A_{2}(64)$ ($\times 2$), ${}^2A_{4}(4)$ ($\times 2$), $A_{2}(9)$ ($\times 3$)  &$ 3$, $4$, $20$, $21$, $22$, $24$, $25$, $26$, $27$, $29$, $30 $ ($\le 30$)  \\
$G_{2}^{16,54,54}$ & $55$  & Yes & No & 2 &  & $A_{2}(3)$ ($\times 2$), $B_{2}(3)$ ($\times 6$), $\textrm{Alt}_{9}$ ($\times 2$), $C_{3}(2)$ ($\times 4$), $\textrm{Alt}_{10}$ ($\times 12$), ${}^2A_{3}(9)$ ($\times 3$), ${}^2A_{2}(64)$ ($\times 2$), $A_{3}(3)$ ($\times 5$), ${}^2A_{4}(4)$ ($\times 5$), $\textrm{Alt}_{11}$ ($\times 6$), $A_{2}(9)$ ($\times 3$)  &$ 3$, $4$, $9$, $10$, $11$, $12$, $13$, $14$, $15$, $16$, $17$, $18$, $19$, $20$, $21$, $22$, $23$, $24$, $25$, $26$, $27$, $28$, $29$, $30 $ ($\le 30$)  \\
$G_{8}^{16,54,54}$ & $55$  & Yes & No & 2 &  & $A_{2}(3)$ ($\times 2$), $B_{2}(3)$ ($\times 4$), $\textrm{Alt}_{9}$ ($\times 2$), ${}^2A_{3}(9)$ ($\times 3$), ${}^2A_{2}(64)$ ($\times 2$), $A_{3}(3)$ ($\times 7$), ${}^2A_{4}(4)$ ($\times 3$), $A_{2}(9)$ ($\times 3$)  &$ 3$, $4$, $9$, $12$, $18$, $19$, $21$, $22$, $23$, $24$, $25$, $26$, $27$, $28$, $29$, $30 $ ($\le 30$)  \\
\hline
$G_0^{18,40,40}$ & $53$  & Yes & No & 0 &  & $\textrm{Alt}_{7}$ ($\times 2$), $B_{2}(3)$ ($\times 5$), $\textrm{M}_{12}$ ($\times 2$), $\textrm{Alt}_{10}$ ($\times 8$), ${}^2A_{3}(9)$, ${}^2A_{4}(4)$ ($\times 3$)  &$ 5$, $7$, $10$, $15$, $17$, $20$, $21$, $22$, $24$, $25$, $26$, $27$, $30 $ ($\le 30$)  \\
\hline
$G_0^{18,40,48}$ & $45$  & Yes & No & 1 & $L_2(3^2)$ & $B_{2}(3)$ ($\times 5$), $\textrm{M}_{12}$ ($\times 7$), $\textrm{Alt}_{10}$ ($\times 2$), ${}^2A_{3}(9)$ ($\times 4$), $A_{3}(3)$ ($\times 10$), $\textrm{Alt}_{11}$ ($\times 5$)  &$ 3$, $5$, $6$, $10$, $11$, $12$, $14$, $15$, $16$, $17$, $18$, $21$, $22$, $23$, $24$, $25$, $26$, $27$, $28$, $29$, $30 $ ($\le 30$)  \\
\hline
$G_0^{18,40,54}$ & $57$  & ? & No & 1 &  & $B_{2}(3)$ ($\times 2$), $\textrm{M}_{12}$ ($\times 2$), $\textrm{Alt}_{10}$ ($\times 4$), $A_{3}(3)$ ($\times 14$), ${}^2A_{4}(4)$ ($\times 3$)  &$ 3$, $5$, $10$, $12$, $15$, $17$, $18$, $19$, $20$, $21$, $22$, $24$, $25$, $26$, $27$, $28$, $29$, $30 $ ($\le 30$)  \\
$G_{2}^{18,40,54}$ & $57$  & Yes & No & 1 &  & $B_{2}(3)$ ($\times 2$), $\textrm{M}_{12}$ ($\times 2$), $\textrm{Alt}_{9}$ ($\times 2$), ${}^2A_{3}(9)$ ($\times 3$), $A_{3}(3)$ ($\times 5$), ${}^2A_{4}(4)$ ($\times 4$)  &$ 3$, $5$, $9$, $12$, $15$, $17$, $18$, $19$, $20$, $21$, $22$, $23$, $24$, $25$, $26$, $27$, $28$, $29$, $30 $ ($\le 30$)  \\
\hline
$G_0^{18,48,48}$ & $37$  & Yes & No & 3 &  & $A_{2}(3)$ ($\times 2$), $B_{2}(3)$ ($\times 3$), $\textrm{Alt}_{10}$ ($\times 3$), ${}^2A_{3}(9)$ ($\times 4$), $A_{3}(3)$ ($\times 9$), $\textrm{Alt}_{11}$ ($\times 2$), $A_{2}(9)$ ($\times 3$)  &$ 3$, $4$, $10$, $11$, $12$, $13$, $14$, $15$, $17$, $18$, $19$, $20$, $21$, $22$, $23$, $24$, $25$, $26$, $27$, $28$, $29$, $30 $ ($\le 30$)  \\
\hline
$G_0^{18,48,54}$ & $49$  & Yes & No & 3 &  & $A_{2}(3)$ ($\times 2$), $B_{2}(3)$ ($\times 2$), $\textrm{Alt}_{10}$ ($\times 2$), ${}^2A_{2}(64)$ ($\times 2$), $A_{3}(3)$ ($\times 8$), $\textrm{Alt}_{11}$ ($\times 4$), $A_{2}(9)$ ($\times 3$)  &$ 3$, $4$, $10$, $11$, $12$, $14$, $15$, $18$, $19$, $20$, $21$, $22$, $23$, $24$, $25$, $26$, $27$, $28$, $29$, $30 $ ($\le 30$)  \\
$G_{2}^{18,48,54}$ & $49$  & Yes & No & 3 &  & $A_{2}(3)$ ($\times 2$), $B_{2}(3)$ ($\times 2$), $\textrm{Alt}_{9}$ ($\times 2$), $\textrm{Alt}_{10}$, ${}^2A_{3}(9)$ ($\times 3$), ${}^2A_{2}(64)$ ($\times 2$), $A_{3}(3)$, ${}^2A_{4}(4)$ ($\times 3$), $\textrm{Alt}_{11}$, $A_{2}(9)$ ($\times 3$)  &$ 3$, $4$, $9$, $10$, $11$, $12$, $13$, $15$, $18$, $19$, $20$, $21$, $22$, $23$, $24$, $25$, $26$, $27$, $28$, $29$, $30 $ ($\le 30$)  \\
\hline
$G_0^{18,54,54}$ & $61$  & ? & No & 3 &  & $A_{2}(3)$ ($\times 2$), $\textrm{Alt}_{9}$ ($\times 2$), ${}^2A_{4}(4)$ ($\times 2$), $A_{2}(9)$ ($\times 3$)  &$ 3$, $9$, $19$, $21$, $22$, $23$, $24$, $25$, $26$, $27$, $28$, $29$, $30 $ ($\le 30$)  \\
$G_{2}^{18,54,54}$ & $61$  & Yes & No & 3 &  & $A_{2}(3)$ ($\times 2$), $B_{2}(3)$ ($\times 2$), $\textrm{Alt}_{9}$ ($\times 10$), ${}^2A_{3}(9)$ ($\times 3$), $A_{3}(3)$, ${}^2A_{4}(4)$ ($\times 9$), $A_{2}(9)$ ($\times 3$)  &$ 3$, $9$, $12$, $15$, $18$, $19$, $21$, $22$, $24$, $25$, $26$, $27$, $28$, $29$, $30 $ ($\le 30$)  \\
$G_{8}^{18,54,54}$ & $61$  & Yes & No & 3 &  & $A_{2}(3)$ ($\times 2$), $B_{2}(3)$ ($\times 2$), $\textrm{Alt}_{9}$ ($\times 8$), $A_{3}(3)$ ($\times 10$), ${}^2A_{4}(4)$ ($\times 4$), $A_{2}(9)$ ($\times 3$)  &$ 3$, $9$, $12$, $18$, $19$, $21$, $22$, $23$, $24$, $25$, $26$, $27$, $28$, $29$, $30 $  ($\le 30$)  \\
\hline
$G_0^{24,40,40}$ & $55$  & Yes & No & 0 & $L_2(\infty^4)$ & $\textrm{Alt}_{7}$ ($\times 2$), $B_{2}(3)$ ($\times 2$), $\textrm{M}_{22}$ ($\times 2$), $\textrm{J}_{2}$ ($\times 2$), $C_{2}(4)$ ($\times 2$), $C_{3}(2)$ ($\times 2$), $\textrm{Alt}_{10}$ ($\times 2$), $B_{2}(5)$ ($\times 10$), $A_{4}(2)$ ($\times 4$), ${}^2A_{4}(4)$ ($\times 4$), $\textrm{Alt}_{11}$ ($\times 3$)  &$ 5$, $6$, $7$, $10$, $11$, $12$, $15$, $17$, $18$, $20$, $21$, $22$, $23$, $24$, $25$, $26$, $27$, $28$, $29$, $30 $ ($\le 30$)  \\
\hline
$G_0^{24,40,48}$ & $47$  & Yes & No & 1 & $L_2(3^2)$, $L_2(3^2)$ & $\textrm{Alt}_{7}$ ($\times 2$), $B_{2}(3)$ ($\times 3$), $\textrm{M}_{22}$ ($\times 2$), $\textrm{J}_{2}$ ($\times 4$), $C_{2}(4)$ ($\times 4$), $C_{3}(2)$ ($\times 3$), $B_{2}(5)$ ($\times 12$), $A_{3}(3)$ ($\times 2$), $A_{4}(2)$ ($\times 5$), ${}^2A_{4}(4)$, $\textrm{Alt}_{11}$ ($\times 4$)  &$ 3$, $5$, $6$, $7$, $11$, $12$, $13$, $15$, $16$, $17$, $18$, $19$, $21$, $22$, $23$, $24$, $25$, $26$, $27$, $28$, $29$, $30 $ ($\le 30$)  \\
\hline
$G_0^{24,40,54}$ & $59$  & Yes & No & 1 & $L_2(3^2)$ & $B_{2}(3)$ ($\times 4$), $\textrm{M}_{12}$ ($\times 6$), $\textrm{Alt}_{10}$ ($\times 12$), ${}^2A_{3}(9)$ ($\times 2$), $A_{3}(3)$ ($\times 3$), ${}^2A_{4}(4)$ ($\times 4$), $\textrm{Alt}_{11}$ ($\times 12$)  &$ 3$, $5$, $6$, $10$, $11$, $12$, $15$, $16$, $17$, $18$, $20$, $21$, $22$, $23$, $24$, $25$, $26$, $27$, $28$, $29$, $30 $ ($\le 30$)  \\
$G_{2}^{24,40,54}$ & $59$  & Yes & No & 1 & $L_2(3^2)$ & $B_{2}(3)$ ($\times 2$), $\textrm{M}_{12}$ ($\times 6$), $\textrm{Alt}_{9}$ ($\times 2$), $C_{3}(2)$ ($\times 4$), $\textrm{Alt}_{10}$ ($\times 7$), ${}^2A_{3}(9)$ ($\times 6$), $A_{3}(3)$ ($\times 7$), ${}^2A_{4}(4)$, $\textrm{Alt}_{11}$ ($\times 6$)  &$ 3$, $5$, $6$, $9$, $10$, $11$, $12$, $13$, $15$, $16$, $17$, $18$, $19$, $20$, $21$, $22$, $23$, $24$, $25$, $26$, $27$, $28$, $29$, $30 $ ($\le 30$)  \\
\hline
$G_0^{24,48,48}$ & $39$  & Yes & No & 3 &  & $\textrm{Alt}_{7}$ ($\times 3$), $\textrm{Alt}_{8}$ or $A_{2}(4)$ ($\times 4$), $B_{2}(3)$ ($\times 3$), $\textrm{M}_{12}$, ${}^2A_{2}(25)$ ($\times 2$), $\textrm{J}_{2}$ ($\times 2$), $C_{3}(2)$ ($\times 11$), $\textrm{Alt}_{10}$, $A_{2}(7)$, ${}^2A_{3}(9)$ ($\times 3$), $B_{2}(5)$ ($\times 7$), $A_{3}(3)$, $A_{4}(2)$ ($\times 2$), ${}^2A_{4}(4)$ ($\times 13$), $\textrm{Alt}_{11}$, $\textrm{HS}_{}$ ($\times 2$)  &$ 3$, $4$, $7$, $8$, $10$, $11$, $12$, $13$, $14$, $15$, $16$, $17$, $18$, $19$, $20$, $21$, $22$, $23$, $24$, $25$, $26$, $27$, $28$, $29$, $30 $ ($\le 30$)  \\
$G_{1}^{24,48,48}$ & $39$  & Yes & No & 3 &  & $B_{2}(3)$ ($\times 4$), $\textrm{Alt}_{9}$, $\textrm{M}_{22}$, $C_{3}(2)$ ($\times 17$), ${}^2A_{3}(9)$ ($\times 8$), $B_{2}(5)$ ($\times 5$), $A_{3}(3)$ ($\times 3$), ${}^2A_{4}(4)$ ($\times 8$), $\textrm{Alt}_{11}$, $\textrm{HS}_{}$  &$ 3$, $4$, $5$, $9$, $11$, $12$, $13$, $14$, $15$, $16$, $17$, $18$, $19$, $20$, $21$, $22$, $23$, $24$, $25$, $26$, $27$, $28$, $29$, $30 $ ($\le 30$)  \\
\hline
$G_0^{24,48,54}$ & $51$  & Yes & No & 3 &  & $B_{2}(3)$ ($\times 4$), $\textrm{Alt}_{9}$ ($\times 3$), $\textrm{Alt}_{10}$ ($\times 5$), ${}^2A_{3}(9)$, $A_{3}(3)$ ($\times 2$), ${}^2A_{4}(4)$ ($\times 3$), $\textrm{Alt}_{11}$ ($\times 4$)  &$ 3$, $4$, $9$, $10$, $11$, $12$, $13$, $14$, $15$, $16$, $17$, $18$, $19$, $20$, $21$, $22$, $23$, $24$, $25$, $26$, $27$, $28$, $29$, $30 $ ($\le 30$)  \\
$G_{2}^{24,48,54}$ & $51$  & ? & No & 3 &  & $B_{2}(3)$ ($\times 2$), $\textrm{Alt}_{9}$ ($\times 3$), $C_{3}(2)$ ($\times 3$), $\textrm{Alt}_{10}$ ($\times 5$), ${}^2A_{3}(9)$ ($\times 5$), $A_{3}(3)$ ($\times 10$), ${}^2A_{4}(4)$ ($\times 12$), $\textrm{Alt}_{11}$ ($\times 2$)  &$ 3$, $4$, $9$, $10$, $11$, $12$, $13$, $15$, $16$, $17$, $18$, $19$, $20$, $21$, $22$, $23$, $24$, $25$, $26$, $27$, $28$, $29$, $30 $ ($\le 30$)  \\
\hline
$G_0^{24,54,54}$ & $63$  & ? & No & 3 &  & $\textrm{Alt}_{9}$ ($\times 6$), $\textrm{Alt}_{10}$ ($\times 2$), $A_{3}(3)$ ($\times 4$), ${}^2A_{4}(4)$ ($\times 8$), $\textrm{Alt}_{11}$ ($\times 2$)  &$ 3$, $4$, $9$, $10$, $11$, $12$, $15$, $18$, $19$, $20$, $21$, $22$, $23$, $24$, $25$, $26$, $27$, $28$, $29$, $30 $ ($\le 30$)  \\
$G_{2}^{24,54,54}$ & $63$  & ? & No & 3 &  & $B_{2}(3)$ ($\times 2$), $\textrm{Alt}_{9}$ ($\times 9$), $C_{3}(2)$ ($\times 6$), $\textrm{Alt}_{10}$ ($\times 22$), ${}^2A_{3}(9)$ ($\times 8$), $A_{3}(3)$ ($\times 26$), ${}^2A_{4}(4)$ ($\times 12$), $\textrm{Alt}_{11}$ ($\times 12$)  &$ 3$, $4$, $9$, $10$, $11$, $12$, $13$, $14$, $15$, $16$, $17$, $18$, $19$, $20$, $21$, $22$, $23$, $24$, $25$, $26$, $27$, $28$, $29$, $30 $q ($\le 30$)  \\
$G_{8}^{24,54,54}$ & $63$  & Yes & No & 3 &  & $B_{2}(3)$ ($\times 4$), $\textrm{Alt}_{9}$ ($\times 14$), $\textrm{Alt}_{10}$, ${}^2A_{4}(4)$ ($\times 9$), $\textrm{Alt}_{11}$  &$ 3$, $4$, $9$, $10$, $11$, $12$, $15$, $18$, $19$, $20$, $21$, $22$, $23$, $24$, $25$, $26$, $27$, $28$, $29$, $30 $ ($\le 30$)  \\
\hline
$G_0^{26,40,40}$ & $55$  & Yes & No & 0 & $L_2(13^2)$ & $A_{3}(3)$ ($\times 3$)  &$ 5$, $20$, $21$, $27$, $28 $ ($\le 30$)  \\
\hline
$G_0^{26,40,48}$ & $47$  & ? & ? & 0 & $L_2(13^2)$ & ${}^2F_4(2)'$  &$ $ ($\le 30$)  \\
\hline
$G_0^{26,40,54}$ & $59$  & Yes & ? & 0 &  & $A_{3}(3)$ ($\times 3$)  &$ 30 $  ($\le 30$)  \\
$G_{2}^{26,40,54}$ & $59$  & ? & ? & 0 &  &   &$ 15 $ ($\le 30$)  \\
\hline
$G_0^{26,48,48}$ & $39$  & Yes & No & 2 &  & $G_{2}(3)$  &$ 3$, $4$, $14$, $28 $ ($\le 30$)  \\
$G_{1}^{26,48,48}$ & $39$  & Yes & No & 2 & $L_2(13)$ & $G_{2}(3)$ ($\times 4$), $A_{3}(3)$  &$ 3$, $13$, $14$, $15$, $16$, $26$, $27$, $28$, $29$, $30 $ ($\le 30$)  \\
\hline
$G_0^{26,48,54}$ & $51$  & ? & ? & 2 &  & $G_{2}(3)$  &$ 3$, $13$, $26$, $27$, $28 $  ($\le 30$)  \\
$G_{2}^{26,48,54}$ & $51$  & ? & No & 2 &  & $G_{2}(3)$, $A_{3}(3)$  &$ 3$, $13$, $26$, $27$, $28$, $29 $  ($\le 30$)  \\
\hline
$G_0^{26,54,54}$ & $63$  & ? & No & 2 &  & $A_{2}(3)$ ($\times 2$), $A_{2}(9)$ ($\times 3$)  &$ 3$, $13$, $26$, $27$, $30 $ ($\le 30$)  \\
$G_{2}^{26,54,54}$ & $63$  & Yes & No & 2 &  & $A_{2}(3)$ ($\times 2$), $A_{3}(3)$ ($\times 20$), $A_{2}(9)$ ($\times 3$)  &$ 3$, $13$, $16$, $19$, $22$, $25$, $26$, $27$, $28$, $29$, $30 $ ($\le 30$)  \\
$G_{8}^{26,54,54}$ & $63$  & Yes & ? & 2 &  & $A_{2}(3)$ ($\times 2$), $A_{3}(3)$ ($\times 6$), $A_{2}(9)$ ($\times 3$)  &$ 3$, $13 $ ($\le 30$)  \\
\end{longtable}

\subsection{Half-girth type $(4,4,4)$}

\leavevmode
\begin{longtable}{lccccp{2cm}p{8cm}p{7cm}} $G$ & p.len. & VTF & (T) & $\dim(H_1)$ & $L_2$-quot's & Small quotients & Degrees of small alternating quotients\\
\hline
$G_0^{40,40,40}$ & $57$  & Yes & No & 0 & $L_2(\infty^4)$, $L_2(\infty^4)$, $L_2(\infty^4)$, $L_2(\infty^4)$ & $\textrm{Alt}_{7}$, $B_{2}(3)$ ($\times 18$), $\textrm{M}_{12}$ ($\times 7$), ${}^2A_{2}(25)$ ($\times 2$), $\textrm{J}_{1}$ ($\times 4$), $A_{2}(5)$ ($\times 2$), $\textrm{J}_{2}$ ($\times 8$), $C_{2}(4)$ ($\times 21$), $\textrm{Alt}_{10}$ ($\times 15$), ${}^2A_{3}(9)$ ($\times 12$), $B_{2}(5)$ ($\times 90$), $A_{3}(3)$ ($\times 7$), $\textrm{HS}_{}$ ($\times 12$)  & $6$, $ 7$, $ 10$, $ 12$, $ 15$, $ 16$, $ 17$, $ 18$, $ 20$, $ 21$, $ 22$, $ 23$, $ 24$, $ 25$, $ 26$, $ 27$, $ 28$, $ 29$, $ 30$, $ 31$, $ 32$, $ 33$, $ 34$, $ 35$, $ 36$, $ 37$, $ 38$, $ 39$, $ 40$ ($\le 40$)  \\
\hline
$G_0^{40,40,48}$ & $49$  & Yes & No & 0 & $L_2(\infty^4)$ & $\textrm{Alt}_{7}$ ($\times 2$), $\textrm{M}_{11}$ ($\times 4$), $B_{2}(3)$ ($\times 8$), ${}^2A_{2}(25)$, $\textrm{M}_{22}$ ($\times 2$), $\textrm{J}_{2}$ ($\times 4$), $C_{2}(4)$ ($\times 2$), $C_{3}(2)$ ($\times 2$), $\textrm{Alt}_{10}$ ($\times 4$), ${}^2A_{3}(9)$ ($\times 4$), $B_{2}(5)$ ($\times 16$), $A_{3}(3)$ ($\times 2$), $A_{4}(2)$ ($\times 4$), ${}^2A_{4}(4)$ ($\times 10$), $\textrm{Alt}_{11}$ ($\times 7$)  & $5$, $ 6$, $ 7$, $ 10$, $ 11$, $ 12$, $ 15$, $ 16$, $ 17$, $ 18$, $ 20$, $ 21$, $ 22$, $ 23$, $ 24$, $ 25$, $ 26$, $ 27$, $ 28$, $ 29$, $ 30$, $ 31$, $ 32$, $ 33$, $ 34$, $ 35$, $ 36$, $ 37$, $ 38$, $ 39$, $ 40$ ($\le 40$)  \\
\hline
$G_0^{40,40,54}$ & $61$  & Yes & No & 0 &  & $\textrm{Alt}_{7}$ ($\times 2$), $B_{2}(3)$ ($\times 5$), $\textrm{M}_{12}$ ($\times 2$), $\textrm{Alt}_{10}$ ($\times 8$), ${}^2A_{3}(9)$ ($\times 15$), $A_{3}(3)$ ($\times 4$), ${}^2A_{4}(4)$ ($\times 7$)  & $5$, $ 7$, $ 10$, $ 15$, $ 16$, $ 17$, $ 19$, $ 20$, $ 21$, $ 22$, $ 23$, $ 24$, $ 25$, $ 26$, $ 27$, $ 28$, $ 29$, $ 30$, $ 31$, $ 32$, $ 33$, $ 34$, $ 35$, $ 36$, $ 37$, $ 38$, $ 39$, $ 40$ ($\le 40$)  \\
\hline
$G_0^{40,48,48}$ & $41$  & Yes & No & 1 & $L_2(3^2)$, $L_2(3^2)$ & $\textrm{Alt}_{7}$ ($\times 2$), $\textrm{M}_{11}$, $B_{2}(3)$ ($\times 18$), $\textrm{M}_{22}$ ($\times 2$), $\textrm{J}_{2}$ ($\times 6$), $C_{2}(4)$ ($\times 4$), $C_{3}(2)$ ($\times 6$), ${}^2A_{3}(9)$ ($\times 10$), $B_{2}(5)$ ($\times 20$), $A_{3}(3)$ ($\times 15$), $A_{4}(2)$ ($\times 8$), ${}^2A_{4}(4)$ ($\times 15$), $\textrm{Alt}_{11}$ ($\times 9$)  & $3$, $ 5$, $ 6$, $ 7$, $ 11$, $ 12$, $ 13$, $ 15$, $ 16$, $ 17$, $ 18$, $ 19$, $ 20$, $ 21$, $ 22$, $ 23$, $ 24$, $ 25$, $ 26$, $ 27$, $ 28$, $ 29$, $ 30$, $ 31$, $ 32$, $ 33$, $ 34$, $ 35$, $ 36$, $ 37$, $ 38$, $ 39$, $ 40$ ($\le 40$)  \\
\hline
$G_0^{40,48,54}$ & $53$  & Yes & No & 1 & $L_2(3^2)$ & $B_{2}(3)$ ($\times 11$), $\textrm{M}_{12}$ ($\times 7$), $\textrm{Alt}_{9}$ ($\times 2$), $C_{3}(2)$ ($\times 4$), $\textrm{Alt}_{10}$ ($\times 7$), ${}^2A_{3}(9)$ ($\times 14$), $A_{3}(3)$ ($\times 16$), ${}^2A_{4}(4)$ ($\times 3$), $\textrm{Alt}_{11}$ ($\times 7$)  & $3$, $ 5$, $ 6$, $ 9$, $ 10$, $ 11$, $ 12$, $ 13$, $ 14$, $ 15$, $ 16$, $ 17$, $ 18$, $ 19$, $ 20$, $ 21$, $ 22$, $ 23$, $ 24$, $ 25$, $ 26$, $ 27$, $ 28$, $ 29$, $ 30$, $ 31$, $ 32$, $ 33$, $ 34$, $ 35$, $ 36$, $ 37$, $ 38$, $ 39$, $ 40$ ($\le 40$)  \\
$G_{2}^{40,48,54}$ & $53$  & Yes & No & 1 & $L_2(3^2)$ & $B_{2}(3)$ ($\times 17$), $\textrm{M}_{12}$ ($\times 7$), $C_{3}(2)$ ($\times 2$), $\textrm{Alt}_{10}$ ($\times 12$), ${}^2A_{3}(9)$ ($\times 20$), $A_{3}(3)$ ($\times 22$), ${}^2A_{4}(4)$ ($\times 24$), $\textrm{Alt}_{11}$ ($\times 15$)  & $3$, $ 5$, $ 6$, $ 10$, $ 11$, $ 12$, $ 14$, $ 15$, $ 16$, $ 17$, $ 18$, $ 20$, $ 21$, $ 22$, $ 23$, $ 24$, $ 25$, $ 26$, $ 27$, $ 28$, $ 29$, $ 30$, $ 31$, $ 32$, $ 33$, $ 34$, $ 35$, $ 36$, $ 37$, $ 38$, $ 39$, $ 40$ ($\le 40$)  \\
\hline
$G_0^{40,54,54}$ & $65$  & Yes & No & 1 &  & $B_{2}(3)$ ($\times 8$), $\textrm{M}_{12}$ ($\times 2$), $\textrm{Alt}_{9}$ ($\times 2$), $\textrm{Alt}_{10}$ ($\times 4$), ${}^2A_{3}(9)$ ($\times 9$), $A_{3}(3)$ ($\times 17$), ${}^2A_{4}(4)$ ($\times 7$)  & $3$, $ 5$, $ 9$, $ 10$, $ 12$, $ 15$, $ 17$, $ 18$, $ 19$, $ 20$, $ 21$, $ 22$, $ 23$, $ 24$, $ 25$, $ 26$, $ 27$, $ 28$, $ 29$, $ 30$, $ 31$, $ 32$, $ 33$, $ 34$, $ 35$, $ 36$, $ 37$, $ 38$, $ 39$, $ 40$ ($\le 40$)  \\
$G_{2}^{40,54,54}$ & $65$  & Yes & No & 1 &  & $B_{2}(3)$ ($\times 12$), $\textrm{M}_{12}$ ($\times 2$), $C_{3}(2)$ ($\times 4$), $\textrm{Alt}_{10}$ ($\times 16$), ${}^2A_{3}(9)$ ($\times 14$), $A_{3}(3)$ ($\times 26$), ${}^2A_{4}(4)$ ($\times 40$), $\textrm{Alt}_{11}$ ($\times 10$)  & $3$, $ 5$, $ 10$, $ 11$, $ 12$, $ 15$, $ 16$, $ 17$, $ 18$, $ 19$, $ 20$, $ 21$, $ 22$, $ 23$, $ 24$, $ 25$, $ 26$, $ 27$, $ 28$, $ 29$, $ 30$, $ 31$, $ 32$, $ 33$, $ 34$, $ 35$, $ 36$, $ 37$, $ 38$, $ 39$, $ 40$ ($\le 40$)  \\
$G_{8}^{40,54,54}$ & $65$  & Yes & No & 1 &  & $B_{2}(3)$ ($\times 8$), $\textrm{M}_{12}$ ($\times 2$), $\textrm{Alt}_{9}$ ($\times 12$), ${}^2A_{3}(9)$ ($\times 12$), $A_{3}(3)$ ($\times 8$), ${}^2A_{4}(4)$ ($\times 16$)  & $3$, $ 5$, $ 9$, $ 12$, $ 15$, $ 17$, $ 18$, $ 19$, $ 20$, $ 21$, $ 22$, $ 23$, $ 24$, $ 25$, $ 26$, $ 27$, $ 28$, $ 29$, $ 30$, $ 31$, $ 32$, $ 33$, $ 34$, $ 35$, $ 36$, $ 37$, $ 38$, $ 39$, $ 40$ ($\le 40$)  \\
\hline
$G_0^{48,48,48}$ & $33$  & ? & No & 3 &  & $A_{2}(3)$, ${}^2A_{2}(9)$ ($\times 2$), $B_{2}(3)$ ($\times 27$), $\textrm{Alt}_{9}$ ($\times 3$), $\textrm{M}_{22}$, $C_{3}(2)$ ($\times 39$), ${}^2A_{3}(9)$ ($\times 21$), $B_{2}(5)$ ($\times 9$), $A_{3}(3)$ ($\times 33$), ${}^2A_{4}(4)$ ($\times 60$), $\textrm{Alt}_{11}$ ($\times 3$), ${}^2A_{2}(81)$ ($\times 2$), $\textrm{HS}_{}$ ($\times 3$)  & $3$, $ 4$, $ 5$, $ 9$, $ 11$, $ 12$, $ 13$, $ 14$, $ 15$, $ 16$, $ 17$, $ 18$, $ 19$, $ 20$, $ 21$, $ 22$, $ 23$, $ 24$, $ 25$, $ 26$, $ 27$, $ 28$, $ 29$, $ 30$, $ 31$, $ 32$, $ 33$, $ 34$, $ 35$, $ 36$, $ 37$, $ 38$, $ 39$, $ 40$ ($\le 40$)  \\
$G_{1}^{48,48,48}$ & $33$  & Yes & No & 3 &  & $\textrm{Alt}_{7}$ ($\times 3$), ${}^2A_{2}(9)$, $\textrm{Alt}_{8}$ or $A_{2}(4)$ ($\times 6$), $B_{2}(3)$ ($\times 24$), $\textrm{M}_{12}$, ${}^2A_{2}(25)$ ($\times 3$), $\textrm{J}_{2}$ ($\times 4$), $C_{3}(2)$ ($\times 27$), $\textrm{Alt}_{10}$ ($\times 3$), $A_{2}(7)$, ${}^2A_{3}(9)$ ($\times 15$), $B_{2}(5)$ ($\times 19$), ${}^2A_{2}(64)$ ($\times 2$), $A_{3}(3)$ ($\times 30$), $A_{4}(2)$ ($\times 4$), ${}^2A_{4}(4)$ ($\times 63$), $\textrm{Alt}_{11}$ ($\times 3$), $A_{2}(9)$, ${}^2A_{2}(81)$ ($\times 2$), $\textrm{HS}_{}$ ($\times 3$)  & $3$, $ 4$, $ 7$, $ 8$, $ 10$, $ 11$, $ 12$, $ 13$, $ 14$, $ 15$, $ 16$, $ 17$, $ 18$, $ 19$, $ 20$, $ 21$, $ 22$, $ 23$, $ 24$, $ 25$, $ 26$, $ 27$, $ 28$, $ 29$, $ 30$, $ 31$, $ 32$, $ 33$, $ 34$, $ 35$, $ 36$, $ 37$, $ 38$, $ 39$, $ 40$ ($\le 40$)  \\
\hline
$G_0^{48,48,54}$ & $45$  & Yes & No & 3 &  & $A_{2}(3)$ ($\times 2$), $B_{2}(3)$ ($\times 19$), $\textrm{Alt}_{9}$ ($\times 3$), $C_{3}(2)$ ($\times 3$), $\textrm{Alt}_{10}$ ($\times 6$), ${}^2A_{3}(9)$ ($\times 17$), $A_{3}(3)$ ($\times 28$), ${}^2A_{4}(4)$ ($\times 40$), $\textrm{Alt}_{11}$ ($\times 6$), $A_{2}(9)$ ($\times 3$)  & $3$, $ 4$, $ 9$, $ 10$, $ 11$, $ 12$, $ 13$, $ 14$, $ 15$, $ 16$, $ 17$, $ 18$, $ 19$, $ 20$, $ 21$, $ 22$, $ 23$, $ 24$, $ 25$, $ 26$, $ 27$, $ 28$, $ 29$, $ 30$, $ 31$, $ 32$, $ 33$, $ 34$, $ 35$, $ 36$, $ 37$, $ 38$, $ 39$, $ 40$ ($\le 40$)  \\
\hline
$G_0^{48,54,54}$ & $57$  & ? & No & 3 &  & $A_{2}(3)$ ($\times 2$), $B_{2}(3)$ ($\times 8$), $\textrm{Alt}_{9}$ ($\times 6$), $\textrm{Alt}_{10}$ ($\times 2$), ${}^2A_{3}(9)$ ($\times 9$), ${}^2A_{2}(64)$ ($\times 2$), $A_{3}(3)$ ($\times 11$), ${}^2A_{4}(4)$ ($\times 25$), $\textrm{Alt}_{11}$ ($\times 4$), $A_{2}(9)$ ($\times 3$)  & $3$, $ 4$, $ 9$, $ 10$, $ 11$, $ 12$, $ 13$, $ 14$, $ 15$, $ 18$, $ 19$, $ 20$, $ 21$, $ 22$, $ 23$, $ 24$, $ 25$, $ 26$, $ 27$, $ 28$, $ 29$, $ 30$, $ 31$, $ 32$, $ 33$, $ 34$, $ 35$, $ 36$, $ 37$, $ 38$, $ 39$, $ 40$ ($\le 40$)  \\
$G_{2}^{48,54,54}$ & $57$  & Yes & No & 3 &  & $A_{2}(3)$ ($\times 2$), $B_{2}(3)$ ($\times 10$), $\textrm{Alt}_{9}$ ($\times 9$), $C_{3}(2)$ ($\times 6$), $\textrm{Alt}_{10}$ ($\times 22$), ${}^2A_{3}(9)$ ($\times 14$), ${}^2A_{2}(64)$ ($\times 2$), $A_{3}(3)$ ($\times 36$), ${}^2A_{4}(4)$ ($\times 28$), $\textrm{Alt}_{11}$ ($\times 20$), $A_{2}(9)$ ($\times 3$)  & $3$, $ 4$, $ 9$, $ 10$, $ 11$, $ 12$, $ 13$, $ 14$, $ 15$, $ 16$, $ 17$, $ 18$, $ 19$, $ 20$, $ 21$, $ 22$, $ 23$, $ 24$, $ 25$, $ 26$, $ 27$, $ 28$, $ 29$, $ 30$, $ 31$, $ 32$, $ 33$, $ 34$, $ 35$, $ 36$, $ 37$, $ 38$, $ 39$, $ 40$ ($\le 40$)  \\
$G_{8}^{48,54,54}$ & $57$  & ? & No & 3 &  & $A_{2}(3)$ ($\times 2$), $B_{2}(3)$ ($\times 18$), $\textrm{Alt}_{9}$ ($\times 14$), $\textrm{Alt}_{10}$, ${}^2A_{3}(9)$ ($\times 15$), ${}^2A_{2}(64)$ ($\times 2$), $A_{3}(3)$ ($\times 19$), ${}^2A_{4}(4)$ ($\times 52$), $\textrm{Alt}_{11}$, $A_{2}(9)$ ($\times 3$)  & $3$, $ 4$, $ 9$, $ 10$, $ 11$, $ 12$, $ 13$, $ 15$, $ 18$, $ 19$, $ 20$, $ 21$, $ 22$, $ 23$, $ 24$, $ 25$, $ 26$, $ 27$, $ 28$, $ 29$, $ 30$, $ 31$, $ 32$, $ 33$, $ 34$, $ 35$, $ 36$, $ 37$, $ 38$, $ 39$, $ 40$ ($\le 40$)  \\
\hline
$G_0^{54,54,54}$ & $69$  & ? & No & 3 &  & $A_{2}(3)$ ($\times 2$), $\textrm{Alt}_{9}$ ($\times 6$), ${}^2A_{4}(4)$ ($\times 10$), $A_{2}(9)$ ($\times 3$)  & $3$, $ 9$, $ 19$, $ 21$, $ 22$, $ 23$, $ 24$, $ 25$, $ 26$, $ 27$, $ 28$, $ 29$, $ 30$, $ 31$, $ 32$, $ 33$, $ 34$, $ 35$, $ 36$, $ 37$, $ 38$, $ 39$, $ 40$ ($\le 40$)  \\
$G_{2}^{54,54,54}$ & $69$  & ? & No & 3 &  & $A_{2}(3)$ ($\times 2$), $B_{2}(3)$ ($\times 8$), $\textrm{Alt}_{9}$ ($\times 24$), ${}^2A_{3}(9)$ ($\times 9$), $A_{3}(3)$ ($\times 13$), ${}^2A_{4}(4)$ ($\times 41$), $A_{2}(9)$ ($\times 3$)  & $3$, $ 9$, $ 12$, $ 15$, $ 18$, $ 19$, $ 21$, $ 22$, $ 23$, $ 24$, $ 25$, $ 26$, $ 27$, $ 28$, $ 29$, $ 30$, $ 31$, $ 32$, $ 33$, $ 34$, $ 35$, $ 36$, $ 37$, $ 38$, $ 39$, $ 40$ ($\le 40$)  \\
\end{longtable}

\end{landscape}

\section{Epimorphisms}\label{sec:Epimorphisms}
\noindent
Among the groups in Section~\ref{sec:small_cubic_graphs} we have described ten homomorphisms that are composites of the seven morphisms between graphs of the following orders: $16 \to 8$, $18 \to 6$, $24 \to 6$, $24 \to 8$, $48 \to 24$, $48 \to 16$, $54 \to 18$. These homomorphisms give rise to homomorphisms among the groups in this appendix. These homomorphisms are indicated in the diagrams below (without multiplicities). Except for the groups $G^{6,40,48}_0$, $G^{6,48,48}_0$, $G^{6,48,54}_0$, $G^{6,48,54}_2$, $G^{6,54,54}_0$, $G^{6,54,54}_2$, $G^{6,54,54}_8$ (which occur multiple times to limit the size of the diagrams), each group occurs at most once in the diagrams. The groups that do not occur are precisely those whose superscripts are all $14$, $26$, or $40$, namely $G^{14,14,14}_0$, $G^{14,14,14}_1$, $G^{14,14,14}_2$, $G^{14,14,14}_6$, $G^{14,14,26}_0$, $G^{14,14,26}_1$, $G^{14,14,26}_3$, $G^{14,14,26}_4$, $G^{14,14,26}_5$, $G^{14,14,26}_7$, $G^{14,26,26}_0$, $G^{14,26,26}_1$, $G^{14,26,26}_3$, $G^{14,26,26}_4$, $G^{14,26,26}_5$, $G^{14,26,26}_{15}$, $G^{26,26,26}_0$, $G^{26,26,26}_1$, $G^{26,26,26}_5$, $G^{26,26,26}_{21}$; $G^{14,14,40}_0$, $G^{14,14,40}_4$, $G^{14,26,40}_0$, $G^{14,26,40}_4$, $G^{26,26,40}_0$, $G^{26,26,40}_4$; $G^{14,40,40}_0$, $G^{26,40,40}_0$; $G^{40,40,40}_0$.
The first five diagrams form the connected component of $G^{6,48,48}_0, G^{6,48,54}_0,G^{6,48,54}_2,G^{6,54,54}_0,G^{6,54,54}_2, G^{6,54,54}_8$, the next two diagrams the component of $G^{6,40,48}_0$ and the remaining ones are single components.

\centering

\begin{landscape}
\begin{tikzpicture}[>=latex,line join=bevel,]
\node (G^{6+48+54}_{2}) at (265.5bp,43.5bp) [] {$G^{6,48,54}_{2}$};
  \node (G^{6+48+54}_{0}) at (509.5bp,43.5bp) [] {$G^{6,48,54}_{0}$};
  \node (G^{6+54+54}_{2}) at (753.5bp,43.5bp) [] {$G^{6,54,54}_{2}$};
  \node (G^{8+54+54}_{0}) at (326.5bp,43.5bp) [] {$G^{8,54,54}_{0}$};
  \node (G^{8+54+54}_{2}) at (692.5bp,43.5bp) [] {$G^{8,54,54}_{2}$};
  \node (G^{24+54+54}_{0}) at (387.5bp,84.5bp) [] {$G^{24,54,54}_{0}$};
  \node (G^{18+48+54}_{2}) at (235.5bp,84.5bp) [] {$G^{18,48,54}_{2}$};
  \node (G^{24+54+54}_{2}) at (631.5bp,84.5bp) [] {$G^{24,54,54}_{2}$};
  \node (G^{16+18+54}_{2}) at (570.5bp,43.5bp) [] {$G^{16,18,54}_{2}$};
  \node (G^{16+18+54}_{0}) at (82.5bp,43.5bp) [] {$G^{16,18,54}_{0}$};
  \node (G^{18+48+54}_{0}) at (509.5bp,84.5bp) [] {$G^{18,48,54}_{0}$};
  \node (G^{8+54+54}_{8}) at (21.5bp,43.5bp) [] {$G^{8,54,54}_{8}$};
  \node (G^{48+54+54}_{2}) at (539.5bp,125.5bp) [] {$G^{48,54,54}_{2}$};
  \node (G^{48+54+54}_{8}) at (205.5bp,125.5bp) [] {$G^{48,54,54}_{8}$};
  \node (G^{16+18+18}_{0}) at (356.5bp,2.5bp) [] {$G^{16,18,18}_{0}$};
  \node (G^{18+24+54}_{0}) at (631.5bp,43.5bp) [] {$G^{18,24,54}_{0}$};
  \node (G^{18+18+24}_{0}) at (417.5bp,2.5bp) [] {$G^{18,18,24}_{0}$};
  \node (G^{16+54+54}_{2}) at (570.5bp,84.5bp) [] {$G^{16,54,54}_{2}$};
  \node (G^{16+54+54}_{0}) at (326.5bp,84.5bp) [] {$G^{16,54,54}_{0}$};
  \node (G^{18+18+48}_{0}) at (387.5bp,43.5bp) [] {$G^{18,18,48}_{0}$};
  \node (G^{16+54+54}_{8}) at (82.5bp,84.5bp) [] {$G^{16,54,54}_{8}$};
  \node (G^{24+54+54}_{8}) at (174.5bp,84.5bp) [] {$G^{24,54,54}_{8}$};
  \node (G^{18+24+54}_{2}) at (204.5bp,43.5bp) [] {$G^{18,24,54}_{2}$};
  \node (G^{6+54+54}_{8}) at (143.5bp,43.5bp) [] {$G^{6,54,54}_{8}$};
  \node (G^{6+54+54}_{0}) at (448.5bp,43.5bp) [] {$G^{6,54,54}_{0}$};
  \node (G^{48+54+54}_{0}) at (357.5bp,125.5bp) [] {$G^{48,54,54}_{0}$};
  \draw [->] (G^{18+48+54}_{0}) -- (G^{16+18+54}_{2});
  \draw [->] (G^{24+54+54}_{2}) -- (G^{18+24+54}_{0});
  \draw [->] (G^{24+54+54}_{2}) -- (G^{18+24+54}_{0});
  \draw [->] (G^{24+54+54}_{8}) -- (G^{8+54+54}_{8});
  \draw [->] (G^{18+24+54}_{2}) -- (G^{18+18+24}_{0});
  \draw [->] (G^{18+48+54}_{0}) -- (G^{18+18+48}_{0});
  \draw [->] (G^{48+54+54}_{2}) -- (G^{16+54+54}_{2});
  \draw [->] (G^{16+18+54}_{0}) -- (G^{16+18+18}_{0});
  \draw [->] (G^{18+48+54}_{0}) -- (G^{6+48+54}_{0});
  \draw [->] (G^{48+54+54}_{0}) -- (G^{16+54+54}_{0});
  \draw [->] (G^{18+18+48}_{0}) -- (G^{16+18+18}_{0});
  \draw [->] (G^{18+48+54}_{0}) -- (G^{18+24+54}_{0});
  \draw [->] (G^{18+48+54}_{2}) -- (G^{16+18+54}_{0});
  \draw [->] (G^{24+54+54}_{8}) -- (G^{18+24+54}_{2});
  \draw [->] (G^{24+54+54}_{8}) -- (G^{18+24+54}_{2});
  \draw [->] (G^{24+54+54}_{8}) -- (G^{6+54+54}_{8});
  \draw [->] (G^{16+54+54}_{0}) -- (G^{16+18+54}_{2});
  \draw [->] (G^{48+54+54}_{0}) -- (G^{24+54+54}_{0});
  \draw [->] (G^{18+48+54}_{2}) -- (G^{18+18+48}_{0});
  \draw [->] (G^{48+54+54}_{8}) -- (G^{16+54+54}_{8});
  \draw [->] (G^{24+54+54}_{0}) -- (G^{8+54+54}_{0});
  \draw [->] (G^{18+48+54}_{2}) -- (G^{18+24+54}_{2});
  \draw [->] (G^{16+54+54}_{2}) -- (G^{16+18+54}_{2});
  \draw [->] (G^{16+54+54}_{2}) -- (G^{16+18+54}_{2});
  \draw [->] (G^{16+54+54}_{8}) -- (G^{8+54+54}_{8});
  \draw [->] (G^{16+18+54}_{2}) -- (G^{16+18+18}_{0});
  \draw [->] (G^{24+54+54}_{0}) -- (G^{18+24+54}_{0});
  \draw [->] (G^{48+54+54}_{2}) -- (G^{24+54+54}_{2});
  \draw [->] (G^{48+54+54}_{8}) -- (G^{18+48+54}_{2});
  \draw [->] (G^{48+54+54}_{8}) -- (G^{18+48+54}_{2});
  \draw [->] (G^{48+54+54}_{8}) -- (G^{24+54+54}_{8});
  \draw [->] (G^{16+54+54}_{0}) -- (G^{8+54+54}_{0});
  \draw [->] (G^{24+54+54}_{0}) -- (G^{6+54+54}_{0});
  \draw [->] (G^{16+54+54}_{8}) -- (G^{16+18+54}_{0});
  \draw [->] (G^{16+54+54}_{8}) -- (G^{16+18+54}_{0});
  \draw [->] (G^{24+54+54}_{2}) -- (G^{8+54+54}_{2});
  \draw [->] (G^{48+54+54}_{2}) -- (G^{18+48+54}_{0});
  \draw [->] (G^{48+54+54}_{2}) -- (G^{18+48+54}_{0});
  \draw [->] (G^{18+24+54}_{0}) -- (G^{18+18+24}_{0});
  \draw [->] (G^{48+54+54}_{0}) -- (G^{18+48+54}_{0});
  \draw [->] (G^{16+54+54}_{2}) -- (G^{8+54+54}_{2});
  \draw [->] (G^{24+54+54}_{2}) -- (G^{6+54+54}_{2});
  \draw [->] (G^{18+18+48}_{0}) -- (G^{18+18+24}_{0});
  \draw [->] (G^{18+48+54}_{2}) -- (G^{6+48+54}_{2});
  \draw [->] (G^{48+54+54}_{0}) -- (G^{18+48+54}_{2});
  \draw [->] (G^{16+54+54}_{0}) -- (G^{16+18+54}_{0});
  \draw [->] (G^{24+54+54}_{0}) -- (G^{18+24+54}_{2});
\end{tikzpicture}

\begin{tikzpicture}[>=latex,line join=bevel,]
\node (G^{16+18+48}_{0}) at (82.5bp,43.5bp) [] {$G^{16,18,48}_{0}$};
  \node (G^{24+24+54}_{0}) at (570.5bp,43.5bp) [] {$G^{24,24,54}_{0}$};
  \node (G^{24+48+54}_{0}) at (387.5bp,84.5bp) [] {$G^{24,48,54}_{0}$};
  \node (G^{16+24+54}_{0}) at (387.5bp,43.5bp) [] {$G^{16,24,54}_{0}$};
  \node (G^{18+48+48}_{0}) at (326.5bp,84.5bp) [] {$G^{18,48,48}_{0}$};
  \node (G^{16+16+18}_{0}) at (52.5bp,2.5bp) [] {$G^{16,16,18}_{0}$};
  \node (G^{16+48+54}_{2}) at (204.5bp,84.5bp) [] {$G^{16,48,54}_{2}$};
  \node (G^{6+48+54}_{2}) at (631.5bp,43.5bp) [] {$G^{6,48,54}_{2}$};
  \node (G^{16+16+54}_{0}) at (21.5bp,43.5bp) [] {$G^{16,16,54}_{0}$};
  \node (G^{6+48+54}_{0}) at (509.5bp,43.5bp) [] {$G^{6,48,54}_{0}$};
  \node (G^{8+48+54}_{0}) at (265.5bp,43.5bp) [] {$G^{8,48,54}_{0}$};
  \node (G^{16+24+54}_{2}) at (204.5bp,43.5bp) [] {$G^{16,24,54}_{2}$};
  \node (G^{24+48+54}_{2}) at (448.5bp,84.5bp) [] {$G^{24,48,54}_{2}$};
  \node (G^{6+48+48}_{0}) at (143.5bp,43.5bp) [] {$G^{6,48,48}_{0}$};
  \node (G^{16+48+54}_{0}) at (265.5bp,84.5bp) [] {$G^{16,48,54}_{0}$};
  \node (G^{8+48+54}_{2}) at (326.5bp,43.5bp) [] {$G^{8,48,54}_{2}$};
  \node (G^{48+48+54}_{0}) at (326.5bp,125.5bp) [] {$G^{48,48,54}_{0}$};
  \node (G^{18+24+48}_{0}) at (448.5bp,43.5bp) [] {$G^{18,24,48}_{0}$};
  \node (G^{18+24+24}_{0}) at (493.5bp,2.5bp) [] {$G^{18,24,24}_{0}$};
  \node (G^{16+18+24}_{0}) at (341.5bp,2.5bp) [] {$G^{16,18,24}_{0}$};
  \draw [->] (G^{24+48+54}_{0}) -- (G^{24+24+54}_{0});
  \draw [->] (G^{48+48+54}_{0}) -- (G^{16+48+54}_{2});
  \draw [->] (G^{24+48+54}_{0}) -- (G^{16+24+54}_{2});
  \draw [->] (G^{16+24+54}_{2}) -- (G^{16+18+24}_{0});
  \draw [->] (G^{18+48+48}_{0}) -- (G^{6+48+48}_{0});
  \draw [->] (G^{16+48+54}_{2}) -- (G^{16+18+48}_{0});
  \draw [->] (G^{16+16+54}_{0}) -- (G^{16+16+18}_{0});
  \draw [->] (G^{16+18+48}_{0}) -- (G^{16+16+18}_{0});
  \draw [->] (G^{16+48+54}_{2}) -- (G^{8+48+54}_{2});
  \draw [->] (G^{16+48+54}_{0}) -- (G^{16+16+54}_{0});
  \draw [->] (G^{48+48+54}_{0}) -- (G^{16+48+54}_{0});
  \draw [->] (G^{18+48+48}_{0}) -- (G^{16+18+48}_{0});
  \draw [->] (G^{18+48+48}_{0}) -- (G^{16+18+48}_{0});
  \draw [->] (G^{18+48+48}_{0}) -- (G^{18+24+48}_{0});
  \draw [->] (G^{18+48+48}_{0}) -- (G^{18+24+48}_{0});
  \draw [->] (G^{16+18+48}_{0}) -- (G^{16+18+24}_{0});
  \draw [->] (G^{16+48+54}_{0}) -- (G^{16+18+48}_{0});
  \draw [->] (G^{24+48+54}_{2}) -- (G^{24+24+54}_{0});
  \draw [->] (G^{48+48+54}_{0}) -- (G^{24+48+54}_{2});
  \draw [->] (G^{24+24+54}_{0}) -- (G^{18+24+24}_{0});
  \draw [->] (G^{24+48+54}_{0}) -- (G^{18+24+48}_{0});
  \draw [->] (G^{24+48+54}_{2}) -- (G^{6+48+54}_{2});
  \draw [->] (G^{24+48+54}_{0}) -- (G^{6+48+54}_{0});
  \draw [->] (G^{48+48+54}_{0}) -- (G^{18+48+48}_{0});
  \draw [->] (G^{24+48+54}_{2}) -- (G^{8+48+54}_{2});
  \draw [->] (G^{16+48+54}_{2}) -- (G^{16+16+54}_{0});
  \draw [->] (G^{48+48+54}_{0}) -- (G^{24+48+54}_{0});
  \draw [->] (G^{18+24+48}_{0}) -- (G^{18+24+24}_{0});
  \draw [->] (G^{16+48+54}_{2}) -- (G^{16+24+54}_{2});
  \draw [->] (G^{24+48+54}_{0}) -- (G^{8+48+54}_{0});
  \draw [->] (G^{16+48+54}_{0}) -- (G^{16+24+54}_{0});
  \draw [->] (G^{18+24+48}_{0}) -- (G^{16+18+24}_{0});
  \draw [->] (G^{16+48+54}_{0}) -- (G^{8+48+54}_{0});
  \draw [->] (G^{24+48+54}_{2}) -- (G^{16+24+54}_{0});
  \draw [->] (G^{16+24+54}_{0}) -- (G^{16+18+24}_{0});
  \draw [->] (G^{24+48+54}_{2}) -- (G^{18+24+48}_{0});
\end{tikzpicture}

\begin{tikzpicture}[>=latex,line join=bevel,]
\node (G^{24+24+48}_{0}) at (265.5bp,43.5bp) [] {$G^{24,24,48}_{0}$};
  \node (G^{24+48+48}_{1}) at (143.5bp,84.5bp) [] {$G^{24,48,48}_{1}$};
  \node (G^{6+48+48}_{0}) at (204.5bp,43.5bp) [] {$G^{6,48,48}_{0}$};
  \node (G^{16+48+48}_{1}) at (82.5bp,84.5bp) [] {$G^{16,48,48}_{1}$};
  \node (G^{16+24+48}_{0}) at (143.5bp,43.5bp) [] {$G^{16,24,48}_{0}$};
  \node (G^{16+16+16}_{1}) at (21.5bp,2.5bp) [] {$G^{16,16,16}_{1}$};
  \node (G^{16+24+24}_{1}) at (189.5bp,2.5bp) [] {$G^{16,24,24}_{1}$};
  \node (G^{24+24+24}_{1}) at (265.5bp,2.5bp) [] {$G^{24,24,24}_{1}$};
  \node (G^{8+48+48}_{0}) at (82.5bp,43.5bp) [] {$G^{8,48,48}_{0}$};
  \node (G^{16+16+48}_{0}) at (21.5bp,43.5bp) [] {$G^{16,16,48}_{0}$};
  \node (G^{16+16+24}_{1}) at (98.5bp,2.5bp) [] {$G^{16,16,24}_{1}$};
  \node (G^{48+48+48}_{0}) at (112.5bp,125.5bp) [] {$G^{48,48,48}_{0}$};
  \draw [->] (G^{24+48+48}_{1}) -- (G^{16+24+48}_{0});
  \draw [->] (G^{24+48+48}_{1}) -- (G^{16+24+48}_{0});
  \draw [->] (G^{24+48+48}_{1}) -- (G^{24+24+48}_{0});
  \draw [->] (G^{24+48+48}_{1}) -- (G^{24+24+48}_{0});
  \draw [->] (G^{16+16+48}_{0}) -- (G^{16+16+16}_{1});
  \draw [->] (G^{16+16+48}_{0}) -- (G^{16+16+24}_{1});
  \draw [->] (G^{16+48+48}_{1}) -- (G^{8+48+48}_{0});
  \draw [->] (G^{24+48+48}_{1}) -- (G^{6+48+48}_{0});
  \draw [->] (G^{24+24+48}_{0}) -- (G^{24+24+24}_{1});
  \draw [->] (G^{24+24+48}_{0}) -- (G^{16+24+24}_{1});
  \draw [->] (G^{16+48+48}_{1}) -- (G^{16+24+48}_{0});
  \draw [->] (G^{16+48+48}_{1}) -- (G^{16+24+48}_{0});
  \draw [->] (G^{16+24+48}_{0}) -- (G^{16+24+24}_{1});
  \draw [->] (G^{16+24+48}_{0}) -- (G^{16+16+24}_{1});
  \draw [->] (G^{24+48+48}_{1}) -- (G^{8+48+48}_{0});
  \draw [->] (G^{48+48+48}_{0}) -- (G^{24+48+48}_{1});
  \draw [->] (G^{48+48+48}_{0}) -- (G^{24+48+48}_{1});
  \draw [->] (G^{48+48+48}_{0}) -- (G^{24+48+48}_{1});
  \draw [->] (G^{16+48+48}_{1}) -- (G^{16+16+48}_{0});
  \draw [->] (G^{16+48+48}_{1}) -- (G^{16+16+48}_{0});
  \draw [->] (G^{48+48+48}_{0}) -- (G^{16+48+48}_{1});
  \draw [->] (G^{48+48+48}_{0}) -- (G^{16+48+48}_{1});
  \draw [->] (G^{48+48+48}_{0}) -- (G^{16+48+48}_{1});
\end{tikzpicture}
\begin{tikzpicture}[>=latex,line join=bevel,]
\node (G^{24+48+48}_{0}) at (143.5bp,84.5bp) [] {$G^{24,48,48}_{0}$};
  \node (G^{16+24+48}_{1}) at (143.5bp,43.5bp) [] {$G^{16,24,48}_{1}$};
  \node (G^{24+24+48}_{1}) at (265.5bp,43.5bp) [] {$G^{24,24,48}_{1}$};
  \node (G^{16+16+16}_{0}) at (21.5bp,2.5bp) [] {$G^{16,16,16}_{0}$};
  \node (G^{16+48+48}_{0}) at (82.5bp,84.5bp) [] {$G^{16,48,48}_{0}$};
  \node (G^{8+48+48}_{1}) at (82.5bp,43.5bp) [] {$G^{8,48,48}_{1}$};
  \node (G^{6+48+48}_{0}) at (204.5bp,43.5bp) [] {$G^{6,48,48}_{0}$};
  \node (G^{24+24+24}_{0}) at (265.5bp,2.5bp) [] {$G^{24,24,24}_{0}$};
  \node (G^{16+16+24}_{0}) at (98.5bp,2.5bp) [] {$G^{16,16,24}_{0}$};
  \node (G^{16+24+24}_{0}) at (189.5bp,2.5bp) [] {$G^{16,24,24}_{0}$};
  \node (G^{16+16+48}_{1}) at (21.5bp,43.5bp) [] {$G^{16,16,48}_{1}$};
  \node (G^{48+48+48}_{1}) at (112.5bp,125.5bp) [] {$G^{48,48,48}_{1}$};
  \draw [->] (G^{24+48+48}_{0}) -- (G^{16+24+48}_{1});
  \draw [->] (G^{24+48+48}_{0}) -- (G^{16+24+48}_{1});
  \draw [->] (G^{24+48+48}_{0}) -- (G^{6+48+48}_{0});
  \draw [->] (G^{24+24+48}_{1}) -- (G^{16+24+24}_{0});
  \draw [->] (G^{24+48+48}_{0}) -- (G^{24+24+48}_{1});
  \draw [->] (G^{24+48+48}_{0}) -- (G^{24+24+48}_{1});
  \draw [->] (G^{16+48+48}_{0}) -- (G^{16+16+48}_{1});
  \draw [->] (G^{16+48+48}_{0}) -- (G^{16+16+48}_{1});
  \draw [->] (G^{16+48+48}_{0}) -- (G^{16+24+48}_{1});
  \draw [->] (G^{16+48+48}_{0}) -- (G^{16+24+48}_{1});
  \draw [->] (G^{48+48+48}_{1}) -- (G^{16+48+48}_{0});
  \draw [->] (G^{48+48+48}_{1}) -- (G^{16+48+48}_{0});
  \draw [->] (G^{48+48+48}_{1}) -- (G^{16+48+48}_{0});
  \draw [->] (G^{16+16+48}_{1}) -- (G^{16+16+16}_{0});
  \draw [->] (G^{48+48+48}_{1}) -- (G^{24+48+48}_{0});
  \draw [->] (G^{48+48+48}_{1}) -- (G^{24+48+48}_{0});
  \draw [->] (G^{48+48+48}_{1}) -- (G^{24+48+48}_{0});
  \draw [->] (G^{16+24+48}_{1}) -- (G^{16+16+24}_{0});
  \draw [->] (G^{16+16+48}_{1}) -- (G^{16+16+24}_{0});
  \draw [->] (G^{16+24+48}_{1}) -- (G^{16+24+24}_{0});
  \draw [->] (G^{24+24+48}_{1}) -- (G^{24+24+24}_{0});
  \draw [->] (G^{16+48+48}_{0}) -- (G^{8+48+48}_{1});
  \draw [->] (G^{24+48+48}_{0}) -- (G^{8+48+48}_{1});
\end{tikzpicture}
\begin{tikzpicture}[>=latex,line join=bevel,]
\node (G^{18+54+54}_{0}) at (82.5bp,84.5bp) [] {$G^{18,54,54}_{0}$};
  \node (G^{18+54+54}_{2}) at (204.5bp,84.5bp) [] {$G^{18,54,54}_{2}$};
  \node (G^{6+54+54}_{2}) at (204.5bp,43.5bp) [] {$G^{6,54,54}_{2}$};
  \node (G^{6+54+54}_{8}) at (82.5bp,43.5bp) [] {$G^{6,54,54}_{8}$};
  \node (G^{54+54+54}_{0}) at (82.5bp,125.5bp) [] {$G^{54,54,54}_{0}$};
  \node (G^{18+18+54}_{0}) at (143.5bp,43.5bp) [] {$G^{18,18,54}_{0}$};
  \node (G^{18+18+18}_{0}) at (143.5bp,2.5bp) [] {$G^{18,18,18}_{0}$};
  \node (G^{6+54+54}_{0}) at (21.5bp,43.5bp) [] {$G^{6,54,54}_{0}$};
  \node (G^{18+54+54}_{8}) at (143.5bp,84.5bp) [] {$G^{18,54,54}_{8}$};
  \node (G^{54+54+54}_{2}) at (143.5bp,125.5bp) [] {$G^{54,54,54}_{2}$};
  \draw [->] (G^{18+18+54}_{0}) -- (G^{18+18+18}_{0});
  \draw [->] (G^{18+54+54}_{0}) -- (G^{6+54+54}_{0});
  \draw [->] (G^{54+54+54}_{2}) -- (G^{18+54+54}_{8});
  \draw [->] (G^{18+54+54}_{8}) -- (G^{18+18+54}_{0});
  \draw [->] (G^{18+54+54}_{8}) -- (G^{18+18+54}_{0});
  \draw [->] (G^{18+54+54}_{2}) -- (G^{18+18+54}_{0});
  \draw [->] (G^{18+54+54}_{2}) -- (G^{18+18+54}_{0});
  \draw [->] (G^{18+54+54}_{2}) -- (G^{6+54+54}_{2});
  \draw [->] (G^{54+54+54}_{2}) -- (G^{18+54+54}_{2});
  \draw [->] (G^{18+54+54}_{0}) -- (G^{18+18+54}_{0});
  \draw [->] (G^{18+54+54}_{0}) -- (G^{18+18+54}_{0});
  \draw [->] (G^{18+54+54}_{8}) -- (G^{6+54+54}_{8});
  \draw [->] (G^{54+54+54}_{2}) -- (G^{18+54+54}_{0});
  \draw [->] (G^{54+54+54}_{0}) -- (G^{18+54+54}_{0});
  \draw [->] (G^{54+54+54}_{0}) -- (G^{18+54+54}_{0});
  \draw [->] (G^{54+54+54}_{0}) -- (G^{18+54+54}_{0});
\end{tikzpicture}
\begin{tikzpicture}[>=latex,line join=bevel,]
\node (G^{40+48+48}_{0}) at (112.5bp,84.5bp) [] {$G^{40,48,48}_{0}$};
  \node (G^{16+24+40}_{0}) at (143.5bp,2.5bp) [] {$G^{16,24,40}_{0}$};
  \node (G^{6+40+48}_{0}) at (204.5bp,2.5bp) [] {$G^{6,40,48}_{0}$};
  \node (G^{8+40+48}_{0}) at (82.5bp,2.5bp) [] {$G^{8,40,48}_{0}$};
  \node (G^{16+16+40}_{0}) at (21.5bp,2.5bp) [] {$G^{16,16,40}_{0}$};
  \node (G^{16+40+48}_{0}) at (82.5bp,43.5bp) [] {$G^{16,40,48}_{0}$};
  \node (G^{24+24+40}_{0}) at (265.5bp,2.5bp) [] {$G^{24,24,40}_{0}$};
  \node (G^{24+40+48}_{0}) at (143.5bp,43.5bp) [] {$G^{24,40,48}_{0}$};
  \draw [->] (G^{24+40+48}_{0}) -- (G^{16+24+40}_{0});
  \draw [->] (G^{24+40+48}_{0}) -- (G^{8+40+48}_{0});
  \draw [->] (G^{40+48+48}_{0}) -- (G^{16+40+48}_{0});
  \draw [->] (G^{40+48+48}_{0}) -- (G^{16+40+48}_{0});
  \draw [->] (G^{24+40+48}_{0}) -- (G^{6+40+48}_{0});
  \draw [->] (G^{24+40+48}_{0}) -- (G^{24+24+40}_{0});
  \draw [->] (G^{16+40+48}_{0}) -- (G^{8+40+48}_{0});
  \draw [->] (G^{16+40+48}_{0}) -- (G^{16+24+40}_{0});
  \draw [->] (G^{40+48+48}_{0}) -- (G^{24+40+48}_{0});
  \draw [->] (G^{40+48+48}_{0}) -- (G^{24+40+48}_{0});
  \draw [->] (G^{16+40+48}_{0}) -- (G^{16+16+40}_{0});
\end{tikzpicture}

\begin{tikzpicture}[>=latex,line join=bevel,]
\node (G^{16+40+54}_{2}) at (265.5bp,43.5bp) [] {$G^{16,40,54}_{2}$};
  \node (G^{18+40+54}_{0}) at (387.5bp,43.5bp) [] {$G^{18,40,54}_{0}$};
  \node (G^{40+54+54}_{0}) at (417.5bp,84.5bp) [] {$G^{40,54,54}_{0}$};
  \node (G^{18+24+40}_{0}) at (204.5bp,2.5bp) [] {$G^{18,24,40}_{0}$};
  \node (G^{24+40+54}_{2}) at (326.5bp,43.5bp) [] {$G^{24,40,54}_{2}$};
  \node (G^{16+18+40}_{0}) at (143.5bp,2.5bp) [] {$G^{16,18,40}_{0}$};
  \node (G^{6+40+48}_{0}) at (21.5bp,2.5bp) [] {$G^{6,40,48}_{0}$};
  \node (G^{16+40+54}_{0}) at (82.5bp,43.5bp) [] {$G^{16,40,54}_{0}$};
  \node (G^{40+54+54}_{8}) at (478.5bp,84.5bp) [] {$G^{40,54,54}_{8}$};
  \node (G^{40+48+54}_{2}) at (143.5bp,84.5bp) [] {$G^{40,48,54}_{2}$};
  \node (G^{18+40+54}_{2}) at (448.5bp,43.5bp) [] {$G^{18,40,54}_{2}$};
  \node (G^{6+40+54}_{2}) at (387.5bp,2.5bp) [] {$G^{6,40,54}_{2}$};
  \node (G^{24+40+54}_{0}) at (204.5bp,43.5bp) [] {$G^{24,40,54}_{0}$};
  \node (G^{6+40+54}_{0}) at (326.5bp,2.5bp) [] {$G^{6,40,54}_{0}$};
  \node (G^{40+48+54}_{0}) at (265.5bp,84.5bp) [] {$G^{40,48,54}_{0}$};
  \node (G^{18+18+40}_{0}) at (448.5bp,2.5bp) [] {$G^{18,18,40}_{0}$};
  \node (G^{40+54+54}_{2}) at (356.5bp,84.5bp) [] {$G^{40,54,54}_{2}$};
  \node (G^{8+40+54}_{0}) at (82.5bp,2.5bp) [] {$G^{8,40,54}_{0}$};
  \node (G^{8+40+54}_{2}) at (265.5bp,2.5bp) [] {$G^{8,40,54}_{2}$};
  \node (G^{18+40+48}_{0}) at (143.5bp,43.5bp) [] {$G^{18,40,48}_{0}$};
  \draw [->] (G^{18+40+54}_{0}) -- (G^{6+40+54}_{0});
  \draw [->] (G^{18+40+48}_{0}) -- (G^{6+40+48}_{0});
  \draw [->] (G^{40+48+54}_{0}) -- (G^{18+40+48}_{0});
  \draw [->] (G^{40+54+54}_{0}) -- (G^{18+40+54}_{2});
  \draw [->] (G^{16+40+54}_{0}) -- (G^{8+40+54}_{0});
  \draw [->] (G^{18+40+54}_{2}) -- (G^{18+18+40}_{0});
  \draw [->] (G^{24+40+54}_{2}) -- (G^{8+40+54}_{2});
  \draw [->] (G^{24+40+54}_{2}) -- (G^{18+24+40}_{0});
  \draw [->] (G^{18+40+54}_{0}) -- (G^{18+18+40}_{0});
  \draw [->] (G^{18+40+48}_{0}) -- (G^{16+18+40}_{0});
  \draw [->] (G^{40+54+54}_{8}) -- (G^{18+40+54}_{2});
  \draw [->] (G^{40+54+54}_{8}) -- (G^{18+40+54}_{2});
  \draw [->] (G^{24+40+54}_{2}) -- (G^{6+40+54}_{2});
  \draw [->] (G^{40+54+54}_{0}) -- (G^{18+40+54}_{0});
  \draw [->] (G^{40+48+54}_{2}) -- (G^{16+40+54}_{0});
  \draw [->] (G^{40+48+54}_{2}) -- (G^{24+40+54}_{0});
  \draw [->] (G^{16+40+54}_{0}) -- (G^{16+18+40}_{0});
  \draw [->] (G^{40+48+54}_{0}) -- (G^{16+40+54}_{2});
  \draw [->] (G^{40+48+54}_{2}) -- (G^{18+40+48}_{0});
  \draw [->] (G^{18+40+48}_{0}) -- (G^{18+24+40}_{0});
  \draw [->] (G^{24+40+54}_{0}) -- (G^{8+40+54}_{0});
  \draw [->] (G^{40+54+54}_{2}) -- (G^{18+40+54}_{0});
  \draw [->] (G^{40+54+54}_{2}) -- (G^{18+40+54}_{0});
  \draw [->] (G^{16+40+54}_{2}) -- (G^{8+40+54}_{2});
  \draw [->] (G^{24+40+54}_{0}) -- (G^{18+24+40}_{0});
  \draw [->] (G^{18+40+54}_{2}) -- (G^{6+40+54}_{2});
  \draw [->] (G^{40+48+54}_{0}) -- (G^{24+40+54}_{2});
  \draw [->] (G^{24+40+54}_{0}) -- (G^{6+40+54}_{0});
  \draw [->] (G^{16+40+54}_{2}) -- (G^{16+18+40}_{0});
\end{tikzpicture}
\end{landscape}

\begin{tikzpicture}[>=latex,line join=bevel,]
\node (G^{14+24+54}_{0}) at (265.5bp,43.5bp) [] {$G^{14,24,54}_{0}$};
  \node (G^{14+16+54}_{2}) at (21.5bp,43.5bp) [] {$G^{14,16,54}_{2}$};
  \node (G^{14+18+24}_{0}) at (143.5bp,2.5bp) [] {$G^{14,18,24}_{0}$};
  \node (G^{14+16+54}_{0}) at (204.5bp,43.5bp) [] {$G^{14,16,54}_{0}$};
  \node (G^{14+18+48}_{0}) at (82.5bp,43.5bp) [] {$G^{14,18,48}_{0}$};
  \node (G^{14+48+54}_{2}) at (82.5bp,84.5bp) [] {$G^{14,48,54}_{2}$};
  \node (G^{14+16+18}_{0}) at (82.5bp,2.5bp) [] {$G^{14,16,18}_{0}$};
  \node (G^{14+24+54}_{2}) at (143.5bp,43.5bp) [] {$G^{14,24,54}_{2}$};
  \node (G^{14+48+54}_{0}) at (204.5bp,84.5bp) [] {$G^{14,48,54}_{0}$};
  \draw [->] (G^{14+48+54}_{0}) -- (G^{14+18+48}_{0});
  \draw [->] (G^{14+48+54}_{2}) -- (G^{14+18+48}_{0});
  \draw [->] (G^{14+18+48}_{0}) -- (G^{14+18+24}_{0});
  \draw [->] (G^{14+48+54}_{0}) -- (G^{14+16+54}_{0});
  \draw [->] (G^{14+48+54}_{2}) -- (G^{14+24+54}_{2});
  \draw [->] (G^{14+24+54}_{2}) -- (G^{14+18+24}_{0});
  \draw [->] (G^{14+18+48}_{0}) -- (G^{14+16+18}_{0});
  \draw [->] (G^{14+16+54}_{0}) -- (G^{14+16+18}_{0});
  \draw [->] (G^{14+48+54}_{2}) -- (G^{14+16+54}_{2});
  \draw [->] (G^{14+48+54}_{0}) -- (G^{14+24+54}_{0});
  \draw [->] (G^{14+16+54}_{2}) -- (G^{14+16+18}_{0});
  \draw [->] (G^{14+24+54}_{0}) -- (G^{14+18+24}_{0});
\end{tikzpicture}

\begin{tikzpicture}[>=latex,line join=bevel,]
\node (G^{16+26+54}_{2}) at (21.5bp,43.5bp) [] {$G^{16,26,54}_{2}$};
  \node (G^{18+24+26}_{0}) at (143.5bp,2.5bp) [] {$G^{18,24,26}_{0}$};
  \node (G^{24+26+54}_{0}) at (265.5bp,43.5bp) [] {$G^{24,26,54}_{0}$};
  \node (G^{24+26+54}_{2}) at (143.5bp,43.5bp) [] {$G^{24,26,54}_{2}$};
  \node (G^{16+26+54}_{0}) at (204.5bp,43.5bp) [] {$G^{16,26,54}_{0}$};
  \node (G^{26+48+54}_{2}) at (204.5bp,84.5bp) [] {$G^{26,48,54}_{2}$};
  \node (G^{26+48+54}_{0}) at (82.5bp,84.5bp) [] {$G^{26,48,54}_{0}$};
  \node (G^{16+18+26}_{0}) at (82.5bp,2.5bp) [] {$G^{16,18,26}_{0}$};
  \node (G^{18+26+48}_{0}) at (82.5bp,43.5bp) [] {$G^{18,26,48}_{0}$};
  \draw [->] (G^{24+26+54}_{0}) -- (G^{18+24+26}_{0});
  \draw [->] (G^{18+26+48}_{0}) -- (G^{16+18+26}_{0});
  \draw [->] (G^{26+48+54}_{0}) -- (G^{24+26+54}_{2});
  \draw [->] (G^{26+48+54}_{0}) -- (G^{18+26+48}_{0});
  \draw [->] (G^{26+48+54}_{0}) -- (G^{16+26+54}_{2});
  \draw [->] (G^{24+26+54}_{2}) -- (G^{18+24+26}_{0});
  \draw [->] (G^{26+48+54}_{2}) -- (G^{16+26+54}_{0});
  \draw [->] (G^{26+48+54}_{2}) -- (G^{18+26+48}_{0});
  \draw [->] (G^{26+48+54}_{2}) -- (G^{24+26+54}_{0});
  \draw [->] (G^{18+26+48}_{0}) -- (G^{18+24+26}_{0});
  \draw [->] (G^{16+26+54}_{0}) -- (G^{16+18+26}_{0});
  \draw [->] (G^{16+26+54}_{2}) -- (G^{16+18+26}_{0});
\end{tikzpicture}

\begin{tikzpicture}[>=latex,line join=bevel,]
\node (G^{24+40+40}_{0}) at (82.5bp,43.5bp) [] {$G^{24,40,40}_{0}$};
  \node (G^{8+40+40}_{0}) at (127.5bp,2.5bp) [] {$G^{8,40,40}_{0}$};
  \node (G^{6+40+40}_{0}) at (51.5bp,2.5bp) [] {$G^{6,40,40}_{0}$};
  \node (G^{18+40+40}_{0}) at (21.5bp,43.5bp) [] {$G^{18,40,40}_{0}$};
  \node (G^{40+40+48}_{0}) at (112.5bp,84.5bp) [] {$G^{40,40,48}_{0}$};
  \node (G^{40+40+54}_{0}) at (21.5bp,84.5bp) [] {$G^{40,40,54}_{0}$};
  \node (G^{16+40+40}_{0}) at (143.5bp,43.5bp) [] {$G^{16,40,40}_{0}$};
  \draw [->] (G^{24+40+40}_{0}) -- (G^{8+40+40}_{0});
  \draw [->] (G^{16+40+40}_{0}) -- (G^{8+40+40}_{0});
  \draw [->] (G^{40+40+48}_{0}) -- (G^{24+40+40}_{0});
  \draw [->] (G^{40+40+54}_{0}) -- (G^{18+40+40}_{0});
  \draw [->] (G^{18+40+40}_{0}) -- (G^{6+40+40}_{0});
  \draw [->] (G^{24+40+40}_{0}) -- (G^{6+40+40}_{0});
  \draw [->] (G^{40+40+48}_{0}) -- (G^{16+40+40}_{0});
\end{tikzpicture}

\begin{tikzpicture}[>=latex,line join=bevel,]
\node (G^{14+48+48}_{1}) at (81.5bp,84.5bp) [] {$G^{14,48,48}_{1}$};
  \node (G^{14+16+24}_{1}) at (82.5bp,2.5bp) [] {$G^{14,16,24}_{1}$};
  \node (G^{14+16+48}_{0}) at (51.5bp,43.5bp) [] {$G^{14,16,48}_{0}$};
  \node (G^{14+24+24}_{1}) at (143.5bp,2.5bp) [] {$G^{14,24,24}_{1}$};
  \node (G^{14+16+16}_{1}) at (21.5bp,2.5bp) [] {$G^{14,16,16}_{1}$};
  \node (G^{14+24+48}_{0}) at (112.5bp,43.5bp) [] {$G^{14,24,48}_{0}$};
  \draw [->] (G^{14+48+48}_{1}) -- (G^{14+24+48}_{0});
  \draw [->] (G^{14+48+48}_{1}) -- (G^{14+24+48}_{0});
  \draw [->] (G^{14+24+48}_{0}) -- (G^{14+24+24}_{1});
  \draw [->] (G^{14+48+48}_{1}) -- (G^{14+16+48}_{0});
  \draw [->] (G^{14+48+48}_{1}) -- (G^{14+16+48}_{0});
  \draw [->] (G^{14+24+48}_{0}) -- (G^{14+16+24}_{1});
  \draw [->] (G^{14+16+48}_{0}) -- (G^{14+16+16}_{1});
  \draw [->] (G^{14+16+48}_{0}) -- (G^{14+16+24}_{1});
\end{tikzpicture}
\begin{tikzpicture}[>=latex,line join=bevel,]
\node (G^{14+48+48}_{0}) at (81.5bp,84.5bp) [] {$G^{14,48,48}_{0}$};
  \node (G^{14+16+24}_{0}) at (82.5bp,2.5bp) [] {$G^{14,16,24}_{0}$};
  \node (G^{14+24+24}_{0}) at (143.5bp,2.5bp) [] {$G^{14,24,24}_{0}$};
  \node (G^{14+16+48}_{1}) at (51.5bp,43.5bp) [] {$G^{14,16,48}_{1}$};
  \node (G^{14+24+48}_{1}) at (112.5bp,43.5bp) [] {$G^{14,24,48}_{1}$};
  \node (G^{14+16+16}_{0}) at (21.5bp,2.5bp) [] {$G^{14,16,16}_{0}$};
  \draw [->] (G^{14+16+48}_{1}) -- (G^{14+16+16}_{0});
  \draw [->] (G^{14+24+48}_{1}) -- (G^{14+24+24}_{0});
  \draw [->] (G^{14+48+48}_{0}) -- (G^{14+16+48}_{1});
  \draw [->] (G^{14+48+48}_{0}) -- (G^{14+16+48}_{1});
  \draw [->] (G^{14+24+48}_{1}) -- (G^{14+16+24}_{0});
  \draw [->] (G^{14+16+48}_{1}) -- (G^{14+16+24}_{0});
  \draw [->] (G^{14+48+48}_{0}) -- (G^{14+24+48}_{1});
  \draw [->] (G^{14+48+48}_{0}) -- (G^{14+24+48}_{1});
\end{tikzpicture}
\begin{tikzpicture}[>=latex,line join=bevel,]
\node (G^{16+24+26}_{1}) at (82.5bp,2.5bp) [] {$G^{16,24,26}_{1}$};
  \node (G^{16+26+48}_{0}) at (51.5bp,43.5bp) [] {$G^{16,26,48}_{0}$};
  \node (G^{16+16+26}_{1}) at (21.5bp,2.5bp) [] {$G^{16,16,26}_{1}$};
  \node (G^{24+24+26}_{1}) at (143.5bp,2.5bp) [] {$G^{24,24,26}_{1}$};
  \node (G^{24+26+48}_{0}) at (112.5bp,43.5bp) [] {$G^{24,26,48}_{0}$};
  \node (G^{26+48+48}_{1}) at (81.5bp,84.5bp) [] {$G^{26,48,48}_{1}$};
  \draw [->] (G^{26+48+48}_{1}) -- (G^{16+26+48}_{0});
  \draw [->] (G^{26+48+48}_{1}) -- (G^{16+26+48}_{0});
  \draw [->] (G^{26+48+48}_{1}) -- (G^{24+26+48}_{0});
  \draw [->] (G^{26+48+48}_{1}) -- (G^{24+26+48}_{0});
  \draw [->] (G^{24+26+48}_{0}) -- (G^{24+24+26}_{1});
  \draw [->] (G^{16+26+48}_{0}) -- (G^{16+24+26}_{1});
  \draw [->] (G^{24+26+48}_{0}) -- (G^{16+24+26}_{1});
  \draw [->] (G^{16+26+48}_{0}) -- (G^{16+16+26}_{1});
\end{tikzpicture}
\begin{tikzpicture}[>=latex,line join=bevel,]
\node (G^{24+24+26}_{0}) at (143.5bp,2.5bp) [] {$G^{24,24,26}_{0}$};
  \node (G^{16+24+26}_{0}) at (82.5bp,2.5bp) [] {$G^{16,24,26}_{0}$};
  \node (G^{16+26+48}_{1}) at (51.5bp,43.5bp) [] {$G^{16,26,48}_{1}$};
  \node (G^{26+48+48}_{0}) at (81.5bp,84.5bp) [] {$G^{26,48,48}_{0}$};
  \node (G^{16+16+26}_{0}) at (21.5bp,2.5bp) [] {$G^{16,16,26}_{0}$};
  \node (G^{24+26+48}_{1}) at (112.5bp,43.5bp) [] {$G^{24,26,48}_{1}$};
  \draw [->] (G^{24+26+48}_{1}) -- (G^{24+24+26}_{0});
  \draw [->] (G^{26+48+48}_{0}) -- (G^{24+26+48}_{1});
  \draw [->] (G^{26+48+48}_{0}) -- (G^{24+26+48}_{1});
  \draw [->] (G^{16+26+48}_{1}) -- (G^{16+24+26}_{0});
  \draw [->] (G^{16+26+48}_{1}) -- (G^{16+16+26}_{0});
  \draw [->] (G^{24+26+48}_{1}) -- (G^{16+24+26}_{0});
  \draw [->] (G^{26+48+48}_{0}) -- (G^{16+26+48}_{1});
  \draw [->] (G^{26+48+48}_{0}) -- (G^{16+26+48}_{1});
\end{tikzpicture}

\begin{tikzpicture}[>=latex,line join=bevel,]
\node (G^{14+18+18}_{0}) at (74.5bp,2.5bp) [] {$G^{14,18,18}_{0}$};
  \node (G^{14+54+54}_{8}) at (21.5bp,84.5bp) [] {$G^{14,54,54}_{8}$};
  \node (G^{14+54+54}_{0}) at (82.5bp,84.5bp) [] {$G^{14,54,54}_{0}$};
  \node (G^{14+54+54}_{2}) at (143.5bp,84.5bp) [] {$G^{14,54,54}_{2}$};
  \node (G^{14+18+54}_{2}) at (112.5bp,43.5bp) [] {$G^{14,18,54}_{2}$};
  \node (G^{14+18+54}_{0}) at (36.5bp,43.5bp) [] {$G^{14,18,54}_{0}$};
  \draw [->] (G^{14+54+54}_{8}) -- (G^{14+18+54}_{0});
  \draw [->] (G^{14+54+54}_{8}) -- (G^{14+18+54}_{0});
  \draw [->] (G^{14+54+54}_{2}) -- (G^{14+18+54}_{2});
  \draw [->] (G^{14+54+54}_{2}) -- (G^{14+18+54}_{2});
  \draw [->] (G^{14+54+54}_{0}) -- (G^{14+18+54}_{0});
  \draw [->] (G^{14+18+54}_{0}) -- (G^{14+18+18}_{0});
  \draw [->] (G^{14+18+54}_{2}) -- (G^{14+18+18}_{0});
  \draw [->] (G^{14+54+54}_{0}) -- (G^{14+18+54}_{2});
\end{tikzpicture}
\begin{tikzpicture}[>=latex,line join=bevel,]
\node (G^{18+26+54}_{0}) at (36.5bp,43.5bp) [] {$G^{18,26,54}_{0}$};
  \node (G^{18+18+26}_{0}) at (74.5bp,2.5bp) [] {$G^{18,18,26}_{0}$};
  \node (G^{26+54+54}_{2}) at (21.5bp,84.5bp) [] {$G^{26,54,54}_{2}$};
  \node (G^{26+54+54}_{8}) at (143.5bp,84.5bp) [] {$G^{26,54,54}_{8}$};
  \node (G^{26+54+54}_{0}) at (82.5bp,84.5bp) [] {$G^{26,54,54}_{0}$};
  \node (G^{18+26+54}_{2}) at (112.5bp,43.5bp) [] {$G^{18,26,54}_{2}$};
  \draw [->] (G^{26+54+54}_{8}) -- (G^{18+26+54}_{2});
  \draw [->] (G^{26+54+54}_{8}) -- (G^{18+26+54}_{2});
  \draw [->] (G^{26+54+54}_{2}) -- (G^{18+26+54}_{0});
  \draw [->] (G^{26+54+54}_{2}) -- (G^{18+26+54}_{0});
  \draw [->] (G^{18+26+54}_{2}) -- (G^{18+18+26}_{0});
  \draw [->] (G^{18+26+54}_{0}) -- (G^{18+18+26}_{0});
  \draw [->] (G^{26+54+54}_{0}) -- (G^{18+26+54}_{2});
  \draw [->] (G^{26+54+54}_{0}) -- (G^{18+26+54}_{0});
\end{tikzpicture}

\begin{tikzpicture}[>=latex,line join=bevel,]
\node (G^{14+14+16}_{1}) at (82.5bp,2.5bp) [] {$G^{14,14,16}_{1}$};
  \node (G^{14+14+48}_{0}) at (51.5bp,43.5bp) [] {$G^{14,14,48}_{0}$};
  \node (G^{14+14+24}_{1}) at (21.5bp,2.5bp) [] {$G^{14,14,24}_{1}$};
  \draw [->] (G^{14+14+48}_{0}) -- (G^{14+14+16}_{1});
  \draw [->] (G^{14+14+48}_{0}) -- (G^{14+14+24}_{1});
\end{tikzpicture}

\begin{tikzpicture}[>=latex,line join=bevel,]
\node (G^{14+14+24}_{0}) at (21.5bp,2.5bp) [] {$G^{14,14,24}_{0}$};
  \node (G^{14+14+16}_{0}) at (82.5bp,2.5bp) [] {$G^{14,14,16}_{0}$};
  \node (G^{14+14+48}_{1}) at (51.5bp,43.5bp) [] {$G^{14,14,48}_{1}$};
  \draw [->] (G^{14+14+48}_{1}) -- (G^{14+14+16}_{0});
  \draw [->] (G^{14+14+48}_{1}) -- (G^{14+14+24}_{0});
\end{tikzpicture}

\begin{tikzpicture}[>=latex,line join=bevel,]
\node (G^{14+14+48}_{4}) at (51.5bp,43.5bp) [] {$G^{14,14,48}_{4}$};
  \node (G^{14+14+24}_{5}) at (21.5bp,2.5bp) [] {$G^{14,14,24}_{5}$};
  \node (G^{14+14+16}_{5}) at (82.5bp,2.5bp) [] {$G^{14,14,16}_{5}$};
  \draw [->] (G^{14+14+48}_{4}) -- (G^{14+14+16}_{5});
  \draw [->] (G^{14+14+48}_{4}) -- (G^{14+14+24}_{5});
\end{tikzpicture}

\begin{tikzpicture}[>=latex,line join=bevel,]
\node (G^{14+14+48}_{5}) at (51.5bp,43.5bp) [] {$G^{14,14,48}_{5}$};
  \node (G^{14+14+24}_{4}) at (21.5bp,2.5bp) [] {$G^{14,14,24}_{4}$};
  \node (G^{14+14+16}_{4}) at (82.5bp,2.5bp) [] {$G^{14,14,16}_{4}$};
  \draw [->] (G^{14+14+48}_{5}) -- (G^{14+14+16}_{4});
  \draw [->] (G^{14+14+48}_{5}) -- (G^{14+14+24}_{4});
\end{tikzpicture}

\begin{tikzpicture}[>=latex,line join=bevel,]
\node (G^{14+26+48}_{0}) at (51.5bp,43.5bp) [] {$G^{14,26,48}_{0}$};
  \node (G^{14+24+26}_{3}) at (21.5bp,2.5bp) [] {$G^{14,24,26}_{3}$};
  \node (G^{14+16+26}_{3}) at (82.5bp,2.5bp) [] {$G^{14,16,26}_{3}$};
  \draw [->] (G^{14+26+48}_{0}) -- (G^{14+24+26}_{3});
  \draw [->] (G^{14+26+48}_{0}) -- (G^{14+16+26}_{3});
\end{tikzpicture}

\begin{tikzpicture}[>=latex,line join=bevel,]
\node (G^{14+16+26}_{7}) at (82.5bp,2.5bp) [] {$G^{14,16,26}_{7}$};
  \node (G^{14+26+48}_{1}) at (51.5bp,43.5bp) [] {$G^{14,26,48}_{1}$};
  \node (G^{14+24+26}_{7}) at (21.5bp,2.5bp) [] {$G^{14,24,26}_{7}$};
  \draw [->] (G^{14+26+48}_{1}) -- (G^{14+24+26}_{7});
  \draw [->] (G^{14+26+48}_{1}) -- (G^{14+16+26}_{7});
\end{tikzpicture}

\begin{tikzpicture}[>=latex,line join=bevel,]
\node (G^{14+16+26}_{0}) at (82.5bp,2.5bp) [] {$G^{14,16,26}_{0}$};
  \node (G^{14+26+48}_{4}) at (51.5bp,43.5bp) [] {$G^{14,26,48}_{4}$};
  \node (G^{14+24+26}_{0}) at (21.5bp,2.5bp) [] {$G^{14,24,26}_{0}$};
  \draw [->] (G^{14+26+48}_{4}) -- (G^{14+16+26}_{0});
  \draw [->] (G^{14+26+48}_{4}) -- (G^{14+24+26}_{0});
\end{tikzpicture}

\begin{tikzpicture}[>=latex,line join=bevel,]
\node (G^{14+24+26}_{1}) at (21.5bp,2.5bp) [] {$G^{14,24,26}_{1}$};
  \node (G^{14+16+26}_{1}) at (82.5bp,2.5bp) [] {$G^{14,16,26}_{1}$};
  \node (G^{14+26+48}_{5}) at (51.5bp,43.5bp) [] {$G^{14,26,48}_{5}$};
  \draw [->] (G^{14+26+48}_{5}) -- (G^{14+16+26}_{1});
  \draw [->] (G^{14+26+48}_{5}) -- (G^{14+24+26}_{1});
\end{tikzpicture}

\begin{tikzpicture}[>=latex,line join=bevel,]
\node (G^{26+26+48}_{0}) at (51.5bp,43.5bp) [] {$G^{26,26,48}_{0}$};
  \node (G^{16+26+26}_{3}) at (82.5bp,2.5bp) [] {$G^{16,26,26}_{3}$};
  \node (G^{24+26+26}_{3}) at (21.5bp,2.5bp) [] {$G^{24,26,26}_{3}$};
  \draw [->] (G^{26+26+48}_{0}) -- (G^{24+26+26}_{3});
  \draw [->] (G^{26+26+48}_{0}) -- (G^{16+26+26}_{3});
\end{tikzpicture}

\begin{tikzpicture}[>=latex,line join=bevel,]
\node (G^{16+26+26}_{0}) at (82.5bp,2.5bp) [] {$G^{16,26,26}_{0}$};
  \node (G^{24+26+26}_{0}) at (21.5bp,2.5bp) [] {$G^{24,26,26}_{0}$};
  \node (G^{26+26+48}_{1}) at (51.5bp,43.5bp) [] {$G^{26,26,48}_{1}$};
  \draw [->] (G^{26+26+48}_{1}) -- (G^{24+26+26}_{0});
  \draw [->] (G^{26+26+48}_{1}) -- (G^{16+26+26}_{0});
\end{tikzpicture}

\begin{tikzpicture}[>=latex,line join=bevel,]
\node (G^{24+26+26}_{5}) at (21.5bp,2.5bp) [] {$G^{24,26,26}_{5}$};
  \node (G^{16+26+26}_{5}) at (82.5bp,2.5bp) [] {$G^{16,26,26}_{5}$};
  \node (G^{26+26+48}_{4}) at (51.5bp,43.5bp) [] {$G^{26,26,48}_{4}$};
  \draw [->] (G^{26+26+48}_{4}) -- (G^{16+26+26}_{5});
  \draw [->] (G^{26+26+48}_{4}) -- (G^{24+26+26}_{5});
\end{tikzpicture}

\begin{tikzpicture}[>=latex,line join=bevel,]
\node (G^{16+26+26}_{1}) at (82.5bp,2.5bp) [] {$G^{16,26,26}_{1}$};
  \node (G^{26+26+48}_{5}) at (51.5bp,43.5bp) [] {$G^{26,26,48}_{5}$};
  \node (G^{24+26+26}_{1}) at (21.5bp,2.5bp) [] {$G^{24,26,26}_{1}$};
  \draw [->] (G^{26+26+48}_{5}) -- (G^{16+26+26}_{1});
  \draw [->] (G^{26+26+48}_{5}) -- (G^{24+26+26}_{1});
\end{tikzpicture}

\begin{tikzpicture}[>=latex,line join=bevel,]
\node (G^{14+16+40}_{0}) at (82.5bp,2.5bp) [] {$G^{14,16,40}_{0}$};
  \node (G^{14+40+48}_{0}) at (51.5bp,43.5bp) [] {$G^{14,40,48}_{0}$};
  \node (G^{14+24+40}_{0}) at (21.5bp,2.5bp) [] {$G^{14,24,40}_{0}$};
  \draw [->] (G^{14+40+48}_{0}) -- (G^{14+16+40}_{0});
  \draw [->] (G^{14+40+48}_{0}) -- (G^{14+24+40}_{0});
\end{tikzpicture}

\begin{tikzpicture}[>=latex,line join=bevel,]
\node (G^{24+26+40}_{0}) at (21.5bp,2.5bp) [] {$G^{24,26,40}_{0}$};
  \node (G^{26+40+48}_{0}) at (51.5bp,43.5bp) [] {$G^{26,40,48}_{0}$};
  \node (G^{16+26+40}_{0}) at (82.5bp,2.5bp) [] {$G^{16,26,40}_{0}$};
  \draw [->] (G^{26+40+48}_{0}) -- (G^{24+26+40}_{0});
  \draw [->] (G^{26+40+48}_{0}) -- (G^{16+26+40}_{0});
\end{tikzpicture}

\begin{tikzpicture}[>=latex,line join=bevel,]
\node (G^{14+18+26}_{3}) at (51.5bp,2.5bp) [] {$G^{14,18,26}_{3}$};
  \node (G^{14+26+54}_{0}) at (21.5bp,43.5bp) [] {$G^{14,26,54}_{0}$};
  \node (G^{14+26+54}_{2}) at (82.5bp,43.5bp) [] {$G^{14,26,54}_{2}$};
  \draw [->] (G^{14+26+54}_{2}) -- (G^{14+18+26}_{3});
  \draw [->] (G^{14+26+54}_{0}) -- (G^{14+18+26}_{3});
\end{tikzpicture}

\begin{tikzpicture}[>=latex,line join=bevel,]
\node (G^{14+26+54}_{4}) at (21.5bp,43.5bp) [] {$G^{14,26,54}_{4}$};
  \node (G^{14+26+54}_{6}) at (82.5bp,43.5bp) [] {$G^{14,26,54}_{6}$};
  \node (G^{14+18+26}_{0}) at (51.5bp,2.5bp) [] {$G^{14,18,26}_{0}$};
  \draw [->] (G^{14+26+54}_{4}) -- (G^{14+18+26}_{0});
  \draw [->] (G^{14+26+54}_{6}) -- (G^{14+18+26}_{0});
\end{tikzpicture}

\begin{tikzpicture}[>=latex,line join=bevel,]
\node (G^{14+18+40}_{0}) at (51.5bp,2.5bp) [] {$G^{14,18,40}_{0}$};
  \node (G^{14+40+54}_{0}) at (21.5bp,43.5bp) [] {$G^{14,40,54}_{0}$};
  \node (G^{14+40+54}_{2}) at (82.5bp,43.5bp) [] {$G^{14,40,54}_{2}$};
  \draw [->] (G^{14+40+54}_{0}) -- (G^{14+18+40}_{0});
  \draw [->] (G^{14+40+54}_{2}) -- (G^{14+18+40}_{0});
\end{tikzpicture}

\begin{tikzpicture}[>=latex,line join=bevel,]
\node (G^{26+40+54}_{0}) at (21.5bp,43.5bp) [] {$G^{26,40,54}_{0}$};
  \node (G^{18+26+40}_{0}) at (51.5bp,2.5bp) [] {$G^{18,26,40}_{0}$};
  \node (G^{26+40+54}_{2}) at (82.5bp,43.5bp) [] {$G^{26,40,54}_{2}$};
  \draw [->] (G^{26+40+54}_{0}) -- (G^{18+26+40}_{0});
  \draw [->] (G^{26+40+54}_{2}) -- (G^{18+26+40}_{0});
\end{tikzpicture}

\begin{tikzpicture}[>=latex,line join=bevel,]
\node (G^{14+14+54}_{0}) at (21.5bp,43.5bp) [] {$G^{14,14,54}_{0}$};
  \node (G^{14+14+18}_{0}) at (21.5bp,2.5bp) [] {$G^{14,14,18}_{0}$};
  \draw [->] (G^{14+14+54}_{0}) -- (G^{14+14+18}_{0});
\end{tikzpicture}

\begin{tikzpicture}[>=latex,line join=bevel,]
\node (G^{14+14+18}_{4}) at (21.5bp,2.5bp) [] {$G^{14,14,18}_{4}$};
  \node (G^{14+14+54}_{4}) at (21.5bp,43.5bp) [] {$G^{14,14,54}_{4}$};
  \draw [->] (G^{14+14+54}_{4}) -- (G^{14+14+18}_{4});
\end{tikzpicture}

\begin{tikzpicture}[>=latex,line join=bevel,]
\node (G^{26+26+54}_{0}) at (21.5bp,43.5bp) [] {$G^{26,26,54}_{0}$};
  \node (G^{18+26+26}_{0}) at (21.5bp,2.5bp) [] {$G^{18,26,26}_{0}$};
  \draw [->] (G^{26+26+54}_{0}) -- (G^{18+26+26}_{0});
\end{tikzpicture}

\begin{tikzpicture}[>=latex,line join=bevel,]
\node (G^{26+26+54}_{4}) at (21.5bp,43.5bp) [] {$G^{26,26,54}_{4}$};
  \node (G^{18+26+26}_{1}) at (21.5bp,2.5bp) [] {$G^{18,26,26}_{1}$};
  \draw [->] (G^{26+26+54}_{4}) -- (G^{18+26+26}_{1});
\end{tikzpicture}

\bibliographystyle{amsalpha}
\bibliography{biblio}

\end{document}